\documentclass[11pt, a4paper]{article} 
\title{\bf Well-posedness theory for \\ non-homogeneous incompressible fluids with odd viscosity}
\author{
\tsl{Francesco Fanelli}$\,^{1}\;$, $\qquad$\textsl{Rafael Granero-Belinch\'on}$\,^{2}\;$,
$\qquad$\textsl{Stefano Scrobogna}$\,^{3}$ 
\vspace{.5cm} \\
\footnotesize{$\,^1\;$ \textsc{Univ. Lyon, Universit\'e Claude Bernard Lyon 1}, CNRS UMR 5208, \textit{Institut Camille Jordan},} \\
\footnotesize{43 blvd. du 11 novembre 1918, F-69622 Villeurbanne cedex, FRANCE} \\
\footnotesize{Email address: \texttt{fanelli@math.univ-lyon1.fr}} \vspace{0.2cm} \\
\footnotesize{$\,^2\;$ \textsc{Universidad de Cantabria}, \textit{Departamento de Matem\'aticas, Estad\'isticas y Computaci\'on},} \\
\footnotesize{Avda. Los Castros s/n, Santander, SPAIN} \\
\footnotesize{Email address: \texttt{rafael.granero@unican.es}} \vspace{0.2cm} \\
\footnotesize{$\,^3\;$ \textsc{Universit\`a degli Studi di Trieste}, \textit{Dipartimento di Matematica e Geoscienze},} \\
\footnotesize{Via Valerio 12/1, 34127 Trieste, ITALY} \\
\footnotesize{Email address: \texttt{stefano.scrobogna@units.it}}
\vspace{.1cm}
}

\usepackage[a4paper]{geometry}
\usepackage{empheq}
\usepackage{times}
\usepackage{stmaryrd}
\usepackage{fourier}
\usepackage[T1]{fontenc}
\usepackage{amssymb, amsmath,  bm, mathrsfs}
\usepackage{amsfonts}
\usepackage[english]{babel}
\usepackage{amssymb,amsthm}
\usepackage{mathtools}
\DeclareMathAlphabet{\mathcal}{OMS}{cmsy}{m}{n}
\DeclareFontFamily{U}{mathc}{}
\DeclareFontShape{U}{mathc}{m}{it}%
{<->x*[1.03] mathc10}{}
\DeclareMathAlphabet{\mathscr}{U}{mathc}{m}{it}

\DeclareMathAlphabet{\mathpzc}{OT1}{pzc}{m}{it}

\usepackage{hyperref}
\usepackage{cleveref}
\usepackage{cite}
\allowdisplaybreaks[1]
\usepackage{xfrac}
\usepackage[utf8]{inputenc}
\usepackage[T1]{fontenc}
\usepackage{enumerate}
\usepackage{accents}
\usepackage{bbm}

\usepackage{tikz}
\usetikzlibrary{arrows.meta}
\usetikzlibrary{bending}
\usetikzlibrary{arrows, automata, backgrounds, shadows, patterns, calc, hobby, plotmarks, shapes}
\usetikzlibrary{shapes.misc}

\tikzset{cross/.style={cross out, draw=black, minimum size=2*(#1-\pgflinewidth), inner sep=0pt, outer sep=0pt},
cross/.default={1pt}}
\allowdisplaybreaks

\newcommand{\dd}{\textnormal{d}}
\renewcommand{\div}{\textnormal{div}}

\newcommand{\curl}{\textnormal{curl}}
\newcommand{\pare}[1]{\left( #1 \right)}

\newcommand{\av}[1]{\left| #1 \right|}
\newcommand{\bra}[1]{\left[ #1 \right]}

\newcommand{\system}[1]{\left\{ #1 \right.}

\newcommand{\supp}{\textnormal{supp}}

\newcommand{\RN}[1]{%
  \textup{\uppercase\expandafter{\romannumeral#1}}%
}

\newcommand{\bR}{\mathbb{R}}

\newcommand{\cG}{\mathcal{G}}

\newcommand{\loc}{\textnormal{loc}}

\theoremstyle{theorem}
\newtheorem{theorem}{Theorem}[section]

\newtheorem*{theorem*}{Theorem}
\newtheorem{prop}[theorem]{Proposition}
\newtheorem{lemma}[theorem]{Lemma}
\newtheorem{cor}[theorem]{Corollary}

\theoremstyle{definition}
\newtheorem{definition}[theorem]{Definition}

\newtheorem{rem}[theorem]{Remark}

\makeatletter
\@addtoreset{proofpart}{theorem}
\makeatother

\makeatletter
\@addtoreset{proofsubpart}{part}
\makeatother

\usepackage{comment}

\newcommand{\tbf}{\textbf}

\newcommand{\tsl}{\textsl}

\newcommand{\mbb}{\mathbb}

\newcommand{\mc}{\mathcal}

\newcommand{\veps}{\varepsilon}

\newcommand{\what}{\widehat}
\newcommand{\wtilde}{\widetilde}
\newcommand{\vphi}{\varphi}

\newcommand{\ra}{\rightarrow}

\newcommand{\g}{\gamma}
\renewcommand{\k}{\kappa}
\newcommand{\s}{\sigma}
\renewcommand{\t}{\tau}
\newcommand{\z}{\zeta}
\newcommand{\lam}{\lambda}
\newcommand{\de}{\delta}
\renewcommand{\o}{\omega}
\newcommand{\wnu}{\wtilde \nu}

\newcommand{\R}{\mathbb{R}}
\newcommand{\Q}{\mathbb{Q}}

\newcommand{\N}{\mathbb{N}}
\newcommand{\Z}{\mathbb{Z}}
\newcommand{\T}{\mathbb{T}}
\renewcommand{\P}{\mathbb{P}}
\newcommand{\DD}{\mathbb{D}}
\newcommand{\B}{\mathbb{B}}
\newcommand{\V}{\mathbb{V}}

\renewcommand{\div}{{\rm div}\,}

\newcommand{\Id}{{\rm Id}\,}

\newcommand{\dt}{ \, {\rm d} t}

\allowdisplaybreaks

\def\d{\partial}
\def\div{{\rm div}\,}

\numberwithin{equation}{section}

\begin{document}
\maketitle

\subsubsection*{Abstract}
{\footnotesize

Several fluid systems are characterised by time reversal and parity breaking. Examples of such phenomena arise both in quantum and classical hydrodynamics.
In these situations, the viscosity tensor, often dubbed ``odd viscosity'', becomes non-dissipative. At the mathematical level,
this fact translates into a loss of derivatives at the level of \tsl{a priori} estimates: while the odd viscosity term depends on derivatives
of the velocity field, no parabolic smoothing effect can be expected.

In the present paper, we establish a well-posedness theory in Sobolev spaces for a system of incompressible non-homogeneous fluids with odd viscosity.
The crucial point of the analysis is the introduction of a set of \emph{good unknowns}, which allow for the emerging of a hidden
hyperbolic structure underlying the system of equations. It is exactly this hyperbolic structure which makes it possible to circumvent
the derivative loss and propagate high enough Sobolev norms of the solution. The well-posedness result is local in time; two continuation criteria
are also established.

}

\paragraph*{\small 2020 Mathematics Subject Classification:}{\footnotesize
35Q35 
\ (primary);
\ 35B65, 
\ 76B03, 
\ 35B45, 
\ 76D09 
\ (secondary).}

\paragraph*{\small Keywords: }{\footnotesize incompressible fluids; odd viscosity; density variations; hidden hyperbolicity; local well-posedness; continuation
criterion; transport and transport-diffusion equations.}

\tableofcontents

\section{Introduction} \label{s:intro}

The goal of this paper is to set up a rigorous well-posedness theory for a system of PDEs describing the motion of non-homogeneous
incompressible fluids which display \emph{non-dissipative viscosity effects} in the dynamics.

Let us start by giving some background on the physics of the problem.

\subsection{Viscosity: a physical perspective} \label{ss:visc_phys}

Viscosity of a fluid is the property of that fluid of resisting (by exerting stresses) to deformations at a given rate.
Such deformations are imposed by velocity gradients,
as they come essentially from internal friction forces between molecules.

Let us denote by $\T$ the stress tensor, by $u$ the velocity field of the fluid and by $\DD\,=\,\big(d_{ij}\big)_{i,j}$ the strain-rate tensor;
recall that $\DD\,=\,\big(Du+\nabla u\big)/2$, where $Du$ is the Jacobian matrix of $u$ and $\nabla u\,=\,^tDu$ its transpose. Assuming Stoke's
postulate (see \tsl{e.g.} \cite{Nov-Strab}) of linear dependence of $\T$ on $\DD$, the relation between those quantities
is given by the phenomenological law
\begin{equation} \label{i_eq:T_classic}
 \T_{ij}\,=\,\s_{ij}^h\,+\,\text{v}_{ij}\,,\qquad\qquad\qquad \mbox{ with }\qquad\qquad \text{v}_{ij}\,=\,\eta_{ijk\ell}\,d_{k\ell}\,,
\end{equation}
where $\s_{ij}^h$ is the hydrostatic stress (typically, in standard fluids one has $\s_{ij}^h\,=\,-\,\pi\,\de_{ij}$, with $\pi$ being the fluid pressure
and $\de_{ij}$ being the Kronecker delta), $\V\,=\,\big(\text{v}_{ij}\big)_{i,j}$ is the tensor of viscous stress and the rank-four
tensor $\eta_{ijk\ell}$ represents the viscosity tensor. We refer to \cite{avron1995viscosity} and to the Introduction of
\cite{K-S-F-V} for more insights about the previous equation.

For systems with microscopic time reversal symmetry, Onsager's reciprocity relations \cite{Onsager} (see also \cite{Land-Lif})
imply that the viscosity tensor must be symmetric, in the sense that
$\eta_{ijk\ell}\,=\,\eta_{k\ell ij}$. In this case, a term of the form $\sum\,\text{v}_{ij}\,\d_iu_j$ appears in the computation of the total energy budget,
so $\V$ contributes to the entropy production through dissipation of kinetic energy.

In sharp contrast to the previous (classical) situation, there exist classes of fluids possessing broken microscopic time-reversal symmetry and broken parity.
For instance, the symmetry breaking may be induced by the presence of microscopic Coriolis or Lorentz forces \cite{souslov2019topological}, but in fact also
other microscopic mechanisms (like \tsl{e.g.} internal torques) can contribute to it.
Examples of parity-breaking phenomena arise in a great variety of contexts, which encompass the realms of both quantum fluids (like \tsl{e.g.}
electron fluids, magnetised plasmas, chiral superfluids, quantum Hall fluids\ldots) and classical fluid systems (as in polyatomic gases, chiral active matter
and vortex dynamics in $2$-D, for instance).
We refer to \cite{souslov2019topological}, \cite{ganeshannon}, \cite{K-S-F-V} and references therein for additional details and examples; see also
the discussion in Paragraph \ref{ss:previous} below.

For fluid systems in which microscopic time reversal and parity are violated (more in general,
for fluids having parity-breaking intrinsic angular momentum of constituent particles \cite{A-C-G-M}),
the viscosity tensor presents a skew-symmetric component
\begin{equation} \label{i_eq:skew}
\eta^0_{ijk\ell}\,=\,-\,\eta^0_{k\ell ij}\,,
\end{equation}
often dubbed \emph{odd viscosity} (or Hall viscosity).
Under \eqref{i_eq:skew}, and assuming the classical form of the hydrostatic stress tensor $\s^h$, the stress tensor $\T$ takes the form\footnote{The precise
formula expressing $\T$ in terms of $\DD$ involves Pauli matrices; it can be found \tsl{e.g.} in
\cite{Avron} (see relation (12) in that paper) and \cite{lapa2014swimming} (see formula (3.7) therein).}
\begin{equation} \label{i_eq:T_odd}
\T\,=\,-\,\pi\,\Id\,+\,\nu^{\rm odd}\,\left(\nabla u^\perp\,+\,\nabla^\perp u\right)\,,
\end{equation}
where $\nu^{\rm odd}$ is the odd viscosity coefficient and where, for any $2$-D vector $a\,=\,\big(a_1,a_2\big)\in\R^2$, we have denoted by
$a^\perp\,:=\,\big(-a_2,a_1\big)$ its rotation of angle $\pi/2$\footnote{According to our notation $\nabla a =\,^tDa$, the symbol
$\nabla^\perp a$ is the transponse of the ``Jacobian matrix'' obtained using the operator $\nabla^\perp$ instead of $\nabla$. Hence, $\nabla^\perp a$
is the $2\times 2$ matrix having, as $j$-th vector column, the vector $\nabla^\perp a_j$. Notice that $\nabla^\perp a\,\equiv\,^t\big(\nabla a^\perp\big)$.}.
Notice that $\nu^{\rm odd}$ needs not be constant, and in fact in general it may depend on
the density $\rho$ of the fluid. We also point out that here we have adopted a different convention with respect to physical papers (see more references below),
where the operator $a^\perp$ is replaced by $a^*\,=\,-\,a^\perp$; in the end, this simply reduces to a change of the sign of 
$\nu^{\rm odd}$, however (as we will see) the sign of $\nu^{\rm odd}$ does not play any role in our study.

Thus, the skew-symmetry \eqref{i_eq:skew} of the tensor $\eta^0$ imposes a non-standard relation between the stress tensor $\T$ and the velocity grandients $Du$
(or, better, the strain-rate tensor $\DD$). As it appears clear from the form of \eqref{i_eq:T_odd}, the most striking consequence is that
$\eta^0$ does \emph{not} contribute to energy dissipation in the fluid motion, so nor to entropy production. Hence,
from a different but related standpoint, fluids with odd viscosity may be defined as a special class of non-Newtonian fluids,
for which the response to internal friction forces is still linear on the velocity gradients, but at the same time non-dissipative.

However, there are also other rather counterintuitive effects which follow from \eqref{i_eq:skew}. As a prototypical example, one may notice that
a rotating disk immersed in a fluid with odd viscosity feels a pressure in the radial direction. In addition, the direction of the pressure force
(either readially inward or outward) depends on the direction of rotation of the disk, a fact which reveals the
broken time-reversal symmetry of the system. We refer to \tsl{e.g.} \cite{lapa2014swimming} and \cite{ganeshan2017odd} for more details and explanations.



\subsection{Previous studies on fluids with odd viscosity} \label{ss:previous}

In the three-dimensional case, examples of fluids with broken parity have been known for a long time
in the context of plasmas in a magnetic field  (see \tsl{e.g.} \cite{Land-Lif_Fluids}) and of hydrodynamics theories of superfluids.
However, in those situations the violation of symmetry is strongly determined by the anisotropy of the fluid.

In the seminal work \cite{avron1995viscosity}, Avron, Seiler and Zograf revealed that the viscosity of two-dimensional quantum Hall fluids occupying a torus 
has a non-dissipative nature. Soon after that, Avron \cite{Avron} highlighted the very special role played by the planar situation:
in $2$-D, the parity breaking, hence an odd viscous response $\eta^0$ of the stress tensor, is compatible with the fluid isotropy.
In the same paper \cite{Avron}, Avron discussed as well basic examples of effects due to odd viscosity,
like the radial pressure on a rotating cylinder and the presence of chiral viscosity waves with quadratic dispersion relation in compressible flows.

Work \cite{Avron} initiated the search for broken parity phenomena and consequences also in classical $2$-D hydrodynamics.
Since then, the physics literature about fluid systems with non-dissipative viscosity effects has been significantly increasing,
especially in the very last years, and has concerned both experimental and theoretical aspects of the problem.
Let us make a brief (and absolutely non-exhaustive) overview.
For instance, in \cite{ganeshan2017odd} it was shown that odd viscosity has very subtle effects in the case of incompressible fluids,
which sensitively depend on the type of boundary conditions imposed. For homogeneous incompressible flows,
a lot of efforts have been devoted to understand free-surface problems and propagation of surface waves: see for instance
\cite{abanov2018odd}, \cite{abanov2019free}, \cite{bililign2022motile}, \cite{doak2022nonlinear}, \cite{ganeshannon}, \cite{soni2019odd} and the references therein.
In \cite{lapa2014swimming}, the authors investigated the theory of swimmers at low Reynolds number in fluids presenting a non-vanishing odd viscosity term,
while \cite{K-S-F-V} dealt with linearised equations and studied the effects of odd viscosity on Stokeslet solutions.
Interestingly, paper \cite{wiegmann2014anomalous} studied the motion of a dense system of vortices and explained its connection with turbulent phenomena.
In the case of compressible flows, odd viscosity effects are probably even more prominent. In this context, we mention works 
\cite{souslov2019topological}, which dealt with propagation of sound waves in a fluid with broken parity, and \cite{banerjee2017odd},
which investigated the dynamics of a chiral active fluid with a special emphasis on the combined effects of compressibility and non-linearities.
Finally, we mention that a first-principles Hamiltonian theory for fluids with odd viscosity was developed in \cite{markovich2021odd};
we refer also to \cite{A-C-G-M} for further investigations on the Hamiltonian structure of the system.

\medbreak
Despite the increasing interest of physical studies for fluid models dominated by odd viscosity effects, one may regret that, apart from \cite{G-O} (devoted to
the free-surface problem for homogeneous incompressible flows), not so many mathematical works have considered this setting.

As already announced at the beginning, our goal is to perform a rigorous mathematical analysis for a system of non-homogeneous incompressible fluids
with odd viscosity. 
For doing so, we resort to the system of equations proposed in \cite{A-C-G-M} by Abanov, Can, Ganesham and Monteiro.

The model considered in \cite{A-C-G-M} describes a planar compressible fluid whose motion is governed by the usual principles of conservations of mass
and of linear momentum, with the exception that the stress tensor $\T$ appearing in the latter relation is not of classical type \eqref{i_eq:T_classic},
but is rather given by \eqref{i_eq:T_odd}. In addition, in \eqref{i_eq:T_odd} the authors made the choice
\[
\nu^{\rm odd}\,=\,\nu^{\rm odd}(\rho)\,=\,\nu_0\,\rho\,,
\]
where $\rho$ describes the fluid density and $\nu_0\in\R$, with $\nu_0\neq0$, is the kinematic odd viscosity coefficient.
The authors of \cite{A-C-G-M} studied several questions linked with the free-surface problem for a compressible broken parity fluid, with
the incompressible limit and with the formation of boundary layers in the propagation of surface waves.

\subsection{The mathematical model and main result} \label{ss:model-result}

In this paper, we consider essentially the incompressible (but still non-homogeneous) counterpart of the model appearing in \cite{A-C-G-M},
where however we formulate some further assumptions and reductions. Before presenting the precise equations, let us summarise
and comment them.

\medbreak
\noindent \tbf{Domain.} ~
In this paper, we consider domains with \emph{no boundaries}, in particular we do not consider the free-boundary problem.
We focus our attention on fluids filling the whole space $\R^2$, although the same analysis would apply (with minor modifications) also
to periodic fluids defined on the flat torus $\T^2$.

\medbreak
\noindent \tbf{Incompressibility.} ~
As already pointed out, we consider incompressible fluids, for which the velocity field $u$ satisfies the divergence-free condition $\nabla\cdot u=0$.
Correspondingly, the pressure term $\nabla\pi$ appearing in the momentum equation takes the role of a Lagrangian multiplier associated to that constraint
and becomes an additional unknown of the problem.

In our work, the incompressibility condition will play an important role, as it will entail fundamental algebraic cancellations in the computations.
We will give more details below about this issue.

\medbreak
\noindent \tbf{Vacuum.} ~
In our study, it is also important to exclude the possible presence of vacuum. Thus, we assume the initial density $\rho_0$ to be bounded away from the $0$ value.
This assumption is propagated by the flow, owing to the conservation of mass and the incompressibility constraint over $u$.

\medbreak
\noindent \tbf{Odd viscosity.} ~
Another important reduction consists in taking a simplified version of the stress tensor \eqref{i_eq:T_odd}. More precisely,
we neglect the last term appearing in that formula and consider the tensor
\[
\T\,=\,-\,\pi\,\Id\,+\,\nu_0\,\rho\,\nabla u^\perp\,.
\]
As $\nabla\cdot\nabla^\perp a\equiv0$, at a first glance it may seem that we have just forgotten a lower order term, whose presence would involve no significant
complications in the analysis. Actually, and quite surprisingly, a more accurate inspection reveals that this is not the case:
the presence of the term $\nabla^\perp u$ would
destroy an underlying hyperbolic structure of the equations (see more details below) and make our approach inconclusive.

On the other hand, $\T$ being no more symmetric, we have to make clear the meaning of the notation $\nabla\cdot \T$ appearing in the equations.
As usual, this writing
for us means the contraction along columns of $\T$. 
Thus, $\nabla\cdot\T$ is the $2$-D vector whose component $j$ (for $j=1,2$) is
\[
 \big(\nabla\cdot \T\big)_j\,=\,\sum_i \d_i\T_{ij}\,.
\]
Recall that, for a vector field $a$ over $\R^2$, for us $\nabla a\,=\,^tDa$ is the transpose of the Jacobian matrix of $a$.

\subsubsection{The system of equations}  \label{sss:equations}
In light of the previous considerations, the system of equations we consider in this paper takes the form
\begin{equation}\label{eq:fluid_odd_visco1}
\system{
\begin{aligned}
& \d_t\rho\, +\, \nabla\cdot \big(\rho\,u\big)\, =\,0 \\
& \partial_t \pare{\rho\, u}\, +\, \nabla\cdot \pare{\rho\, u \otimes u}\, +\, \nabla \pi\,  +\, \nu_0\, \nabla\cdot \pare{\rho \ \nabla u^\perp}\, =\, 0 \\
& \nabla\cdot u\, =\,0\,,
\end{aligned}
}
\end{equation}
where, as above, $\rho=\rho(t,x)\geq0$ represents the density of the fluid, $u=u(t,x)\in\R^2$ its velocity field and $\pi=\pi(t,x)\in\R$ its pressure field.
The kinematic odd viscosity coefficient $\nu_0\in\R\setminus\{0\}$ will be assumed to be constant; as nor its precise value, nor its sign are important in our study,
we immediately set $\nu_0 = 1$.

System \eqref{eq:fluid_odd_visco1} is defined for
\[
(t,x)\in\R_+\times\R^2\,. 
\]
The first equation appearing therein represents the law of conservation of mass, whereas the second equation represents the law of conservation of linear momentum;
as already mentioned, the third equation encodes the incompressibility property of the fluid.

In this work we will study the well-posedness of system \eqref{eq:fluid_odd_visco1} in the class of
Sobolev spaces $H^s=H^s(\R^2)$. This framework is natural in our context,
because of the skew-symmetry of the odd viscosity term with respect to the $L^2$ scalar product.

\subsubsection{Loss of derivatives} \label{sss:loss}

Despite the apparent simplicity of formulation of the equations, establishing a well-posedness theory for system \eqref{eq:fluid_odd_visco1} does not
look so obvious. 
Recall that the odd viscosity term is non-dissipative, hence (differently from the classical case of viscous fluids) we cannot expect any gain of
regularity from it. On the contrary, a quick look at the momentum equation shows that the odd viscosity term is responsible for a loss of derivatives, which
occur at least at two levels.
Let us explain better this issue.

First of all, as we have no parabolic smoothing effect in the equations, it seems reasonable to resort to hyperbolic theory and to work in high regularity
Sobolev spaces $H^s$, with $s>2$. So, assume to dispose of a solution $\big(\rho,u,\nabla\pi\big)$ to \eqref{eq:fluid_odd_visco1},
with the velocity field $u\in H^s$. Then, classical results on transport equations allow to propagate $H^s$ regularity also for $\rho$.
Nonetheless, if we expand the odd viscosity term, we see that
\begin{equation} \label{i_eq:odd_1}
-\,\nabla\cdot\big(\rho\,\nabla u^\perp\big)\,=\,-\,\rho\,\Delta u^\perp\,-\,\big(\nabla\rho\cdot\nabla\big)u^\perp\,.
\end{equation}
While we may reasonably expect that the former term on the right will be killed by skew-symmetry (at least at higher order), the same cannot be said
about the latter term: we can only get $\big(\nabla\rho\cdot\nabla\big)u^\perp\,\in\,H^{s-1}$.

The second issue is linked with the regularity of the pressure term, and it looks even more dangerous than the previous one.
To understand it, let us consider the simple case of a homogeneous fluid,
for which $\rho\equiv1$. Then, the momentum equation becomes
\begin{equation} \label{eq:odd_homog}
\d_tu\,+\,\nabla\cdot\big(u\otimes u\big)\,+\,\nabla\pi\,+\,\Delta u^\perp\,=\,0\,.
\end{equation}
At this point, because of incompressibility, we see that the odd term is actually a gradient: one has  $\Delta u^\perp\,=\,\nabla\vphi$, for a suitable scalar
function $\vphi$. Then, the study of \eqref{eq:odd_homog} can be easily reconducted to the one of a classical incompressible Euler system
\[ 
 \d_tu\,+\,\nabla\cdot\big(u\otimes u\big)\,+\,\nabla\wtilde \pi\,=\,0\,,\qquad\qquad \mbox{ with }\qquad \wtilde \pi\,=\,\pi\,+\,\vphi\,.
\] 
However, the theory gives us that $\nabla\wtilde \pi$ has the same regularity as $u$, namely $\nabla\wtilde\pi\in H^s$, which in turn implies
that the original hydrodynamic pressure $\nabla\pi$ only belongs to $H^{s-2}$.
As, for non-homogeneous fluids (see \tsl{e.g.} \cite{D_2010} and \cite{D-F}), there is no way to avoid the presence of the pressure gradient
(either in the momentum equation or in the vorticity formulation) when propagating higher order norms, this loss of two derivatives for $\nabla\pi$
represents another major obstacle to the analysis of system \eqref{eq:fluid_odd_visco1}.

\subsubsection{Main result} \label{sss:result}

Despite the difficulties highlighted in the previous paragraph, we claim that system \eqref{eq:fluid_odd_visco1} is well-posed, locally in time,
in the class of Sobolev spaces of sufficiently high regularity.

The precise statement is contained in the next theorem, which represents the main result of the paper. It claims the local in time existence
and uniqueness of solutions emanating from smooth enough initial data, together with two continuation criteria, similar in spirit to the Beale-Kato-Majda
criterion \cite{B-K-M} for the classical incompressible Euler equations.

\begin{theorem}\label{teo1}
Let $s>2$. Take an initial datum $(\rho_0,u_0)\in L^\infty(\R^2)\times H^{s}(\R^2)$ and assume that
$$
\exists\,\rho_*\,>0\qquad \mbox{ such that }\qquad\qquad \rho_0(x)\geq \rho_*>0\,.
$$
In addition, assume that $\rho_0-1\in H^{s+1}(\R^2)$.

Then, there exists $0<T=T(\rho_0,u_0)\leq +\infty$ and a unique solution $(\rho,u,\nabla \pi)$ to system \eqref{eq:fluid_odd_visco1} on $[0,T]\times\R^2$ such that:
\begin{itemize}
 \item $\rho\in L^\infty\big([0,T]\times\R^2\big)$ verifies $\rho\geq\rho_* $ and $\rho-1\in C\big([0,T];H^{s+1}(\R^2)\big)$;
 \item $u$ belongs to $C\big([0,T];H^s(\R^2)\big)$;
 \item $\nabla \pi\in C\big([0,T];H^{s-2}(\R^2)\big)$.
\end{itemize}
In addition, let $T^*>0$ be such that $\big(\rho,u,\nabla\pi\big)$ is defined on $\left[0,T^*\right[\,$. 
If $T^*<+\infty$, assume either that one has
\begin{align*} 
&\int_0^{T^*}
\Big(\left\|\nabla u(t)\right\|^2_{L^\infty}\,+\,\left\|\nabla \rho(t)\right\|^s_{L^\infty}\,+\,
\left\|\nabla\rho(t)\right\|^{s-1}_{L^\infty}\,\left\|\nabla u(t)\right\|_{L^\infty}\,+\,\left\|\Delta \rho(t)\right\|_{L^\infty}\,+\,
\left\|\nabla\pi(t)\right\|^{s/(s-1)}_{L^\infty} 
\Big)\,\dt\,<\,+\,\infty
\end{align*}
or that, after defining $\o\,:=\,\nabla^\perp\cdot u\,=\,\d_1u_2-\d_2u_1$ the vorticity of the fluid, one has
\begin{align*} 
&\int_0^{T^*}
\Big(\left\|\nabla u(t)\right\|^2_{L^\infty}\,+\,\left\|\nabla \rho(t)\right\|^s_{L^\infty}\,+\,
\left\|\nabla\rho(t)\right\|^{\max\{2,s-1\}}_{L^\infty}\,\left\|\nabla u(t)\right\|_{L^\infty} \\
&\qquad\qquad\qquad\qquad\qquad\qquad\qquad\qquad\qquad\qquad
\,+\,\left\|\Delta \rho(t)\right\|_{L^\infty}\,+\,
\left\|\nabla\big(\pi(t)-\rho(t)\o(t)\big)\right\|^{s/(s-1)}_{L^\infty}
\Big)\,\dt\,<\,+\,\infty\,.
\end{align*}
Then $\big(\rho,u,\nabla\pi\big)$ can be continued beyond the time $T^*$ into a solution of system \eqref{eq:fluid_odd_visco1} with the same regularity.
\end{theorem}

Let us explain better the previous statement by formulating some remarks.
First of all, we want to comment on the $H^{s+1}$ regularity assumption on the initial density variations $\rho_0-1$.

\begin{rem} \label{r:rho_0}
Observe that the assumption $\rho_0-1\in H^{s+1}$ alone is not sufficient to solve the first issue mentioned in Paragraph \ref{sss:loss},
namely the problem of loss of derivatives deriving from \eqref{i_eq:odd_1}.

As a matter of fact, $u$ being only $H^s$ with respect to the space variable, transport theory does not allow us to propagate more than $H^s$ regularity
of $\rho-1$. In other words, in order to get $\rho-1\in H^{s+1}$ we must resort to different ingredients. We will give more details about this point
in Subsection \ref{ss:strategy} below.
\end{rem}

Next, we want to spend a few words about the regularity of the pressure function.
\begin{rem} \label{r:intro_press}
Theorem \ref{teo1} only gives $\nabla \pi\in C\big([0,T];H^{s-2}(\R^2)\big)$, but our analysis will show that the difference
$\nabla\big(\pi-\rho\o\big)$ is actually more regular, namely it belongs to $C\big([0,T];H^{s-1}(\R^2)\big)$.

This is a key point of our study: it represents the first of the two ingredients which allow us to solve the problem caused by the loss of derivatives
coming from the pressure term.
\end{rem}

Finally, let us comment the two continuation criteria stated in the previous theorem.

\begin{rem} \label{r:continuation}
The two continuation criteria of Theorem \ref{teo1} look a bit more complicated than the classical Beale-Kato-Majda criterion for the incompressible
Euler equation. On the one hand, this is a natural consequence of the more complicated structure of system \eqref{eq:fluid_odd_visco1} with respect to that model.
On the other hand, there are important aspects of our continuation criteria that we want to point out here.

\begin{enumerate}[(i)]
 \item The presence of suitable powers (higher than $1$) of the $L^\infty$ norms of $\nabla u$, $\nabla\rho$ and $\nabla\pi$
are absolutely natural for our system; see \tsl{e.g.} \cite{F-L} for analogous continuation criteria in the context of a system for
non-homogeneous inviscid quasi-incompressible fluids.

\item On the other hand, it is appreciable the presence of the term $\|\Delta\rho\|_{L^\infty}$, instead of the corresponding norm of the full
Hessian matrix $\nabla^2\rho$. This improvement relies on the use of the logarithmic interpolation inequality from \cite{KT}, but,
due to the complicate structure of the odd system, its application in our context is not obvious and requires a very precise estimate process.

\item The importance of the second continuation criterion relies on the fact that we replace the $\|\nabla\pi\|_{L^\infty}$ with
the corresponding norm of the difference $\nabla\big(\pi-\rho\o\big)$, which is a much more natural quantity in our problem (keep in mind
the discussion of  Remark \ref{r:intro_press}). As a matter of fact,
observe that the former norm is not under control by our theory (as $H^{s-2}$ fails to embed into $L^\infty$ when $2<s\leq 3$), while the latter does
(because $\nabla\big(\pi-\rho\o\big)$ belongs to $H^{s-1}$, with $s-1>1$).
\end{enumerate}
\end{rem}

This having been clarified, let us move further and explain the strategy of the proof to Theorem \ref{teo1}.

\subsection{Strategy of the proof} \label{ss:strategy}

As already hinted in Paragraph \ref{sss:loss}, the structure of equations \eqref{eq:fluid_odd_visco1} is rather intricate and this makes
the search for \tsl{a priori} estimates a quite hard task.
Actually, this complex structure transforms into deep and delicate issues which arise also in all the other
levels of the proof of Theorem \ref{teo1}, 
namely in the proof of existence of a solution, in the stability estimates (which lead to uniqueness) and in the estimates for deriving
the continuation criteria.

The goal of this subsection is to discuss those issues in detail and explain our strategy to overcome them in the analysis and obtain the sought result.
We start by focusing on the \tsl{a priori} estimates; we will comment about the other parts of the proof in Paragraph \ref{sss:intro_rest}.

\subsubsection{\tsl{A priori} estimates} \label{sss:intro_a-priori}

The difficulties arising when trying to derive \tsl{a priori} estimates for system \eqref{eq:fluid_odd_visco1} have already been presented in
Paragraph \ref{sss:loss}. There, we have shown that the odd viscosity term is responsible for the appearing of a loss of derivatives at two
different levels of the analysis. Correspondingly, we need to introduce two new ingredients to solve those problems.

The first crucial idea that we use in this paper is to introduce a new set of unknowns, which are better adapted 
for propagating higher order norms of $\rho$ and $u$. We will refer to them as the \emph{good unknowns} for system \eqref{eq:fluid_odd_visco1},
as their introduction allows to highlight a hidden \emph{hyperbolic structure} of the equations.

The good unknowns for our system with odd viscosity are the vorticity of the fluid
\[
\omega\,:=\,\nabla^\perp\cdot u\,=\,\d_1u_2-\d_2u_1
\]
and the function $\theta$ defined as
\[
\theta\,:=\,\eta\,-\,\Delta\rho\,,\qquad\qquad\qquad\qquad \mbox{ where }\qquad\quad \eta\,:=\,\nabla^\perp\cdot\big(\rho\,u\big)\,.
\]
The ``vorticity'' component $\eta$ of the momentum $\rho\,u$  (which is not divergence-free, though) has the role of trading regularity
from the set of good unknowns $\big(\o,\theta\big)$ to the original set of unknowns $\big(\rho,u\big)$. Indeed, we notice that $u\in H^s$
if\footnote{Observe that here we are interested only in propagation of higher regularity norms, or in order words in high frequency analysis.
Indeed, the control on low frequencies will be provided by a standard energy conservation property, which system \eqref{eq:fluid_odd_visco1} enjoys.}
and only if $\o \in H^{s-1}$; on the other hand, if $\rho-1$ and $u$ belong to $H^s$, then $\eta$ possesses $H^{s-1}$ regularity, so that $\rho-1$ belongs
to $H^{s+1}$ if and only if $\theta\in H^{s-1}$.

What makes all this argument work is that both $\o$ and $\theta$ satisfy very simple transport equations by $H^s$ vector fields. As a matter of fact,
direct computations (for the details of which we refer to Subsection \ref{ss:new-en}) show that $\theta$ solves
\begin{equation} \label{i_eq:theta}
\d_t\theta\,+\,u\cdot\nabla\theta\,=\,\frac{1}{2}\,\nabla^\perp\rho\cdot\nabla|u|^2\,+\,\mc B\big(\nabla u,\nabla^2\rho\big)\,,
\end{equation}
where we have defined, for any divergence-free vector field $v\in\R^2$ and any scalar field $\alpha\in\R$, the bilinear operator
\begin{equation} \label{def:B}
\mc B\big(\nabla v,\nabla^2\alpha\big)\,:=\,\nabla^\perp\cdot\Big(\big(\nabla \alpha\cdot\nabla\big)v^\perp\Big)\,=\,
\d_1\d_2\alpha\,\big(\d_1v_2\,+\,\d_2v_1\big)\,+\,\d_1v_1\,\big(\d_1^2\alpha\,-\,\d_2^2\alpha\big)\,,
\end{equation}
whereas $\o$ satisfies the equation
\begin{equation} \label{i_eq:vort}
\d_t\omega\,+\,u\cdot\nabla\omega\,+\,\nabla^\perp\left(\frac{1}{\rho}\right)\cdot\nabla \pi\,+\,\mc B\big(\nabla u,\nabla^2\log\rho\big)\,=\,0\,.
\end{equation}
We remark that the derivation of the previous equations strongly relies on the hypothesis of initial absence of vacuum (a fact which is preserved
by the flow of the equations) and on the incompressibility $\nabla\cdot u=0$ of the fluid velocity.
We also point out that, although equation \eqref{i_eq:vort} presents the same phenomenon of loss of regularity discussed in Paragraph \ref{sss:loss},
the two pathological terms $\mc B$ and $\nabla\pi$ now come into play at a lower order, inasmuch as, in order to close the estimates, it is not necessary to control
them in $H^s$, but only in $H^{s-1}$: this is yet another advantage of passing to the vorticity formulation.

Now, repeating the analysis of Paragraph \ref{sss:loss}, and forgetting about the pressure term for a while, it is easy to see that all the terms
on the right-hand side of \eqref{i_eq:theta} and the bilinear term in \eqref{i_eq:vort} belong to $H^{s-1}$; as $u\in H^s$, by standard results
on transport equations we expect to be able to propagate $H^{s-1}$ regularity for both $\o$ and $\theta$, which, as already noticed,
in turn implies $\rho-1\in H^{s+1}$.

The bug in the previous argument is that, because of the low regularity of the pressure gradient $\nabla\pi \in H^{s-2}$, we are not really able
to propagate $H^{s-1}$ regularity of $\o$, so we cannot make sure yet that $u$ belongs to $H^s$.
Here comes into play the second main idea of our work: by a precise analysis of the pressure equation
\begin{equation} \label{eq:ell-p}
-\nabla\cdot\left(\frac{1}{\rho}\,\nabla \pi\right)\,=\,\nabla\cdot\Big((u\cdot\nabla)u\,+\,(\nabla\log\rho\cdot\nabla)u^\perp\Big)\,-\,\Delta\omega\,,
\end{equation}
we will show that one actually has
$\nabla\pi\approx\nabla\big(\rho\o\big)$, in the sense that their difference is more regular, \tsl{i.e.} it belongs to $H^{s-1}$.
Plugging the \tsl{ansatz} $\nabla\big(\pi-\rho\o\big)\in H^{s-1}$ into the vorticy equation \eqref{i_eq:vort} and using the cancellation
$\nabla^\perp(1/\rho)\cdot\nabla\rho=0$, we see that the pressure gradient actually splits into a $H^{s-1}$ term plus a transport term of the form
$-\nabla^\perp f(\rho)\cdot\nabla\o$, for a suitable function $f(\rho)$ of the density. As $-\nabla^\perp f(\rho)$ is a $H^s$ vector field having zero divergence,
we can finally transport the $H^{s-1}$ regularity of $\o$ and close the estimates, at least locally in time.

To conclude this part, we observe that resorting to the good unknowns seems to be really necessary to prove \tsl{a priori} bounds for our system.
As a matter of fact, one may object that the same cancellations we are able to see on the system for $\big(\o,\theta\big)$ might be obtained also by directly
working on the original system \eqref{eq:fluid_odd_visco1}. Yet, even if this were the case, there would be no way to get $\rho-1\in H^{s+1}$
(as $u\in H^s$, it is not possible to derive that regularity from the equation of conservation of mass). Now, the analysis cannot really work without this property,
as  $\rho-1\in H^{s+1}$ is needed to absorbe the derivative loss mentioned in Paragraph \ref{sss:loss} and also to establish the regularity of the pressure.

\subsubsection{The rest of the proof of Theorem \ref{teo1}} \label{sss:intro_rest}

The use of the good unknowns $\big(\o,\theta\big)$ and of the hyperbolic structure resulting from \eqref{i_eq:theta} and \eqref{i_eq:vort}
plays a key role not only in the derivation of \tsl{a priori} estimates for equations \eqref{eq:fluid_odd_visco1}, but also in all the other parts
of the proof of Theorem \ref{teo1}. However, at each step we will have to face different difficulties: let us comment them here below.

\paragraph*{Existence.}
We start by reporting on the proof of existence of a solution. Here, we decide to proceed by viscous regularisation: we will add a hyperviscosity term
$+\veps\Delta^2u$ to the momentum equation and construct smooth approximate solutions $\big(\rho_\veps,u_\veps,\nabla\pi_\veps\big)_\veps$, locally in time.
The plan for the rest of the proof follows a standard scheme: first of all, we need to derive uniform estimates 
for the sequence $\big(\rho_\veps,u_\veps,\nabla\pi_\veps\big)_\veps$ in some uniform time interval $[0,T_0]$, and then we pass to the limit $\veps\ra0^+$
by means of a compactness argument.

The problem is that the underlying hyperbolic structure is not really compatible with the (new) parabolic structure of the viscous system. This contrast emerges
in particular when computing the equation for $\theta_\veps$, in which a dangersous $+\veps\Delta^2\o_\veps$ appears: it is easy to observe
that this term cannot be reconducted to $\theta_\veps$ without making higher order terms of the density enter into play. Thus, we can only treat it
as a forcing term in the equation for $\theta_\veps$. Obviously, this is a problem when deriving the uniform
bounds for $\big(\rho_\veps,u_\veps,\nabla\pi_\veps\big)_\veps$ mentioned above.
Observe also that, in this way, we do not gain any parabolic smoothing effect on $\theta_\veps$.
However, the fundamental point is that the equation for $\o_\veps$ does present a hyperviscosity term (with non-constant coefficients), which yields a gain of regularity
for that quantity, up to remainders of lower order. In this way, we will be able to close the estimates for $\o_\veps$ in terms of the energy functional itself,
without requiring any additional smoothness on $\theta_\veps$ and $\rho_\veps-1$: plugging those estimates into the ones obtained
(by transport theory) for $\theta_\veps$ will end the derivation of suitable uniform bounds. In passing, we point out that, for this argument to work,
it is absolutely fundamental to use transport-diffusion estimates in Chemin-Lerner spaces $\wtilde L^r_T\big(H^s\big)$:
indeed, if we did not resort to those sharp estimates, the term $+\veps\Delta^2\o_\veps$ appearing in the equation for $\theta_\veps$ would be out of control.

Thanks to the previous uniform bounds, we can pass to the limit in the approximation parameter $\veps\ra0^+$ to prove existence of a solution to the original
problem \eqref{eq:fluid_odd_visco1}. Here, yet another problem arises: let us briefly explain it. Of course, for treating the non-linearities appearing in the equations,
we need strong convergence of both $\rho_\veps$ and $u_\veps$. However, owing to the low time regularity of the pressure term (because of the transport-diffusion
structure of the equations), getting strong convergence of $u_\veps$ by making use of the equations reveals to be technically involved. As a first step
we need to perform a very careful analysis of the pressure term to gain that both $\nabla\pi_\veps$ and $\nabla\big(\pi_\veps-\rho_\veps\o_\veps\big)$
belong to some $L^p_T\big(H^{s-2}\big)$ space, for a suitable (small) $1<p\leq2$. Only then we can use the momentum equation divided by $\rho_\veps$ to get a uniform
bound for $\big(\d_tu_\veps\big)_\veps$ in some low regularity space, hence (by Ascoli-Arzel\`a theorem) compactness.

To conclude, we mention that a further, final complication consists in establishing the right time regularity of the solution. Again, the pressure term is the cause
of the issue. To bypass the problem, we will need to perform a ``two-steps'' analysis, where we first establish that the pressure terms
$\nabla\pi$ and $\nabla\big(\pi-\rho\o\big)$ are $L^\infty$ in time and then we get the right time regularity of the vorticity function $\o$.
With that information at hand, it is not difficult to get the sought time continuity of both $u$ and $\nabla\pi$ in the respective spaces.

\paragraph*{Uniqueness.}
In order to prove uniqueness of solutions at the claimed level of regularity, we follow a classical approach and perform a stability estimate in $L^2$.
Notice however that several complications arise.

First of all, we remark that the $L^2$ estimate alone cannot suffice,
because of the presence of the odd viscosity term, which requires a control of the difference of solutions in a higher regularity norm.
Thus, we will proceed to the estimate of the difference of the vorticities, which in turn requires (because of the presence of the bilinear term
$\mc B\big(\nabla u,\nabla^2\log\rho\big)$ in \eqref{i_eq:vort},
which involves two derivatives of the density) an estimate also for the difference of the functions $\theta$.

However, and this is connected with the second difficult point of our computations, in the stability estimates some quantities seem to need a control
in $L^\infty$ in order to close the argument: those quantities are $\nabla\pi$ (coming from the computations for the difference of the velocity fields),
$\nabla\o$ and $\nabla\theta$ (the latters, both coming from the computations performed in the respective equations). The problem is that
all those quantites are only $H^{s-2}$, which is not embedded in $L^\infty$ if $2<s\leq3$. This would make a gap appear between
the regularities needed for the existence theory and the uniqueness theory, but of course this is not desirable.
While one may expect to solve this issue for $\nabla\o$ and
$\nabla\theta$ by a wise use of Lagrangian coordinates\footnote{Remark that, while $\theta$ is transported by $u$, the vorticity $\o$ is instead transported by
$u-\nabla^\perp f(\rho)$, so it is not really clear to us how to use Lagrangian coordinates in this context.},
the same cannot be said about the pressure term, whose presence does not depend on the use of Eulerian or Lagrangian coordinates.

The solution to the previous \tsl{impasse} will be given by a clever use of Sobolev embeddings and Gagliardo-Nirenberg inequalities in dimension $d=2$.
As a matter of fact, whenever dealing with a term presenting a high (critical) number of derivatives in its expression, 
by performing the stability estimates carefully we will be able to bound the differences
$\de u$ of the velocity fields and $\de\rho$ of the densities and their gradients in some $L^p$ space, for a well-chosen $p\in[2,+\infty[\,$.
Now, as for those quantities we have
one derivative more (the difference of the vorticities $\de\o$ and the difference of the Laplacians $\Delta\de\rho$) at our disposal, the Gagliardo-Nirenberg
inequalities will allow us to reconduct those terms to a suitable energy norm.
It appears clear from our discussion that, also in this argument, the gain of one derivative for the density will be essential.
We also remark that, here, the fact of working in subcritical spaces $H^s$ with $s>2$ will play an important role as well.

Finally, we point out that, as usual, the pressure term will require a special treatment,
first of all because it is a quadratic quantity, recall equation \eqref{eq:ell-p}, and secondly because, in fact, we will need to
estimate the difference of the functions $\nabla\big(\pi-\rho\o\big)$, instead of the difference of the pressure terms.
Thus, this part of the analysis will be quite involved.

\paragraph*{Continuation criteria.}
Finally, we comment on the derivation of the two continuation criteria. Also in this case, we follow a classical scheme, which consists in obtaining,
by a massive use of commutator and tame (also called Moser's) estimates, an inequality of the form
\begin{equation} \label{i_est:cont-crit}
\mc E(t)\,\lesssim\,\mc E(0)\,+\,\int^t_0J(\t)\,\mc E(\t)\,\dd\t\,,
\end{equation}
where $\mc E(t)$ is a suitable energy function associated to the solution and $J$ is a $L^1_t$ function of the time variable, defined as the sum of
a number of $L^\infty=L^\infty_x$ norms (in space) of the solution (precisely, the norms appearing under the integral of our continuation criteria).

In this argument, there are a few points which deserve some attention. It goes without saying that the first problem comes again from the analysis of the pressure.
Here we face similar difficulties as the ones we treated in the proof of the stability estimates. As a matter of fact,
from \eqref{eq:ell-p} it is clear that $\nabla\pi$ is a quadratic quantity, and so does $\nabla\big(\pi-\rho\o\big)$ (which is the right quantity to control here),
hence one easily obtains an estimate of the type
\[
\left\|\nabla\big(\pi-\rho\o\big)\right\|_{H^{s-1}}\,\lesssim\,\mc E(t)\,,
\]
but this inequality is not useful to derive \eqref{i_est:cont-crit}. Therefore, we will strive to get a more precise bound,
where $\mc E(t)$ on the right of the previous estimate will be replaced by $J_{\pi}(t)\,\sqrt{\mc E(t)}$, where $J_{\pi}$ is a $L^1_t$ function of time,
which comes from the analysis of the pressure and which depends on suitable $L^\infty_x$ norms of the solution.

In addition, also here (as for uniqueness) we need to avoid to estimate $\nabla\o$ and $\nabla\theta$ in $L^\infty_x$, although those bounds come out naturally
from the transport terms when performing $H^{s-1}$ estimates in equations \eqref{i_eq:theta} and \eqref{i_eq:vort}. To solve the previous issue, our idea is to resort
to paradifferential calculus to treat the corresponding commutator terms: this will enable us to ``move derivatives'' from the transported quantities
to the transport fields. Remark that this is not hurtful, as the transport fields belong to $H^s$.
On the contrary, since the pressure is an important hydrodynamic quantity, in the first continuation criterion we will allow ourselves
to make the $L^\infty$ norm of $\nabla\pi$ appear. On the other hand, because of the same previous reasons (see also Remarks \ref{r:intro_press} and
\ref{r:continuation}), we will also derive the second continuation criterion, based on a $L^\infty$ control of $\nabla\big(\pi-\rho\o\big)$ instead.

To conclude, we remark that, by performing the estimates carefully, we will be able to keep a linear dependence of the function $J(t)$, appearing in
\eqref{i_est:cont-crit}, on the quantity $\left\|\nabla^2\rho\right\|_{L^\infty}$. Therefore, applying the logarithmic interpolation inequality from
\cite{KT} and the Osgood lemma will enable us to replace that norm by the $L^\infty$ norm of the Laplacian $\Delta\rho$ only.

\addcontentsline{toc}{section}{Organisation of the paper}
\section*{Organisation of the paper}
After this long introduction, let us give a detailed overview of the paper.

We start, in Section \ref{s:tools}, by presenting a toolbox of Littlewood-Paley theory and paradifferential calculus, which will be needed in the course of 
the proof and actually, as already pointed out, whose use will play a key role in our work. In the same section, we also state some well-known
estimates for elliptic, transport and transport-diffusion equations in Sobolev spaces.

Section \ref{s:a-priori} contains the proof of the \tsl{a priori} estimates for smooth (supposed to exist) solutions to the odd viscosity fluid
system \eqref{eq:fluid_odd_visco1}. In the last part of the section (see Subsection \ref{ss:a-priori_cont}),
we also present the estimates leading to the two continuation criteria.
All those estimates are collected in Theorem \ref{t:uniform_estimates}, which represents the main result of that part.

In Section \ref{s:Stokes}, we study a hyperviscous Stokes problem with variable coefficients, where the unknowns are
the velocity field and the pressure field (we neglect the mass equation, although the variable coefficients account for the variations of the density)
and the transport term is linearised. We present a global in time existence and uniqueness
theory in Sobolev spaces for that system, which is of independent interest. That analysis will be needed in Subsection \ref{ss:proof-e},
especially in Proposition \ref{p:viscous}, where we construct smooth approximate solutions to the original system.

Finally, in Section \ref{s:proof} we complete the proof of Theorem \ref{teo1}. In Subsection \ref{ss:proof-e} we show the proof of existence of solutions.
After that, in Subsection \ref{ss:proof-u} we turn our attention to the proof of the stability estimates and uniqueness of solutions at the claimed
level of regularity.

\medbreak
Before going on, let us collect some notation that we have freely used throughout this paper.

\subsection*{Notation}
To begin with, let us make more explicit the meaning of the operators $\nabla^\perp a$, $\nabla a^\perp$, although we have already discussed them before.
We denote
$$
\nabla\,=\,^t\left(\partial_1,\partial_2\right)\qquad\qquad \mbox{ and }\qquad\qquad \nabla^\perp\,=\,^t\left(-\partial_2,\partial_1\right)\,.
$$
With this notation, we have
\begin{align*}
\nabla u\,=\,\left(\begin{array}{c}\partial_1 \\
\partial_2 \end{array}\right)(u_1,u_2)\,=\,\left(\begin{array}{cc}\partial_1 u_1 & \partial_1 u_2\\
\partial_2 u_1 & \partial_2 u_2\end{array}\right)
&&
\mbox{ and }
&&
\nabla u^\perp\,=\,\left(\begin{array}{c}\partial_1 \\
\partial_2 \end{array}\right)(-u_2,u_1)\,=\, \left(\begin{array}{cc}-\partial_1 u_2 & \partial_1 u_1\\
-\partial_2 u_2 & \partial_2 u_1\end{array}\right)\,.
\end{align*}
In addition, we can write the classical $\div$ and $\curl$ operators as
\[
\div f\,=\,\nabla\cdot f\,=\,\nabla^\perp\cdot f^\perp\qquad\qquad \mbox{ and }\qquad\qquad 
\curl f\,=\,\nabla^\perp\cdot f \,=\,\d_1f_2\,-\,\d_2f_1\,.
\]
When applyied to matrices (say) $M$, the notation $\nabla\cdot M$ means the contraction of indeces along columns of $M$. Thus, given a vector $v=(v_1,v_2)\in\R^2$,
we have that
\[
\nabla\cdot\nabla v\,=\,\Delta v\,=\,\big(\Delta v_1\,,\,\Delta v_2\big)
\]
is simply the vector which has, for components, the Laplacian operator applied to the components of $v$.

To simplify the notation, when performing energy estimates we will write
\[
\int_{\R^2}\bullet\,\dd x\,=\,\int\bullet\,\dd x\,\,.
\]

Because of the massive use of Littlewood-Paley theory, we will often have to deal with commutator operators. So, given
two operators $A$ and $B$, we will denote their commutator operator by
\[
[A,B]\,=\,AB\,-\,BA\,.
\]

Finally, in order to make the writing easier and the reading ligther, in our estimates we will often avoid to write the explicit multiplicative constants
which allow to pass from one line to the other. In this case, we will write $A\,\lesssim\, B$
meaning that there exists a universal (\tsl{i.e.} independent of the initial datum and solution of our problem) $C>0$ such that $A\,\leq\,C\,B$.
Similarly, we will write $A\,\approx\,B$ when $A\,\lesssim\,B$ and $B\lesssim\,A$.

\section{Tools} \label{s:tools}
We collect here the main tools needed in our analysis. We start by recalling the basic facts of Littlewood-Paley theory in $\R^2$
and the dyadic characterisation of Sobolev spaces. Then, in Subsection \ref{ss:para} we introduce some elements of paradifferential calculus.
In Subsection \ref{ss:tools-est}, we apply the previous theory to recall some well-known estimates for smooth solutions to transport and elliptic equations.
Finally, in Subsection \ref{ss:transp-diff} we refine the previous analysis by introducing the class of Chemin-Lerner spaces and recalling sharp
estimates in those spaces for solutions to transport-diffusion equations.

\subsection{Littlewood-Paley theory and Sobolev spaces} \label{ss:LP-Sobolev}

We recall here the main ideas of Littlewood-Paley theory in $\R^d$. These are classical results, for which,
if not otherwise specified, we refer to Chapter 2 of \cite{B-C-D} for details.

The starting point is to remark that, by Paley-Wiener theorem, distributions which have compact spectrum\footnote{By \emph{spectrum} of a tempered
distribution $u$, we mean the support of its Fourier transform $\mc Fu=\what u$.} are in fact smooth functions.
The following lemma, which presents the so-called \emph{Bernstein inequalities}, gives a quantitative control of the Lebesgue norms of a tempered
distribution having spectrum contained in rings and of its derivatives. 
  \begin{lemma} \label{l:bern}
Let  $0<r<R$.   A constant $C>0$ exists so that, for any nonnegative integer $k$, any couple $(p,q)$ 
in $[1,+\infty]^2$, with  $p\leq q$,  and any function $u\in L^p$,  we  have, for all $\lambda>0$,
\begin{align*}
{\supp}\, \widehat u\, \subset\,   B(0,\lambda R)\,=\,\big\{\xi\in\R^d\,\big|\,|\xi|\leq\lambda R \big\}\qquad
&\Longrightarrow\qquad
\|\nabla^k u\|_{L^q}\, \leq\,
 C^{k+1}\,\lambda^{k+d\left(\frac{1}{p}-\frac{1}{q}\right)}\,\|u\|_{L^p} \\[1ex]
{\supp}\, \widehat u   \, \subset\, \big\{\xi\in\R^d\,\big|\, r\lambda\leq|\xi|\leq R\lambda\big\}
\qquad&\Longrightarrow\qquad C^{-k-1}\,\lambda^k\|u\|_{L^p}\,\leq\,
\|\nabla^k u\|_{L^p}\,
\leq\,C^{k+1} \, \lambda^k\|u\|_{L^p}\,.
\end{align*}
\end{lemma}   

This property represents a strong motivation to look at the so-called ``Littlewood-Paley decomposition'' of a tempered distribution $u$,
which 
we are going to introduce now.

Fix a smooth radial function $\chi$ supported in the ball $B(0,2)$, equal to $1$ in a neighbourhood of $B(0,1)$
and such that $r\mapsto\chi(r\,e)$ is non-increasing over $\R_+$ for all unitary vectors $e\in\R^d$. Set
$\varphi\left(\xi\right)=\chi\left(\xi\right)-\chi\left(2\xi\right)$ and
$\vphi_j(\xi):=\vphi(2^{-j}\xi)$ for all $j\geq0$.
The dyadic blocks $(\Delta_j)_{j\in\Z}$ are defined by\footnote{Throughout we agree  that  $f(D)$ stands for 
the pseudo-differential operator $u\mapsto\mc{F}^{-1}[f(\xi)\,\what u(\xi)]$, where
$\mc F^{-1}$ denotes the inverse Fourier transform operator.} 
$$
\Delta_j\,:=\,0\quad\mbox{ if }\; j\leq-2,\qquad\Delta_{-1}\,:=\,\chi(D)\qquad\mbox{ and }\qquad
\Delta_j\,:=\,\varphi(2^{-j}D)\quad \mbox{ if }\;  j\geq0\,.
$$
We  also introduce the following low frequency cut-off operators: for $j\geq0$, we define
\begin{equation*} \label{eq:S_j}
S_j\,:=\,\chi(2^{-j}D)\,=\,\sum_{k\leq j-1}\Delta_{k}\,. 
\end{equation*}
Notice that $S_j$ is a convolution operator with a function $K_j(x) = 2^{dj}K_1(2^j x) = \mathcal{F}^{-1}[\chi (2^{-j} \xi)] (x)$ of constant $L^1$ norm,
hence it defines a continuous self-map of $L^p$ for any $1 \leq p \leq +\infty$.
The same holds true for the operators $\Delta_j$.

Observe that, owing to spectral localisation, we have
\begin{equation} \label{eq:loc-prop}
\Delta_k\Delta_jf\,\equiv\,0\qquad \mbox{ if }\quad |k-j|\geq2\qquad\qquad\mbox{ and }\qquad\qquad
\Delta_k\big(S_{j-1}f\,\Delta_jg\big)\,\equiv\,0\qquad \mbox{ if }\quad |k-j|\geq5\,.
\end{equation}
The previous properties will be extensively used in the estimates of Section \ref{s:Stokes}.

As announced above, we are now ready to derive the Littlewood-Paley decomposition of any tempered distribution:
\[
\forall\,u\,\in\,\mc S'\,,\qquad\qquad\qquad u\,=\,\sum_{j\geq-1}\Delta_ju\qquad \mbox{ in the sense of }\quad \mc S'\,.
\]

By use of Littlewood-Paley decomposition, we can also define the class of Besov spaces.
\begin{definition} \label{d:B}
  Let $\s\in\R$ and $1\leq p,r\leq+\infty$. The \emph{non-homogeneous Besov space}
$B^{\s}_{p,r}$ is defined as the subset of tempered distributions $u$ for which
$$
\|u\|_{B^{\s}_{p,r}}\,:=\,
\left\|\left(2^{j\s}\,\|\Delta_ju\|_{L^p}\right)_{j\geq -1}\right\|_{\ell^r}\,<\,+\infty\,.
$$
\end{definition}
Besov spaces are interpolation spaces between Sobolev spaces. In fact, for any $k\in\N$ and~$p\in[1,+\infty]$
we have the following chain of continuous embeddings: $ B^k_{p,1}\hookrightarrow W^{k,p}\hookrightarrow B^k_{p,\infty}$,
where  $W^{k,p}$ denotes the classical Sobolev space of $L^p$ functions with all the derivatives up to the order $k$ in $L^p$.
When $1<p<+\infty$, we can refine the previous result (this is the non-homogeneous version of Theorems 2.40 and 2.41 in \cite{B-C-D}): we have
$$
 B^k_{p, \min (p, 2)}\hookrightarrow W^{k,p}\hookrightarrow B^k_{p, \max(p, 2)}\,.
$$
In particular, for all $\s\in\R$ we deduce the equivalence $B^{\s}_{2,2}\equiv H^\s$, with equivalence of norms:
\begin{equation} \label{eq:LP-Sob}
\forall\,f\in H^{\s}\,,\qquad\qquad \|f\|_{H^{\s}}\,\approx\,\left(\sum_{j\geq-1}2^{2 j \s}\,\|\Delta_jf\|^2_{L^2}\right)^{1/2}\,.
\end{equation}
More precisely, we have the following statement, for which we refer \tsl{e.g.} to \cite{Met}.
\begin{prop} \label{p:LP-H}
Let $\s\in\R$. Then $f\in\mc{S}'$ belongs to the space $H^{\s}$ if and only if:
\begin{itemize}
 \item[(i)] for all integer $k\geq-1$, $\Delta_kf\in L^2$;
\item[(ii)] for all integer $k\geq -1$, set $\,\delta_k\,:=\,2^{k\s}\,\|\Delta_kf\|_{L^2}$, then the sequence
$\left(\delta_k\right)_k$ belongs to $\ell^2\big(\N\cup\{-1\}\big)$.
\end{itemize}
Moreover, one has $\|f\|_{H^{\s}}\,\approx\,\left\|\left(\delta_k\right)_k\right\|_{\ell^2}$.
\end{prop}

It is also true that the classical H\"older spaces $C^{k,\alpha}$, with $k\in\N$ and $\alpha\in\,]0,1[\,$, coincides with the Besov spaces
$B^{k+\alpha}_{\infty,\infty}$, whereas the space $C^k$ (and $C^{k-1,1}$ if $k\in\N\setminus\{0\}$) is only strictly
embedded in $B^k_{\infty,\infty}$.

As an immediate consequence of the first Bernstein inequality, one gets the following embedding result, which is the natural generalisation of the
Sobolev embeddings.
\begin{prop}\label{p:embed}
The continuous embedding $B^{\s_1}_{p_1,r_1}\,\hookrightarrow\,B^{\s_2}_{p_2,r_2}$ holds whenever $p_1\,\leq\,p_2$ and
$$
\s_2\,<\,s_1-d\left(\frac{1}{p_1}-\frac{1}{p_2}\right)\,,\qquad\mbox{ or }\qquad
\s_2\,=\,s_1-d\left(\frac{1}{p_1}-\frac{1}{p_2}\right)\;\;\mbox{ and }\;\;r_1\,\leq\,r_2\,. 
$$
\end{prop}

If we limit our focus to the case of $d=2$ and embeddings of Sobolev spaces into H\"older type spaces $B^\alpha_{\infty,\infty}$,
which is the only relevant case for our study, the statement becomes the following one.
\begin{cor}\label{c:embed}
For any $\s\in\R$, one has the continuous embedding $H^{\s}(\R^2)\,\hookrightarrow\,B^{\s-1}_{\infty,\infty}(\R^2)$.

In particular, the space $H^{\s}(\R^2)$ is continuously embedded in the space $L^\infty(\R^2)$ whenever $\s>1$, in the space $W^{1,\infty}(\R^2)$ whenever $\s>2$.
\end{cor}

To conclude this part, we prove some useful inequalities by means of Littlewood-Paley decomposition.
They are all consequences of cutting a tempered distribution into low and high frequencies, together with the dyadic characterisation \eqref{eq:LP-Sob} of
Sobolev spaces and the Bernstein inequalities.

To begin with, we see that, for any $\s\geq0$ and any $f\in H^{\s}$, one has, for any $k\in\N$, the estimate
\begin{equation} \label{est:basic}
\|f\|_{H^\s}\,\approx\,\|f\|_{L^2}\,+\,\left\|D^kf\right\|_{H^{\s-k}}\,,
\end{equation}
for some (implicit) multiplicative constant $C=C(\s,k)>0$ only depending on the quantities inside the brackets.
In order to see this, one inequality is trivial, since the $H^\s$ norm controls the quantities on the right-hand side. So, we focus on
the reverse inequality, for which we start by writing
\begin{align*}
\|f\|^2_{H^\s}\,&\lesssim\,2^{-2\s}\,\left\|\Delta_{-1}f\right\|_{L^2}^2\,+\,\sum_{j\geq0}2^{2j\s}\,\left\|\Delta_{j}f\right\|_{L^2}^2\,\leq\,
\|f\|^2_{L^2}\,+\,C(k)\,\sum_{j\geq0}2^{2j(\s-k)}\,\left\|\Delta_{j}D^kf\right\|_{L^2}^2\,,
\end{align*}
where we have applied both Bernstein's inequalities for passing from the first to the second inequality.
At this point, the last series on the right-hand side can be easily controlled by the $H^{\s-k}$ norm of $D^kf$, thus completing the proof
of the claimed bound \eqref{est:basic}.

Applying this inequality to the density function $\rho$ with $k=2$, by Calder\'on-Zygmund theory we deduce that, for any $\s\geq-1$, one has
\begin{equation} \label{est:dens-low-high}
\left\|\rho-1\right\|_{H^{\s+1}}\,\approx\,\left\|\rho-1\right\|_{L^2}\,+\,\left\|\Delta\rho\right\|_{H^{\s-1}}\,.
\end{equation}

Analogously, let $u$ be a divergence-free vector field, and let $\omega\,:=\,\d_1u_2-\d_2u_1$ its vorticity. Then $u$ can be recovered from $\omega$
by applying the Biot-Savart law
\[
u\,=\,-\,\nabla^\perp(-\Delta)^{-1}\omega\,,
\]
which in particular implies that $\nabla u\,=\,-\nabla \nabla^\perp(-\Delta)^{-1}\omega$. Repeating the computations leading to \eqref{est:basic},
where we take $k=1$, and using Calder\'on-Zygmund theory in order to control the $L^2$ norm of $\Delta_j\nabla u$ with the $L^2$ norm of $\Delta_j\omega$,
we see that, for any $\s\geq0$, we can estimate
\begin{equation} \label{est:u-omega}
\|u\|_{H^\s}\,\approx\,\|u\|_{L^2}\,+\,\|\o\|_{H^{\s-1}}\,.
\end{equation}

A similar argument allows us to show also the following interpolation inequality: for any $\s>2$, one has
\begin{equation} \label{est:interp}
\forall\,\rho\in H^{s}(\R^2)\,,\qquad\qquad\|\Delta\rho\|_{H^{\s-2}}\,\lesssim\,\|\rho-1\|_{L^2}^{1-\beta}\;\|\Delta\rho\|_{H^{\s-1}}^{\beta}\,,\qquad
\mbox{ with } \qquad \beta\,=\,\beta(\s)\,=\,\frac{\s}{\s+1}\,,
\end{equation}
where the (implicit) multiplicative constant is independent of $\rho$. In order to see this, we use again \eqref{eq:LP-Sob} to write, for any
$N\in\N$ to be fixed later, the following estimate:
\[
 \left\|\Delta{\rho}\right\|_{H^{\s-2}}^2\,\lesssim\,\sum_{j=-1}^N2^{2j(\s-2)}\,\left\|\Delta_j\Delta\rho\right\|_{L^2}^2\,+\,
 \sum_{j\geq N+1}2^{2j(\s-2)}\,\left\|\Delta_j\Delta\rho\right\|_{L^2}^2\,.
\]
Now, we use the first Bernstein inequality to control the low frequency term: we get
\begin{align*}
 \left\|\Delta{\rho}\right\|_{H^{\s-2}}^2\,&\lesssim\,\sum_{j=-1}^N2^{2j\s}\,\left\|\Delta_j\big(\rho-1\big)\right\|_{L^2}^2\,+\,
  \sum_{j\geq N+1}2^{-2j}\,2^{2j(\s-1)}\,\left\|\Delta_j\Delta\rho\right\|_{L^2}^2 \\
&\lesssim\,2^{2N\s}\,\left\|\rho-1\right\|^2_{L^2}\,+\,2^{-2N}\,\left\|\Delta\rho\right\|^2_{H^{\s-1}}\,.
\end{align*}
Then, if we choose $N\in\N$ such that
\[
 2^{2N\s}\,\left\|\rho-1\right\|^2_{L^2}\,\approx\,2^{-2N}\,\left\|\Delta\rho\right\|^2_{H^{\s-1}}\,,\qquad\qquad 
\mbox{ \tsl{i.e.} }\qquad 2^N\,\approx\,\left(\frac{\left\|\Delta\rho\right\|_{H^{\s-1}}}{\left\|\rho-1\right\|_{L^2}}\right)^{1/(\s+1)}\,,
\]
we immediately infer inequality \eqref{est:interp}.

\subsection{Paradifferential calculus} \label{ss:para}

We now apply Littlewood-Paley decomposition to state some useful results from paradifferential calculus. Again, we refer to Chapter 2
of \cite{B-C-D} for details and further results.

To begin with, let us introduce the \emph{paraproduct operator} (after J.-M. Bony, see \cite{Bony}).
Constructing the paraproduct operator relies on the observation that, 
formally, any product  of two tempered distributions $u$ and $v$ may be decomposed into 
\begin{equation}\label{eq:bony}
u\,v\;=\;\mathcal{T}_uv\,+\,\mathcal{T}_vu\,+\,\mathcal{R}(u,v)\,,
\end{equation}
where we have defined
$$
\mathcal{T}_uv\,:=\,\sum_jS_{j-1}u\,\Delta_j v,\qquad\qquad\mbox{ and }\qquad\qquad
\mathcal{R}(u,v)\,:=\,\sum_j\sum_{|k-j|\leq1}\Delta_j u\,\Delta_{k}v\,.
$$
The above operator $\mc T$ is called ``paraproduct'' whereas
$\mc R$ is called ``remainder''.
The paraproduct and remainder operators have many nice continuity properties over Besov spaces, which are collected in the next statement.
\begin{prop}\label{p:op}
For any $(\s,p,r)\in\R\times[1,+\infty]^2$ and $t>0$, the paraproduct operator 
$\mathcal{T}$ maps continuously $L^\infty\times B^\s_{p,r}$ in $B^\s_{p,r}$ and  $B^{-t}_{\infty,\infty}\times B^\s_{p,r}$ in $B^{\s-t}_{p,r}$.
Moreover, the following estimates hold:
$$
\|\mathcal{T}_uv\|_{B^\s_{p,r}}\,\lesssim\,\|u\|_{L^\infty}\,\|\nabla v\|_{B^{\s-1}_{p,r}}\qquad\mbox{ and }\qquad
\|\mathcal{T}_uv\|_{B^{\s-t}_{p,r}}\,\lesssim\,\|u\|_{B^{-t}_{\infty,\infty}}\,\|\nabla v\|_{B^{\s-1}_{p,r}}\,.
$$

For any $(\s_1,p_1,r_1)$ and $(\s_2,p_2,r_2)$ in $\R\times[1,+\infty]^2$ such that 
$\s_1+\s_2>0$, $1/p:=1/p_1+1/p_2\leq1$ and~$1/r:=1/r_1+1/r_2\leq1$,
the remainder operator $\mathcal{R}$ maps continuously~$B^{\s_1}_{p_1,r_1}\times B^{\s_2}_{p_2,r_2}$ into~$B^{\s_1+\s_2}_{p,r}$.
In the case $\s_1+\s_2=0$, provided $r=1$, operator $\mathcal{R}$ is continuous from $B^{\s_1}_{p_1,r_1}\times B^{\s_2}_{p_2,r_2}$ with values
in $B^{0}_{p,\infty}$.
\end{prop}

The consequence of this proposition is that
the spaces $H^\s(\R^2)$ are Banach algebras embedded in $L^\infty(\R^2)$ as long as $\s>1$. Moreover, in that case, we have the so-called \emph{tame estimates}.

\begin{cor}\label{c:tame}
For any $\s>1$ and any $f,g$ both belonging to $H^\s(\R^2)$, we have that $fg\in H^\s(\R^2)$, with the estimate
\begin{equation*}
\| f\,g \|_{H^\s}\, \lesssim\, \| f \|_{L^\infty}\, \|g\|_{H^\s}\, +\, \| f \|_{H^\s}\, \| g \|_{L^\infty}\,.
\end{equation*}
\end{cor}

Next, we recall the following result (see its proof in \cite{D_2010}), in the same spirit of the classical Meyer's \emph{paralinearisation} theorem.
It will be important when comparing the Sobolev norms of the density $\rho$ with the ones of its inverse $1/\rho$ and of $\log\rho$.
Again, we limit ourselves to treat the case of Sobolev spaces $H^\s(\R^2)$.
\begin{prop}\label{p:comp}
Let $I$ be an open  interval of $~\R$ and let $F:I\rightarrow\R$ be a smooth function. 

Then for all compact subset $J\subset I$ and $\s>0$, there exists a constant $C$
such that, for all function $a$ valued in $J$ and with gradient in $H^{\s-1}$,  we have
$\nabla(F\circ a)\in H^{\s-1}$ together with the estimate
$$
\|\nabla(F\circ a)\|_{H^{\s-1}}\,\leq\,C\,\|\nabla a\|_{H^{\s-1}}\,.
$$
\end{prop}

Finally, we need a few commutator estimates.
The first one corresponds to Lemma 2.100 of \cite{B-C-D} (see also Remark 2.101 therein), stated in the framework of Sobolev spaces.
\begin{lemma}\label{l:CommBCD}
Let $\s>d/2$ and assume that $v \in H^\s$. Then one has
\begin{equation*}
\forall\, f \in H^\s\,, \qquad\qquad  2^{j\s} \left\| \big[ v \cdot \nabla, \Delta_j \big] f  \right\|_{L^2}\,\lesssim\,
c_j\, \Big( \|\nabla v \|_{L^\infty}\, \| f \|_{H^\s} \,+\, \|\nabla v \|_{H^{\s-1}}\, \|\nabla f \|_{L^\infty} \Big)\,,
\end{equation*}
for some sequence $\big(c_j\big)_{j\geq -1}$ belonging to the unit ball of $\ell^2$. 
\end{lemma}

The next commutator estimate is contained in Lemma 2.99 of \cite{B-C-D} (see also Remark A.11 of \cite{C-F} for the last part of the statement).
\begin{lemma}\label{l:ParaComm}
Let $\psi$ be a smooth function on $\mathbb{R}^d$, which is homogeneous of degree $m$ away from a neighbourhood of $\,0$.
Then, for any vector field $v$ such that $\nabla v \in L^\infty$ and for any $\s\in\R$, one has:
\begin{equation*}
\forall\, f \in H^\s\,, \qquad\qquad \left\| \big[ \mathcal{T}_v, \psi(D) \big] f \right\|_{H^{\s-m+1}}\, \lesssim\,
\|\nabla v\|_{L^\infty} \|f\|_{H^\s}\,.
\end{equation*}
The same result, with $m=0$, holds true even if $\psi$ is a bounded homogeneous function of degree zero. In particular, the previous estimate,
with $m=0$, applies also if $\psi(D)$ is the Leray-Helmholtz projector $\P$ or its orthogonal projection $\Q\,=\,\Id-\P$.
\end{lemma}

We recall that the \emph{Leray-Helmholtz projector} $\P$ is the $L^2$-orthogonal projection onto the subspace of divergence-free vector fields.
The two orthogonal projectors $\P$ and $\Q\,=\,\Id-\P$ are defined by the formulas
\[ 
\mathbb{P}\,=\, \Id \,+\, \nabla (- \Delta)^{-1}\nabla\cdot\qquad\qquad\mbox{ and }\qquad\qquad
\mbb Q\,=\,-\nabla(-\Delta)^{-1}\nabla\cdot\;.
\] 
The previous formulas have to be interpreted in the sense of Fourier multipliers: for instance,
\[ 
\forall\, f \in L^2(\mathbb{R}^d), \qquad\qquad
\big[\mc F(\mathbb{P} f)\big]_j(\xi) = \sum_{k=1}^d\left( 1 - \frac{\xi_j \xi_k}{|\xi|^2} \right) \what{f_k}(\xi)\,,
\] 
and analogously for $\mbb Q$.

We also recall that both $\P$ and $\Q$ are singular integral operators.
By Calder\'on-Zygmund theory, they are therefore continuous operators from $L^p$ into itself for any $1<p<+\infty$.
Thanks to this property, it is easy to see that $\P$ and $\Q$ are also continuous
from $B^s_{p,r}$ into itself, for all $(s, r) \in \mathbb{R} \times [1, +\infty]$, as long as $1 < p < +\infty$.
In particular, they act continuously on every space $H^\s$, for any $\s\in\R$.

\subsection{Transport and elliptic estimates in Sobolev spaces} \label{ss:tools-est}

In this part of Section \ref{s:tools}, we apply Littlewood-Paley theory and dyadic characterisation of Sobolev spaces to derive Sobolev estimates
for smooth solutions to transport equations and elliptic equations with variable coefficients.

\medbreak
First of all, we focus on the case of transport equations in Sobolev spaces $H^\s(\R^2)$.
We refer to Chapter 3 of \cite{B-C-D} for additional details. We study the initial value problem
\begin{equation}\label{eq:TV}
\begin{cases}
\partial_t f\, +\, v \cdot \nabla f\, = \,g \\
f_{|t = 0}\, =\, f_0\,.
\end{cases}
\end{equation}
The velocity field $v=v(t,x)$ will always assumed to be divergence-free, \tsl{i.e.} $\nabla\cdot v = 0$, and Lipschitz continuous with respect to the space variable.

The following statement is an adaptation of Theorems 3.14 and 3.19 of \cite{B-C-D} to our working setting $H^\s(\R^2)$.
\begin{theorem}\label{th:transport}
Let $\s>0$ and $T>0$ be fixed. Let $g \in L^1\big([0,T];H^\s\big)$. Assume that $v$ is a divergence-free vector field such that, for some
$q > 1$ and $M > 0$, $v \in L^q\big([0,T];B^{-M}_{\infty, \infty}\big)$. Assume moreover the following condition:
\begin{itemize}
 \item if $\s\,<\,2$, assume that $~\nabla v\in L^1\big([0,T];B^1_{2,\infty}\cap L^\infty\big)$;
 \item if $\s>2$, assume instead that $~\nabla v\in L^1\big([0,T];H^{\s-1}\big)$.
\end{itemize}

Then, for any initial datum $f_0 \in H^{\s}$, the transport equation \eqref{eq:TV} has a unique solution $f$ in the space 
$C\big([0,T];H^\s\big)$. Moreover, after defining the function $V(t)$ such that
\[
V'(t)\,=\,\left\{
\begin{array}{ll}
\!\|\nabla v(t)\|_{B^1_{2,\infty}\cap L^\infty}\! & \mbox{if}\!\!\!\quad \s<2\,,\\[2ex]
\!\|\nabla v(t)\|_{H^{\s-1}}\! & \mbox{if}\!\!\!\quad \s>2\,, 
\end{array}\right.
\]
the unique solution $f$ satisfies the following estimate, for a suitable universal constant $C>0$:
\begin{equation*} 
\forall\,t\in[0,T]\,,\qquad
\| f(t) \|_{H^\s}\, \leq\, e^{C\,V(t)}\,\left(\| f_0 \|_{H^\s} + \int_0^t e^{-C\,V(\tau)}\, \| g(\t) \|_{H^\s} \dd\t\right)\,.
\end{equation*}
In the case when $v=f$, the previous estimate holds true with $V'(t)\,=\,\left\|\nabla f(t)\right\|_{L^\infty}$.
\end{theorem}

\begin{rem} \label{r:transport}
In particular, as an immediate consequence of the previous theorem, we deduce the following estimate for
solutions to the transport problem \eqref{eq:TV}: for $s>2$ and for any $0<\s\leq s$, one has
\begin{equation*} 
\forall\,t\in[0,T]\,,\qquad
\| f(t) \|_{H^{\s}}\, \leq\, 
\exp\left( C\!\! \int_0^t \| \nabla v \|_{H^{s-1}} \right)
\left( \| f_0 \|_{H^{\s}} + \int_0^t \exp \left( - C\!\! \int_0^\t \| \nabla v \|_{H^{s-1}} \right)\, \| g(\t) \|_{H^{\s}}  \dd \t  \right).
\end{equation*}
\end{rem}

Now, we turn our attention to the following elliptic equation:
\begin{equation}\label{eq:elliptic}
-\,\div(a\,\nabla\Pi)\,=\,\div F \qquad\qquad\qquad \mbox{ in }\quad \R^2\,,
\end{equation}
where $a=a(x)$ is a given smooth bounded function satisfying 
\begin{equation}\label{eq:ellipticity}
a_*\,:=\,\inf_{x\in\R^2}a(x)\,>\,0\,.
\end{equation}
We shall use  the following  result,
based on Lax-Milgram's theorem (this is Lemma 2 in \cite{D_2010}). 
\begin{lemma}\label{l:laxmilgram}
For all vector field $F$ with coefficients in $L^2$, there exists a tempered distribution $\Pi$,
unique up to  constant functions, such that  $\nabla\Pi\in L^2$ and  
equation $\eqref{eq:elliptic}$ is satisfied. 
In addition, we have 
$$
a_*\,\|\nabla\Pi\|_{L^2}\,\leq\,\|F\|_{L^2}\,.
$$
\end{lemma}

Next, let us discuss how to derive higher regularity estimates for equation \eqref{eq:elliptic}. As we will explain later on,
such estimates are not really needed in the present work; however their derivation is instructive, because it will suggest us how to deal with the pressure function
appearing in \eqref{eq:fluid_odd_visco1}.

In order to control higher order Sobolev norms of \eqref{eq:elliptic}, we follow the approach of \cite{D-F}.
Thus, let $\s\geq1$.
First of all, applying \eqref{est:basic} with $k=1$, we see that
\[
 \|\nabla\Pi\|_{H^\s}\,\lesssim\,\|\nabla\Pi\|_{L^2}\,+\,\|\Delta\Pi\|_{H^{\s-1}}\,.
\]
So, we seek an estimate on $\Delta \Pi$. We notice that, using the ellipticity condition \eqref{eq:ellipticity},
we can work on equation \eqref{eq:elliptic} to get a relation for that quantity:
\[
-\,\Delta\Pi\,=\,\nabla\log a\cdot\nabla\Pi\,+\,\frac{1}{a}\,\div F\,.
\]
The key point is that, when performing $H^{\s-1}$ estimates for $\Delta \Pi$, the first term on the right-hand side of the previous equation is of lower order,
thus an interpolation argument, an application of the Young inequality and Lemma \ref{l:laxmilgram} allow to close the estimates for $\Delta \Pi$
in terms of suitable norms of $a$ and $F$ only.

In our analysis, we will need to resort to this kind of argument in order to study the pressure term $\nabla \pi$ in system \eqref{eq:fluid_odd_visco1}.
However, we will need to go much further, and to find a precise expression of $\nabla \pi$ in order to justify
the heuristics $\nabla \pi \approx \nabla \omega$ (keep in mind \eqref{eq:odd_homog}, where in fact $\vphi=-\o$ by the Biot-Savart law)
with a high degree of precision.

\subsection{Time-dependent Besov spaces and transport-diffusion equations} \label{ss:transp-diff}

In the present subsection, we recall the definition of time-dependent Besov spaces, often named \emph{Chemin-Lerner spaces}
(introduced for the first time in the paper \cite{Chem-Ler}),
and their use in establishing sharp estimates for smooth solutions to transport-diffusion equations.

The reason for separating this part from the previous Subsections \ref{ss:para} and \ref{ss:tools-est}
is to have here a somehow self-cointained toolbox to be used for the existence proof.
As a matter of fact, the definitions and results of this part will find an application 
only in Section \ref{s:Stokes} and Subsection \ref{ss:proof-e},
where we will construct smooth approximate solutions to equations \eqref{eq:fluid_odd_visco1} by a
viscous regularisation of that system.

\medbreak
Introducing Chemin-Lerner spaces relies on the observation that,
when solving evolutionary PDEs, it is natural to use
spaces of type $L^q_T\big(X\big)=L^q\big([0,T];X\big)$, with $X$ denoting some Banach space. However, when $X$ is a Besov space (as in our case),
one has to localise first the equations by Littlewood-Paley decomposition and make estimates for each dyadic block. This will provide one
with  estimates of the $L^q_T\big(L^p\big)$ norm of  each dyadic block $\Delta_j$ \emph{before} performing the $\ell^r$ summation over $j$.
This  leads to the following definition (see \cite{Chem-Ler}, \cite{B-C-D}).
\begin{definition}\label{def:Besov,tilde}
Let $s\in \R$, $(q,p,r)\in [1,+\infty]^3$ and $T\in [0,+\infty]$. We set
$$
\|u\|_{\wtilde L^q_T(B^s_{p,r})}
\,:=\,\left\| \Bigl(  2^{js}  \|\Delta_j u(t)\|_{L^q_T(L^p)} \Bigr)_{j\geq -1}\right\|_{\ell^r}\,.
$$
We also set $\wtilde C_T\big(B^s_{p,r}\big)\,:=\,\wtilde L_T^\infty\big(B^s_{p,r}\big)\cap C\big([0,T];B^s_{p,r}\big)$.
\end{definition}

The relation between this new class of spaces and the classical spaces $L^q_T\big(B^s_{p,r}\big)$ can be easily recovered by use of the Minkowski
inequality: we have
\begin{equation} \label{est:Chem-Ler}
\|u\|_{\wtilde{L}^q_T(B^s_{p,r})}\;\leq\;\|u\|_{L^q_T(B^s_{p,r})}\qquad \mbox{ if }\ q\,\leq\,r\,,\qquad\qquad\qquad
\|u\|_{\wtilde{L}^q_T(B^s_{p,r})}\;\geq\;\|u\|_{L^q_T(B^s_{p,r})} \qquad \mbox{ if }\ q\,\geq\,r\,.
\end{equation}
In particular, in the case of Sobolev spaces, for which $p=r=2$, we have
\[
L^1_T\big(H^\s\big)\,\hookrightarrow\,\wtilde L^1_T\big(H^\s\big)\qquad\qquad\mbox{ and }\qquad\qquad
\wtilde L^\infty_T\big(H^\s\big)\,\hookrightarrow\,L^\infty_T\big(H^\s\big)\,.
\]
However, owing to the embeddings of Proposition \ref{p:embed}, we also have the following inclusions, which will be of constant use in Section \ref{s:Stokes}:
\begin{equation} \label{est:emb-time}
\forall\,\delta>0\,,\qquad\qquad \wtilde L^1_T\big(H^{\s+\de}\big)\,\hookrightarrow\, L^1_T\big(H^\s\big)\,.
\end{equation}

Time-dependent Besov spaces behave well with respect to product and paralinearisation operations. For instance, concerning
products, Bony's decomposition \eqref{eq:bony} and Proposition \ref{p:op} imply the following statement.
\begin{cor}\label{c:op_time}
Let $s>0$ and $(p,r,q_1\ldots q_4)\in[1,+\infty]^6$ such that
\[
\frac 1q\,:=\,\frac{1}{q_1}\,  +\,\frac{1}{q_2}\,=\,\frac{1}{q_3}\, +\,\frac{1}{q_4}\,,\qquad\qquad\mbox{ with }\qquad q\in[1,+\infty]. 
\]

Then, there exists a constant $C>0$, depending only on $(d, s, p, r)$, such that 
$$
\|uv\|_{\wtilde L^q_T(B^s_{p,r})}\,
\leq\, C\,\left(\|u\|_{L^{q_1}_T(L^\infty)}\,\|v\|_{\wtilde L^{q_2}_T(B^{s}_{p,r})}
\,+\, \|u\|_{\wtilde L^{q_3}_T(B^{s}_{p,r})}\, \| v\|_{L^{q_4}_T(L^\infty)}\right)\,.
$$
\end{cor}

From the previous result, it goes without saying that the equivalent versions in Chemin-Lerner spaces of the commutator estimates stated above hold true
as well. In particular, the following result was proved in \cite{D_2006} (see Lemma 8.11 of that work).
\begin{prop} \label{p:comm-time}
Fix $\s>0$ and $T>0$. Assume that $\nabla\cdot v=0$.
Then, there exists a constant $C=C(\s,d)>0$, depending only on the quantities inside the brackets, such that
\[
\left(\sum_{j\geq-1}2^{2\,j\,\s}\,\left\|\big[v\cdot\nabla,\Delta_j\big]u\right\|_{L^1_T(L^2)}\right)^{1/2}\,\lesssim\,
\left\{
\begin{array}{lcl}
\displaystyle
\int^T_0\left\|\nabla v(\t)\right\|_{L^\infty\cap B^{d/2}_{2,\infty}}\,\left\|\nabla u(\t)\right\|_{H^{\s-1}}\,\dd \t 
& \qquad\mbox{ if } & 0<\s<1+\dfrac{d}{2} \\[4ex]
\displaystyle
\int^T_0\left\|\nabla v(\t)\right\|_{H^{\s-1}}\,\left\|\nabla u(\t)\right\|_{H^{\s-1}}\,\dd \t & \qquad\mbox{ if } & \s>1+\dfrac{d}{2}\,.
\end{array}
\right.
\]
In the case when $v=u$, for any $\s\geq0$ one has
\[
\left(\sum_{j\geq-1}2^{2\,j\,\s}\,\left\|\big[v\cdot\nabla,\Delta_j\big]u\right\|_{L^1_T(L^2)}\right)^{1/2}\,\lesssim\,
\int^T_0\left\|\nabla u(\t)\right\|_{L^\infty}\,\left\|\nabla u(\t)\right\|_{H^{\s-1}}\,\dd \t\,.
\]
\end{prop}


As for composition with smooth functions, we have the following counterpart of Proposition \ref{p:comp} above. Once again, we limit ourselves
to consider the framework of Sobolev spaces.
\begin{prop}\label{p:comp_time}
Let $I$ be an open  interval of $~\R$ and let $F:I\rightarrow\R$ be a smooth function. 

Then for all compact subset $J\subset I$, $T>0$, $\s>0$ and $q\in[1,+\infty]$, there exists a constant $C$
such that, for all function $a$ valued in $J$ and with gradient in $\wtilde L^q_T\big(H^{\s-1}\big)$,  we have
$\nabla(F\circ a)\in \wtilde L^q_T\big(H^{\s-1}\big)$ together with the estimate
$$
\|\nabla(F\circ a)\|_{\wtilde L^q_T(H^{\s-1})}\,\leq\,C\,\|\nabla a\|_{\wtilde L^q_T(H^{\s-1})}\,.
$$
If furthermore   $F(0)=0$ and $a\in \wtilde L^q_T\big(H^\s\big)$, then one has
$\left\|F(a)\right\|_{\wtilde L^q_T(H^\s)}\,\leq\,C(F',\|a\|_{L^\infty_T(L^\infty)})\,\|a\|_{\wtilde L^q_T(H^\s)}$. 
\end{prop}

Chemin-Lerner spaces behave well also with respect to interpolation: in the class of Sobolev spaces, we have
the inequality
\begin{equation} \label{est:CL-interp}
\left\|u\right\|_{\wtilde L^q_T(H^\s)}\,\leq\,\left\|u\right\|_{\wtilde L^{q_1}_T(H^{\s_1})}^\beta\;\left\|u\right\|_{\wtilde L^{q_2}_T(H^{\s_2})}^{1-\beta}
\qquad \mbox{ provided }\quad \frac{1}{q}\,=\,\frac{\beta}{q_1}\,+\,\frac{1-\beta}{q_2}\quad \mbox{ and }\quad \s\,=\,\beta\,\s_1\,+\,(1-\beta)\,\s_2\,.
\end{equation}

Chemin-Lerner spaces are particularly useful when solving transport-diffusion equations of the type
\begin{equation}\label{eq:T-diff}
\begin{cases}
\partial_t f\, +\, v \cdot \nabla f\,-\,\nu\,\Delta f\, = \,g \\
f_{|t = 0}\, =\, f_0\,.
\end{cases}
\end{equation}
As a matter of fact, the use of such spaces allows to obtain a gain
of two full derivatives which otherwise would be out of reach in the classical $L^q_T\big(B^s_{p,r}\big)$ setting.
This is stated by the following theorem, which is the adptation of Theorem 3.38 of \cite{B-C-D} to the Sobolev spaces framework
$H^\s(\R^2)$.

\begin{theorem} \label{t:tr-diff}
Let $\nu\geq0$, $T>0$ and $\s>0$ be fixed. Let the vector field $v$ satisfy the same assumptions as in Theorem \ref{th:transport}, and define
the function $V'(t)$ in the same way. Take $g\,\in\,\wtilde L^1_T\big(H^\s(\R^2)\big)$ and $f_0\,\in\,H^\s(\R^2)$.

Then, there exists a constant $C>0$, depending only on the regularity index $\s$, such that, for any smooth solution
$f$ of equation \eqref{eq:T-diff} and any $q\in[1,+\infty]$, one has
\begin{align*}
\nu^{1/q}\,\left\|f\right\|_{\wtilde L^q_T(H^{\s+2/q})}\,\leq\,C\,(1+\nu T)^{1/q}\,e^{C\,V(T)\,(1+\nu T)^{1/q}}\,\left(\left\|f_0\right\|_{H^\s}\,+\,
\left\|g\right\|_{\wtilde L^1_T(H^\s)}\right)\,.
\end{align*}
\end{theorem}

Notice that the previous theorem holds also in the limit case $\nu=0$ and $q=+\infty$,
for which it yields a finer result than the one of Theorem \ref{th:transport}
above. We will need this specific result in Section \ref{ss:proof-e}, when constructing smooth approximate solutions to system \eqref{eq:fluid_odd_visco1}.
We point out that, in Sections \ref{s:Stokes} and \ref{s:proof}, we will have to deal with transport-diffusion equations having diffusion operators
of higher order and, more importantly, with variable coefficients: so, Theorem \ref{t:tr-diff} gives us a flavour of the kind
of estimates one can hope for, paying the price of higher technicalities for handling the presence of the variable coefficients.

\medbreak
To conclude this part, we come back to the elliptic equation \eqref{eq:elliptic}, under the ellipticity condition \eqref{eq:ellipticity},
and present a time-dependent version of Lemma \ref{l:laxmilgram}, where in addition higher order estimates are established.
This is a reduced version of Proposition 8.5 from \cite{D_2006}, which however contains all what is needed for our study.
\begin{prop} \label{p:ell-time}
Let $\s>1$ and $T>0$ fixed. Let $a=a(t,x)$ be a scalar function such that
\[
a-1\,\in\,\wtilde L^\infty_T\big(H^\s\big)\qquad\qquad \mbox{ and }\qquad\qquad a_*\,:=\,\inf_{[0,T]\times\R^2}a\,>\,0\,.
\]

Then, there exists a constant $C>0$, depending only on $\|a\|_{L^\infty([0,T]\times\R^2)}$, such that, for any $F\,\in\,\wtilde L^1_T\big(H^\s\big)$,
the unique (up to additive constants) solution of the elliptic equation \eqref{eq:elliptic} satisfies the following estimate:
\[
a_*\,\left\|\nabla \Pi\right\|_{\wtilde{L}^1_T(H^\s)}\,\leq\,
C\,\left(1\,+\,a_*^{-1}\,\left\|\nabla a\right\|_{\wtilde L^\infty_T(H^{\s-1})}\right)^{\s}\,\left\|\Q F\right\|_{\wtilde L^1_T(H^\s)}\,.
\]
In addition, the solution operator $\mc Q_a:\,F\,\mapsto\,\nabla\Pi$ associated to equation \eqref{eq:elliptic}
acts continuously from the space $\wtilde L^1_T\big(H^\s\big)$ into itself.
\end{prop}

\section{Estimates for smooth solutions} \label{s:a-priori}

In this section we show \tsl{a priori} estimates for a smooth solution $(\rho,u,\nabla \pi)$ to system \eqref{eq:fluid_odd_visco1}, assuming that
such a solution exists and is defined on $\R_+\times\R^2$.
We fix $s>2$. The main goal is to obtain, in some possibly small time interval, a uniform bound for the energy of the solution
\begin{equation} \label{def:E}
E(t)\,:=\,\left\|\rho(t)-1\right\|_{H^{s+1}}\,+\,\left\|u(t)\right\|_{H^s}
\end{equation}
in terms of the positivity constants $\rho_*$ and $\rho^*\,:=\,\left\|\rho_0\right\|_{L^\infty}$, of $\left\|\rho_0-1\right\|_{L^2}$ and of the initial kinetic energy
$\left\|u_0\right\|_{L^2}$ and of the energy of the initial datum
\[
 E_0\,:=\,\left\|\rho_0-1\right\|_{H^{s+1}}\,+\,\left\|u_0\right\|_{H^s}\,.
\]
Notice that the function $E(t)$ does not involve the pressure term $\nabla\pi$. However, a control of $\nabla\pi$ in suitable Sobolev norms
will be needed in order to find the sought bound for $E(t)$: we will come back to this issue in Subsection \ref{ss:high}.

The main result of this section is contained in the following statement.

\begin{theorem}\label{t:uniform_estimates}
Let $ \pare{\rho, u, \nabla \pi} $ be a smooth global solution of system \eqref{eq:fluid_odd_visco1}, stemming from the smooth initial datum
$ \pare{\rho_0, u_0} $.

Then, there exists a time $ T = T\pare{\rho_0, u_0} > 0 $, a constant $C\,=\,C\left(s,\rho_*,\rho^*,\left\|\rho_0-1\right\|_{L^2},\|u_0\|_{L^2}\right)\,>\,0$
(only depending on the quantities inside the brackets), and a non-decreasing function
$ \cG \in L^\infty_{\loc}\pare{\bR_+; \bR_+}$, with $ \cG\pare{0}=0 $, such that
\begin{align} \label{est:energy_to-prove}
\forall\,t\,\in\,[0,T]\,,\qquad\qquad 
E\pare{t}\,\leq\,C\,\cG\pare{  E_0 }\,. 
\end{align}
In addition, given any finite $T^*>0$, if
\begin{align*}
&\int_0^{T^*}
\Big(\left\|\nabla u(t)\right\|^2_{L^\infty}\,+\,\left\|\nabla \rho(t)\right\|^s_{L^\infty}\,+\,
\left\|\nabla\rho(t)\right\|^{s-1}_{L^\infty}\,\left\|\nabla u(t)\right\|_{L^\infty}\,+\,\left\|\Delta \rho(t)\right\|_{L^\infty}\,+\,
\left\|\nabla\pi(t)\right\|^{s/(s-1)}_{L^\infty} 
\Big)\,\dt\,<\,+\,\infty\,,
\end{align*}
then inequality \eqref{est:energy_to-prove} holds up to time $T=T^*$ (possibly with a different multiplicative constant $C>0$). The same conclusion holds true
also if
\begin{align*}
&\int_0^{T^*}
\Big(\left\|\nabla u(t)\right\|^2_{L^\infty}\,+\,\left\|\nabla \rho(t)\right\|^s_{L^\infty}\,+\,
\left\|\nabla\rho(t)\right\|^{\max\{2,s-1\}}_{L^\infty}\,\left\|\nabla u(t)\right\|_{L^\infty} \\
&\qquad\qquad\qquad\qquad\qquad\qquad\qquad\qquad\qquad\qquad
\,+\,\left\|\Delta \rho(t)\right\|_{L^\infty}\,+\,
\left\|\nabla\big(\pi(t)-\rho(t)\o(t)\big)\right\|^{s/(s-1)}_{L^\infty} 
\Big)\,\dt\,<\,+\,\infty\,.
\end{align*}
\end{theorem}

As the proof of the previous theorem is rather intrincate, we divide it into several steps.

\begin{enumerate}[(i)]
 \item In Subsection \ref{ss:low}
we obtain estimates on the low regularity norms of the solution, namely on suitable Lebesgue norms of $\rho$ and $u$.

\item In Subsection \ref{ss:new-en} we introduce two new quantities, the vorticity $\o=\nabla^\perp\cdot u$ of the fluid and $\theta$,
which we will refer to as the \emph{good unknowns} of the system, and an auxiliary unknown $\eta$. There,
we also define a new energy of the solution in terms of those quantities. As a matter of fact, $\o$ and $\theta$ are strictly linked with
the solutions, but they solve easier equations; $\eta$ will represent the bridge to switch from one energy to the other.

\item In Subsection \ref{ss:high} we will propagate higher order regularity norms of the solution: this will be done by rather bounding
the higher order norms of $\o$ and $\theta$, instead of the ones of $\rho$ and $u$ directly.

\item In Subsection \ref{ss:est-close} we will close the estimates and show the claimed uniform bound for $E(t)$ in some
interval $[0,T]$, for a suitable (possibly small) time $T>0$.

\item In Subsection \ref{ss:a-priori_cont}, we will prove the two blow-up criteria.
\end{enumerate}

\subsection{Estimates for the low regularity norms} \label{ss:low}

Here we show estimates for suitable $L^p$ norms of the (supposed to exist)
smooth functions $\rho$ and $u$, solving \eqref{eq:fluid_odd_visco1} on $\R_+\times\R^2$.

First of all, by assumption, there exists a constant $\rho^*\geq\rho_*$ such that
\begin{equation} \label{est:rho_0-inf}
\forall\,x\in\R^2\,,\qquad\qquad \rho_*\,\leq\,\rho_0(x)\,\leq\,\rho^*\,.
\end{equation}
Thus, because $\rho$ is transported by the smooth divergence-free vector field $u$, one immediately obtains that
\begin{equation} \label{est:rho-inf}
\forall\,(t,x)\in\R_+\times\R^2\,,\qquad\qquad \rho_*\,\leq\,\rho(t,x)\,\leq\,\rho^*\,.
\end{equation}
Analogously, as also the quantity $\rho-1$ is transported by $u$, we get
\begin{equation} \label{est:rho-2}
\forall\,t\geq0\,,\qquad\qquad \left\|\rho(t)-1\right\|_{L^2}\,=\,\left\|\rho_0-1\right\|_{L^2}\,.
\end{equation}
As a matter of fact, the same equality holds in $L^p$, for any $p\in[2,+\infty]$.

Next, we perform an energy estimate on the momentum equation, namely the second equation appearing in system \eqref{eq:fluid_odd_visco1}.
Since the velocity field $u$ is divergence-free, the pressure term identically vanishes when multiplied in $L^2$ by $u$. In addition, we observe that
the odd viscosity term is skew-symmetric with respect to the $L^2$ scalar product, so it does not contribute to the energy balance of the system:
\[
 \int \nabla\cdot\big(\rho\,\nabla u^\perp\big)\cdot u\,\dd x\,=\,\sum_{j=1,2}\int\d_j\big(\rho\,\d_j u^\perp\big)\cdot u\,\dd x\,
=\,-\,\sum_{j=1,2}\int\rho\,\d_j u^\perp\cdot\d_j u\,\dd x\,=\,0\,.
\]
Finally, the transport term can be dealt with in a classical way: we can combine it with the term presenting the time derivative and use the mass equation to
get
\[
\int\Big(\d_t(\rho u)\,+\,\nabla\cdot(\rho u\otimes u)\Big)\cdot u\,\dd x\,=\,\frac{\dd}{\dd t}\int\rho\,|u|^2\,\dd x\,.
\]
In light of the previous computations, in the end we get that this last term has to vanish, that is, after an intergration in time, one must have
\begin{equation} \label{est:momentum}
\forall\,t\geq0\,,\qquad\qquad
\left\|\sqrt{\rho(t)}u(t)\right\|_{L^2}^2\,=\,\left\|\sqrt{\rho_0}u_0\right\|_{L^2}^2\,.
\end{equation}
Using \eqref{est:rho-inf}, we immediately deduce that
\begin{equation} \label{est:u-2}
\forall\,t\geq0\,,\qquad\qquad
\left\|u(t)\right\|_{L^2}\,\lesssim\,\left\|u_0\right\|_{L^2}\,,
\end{equation}
where the (implicit) multiplicative constant depends on $\rho_*$ and $\rho^*$. 

\medbreak

As mentioned in the Introduction, higher order estimates performed on equations \eqref{eq:fluid_odd_visco1} would cause an apparently catastrophic
loss of derivatives. In order to propagate high regularity norms of the solution, we need instead to introduce the \emph{good unknowns} linked to our system:
this is the scope of the next subsection.

\subsection{A new energy} \label{ss:new-en}
Here, we are going to explore more in detail the structure of the odd viscosity term and introduce the good
unknowns of our system. As we will see in Subsection \ref{ss:high}, this analysis will enable us to bypass the
two main difficulties which have been put in evidence in Paragraph \ref{sss:loss} of the Introduction, linked with the propagation of a
higher Sobolev norm of the density and with the regularity of the pressure term.


\subsubsection{The \emph{good unknowns} of the system} \label{sss:good}

Inspired on the one hand by what happens for inviscid incompressible fluids wih variable density (see \cite{D-F} and \cite{F_2012}) and, on the other hand,
by some structure appearing in the context of fast rotating fluids \cite{F-G}, the idea we want to explore here is to resort to the ``vorticity functions''
\[
\o\,:=\,\nabla^\perp\cdot u\,=\,\d_1u_2-\d_2u_1\qquad\qquad \mbox{ and }\qquad\qquad \eta\,:=\,\nabla^\perp\cdot\big(\rho\,u\big)
\]
in order to propagate higher order regularity norms of the solution. Notice that $\o$ is the true vorticity of the fluid, while $\eta$ represents
the ``vorticity'' component of the linear momentum $m\,:=\,\rho\,u$.
For later use, we also introduce the quantity
\begin{equation} \label{def:theta}
\theta\,:=\,\eta\,-\,\Delta\rho\,.
\end{equation}

Notice that the momentum $m\,=\,\rho u$ is not divergence-free, so it cannot be fully reconstructed by the knowledge of its vorticity $\eta$.
However, owing to the relation 
\begin{equation} \label{def:eta}
\eta\,:=\, \nabla^\perp\cdot\pare{\rho u}\,=\,\rho\,\omega\,+\,\nabla^\perp\rho\cdot u\,,
\end{equation}
the function $\eta$ can be easily compared with $\omega$.
Indeed, since $s>2$, the space $H^{s-1}$ is a Banach algebra embedded in $L^\infty$. Thus, by definition of $\eta$ and using the paraproduct
estimates of Subsection \ref{ss:para}, we easily get that
\[
\left\|\eta\right\|_{H^{s-1}}\,\lesssim\,\Big(1\,+\,\left\|\rho-1\right\|_{H^{s-1}}\Big)\,\|\o\|_{H^{s-1}}\,+\,\left\|\nabla\rho\right\|_{H^{s-1}}\,\|u\|_{H^{s-1}}\,.
\]
From this, applying repeatedly the arguments used for \eqref{est:u-omega} and \eqref{est:dens-low-high}, based on cutting into low and high frequencies,
we easily get that
\begin{equation} \label{est:eta_1}
\left\|\eta\right\|_{H^{s-1}}\,\lesssim\,\Big(1\,+\,\left\|\rho-1\right\|_{L^2}\,+\,\left\|\Delta\rho\right\|_{H^{s-2}}\Big)\,
\Big(\left\|u\right\|_{L^2}\,+\,\|\o\|_{H^{s-1}}\Big)\,.
\end{equation}

We remark the presence of the $H^{s-2}$ norm of $\Delta\rho$ on the right of the previous inequality: this is a key point of our analysis.
Indeed, on the one hand, estimate \eqref{est:eta_1} implies that 
\begin{equation} \label{est:eta-E}
\left\|\eta\right\|_{H^{s-1}}\,\lesssim\,(1+E)\,E\,,
\end{equation}
where the energy function $E$ has been defined in \eqref{def:E}. Thus, from the definition of $\theta$ we also obtain
\begin{equation} \label{est:sigma-E}
\|\theta\|_{H^{s-1}}\,\lesssim\,\|\eta\|_{H^{s-1}}\,+\,\|\Delta\rho\|_{H^{s-1}}\,\lesssim\,E\,(1\,+\,E)\,.
\end{equation}
On the other hand, we claim that a $H^{s-1}$ bound on $\o$ and $\theta$ together gives us a control on $\Delta\rho$ in the same space, so the sought control
on $\rho-1$ in $H^{s+1}$, owing to \eqref{est:dens-low-high}. Thus, $\o$ and $\theta$ are the good quantities to look at, in order to propagate
higher order regularity of $\rho$.
Proving this claim is the goal of the next lemma.

\begin{lemma} \label{l:F}
For any time $t\geq0$, define the new energy function $F(t)$ by the formula
\[ 
F(t)\,:=\,\left\|\rho(t)-1\right\|_{L^2}\,+\,\left\|u(t)\right\|_{L^2}\,+\,\left\|\theta(t)\right\|_{H^{s-1}}\,+\,
\left\|\o(t)\right\|_{H^{s-1}}\,,
\] 
with the convention that $F(0)\,=\,F_0$ is the same quantity computed on the initial data $(\rho_0,u_0,\theta_0)$, where we have set
$~\theta_0\,:=\,\nabla^\perp\cdot(\rho_0u_0)-\Delta\rho_0$.

Then, there exists a universal constant $C>0$, depending only on $s$ but independent of the solution, such that, for any $t\geq0$, one has the following inequalities:
\[
F(t)\,\leq\,C\,E(t)\,\big(1+E(t)\big)\qquad\qquad\mbox{ and }\qquad\qquad 
E(t)\,\leq\,C\,F(t)\,\Big(1\,+\,\big(F(t)\big)^{s-1}\Big)\,.
\]
In particular, for any $T>0$, one has $E\in L^\infty\big([0,T]\big)$ if and only if $F\in L^\infty\big([0,T]\big)$.
\end{lemma}

\begin{proof}
Bounding $F$ in terms of $E$ is an easy consequence of inequality \eqref{est:sigma-E}. Hence, let us prove the reverse estimate. Owing to
\eqref{est:u-omega} and \eqref{est:dens-low-high} again, it is enough to bound the $H^{s-1}$ norm of $\Delta\rho$ in terms of the new energy function $F(t)$.

Our starting point is estimate \eqref{est:eta_1}, which, together with the definition of $\theta$, implies that
\begin{align*}
\left\|\Delta\rho\right\|_{H^{s-1}}\,\lesssim\,\|\eta\|_{H^{s-1}}\,+\,\|\theta\|_{H^{s-1}}\,&\lesssim\,
\Big(1\,+\,\left\|\rho-1\right\|_{L^2}\,+\,\left\|\Delta\rho\right\|_{H^{s-2}}\Big)\,
\Big(\left\|u\right\|_{L^2}\,+\,\|\o\|_{H^{s-1}}\Big)\,+\,\|\theta\|_{H^{s-1}} \\
&\lesssim\,\big(1\,+\,F(t)\big)\,F(t)\,+\,\left\|\Delta\rho\right\|_{H^{s-2}}\,F(t)\,.
\end{align*}
At this point, since $s>2$, we can use the interpolation inequality \eqref{est:interp} with $\s=s$:
\[
\|\Delta\rho\|_{H^{s-2}}\,\lesssim\,\|\rho-1\|_{L^2}^{1-\beta}\;\|\Delta\rho\|_{H^{s-1}}^{\beta}\,,\qquad\qquad\qquad
\mbox{ with }\qquad\quad \beta\,=\,\frac{s}{s+1}\,.
\]
Then, an application of the Young inequality yields, for some universal constant $C>0$ independent of the solution, the bound
\begin{align*}
\left\|\Delta\rho\right\|_{H^{s-1}}\,\leq\,C\,\big(1\,+\,F(t)\big)\,F(t)\,+\,C\,\left\|\Delta\rho\right\|^{\beta}_{H^{s-1}}\,\big(F(t)\big)^{\beta}\,\leq\,
C\,\big(1\,+\,F(t)\big)\,F(t)\,+\,C(s)\,\big(F(t)\big)^{s}\,+\,\frac{1}{2}\,\left\|\Delta\rho\right\|_{H^{s-1}}\,,
\end{align*}
where the last constant $C(s)>0$ only depends on the regularity index $s$. Since $s>2$, we finally get
\begin{equation}
\label{eq:control_density_F}
\left\|\Delta\rho\right\|_{H^{s-1}}\,\lesssim\,F(t)\,+\,\big(F(t)\big)^s\,.
\end{equation}
This completes the proof of the inequality and of the whole lemma.
\end{proof}


%

Therefore, our new goal is to obtain uniform estimates for the new energy function $F$ on some interval $[0,T]$. This will immediately imply the bound
for the energy $E$ of the solution, thus concluding the proof of the \tsl{a priori} estimates.

\subsubsection{The equation for $\theta$} \label{sss:eq-theta}

The key point of all this argument is that $\theta$ solves a very simple equation, namely a transport equation by $u$, with the presence of some forcing terms.
In order to see this, we remark that, differently from \eqref{i_eq:odd_1}, we can write the odd term as
\[
\nabla\cdot\big(\rho\,\nabla u^\perp\big)\,=\,\Delta\big(\rho\,u\big)\,-\,\big(\nabla\rho\cdot\nabla\big)u^\perp\,.
\]
We also observe that, owing to the divergence-free condition on $u$, one has
\[
\nabla\cdot(\rho u\otimes u)\,=\,(u\cdot\nabla)(\rho u)\,.
\]

Using those relations, we can compute an evolution equation for $ \eta $: taking the rotational $\nabla^\perp\cdot$ of the momentum equation
in \eqref{eq:fluid_odd_visco1} yields
\[ 
\d_t\eta + u\cdot \nabla \eta + \left[\nabla^\perp\cdot\,,\,u\cdot\nabla\right](\rho u) + \Delta\div(\rho u) +
\nabla^\perp\cdot\Big(- \Delta\rho \ u^\perp - \pare{\nabla \rho\cdot \nabla} u^\perp\Big) = 0\,.
\] 
Now, since for any vector field $v\in\R^2$ one has $\nabla^\perp\cdot v^\perp= \nabla\cdot v$, owing to the divergence-free condition $\nabla\cdot u=0$ on $u$
we derive fundamental cancellations and get the following equalities:
\begin{align*}
\left[\nabla^\perp\cdot\,,\,u\cdot\nabla\right](\rho u)\,&=\,u_2\,\d_1u\cdot\nabla\rho\,-\,u_1\,\d_2u\cdot\nabla\rho
\,=\,-\,\dfrac{1}{2}\,\nabla^\perp\rho \cdot \nabla\av{u}^2 \\
\nabla^\perp\cdot\big(- \Delta\rho \ u^\perp\big)\, &=\,-\, u\cdot\nabla\Delta\rho \\
\nabla^\perp\cdot\Big(\pare{\nabla \rho\cdot \nabla} u^\perp\Big)\,&=\,
-\,\partial_2\Big(\big( \partial_1\rho\,\partial_1+\partial_2\rho\,\partial_2\big)\big(-u_2\big)\Big)\,+\,
\partial_1\Big(\big( \partial_1\rho\,\partial_1+\partial_2\rho\,\partial_2\big)u_1\Big)\;=\;\mc B\big(\nabla u,\nabla^2\rho\big)\,,
\end{align*}
where the operator $\mc B$ has been defined in \eqref{def:B} above.
In addition, using the first relation in \eqref{eq:fluid_odd_visco1}, we deduce
\begin{equation*}
\Delta\div\pare{\rho u} = -\Delta \d_t\rho\,.
\end{equation*}
As a consequence, we finally find the sought equation for $\theta\,=\,\eta\,-\,\Delta\rho$: we have
\begin{equation} \label{eq:theta}
\d_t\theta\, +\, u\cdot \nabla \theta\, =\,\frac{1}{2}\,\nabla^\perp\rho \cdot \nabla\av{u}^2\,+\,\mc B\pare{\nabla u, \nabla^2 \rho}\,.
\end{equation}

As already pointed out, the previous equation \eqref{eq:theta} allows us to estimate $\theta$ in $H^{s-1}$; this in turn gives us a control
on $\Delta\rho$ in the same space $H^{s-1}$. However, even if this new unknown solves the first issue presented in the Introduction
(recall the discussion in Paragraph \ref{sss:loss}), in order to bound the vorticity $\omega$ in $H^{s-1}$ we still have to understand
how to deal with the pressure term. This will be done in Paragraph \ref{sss:pressure}.

\subsection{High regularity estimates} \label{ss:high}

As we have already explained, in order to close the estimates for the new energy $F$ we need a bound for $\omega$ and $\theta$ in $H^{s-1}$. Finding such bounds is the goal of the present subsection.

\subsubsection{Estimates for $\theta$} \label{sss:theta-est}
We start by working on the quantity $\theta$, which solves the transport equation \eqref{eq:theta}.
Applying the transport estimate of Remark \ref{r:transport} with $\s=s-1$ to that equation, we immediately obtain that
\[
\left\|\theta(t)\right\|_{H^{s-1}}\,\leq\,\exp\left(C\int^t_0\left\|\nabla u(\t)\right\|_{H^{s-1}}\,d\t\right)\,
\left(\left\|\theta_0\right\|_{H^{s-1}}\,+\,\int^t_0\Big(\left\|\nabla^\perp\rho \cdot \nabla\av{u}^2\right\|_{H^{s-1}}\,+\,
\left\|\mc B\big(\nabla u,\nabla^2\rho\big)\right\|_{H^{s-1}}\Big)\,\dd\t\right)
\]
for all $t\geq0$, where we recall that we have set $\theta_0\,=\,\nabla^\perp\cdot(\rho_0u_0)-\Delta\rho_0$.

In order to control the first term appearing inside the time integral on the right-hand side of the previous estimate, we observe that
$\nabla^\perp\rho \cdot \nabla\av{u}^2$ is in fact a trilinear term, as written above:
\begin{equation} \label{eq:trilinear}
\nabla^\perp\rho \cdot \nabla\av{u}^2\,=\,-\,2\,\left(u_2\,\d_1u\cdot\nabla\rho\,-\,u_1\,\d_2u\cdot\nabla\rho\right)\,.
\end{equation}
Since $s>2$, the space $H^{s-1}$ is still a Banach algebra; hece, we can apply the product estimates of Corollary \ref{c:tame}
to bound it in the following way:
\begin{equation} \label{est:trilin}
\left\|\nabla^\perp\rho \cdot \nabla\av{u}^2\right\|_{H^{s-1}}\,\lesssim\,\left\|\nabla\rho\right\|_{H^{s-1}}\,\left\|u\right\|_{H^s}^2\,.
\end{equation}

Similarly, we can estimate
\[
\left\|\mc B\big(\nabla u,\nabla^2\rho\big)\right\|_{H^{s-1}}\,\lesssim\,\left\|\nabla u\right\|_{H^{s-1}}\,\left\|\nabla^2\rho\right\|_{H^{s-1}}\,.
\]
At this point, using the definition of the function $\theta$, we write
\[
\nabla^2\rho\,=\,\nabla^2(-\Delta)^{-1}\Delta\rho\,=\,\nabla^2(-\Delta)^{-1}\Big(\eta\,-\,\theta\Big)\,.
\]
Applying Calder\'on-Zygmund theory, from the previous expression we deduce that
\[
\left\|\nabla^2\rho\right\|_{H^{s-1}}\,\lesssim\,\left\|\eta\right\|_{H^{s-1}}\,+\,\left\|\theta\right\|_{H^{s-1}}\,\lesssim\,
\Big(1\,+\,\left\|\rho_0-1\right\|_{L^2}\,+\,\left\|\nabla\rho\right\|_{H^{s-1}}\Big)\,\|u\|_{H^s}\,+\,\left\|\theta\right\|_{H^{s-1}}\,,
\]
where we have also used \eqref{est:eta_1} and the preservation \eqref{est:rho-2} of the $L^2$ norm of $\rho-1$. Putting everything together, we finally deduce
\begin{equation} \label{est:B-theta}
\left\|\mc B\big(\nabla u,\nabla^2\rho\big)\right\|_{H^{s-1}}\,\lesssim\,\left\|\nabla u\right\|_{H^{s-1}}\,
\Bigg(\Big(1\,+\,\left\|\nabla\rho\right\|_{H^{s-1}}\Big)\,\left\|u\right\|_{H^s}\,+\,\left\|\theta\right\|_{H^{s-1}}\Bigg)\,,
\end{equation}
for a new multiplicative constant, possibly depending on the $L^2$ norm of $\rho_0-1$.

Plugging \eqref{est:trilin} and \eqref{est:B-theta} into the previous control for $\theta$, we obtain, for all $t\geq0$, the bound
\begin{align*}
\left\|\theta(t)\right\|_{H^{s-1}}\,&\lesssim\,\exp\left(C\int^t_0\left\|\nabla u(\t)\right\|_{H^{s-1}}\,\dd\t\right)\,
\left(\left\|\theta_0\right\|_{H^{s-1}}\,+\,\int^t_0\Big(\big(1\,+\,\left\|\nabla\rho\right\|_{H^{s-1}}\big)\,\left\|u\right\|_{H^s}^2\,+\,
\left\|\nabla u\right\|_{H^{s-1}}\left\|\theta\right\|_{H^{s-1}}\Big)\,\dd\t\right) \\
&\lesssim\,\exp\left(C\int^t_0F(\t)\,\dd\t\right)\,
\left(\left\|\theta_0\right\|_{H^{s-1}}\,+\,\int^t_0\Big(\big(1\,+\,E(\t)\big)\,\big(F(\t)\big)^2\,+\,\big(F(\t)\big)^2\Big)\,\dd\t\right)\,,
\end{align*}
where, for passing from the first to the second inequality, we have used the definition \eqref{def:E} of $E$, the equivalence
\eqref{est:u-omega} and the definition of the new energy functional $F$. Finally, using Lemma \ref{l:F}, we conclude that
\begin{equation} \label{est:theta-transp}
\forall\,t\geq0\,,\qquad\qquad
\left\|\theta(t)\right\|_{H^{s-1}}\,\lesssim\,\exp\left(C\int^t_0F(\t)\,\dd\t\right)\,
\left(\left\|\theta_0\right\|_{H^{s-1}}\,+\,\int^t_0\Big(\big(F(\t)\big)^2\,+\,\big(F(\t)\big)^{s+2}\Big)\,\dd\t\right)\,.
\end{equation}

\subsubsection{Analysis of the pressure} \label{sss:pressure}
The energy $F$ contains the term $\|\omega\|_{H^{s-1}}$, as a consequence we need an estimate for $\omega$ in  $H^{s-1}$.
Looking at equation \eqref{i_eq:vort}, it might be natural to seek after a bound on $\nabla \pi$ in the same space $H^{s-1}$. 
However, as already emphasised in the Introduction, a \tsl{na\"if} approach based on the analysis of the pressure equation \eqref{eq:ell-p}
would only give $\nabla \pi\in H^{s-2}$, which is of course not enough to estimate $\o$ in $H^{s-1}$, thus causing a new loss of derivatives in the energy estimates.

The main goal of this paragraph is to make a precise analysis of the pressure term and prove the following representation formula for the pressure:
\begin{align*}
\nabla \pi \,=\,\rho\, \nabla \omega\, +\, \phi\,, && \mbox{ where }\ \phi\ \mbox{ is a remainder of lower order.} 
\end{align*}
Notice that this solves the above mentioned problem linked to the inner structure of the equation. Indeed, inserting
this \tsl{ansatz} into equation \eqref{i_eq:vort}, we will see that the pressure term $\nabla^\perp\big(1/\rho\big)\cdot\nabla \pi$
is nothing but a transport term for $\omega$ (plus lower order remainder terms) by a divergence-free vector field of the form $\nabla^\perp f(\rho)$.
We refer to the next paragraph for the detailed computations.

\medbreak
The starting point of our analysis is the equation \eqref{eq:ell-p} for $\pi$, which we recall here:
\begin{equation} \label{eq:ell-p_2}
-\nabla\cdot\left(\frac{1}{\rho}\,\nabla \pi\right)\,=\,\nabla\cdot\Big((u\cdot\nabla)u\,+\,(\nabla\log\rho\cdot\nabla)u^\perp\Big)\,-\,\Delta\omega\,.
\end{equation}
It is obtained by taking the divergence of the momentum equation in \eqref{eq:fluid_odd_visco1}.

Applying to this equation the elliptic estimates of Lemma \ref{l:laxmilgram}, we immediately get a control for the lower order norm
of the pressure:
\begin{equation} \label{est:p-L^2_mild}
\left\|\nabla \pi\right\|_{L^2}\,\lesssim\,\left\|(u\cdot\nabla)u\,+\,(\nabla\log\rho\cdot\nabla)u^\perp\,+\,\nabla\omega\right\|_{L^2}\,\lesssim\,\left\|u\right\|_{L^2}\,\Big(\left\|\nabla u\right\|_{L^\infty}\,+\,\left\|\nabla\rho\right\|_{L^\infty}\Big)\,+\,
\left\|\nabla\omega\right\|_{L^2}\,.
\end{equation}
Since $s-2>0$, from the previous bound we deduce
\begin{equation} \label{est:p-L^2}
 \left\|\nabla \pi\right\|_{L^2}\,\lesssim\,\left\|\nabla\rho\right\|_{H^{s-1}}\,+\,\|u\|_{H^s}\,\lesssim\,E\,,
\end{equation}
where the implicit multiplicative constant depends also on $\left\|u_0\right\|_{L^2}$.

Next, by expanding the left-hand side of \eqref{eq:ell-p_2}, straightforward computations show that
\begin{align*}
-\,\Delta \pi\,&=\,-\,\nabla\log\rho\cdot\nabla \pi\,+\,\rho\,\nabla\cdot\Big((u\cdot\nabla)u\,+\,
(\nabla\log\rho\cdot\nabla)u^\perp\Big)\,-\,\rho\,\Delta\omega \\
&=\,-\,\nabla\log\rho\cdot\nabla \pi\,+\,\rho\,\nabla\cdot\Big((u\cdot\nabla)u\,+\,(\nabla\log\rho\cdot\nabla)u^\perp\Big)\,-\,
\Delta(\rho\,\omega)\,-\,\big[\rho-1,\Delta\big]\omega\,.
\end{align*}
From this, we formally obtain that
\begin{equation} \label{eq:p-rho_vort}
\nabla\big(\pi\,-\,\rho\,\omega\big)\,=\,\nabla(-\Delta)^{-1}\Big(-\,\nabla\log\rho\cdot\nabla \pi\,+\,\rho\,\nabla\cdot\Big((u\cdot\nabla)u\,+\,(\nabla\log\rho\cdot\nabla)u^\perp\Big)\,-\,\big[\rho-1,\Delta\big]\omega\Big)\,.
\end{equation}

Equation \eqref{eq:p-rho_vort} is the key point to obtain the sought estimate for the vorticity. We will see in the next paragraph how to take advantage of it.

\subsubsection{Estimates for $\omega$} \label{sss:omega}
In this paragraph, we finally carry out $H^{s-1}$ estimates on $\omega$. The key ingredient will be the analysis of the pressure term, performed above.

To begin with, let us compute the evolution equation for $\o$. Owing to the absence of vacuum and writing the odd viscosity term as in equation \eqref{i_eq:odd_1},
we can recast the momentum equation in the following form:
\begin{equation} \label{eq:u}
\d_tu\,+\,(u\cdot\nabla)u\,+\,\frac{1}{\rho}\,\nabla \pi\,+\,\Delta u^\perp\,+\,(\nabla\log\rho\cdot\nabla)u^\perp\,=\,0\,.
\end{equation}
Now, applying the $\curl$ operator $\nabla^\perp\cdot$ and performing similar computations as the ones leading to \eqref{eq:theta} above
yield the sought equation for $\o$:
\begin{equation} \label{eq:vort}
\d_t\omega\,+\,u\cdot\nabla\omega\,+\,\nabla^\perp\left(\frac{1}{\rho}\right)\cdot\nabla \pi\,+\,\mc B\big(\nabla u,\nabla^2\log\rho\big)\,=\,0\,,
\end{equation}
Remark that passing from equation \eqref{eq:u} to equation \eqref{eq:vort} heavily relies on the incompressibility condition over $u$. As a matter of fact,
this fact entails fundamental cancellations in the computation of $\nabla^\perp\cdot\Delta u^\perp\,=\,\nabla\cdot\Delta u=0$ and of
the bilinear term $\mc B$.

Next, we perform estimates on equation \eqref{eq:vort}.

\paragraph{The good.}
First of all, we observe that, owing to the cancellation $\nabla\alpha\cdot\nabla^\perp\alpha=0$, which holds true for any smooth enough
scalar function $\alpha$, one has
\[
 \nabla^\perp\left(\frac{1}{\rho}\right)\cdot\nabla(\rho\,\omega)\,=\,-\,\frac{1}{\rho^2}\,\nabla^\perp\rho\cdot\Big(\rho\,\nabla\omega\,+\,\omega\,\nabla\rho\Big)\,=\,-\,\nabla^\perp\log\rho\cdot\nabla\omega\,.
\]
Therefore, the transport equation \eqref{eq:vort} for $\omega$ can be rewritten as
\begin{equation} \label{eq:vort-2}
\d_t\omega\,+\,\big(u\,-\,\nabla^\perp\log\rho\big)\cdot\nabla\omega\,=\,
-\,\nabla^\perp\left(\frac{1}{\rho}\right)\cdot\nabla\big(\pi\,-\,\rho\,\omega\big)\,-\,\mc B\big(\nabla u,\nabla^2\log\rho\big)\,.
\end{equation}
We remark that the new transport field $u-\nabla^\perp\log\rho$ has still the right space regularity and is divergence-free:
\[
\forall\,t\geq0\,,\qquad\qquad u(t)\,-\,\nabla^\perp\log\rho(t)\;\in\,H^s\qquad \mbox{ and }\qquad \nabla\cdot\left(u(t)\,-\,\nabla^\perp\log\rho(t)\right)\,=\,0\,.
\]
Hence, applying Theorem \ref{th:transport} and Remark \ref{r:transport} to equation \eqref{eq:vort-2} yields, for all $t\geq0$, the bound
\begin{align}
\left\|\o(t)\right\|_{H^{s-1}}\,&\lesssim\,
\exp\left(C\int^t_0\Big(\left\|\nabla u(\t)\right\|_{H^{s-1}}\,+\,\left\|\nabla\rho(t)\right\|_{H^{s}}\Big)\,\dd\t\right) \label{est:vort_part} \\
&\quad \times\,\left(\left\|\o_0\right\|_{H^{s-1}}\,+\,
\int^t_0\Big(\left\|\nabla\rho\right\|_{H^{s-1}}\,\left\|\nabla\big(\pi\,-\,\rho\,\o\big)\right\|_{H^{s-1}}\,+\,
\left\|\mc B\big(\nabla u,\nabla^2\log\rho\big)\right\|_{H^{s-1}}\Big)\,\dd\t\right)\,, \nonumber
\end{align}
where we have used also the fact that $H^{s-1}(\R^2)$ is an algebra under our assumption $s>2$ and that, owing to Proposition \ref{p:comp}, one has
\begin{align*}
\left\|\nabla^2\log\rho\right\|_{H^{s-1}}\,\lesssim\,\left\|\nabla\rho\right\|_{H^s}\qquad\qquad \mbox{ and }\qquad\qquad
\left\|\nabla\left(\frac{1}{\rho}\right)\right\|_{H^{s-1}}\,\lesssim\,\left\|\nabla\rho\right\|_{H^{s-1}}\,.
\end{align*}

Thus, in order to complete our derivation of \tsl{a priori} bounds for \eqref{eq:fluid_odd_visco1}, we need to estimate the $H^{s-1}$ norm of
$\nabla\big(\pi\,-\,\rho\,\o\big)$ and $\mc B\big(\nabla u,\nabla^2\log\rho\big)$.

\paragraph{The bad.}
We start by considering the bilinear term $\mc B\big(\nabla u,\nabla^2\log\rho\big)$. Keeping in mind definition \eqref{def:B} of the operator $\mc B$,
we can explicitly compute the second order derivatives to get
\[
\mc B\big(\nabla u,\nabla^2\log\rho\big)\,=\,\frac{1}{\rho}\,\mc B\big(\nabla u,\nabla^2\rho\big)\,-\,
\frac{1}{\rho^2}\,\Big(\d_1\rho\,\d_2\rho\,\big(\d_1u_2\,+\,\d_2u_1\big)\,+\,\d_1u_1\,\left((\d_1\rho)^2\,-\,(\d_2\rho)^2\right)\Big)\,.
\]
After decomposing $1/\rho\,=\,1\,+\,\big(1/\rho-1\big)$ and using \eqref{est:basic} together with Proposition \ref{p:comp}, we see that
\begin{align*}
\left\|\frac{1}{\rho}\,-\,1\right\|_{H^{s-1}}\,&\lesssim\,\left\|\frac{1}{\rho}\,-\,1\right\|_{L^2}\,+\,
\left\|\nabla\left(\frac{1}{\rho}\right)\right\|_{H^{s-2}} \\
&\lesssim\,\left\|\rho_0\,-\,1\right\|_{L^2}\,+\,\left\|\nabla\rho\right\|_{H^{s-2}}\,,
\end{align*}
for some (implicit) multiplicative constant depending only on $\rho_*$ and $\rho^*$ appearing in \eqref{est:rho-inf}.
Thus, arguing as in \eqref{est:B-theta} and applying the product rules of Corollary \ref{c:tame}, we get
\begin{align*}
\left\|\frac{1}{\rho}\,\mc B\big(\nabla u,\nabla^2\rho\big)\right\|_{H^{s-1}}\,&\lesssim\,
\Big(1\,+\,\left\|\nabla\rho\right\|_{H^{s-1}}\Big)^2\,\left\|u\right\|_{H^s}^2\,+\,
\Big(1\,+\,\left\|\nabla\rho\right\|_{H^{s-1}}\Big)\,\|u\|_{H^s}\,\left\|\theta\right\|_{H^{s-1}} \\
&\lesssim\,\big(1\,+\,E^2\big)\,F^2\,,
\end{align*}
where now the multiplicative constant may also depend on the $L^2$ norm of the initial datum $\rho_0-1$.
In an analogous way, using this time that
\[
 \frac{1}{\rho^2}\,=\,1\,+\,\left(\frac{1}{\rho}\,-\,1\right)\,\left(2\,+\,\frac{1}{\rho}\,-\,1\right)\,,
\]
we find the bound
\begin{align*}
\left\|\frac{1}{\rho^2}\,\Big(\d_1\rho\,\d_2\rho\,\big(\d_1u_2\,+\,\d_2u_1\big)\,+\,\d_1u_1\,\left((\d_1\rho)^2\,-\,(\d_2\rho)^2\right)\Big)\right\|_{H^{s-1}}
\,&\lesssim\,\Big(1\,+\,\left\|\nabla\rho\right\|_{H^{s-2}}^2\Big)\,\left\|\nabla u\right\|_{H^{s-1}}\,\left\|\nabla\rho\right\|^2_{H^{s-1}} \\
&\lesssim\,\big(1\,+\,E^2\big)\,F\,E^2\,.
\end{align*}
Hence, collecting these inequalities and using that the energy $E$ can be bounded in terms of the energy $F$, we find that
\begin{equation} \label{est:B-omega}
 \left\|\mc B\big(\nabla u,\nabla^2\log\rho\big)\right\|_{H^{s-1}}\,\lesssim\,
F^2\,+\,F^{4s+1}\,.
\end{equation}
As above, the multiplicative constant in the previous estimate may depend on the $L^\infty$ and $L^2$ norms of the initial datum $\rho_0-1$.

\paragraph{The ugly.}
We now pass to the control of the term presenting the difference $\nabla\big(\pi-\rho\omega\big)$, appearing in \eqref{eq:vort-2}. As we will see,
its estimate is much more involved than the previous ones. The crucial point of the argument is to make expression \eqref{eq:p-rho_vort} rigorous. For this,
we cut that term into low and high frequencies to get
\[
 \nabla\big(\pi\,-\,\rho\,\omega\big)\,=\,\Delta_{-1}\nabla\big(\pi\,-\,\rho\,\omega\big)\,+\,(\Id-\Delta_{-1})\nabla\big(\pi\,-\,\rho\,\omega\big)\,,
\]
Since the second term on the right is spectrally supported away from the origin, we can now rigorously invert the Laplace operator and write
equation \eqref{eq:p-rho_vort}, on the spectrum of $\Id-\Delta_{-1}$. Thus, if we set
\[
 \Phi\,:=\,-\,\nabla\log\rho\cdot\nabla \pi\,+\,\rho\,\nabla\cdot\Big((u\cdot\nabla)u\,+\,(\nabla\log\rho\cdot\nabla)u^\perp\Big)\,-\,
 \big[\rho-1,\Delta\big]\omega\,,
\]
we immediately infer the equality
\[
 \nabla\big(\pi\,-\,\rho\,\omega\big)\,=\,\Delta_{-1}\nabla\big(\pi\,-\,\rho\,\omega\big)\,+\,(\Id-\Delta_{-1})\nabla(-\Delta)^{-1}\Phi\,.
\]
In addition, as a consequence of this decomposition and the characterisation \eqref{eq:LP-Sob} of Sobolev spaces by
Littlewood-Paley theory, we obtain
\begin{equation} \label{est:p-rho-omega}
\left\| \nabla\big(\pi\,-\,\rho\,\omega\big)\right\|_{H^{s-1}}\,\lesssim\,\left\|\Delta_{-1}\nabla\big(\pi\,-\,\rho\,\omega\big)\right\|_{L^2}
\,+\,\left(\sum_{j\geq0}2^{2j(s-1)}\,\left\|\Delta_j\nabla(-\Delta)^{-1}\Phi\right\|_{L^2}^2\right)^{1/2}\,.
\end{equation}

First of all, we observe that, by the first Bernstein inequality, we have
\begin{equation} \label{est:p-rho-o_L^2}
 \left\|\Delta_{-1}\nabla\big(\pi\,-\,\rho\,\omega\big)\right\|_{L^2}\,\lesssim\,\left\|\nabla\pi\right\|_{L^2}\,+\,\left\|\rho\,\omega\right\|_{L^2}
\,\lesssim\,\left\|\nabla\pi\right\|_{L^2}\,+\,\rho^*\,\left\|\omega\right\|_{L^2}\,.
\end{equation}
Then, using \eqref{est:p-L^2} and the definitions of the energies $E$ and $F$, we infer
\begin{equation} \label{est:p-low}
 \left\|\Delta_{-1}\nabla\big(\pi\,-\,\rho\,\omega\big)\right\|_{L^2}\,\lesssim\,E\,\lesssim\,F\,+\,F^{s}\,,
\end{equation}
for a multiplicative constant which also depends on $\rho^*\,=\,\left\|\rho_0\right\|_{L^\infty}$. 

Next, we consider the sum appearing in \eqref{est:p-rho-omega}: it is tempting to reconduct it to bound a $H^{s-2}$ norm of $\Phi$,
but, as $H^{s-2}$ is not necessarily an algebra, this would entail some problems. Instead, it is better to perform suitable integration by parts, when possible.
Let us write
\begin{align*}
\Phi\,=\,\Phi_1\,+\,\Phi_2\,+\,\Phi_3\,,\qquad\qquad\mbox{ where }\qquad \Phi_1\,&:=\,-\,\nabla\log\rho\cdot\nabla \pi \\
\Phi_2\,&:=\,\rho\,\nabla\cdot\Big((u\cdot\nabla)u\,+\,(\nabla\log\rho\cdot\nabla)u^\perp\Big) \\
\Phi_3\,&:=\,\big[\rho-1,\Delta\big]\omega\,.
\end{align*}

To start, we consider the $\Phi_3$ term, which can be explicitly written as
\[
\Phi_3\,=\,\big[\rho-1,\Delta\big]\omega\,=\,-\,2\,\nabla\rho\cdot\nabla\omega\,-\,\omega\,\Delta\rho\,=\,
-\,2\,\nabla\cdot\big(\omega\,\nabla\rho\big)\,+\,\omega\,\Delta\rho\,.
\]
Notice that both $\omega\,\nabla\rho$ and $\omega\,\Delta\rho$ belong to $H^{s-1}$. Hence, in view also of Proposition \ref{p:LP-H}, we can estimate
\begin{align*}
\left\|\Delta_j\nabla(-\Delta)^{-1}\Phi_3\right\|_{L^2}\,&\lesssim\,
\left\|\nabla(-\Delta)^{-1}\nabla\cdot\Delta_j\big(\omega\,\nabla\rho\big)\right\|_{L^2}\,+\,
\left\|\nabla(-\Delta)^{-1}\Delta_j\big(\omega\,\Delta\rho\big)\right\|_{L^2} \\
&\lesssim\,2^{-j(s-1)}\,\delta_j\,\left\|\omega\right\|_{H^{s-1}}\,\left(\left\|\nabla\rho\right\|_{H^{s-1}}\,+\,2^{-j}\,\left\|\Delta\rho\right\|_{H^{s-1}}\right)\,,
\end{align*}
for a suitable sequence $\big(\delta_j\big)_{j\geq-1}$ belonging to the unit sphere of $\ell^2$. In light of this control, we get 
\begin{equation} \label{est:Phi_3}
\left(\sum_{j\geq0}2^{2j(s-1)}\,\left\|\Delta_j\nabla(-\Delta)^{-1}\Phi_3\right\|_{L^2}^2\right)^{1/2}\,\lesssim\,
\left\|\omega\right\|_{H^{s-1}}\,\left(\left\|\nabla\rho\right\|_{H^{s-1}}\,+\,\left\|\Delta\rho\right\|_{H^{s-1}}\right)\,\lesssim\,F\,E\,
\lesssim\,F^2\,+\,F^{s+1}\,.
\end{equation}

Similarly, for $\Phi_2$ we can write
\[
\Phi_2\,=\,\nabla\cdot\Big(\rho\,(u\cdot\nabla)u\,+\,(\nabla\rho\cdot\nabla)u^\perp\Big)\,-\,\nabla\rho\cdot\Big((u\cdot\nabla)u\,+\,(\nabla\log\rho\cdot\nabla)u^\perp\Big)\,,
\]
from which we deduce that $\Phi_2$ belongs to $H^{s-2}$, together with the estimates
\begin{align*}
\left\|\rho\,(u\cdot\nabla)u\,+\,(\nabla\rho\cdot\nabla)u^\perp\right\|_{H^{s-1}}\,&\lesssim\,\left(1\,+\,\left\|\nabla\rho\right\|_{H^{s-2}}\right)
\,\left\|u\right\|_{H^s}^2\,+\,\left\|\nabla\rho\right\|_{H^{s-1}}\,\|u\|_{H^s} \\
\left\|\nabla\rho\cdot\Big((u\cdot\nabla)u\,+\,(\nabla\log\rho\cdot\nabla)u^\perp\Big)\right\|_{H^{s-1}}&\lesssim\,
\left\|\nabla\rho\right\|_{H^{s-1}}\,\left\|u\right\|_{H^s}^2\,+\,\left\|\nabla\rho\right\|^2_{H^{s-1}}\,\|u\|_{H^s}\,,
\end{align*}
which are derived again from \eqref{est:basic}, Corollary \ref{c:tame} and the paralinearisation result of Proposition \ref{p:comp}.
Thus, we discover that
\begin{align*}
\sum_{j\geq0}2^{2j(s-1)}\,\left\|\Delta_j\nabla(-\Delta)^{-1}\Phi_2\right\|_{L^2}^2\,&\lesssim\,\left(1\,+\,\left\|\nabla\rho\right\|_{H^{s-1}}\right)
\,\left(\left\|u\right\|_{H^s}^2\,+\,\left\|\nabla\rho\right\|_{H^{s-1}}\,\|u\|_{H^s}\right)\,
\lesssim\,\big(1\,+\,E\big)\,\left(F^2\,+\,E\,F\right)\,. 
\end{align*}
where we have used also \eqref{est:u-omega}. As $E\,\lesssim\,F+F^s$, in the end we find that
\begin{equation} \label{est:Phi_2}
\left(\sum_{j\geq0}2^{2j(s-1)}\,\left\|\Delta_j\nabla(-\Delta)^{-1}\Phi_2\right\|_{L^2}^2\right)^{1/2}\,\lesssim\,
\left(1\,+\,F^s\right)\,\left(F^2\,+\,F^{s+1}\right)\,\lesssim\,F^2\,+\,F^{2s+1}\,.
\end{equation}

The analysis of the term $\Phi_1$ is more challenging. We start by observing that
\begin{equation} \label{est:Phi_1-primo}
\sum_{j\geq0}2^{2j(s-1)}\,\left\|\Delta_j\nabla(-\Delta)^{-1}\Phi_1\right\|_{L^2}^2\,\leq\,\left\|\Phi_1\right\|^2_{H^{s-2}}\,.
\end{equation}
Next, if $s-2>1$ we see that $H^{s-2}$ is an algebra embedded in $L^\infty$, as a consequence of Corollary \ref{c:tame}. So, we can bound
\[
\left\|\Phi_1\right\|_{H^{s-2}}\,\lesssim\,\left\|\nabla\log\rho\right\|_{H^{s-2}}\,\left\|\nabla\pi\right\|_{H^{s-2}}\,.
\]
If instead $0<s-2\leq 1$, we write the Bony paraproduct decomposition of $\Phi_1$:
\[
 \Phi_1\,=\,-\,\nabla\log\rho\cdot\nabla \pi\,=\,-\,
 \Big(\mc T_{\nabla\log\rho}\nabla\pi\,+\,\mc T_{\nabla\pi}\nabla\log\rho\,+\,\mc R(\nabla\log\rho,\nabla\pi)\Big)\,,
\]
where, in view of Proposition \ref{p:op}, we can bound
\[
\left\|\mc T_{\nabla\log\rho}\nabla\pi\right\|_{H^{s-2}}\,+\,\left\|\mc R(\nabla\log\rho,\nabla\pi)\right\|_{H^{s-2}}\,\lesssim\,
\left\|\nabla\log\rho\right\|_{L^\infty}\,\left\|\nabla\pi\right\|_{H^{s-2}}\,\lesssim\,\left\|\nabla\rho\right\|_{H^{s-1}}\,\left\|\nabla\pi\right\|_{H^{s-2}}\,.
\]
Next, we use the continuous embeddings of Proposition \ref{p:embed} and Corollary \ref{c:embed}, namely
\[
H^{s-2}(\R^2)\,\hookrightarrow\,B^{s-3}_{\infty,\infty}(\R^2)\,\hookrightarrow\,B^{s-3-\delta}_{\infty,\infty}(\R^2) \quad \forall\,\delta>0\,. 
\]
where we choose $\delta=0$ if $s-3<0$ and $\delta>0$ small enough such that $s-1>1+\delta$ if $s-3=0$. Then, Proposition \ref{p:op} yields
\[
\left\|\mc T_{\nabla\pi}\nabla\log\rho\right\|_{H^{s-2}}\,\lesssim\,
\left\|\nabla\log\rho\right\|_{H^{1+\delta}}\,\left\|\nabla\pi\right\|_{B^{s-3-\delta}_{\infty,\infty}}\,
\lesssim\,\left\|\nabla\log\rho\right\|_{H^{s-1}}\,\left\|\nabla\pi\right\|_{H^{s-2}}\,.
\]
Thus, after an application of Proposition \ref{p:comp}, in the end we always get, for any $s>2$, the estimate
\[
 \left\|\Phi_1\right\|_{H^{s-2}}\,\lesssim\,\left\|\nabla\rho\right\|_{H^{s-1}}\,\left\|\nabla\pi\right\|_{H^{s-2}}\,.
\]
From it and \eqref{est:Phi_1-primo}, we finally discover the following bound for the $\Phi_1$ term:
\begin{equation} \label{est:Phi_1}
\left(\sum_{j\geq0}2^{2j(s-1)}\,\left\|\Delta_j\nabla(-\Delta)^{-1}\Phi_1\right\|_{L^2}^2\right)^{1/2}\,\lesssim\,E\,\left\|\nabla\pi\right\|_{H^{s-2}}\,\lesssim\,
\left(F\,+\,F^s\right)\,\left\|\nabla\pi\right\|_{H^{s-2}}\,.
\end{equation}

At this point, we can plug inequalities \eqref{est:p-low}, \eqref{est:Phi_3}, \eqref{est:Phi_2} and \eqref{est:Phi_1} into
\eqref{est:p-rho-omega}: we get
\begin{align} \label{est:p-rho-o_part}
\left\| \nabla\big(\pi\,-\,\rho\,\omega\big)\right\|_{H^{s-1}}\,\lesssim\,F\,+\,F^{2s+1}\,+\,\left(F\,+\,F^s\right)\,\left\|\nabla\pi\right\|_{H^{s-2}}\,.
\end{align}
Using again the decomposition $\nabla\pi\,=\,\nabla(\rho\,\o)\,+\,\nabla\big(\pi-\rho\o\big)$, we can estimate
\begin{align*}
\left\|\nabla\pi\right\|_{H^{s-2}}\,&\leq\,\left\|\nabla(\rho\,\o)\right\|_{H^{s-2}}\,+\,\left\|\nabla\big(\pi\,-\,\rho\,\o\big)\right\|_{H^{s-2}} \\
&\lesssim\,\left(1+\left\|\nabla\rho\right\|_{H^{s-2}}\right)\,\|\o\|_{H^{s-1}}\,+\,
\left\|\nabla\big(\pi\,-\,\rho\,\o\big)\right\|_{L^{2}}^{1/(s-1)}\,\left\|\nabla\big(\pi\,-\,\rho\,\o\big)\right\|_{H^{s-1}}^{(s-2)/(s-1)}\,,
\end{align*}
where we have also used Corollary \ref{c:tame} and an interpolation argument in Sobolev spaces for passing from the first to the second inequality.
Observe that, in order to bound the $L^2$ norm which appears, we cannot rely anymore on \eqref{est:p-rho-o_L^2},
as the frequency localisation operator $\Delta_{-1}$ is now missing. Instead, we proceed in the following way:
\begin{align*}
\left\|\nabla\big(\pi\,-\,\rho\,\o\big)\right\|_{L^{2}}\,&\lesssim\,\left\|\nabla\pi\right\|_{L^2}\,+\,\left\|\rho\,\o\right\|_{H^1}\,\lesssim\,
\left\|\nabla\pi\right\|_{L^2}\,+\,\left\|\o\right\|_{H^{s-1}}\,+\,\left\|\rho-1\right\|_{H^{s-1}}\,\left\|\o\right\|_{H^{s-1}} \\
&\lesssim\,E\,+\,E^2\,, 
\end{align*}
where we have also used estimate \eqref{est:p-L^2} and the fact that $s-1>1$. 
In turn, making also use of Lemma \ref{l:F}, this inequality implies
\begin{align*}
\left\|\nabla\pi\right\|_{H^{s-2}}\,&\lesssim\,(1+E)\,F\,+\,\left(E+E^{2}\right)^{1/(s-1)}\,
\left\|\nabla\big(\pi\,-\,\rho\,\o\big)\right\|_{H^{s-1}}^{(s-2)/(s-1)} \\
&\lesssim\,
F\,+\,F^{s+1}\,+\,\left(F+F^{2s}\right)^{1/(s-1)}\,\left\|\nabla\big(\pi\,-\,\rho\,\o\big)\right\|_{H^{s-1}}^{(s-2)/(s-1)}\,.
\end{align*}
Inserting this bound into \eqref{est:p-rho-o_part} yields
\begin{align*}
\left\| \nabla\big(\pi\,-\,\rho\,\omega\big)\right\|_{H^{s-1}}\,\lesssim\,F\,+\,F^{2s+1}\,+\,
\left(F\,+\,F^s\right)\,\left(F+F^{2s}\right)^{1/(s-1)}\,\left\|\nabla\big(\pi\,-\,\rho\,\o\big)\right\|_{H^{s-1}}^{(s-2)/(s-1)}\,.
\end{align*}
Therefore, using Young's inequality, we finally discover the sought bound for $\nabla\big(\pi\,-\,\rho\,\o\big)$ in $H^{s-1}$: we have
\begin{equation} \label{est:p-r-o_final}
 \left\| \nabla\big(\pi\,-\,\rho\,\omega\big)\right\|_{H^{s-1}}\,\lesssim\,F\,+\,F^{\g_0}\,,\qquad\qquad\mbox{ with }\qquad
 \g_0\,:=\,s^2+s+1\,.
\end{equation}
Notice that it would be enough to take $\g_0=s^2+s$, but we take the previous value for convenience, as it allows to simplify some computations below.

\subsection{Closing the estimates} \label{ss:est-close}

With the previous estimate \eqref{est:p-r-o_final} at hand, we are finally in the position of closing the \tsl{a priori} estimates for system
\eqref{eq:fluid_odd_visco1}, ending in this way the proof of the \tsl{a priori} bounds stated in Theorem \ref{t:uniform_estimates}.

First of all, we insert the two bounds \eqref{est:B-omega} and \eqref{est:p-r-o_final} into estimate \eqref{est:vort_part} for the vorticity
to deduce, for any $t\geq0$, the inequality
\begin{align*}
\left\|\o(t)\right\|_{H^{s-1}}\,&\lesssim\,
\exp\left(C\int^t_0\left(F\,+\,F^s\right)\,\dd\t\right)\,
\left(\left\|\o_0\right\|_{H^{s-1}}\,+\,\int^t_0\Big(\left(F\,+\,F^s\right)\,\left(F\,+\,F^{\g_0}\right)\,+\,F^2\,+\,F^{4s+1}\Big)\,\dd\t\right)\,, 
\end{align*}
where we have also used \eqref{est:dens-low-high} and \eqref{est:u-omega} to control the various Sobolev norms of $\nabla u$ and $\nabla\rho$ appearing in
\eqref{est:vort_part}.
Denoting $\g\,:=\,\g_0+s\,=\,(s+1)^2$ and observing that $\g>4s+1$ (because $s>2$), we finally find
\begin{equation} \label{est:vort_fin}
\forall\,t\geq0\,,\qquad\quad \left\|\o(t)\right\|_{H^{s-1}}\,\lesssim\,
\exp\left(C\int^t_0\Big(F(\t)\,+\,\big(F(\t)\big)^s\Big)\,\dd\t\right)\,
\left(\left\|\o_0\right\|_{H^{s-1}}\,+\,\int^t_0\Big(\big(F(\t)\big)^2\,+\,\big(F(\t)\big)^\g\Big)\,\dd\t\right)\,.
\end{equation}

Now, we can sum up estimate \eqref{est:theta-transp} and the previous estimate \eqref{est:vort_fin}. Keeping in mind the definition of the energy function
$F$ and the conservation of the low regularity norms \eqref{est:rho-2} and \eqref{est:u-2},
we obtain that, for all time $t\geq0$, the energy $F$ satisfies
\begin{equation} \label{est:F}
F(t)\,\lesssim\,\exp\left(C\int^t_0\Big(F(\t)\,+\,\big(F(\t)\big)^s\Big)\,\dd\t\right)\,
\left(F_0\,+\,\int^t_0\Big(F(\t)\,+\,\big(F(\t)\big)^\g\Big)\,\dd\t\right)\,,
\end{equation}
where we recall that we have adopted the convection that $F_0$ denotes the same quantity $F$, computed on the initial datum.

At this point, the end of the proof follows a classical argument. Define the time $T>0$ as
\[
T\,:=\,\sup\left\{t\geq0\;\Big|\quad \int^t_0\Big(F(\t)\,+\,\big(F(\t)\big)^\g\Big)\,d\t\,\leq\,2\,\min\big\{F_0\,,\,\log2\big\}\quad \right\}\,.
\]
Then, since $s<\g$, the estimate given in \eqref{est:F} implies that the energy function $F$ stays bounded in $[0,T]$: more precisely, one has
\[
\forall\,t\in[0,T]\,,\qquad\qquad F(t)\,\leq\,C\,F_0\,,
\]
for a suitable large constant $C>0$, depending only on $\rho_*$, on the $L^2$ and $L^\infty$ norms of the initial density $\rho_0-1$ and on the $L^2$ norm of $u$.
Since the energy $F$ controls the former energy $E$ and \tsl{viceversa}, see Lemma \ref{l:F}, we have that
\begin{align*}
\forall\,t\in[0,T]\,,\qquad\qquad E(t)\,\leq\,C\,\mc H\pare{F_0}\,\leq\,C\, \cG \pare{  E_0 }\,,
\end{align*}
for suitable non-linear expressions $\mc H$ and $\mc G$ and for possibly new universal constants $C>0$. The \tsl{a priori} bounds are thus completely proved.

\subsection{Continuation criteria} \label{ss:a-priori_cont}

This section is devoted to the proof of the continuation criteria stated in Theorem \ref{t:uniform_estimates}. Roughly speaking, we have to perform
the \tsl{a priori} estimates again, but in a more careful way, in order to get an estimate whose right-hand side is linear in the energy of the solution, up to
the presence of a factor which will contain some $L^\infty$ norms of the solution itself.

We start by defining, for all $t\geq0$, the quantities
\begin{align*} 
M(t)\,&:=\,\left\|\nabla u(t)\right\|^2_{L^\infty}+\left\|\nabla \rho(t)\right\|^s_{L^\infty}+
\left\|\nabla\rho(t)\right\|^{s-1}_{L^\infty}\,\left\|\nabla u(t)\right\|_{L^\infty}+\left\|\nabla^2 \rho(t)\right\|_{L^\infty}+
\left\|\nabla\pi(t)\right\|^{s/(s-1)}_{L^\infty} \\
\wtilde M(t)\,&:=\,\left\|\nabla u(t)\right\|^2_{L^\infty}+\left\|\nabla \rho(t)\right\|^s_{L^\infty}+
\left\|\nabla\rho(t)\right\|^{\max\{2,s-1\}}_{L^\infty}\,\left\|\nabla u(t)\right\|_{L^\infty}+\left\|\nabla^2 \rho(t)\right\|_{L^\infty}+
\left\|\nabla\big(\pi(t)-\rho(t)\o(t)\big)\right\|^{s/(s-1)}_{L^\infty}\,.
\end{align*} 
Given some $T^*>0$, we are going to prove that
\begin{equation} \label{eq:cont_to-prove}
\mbox{if }\qquad \min\left\{\int^{T^*}_0M(t)\,\dt\;,\;\int^{T^*}_0\wtilde M(t)\,\dt\right\}\,<\,+\infty\,,
\qquad\qquad\qquad \mbox{ then }\qquad\quad \sup_{t\in[0,T^*[}F(t)\,<\,+\infty\,.
\end{equation}
Indeed, owing to Lemma \ref{l:F} and a standard continuation argument (based also on uniqueness of solutions in the considered class, which we are going to
prove in Subsection \ref{ss:proof-u}), the previous property \eqref{eq:cont_to-prove} will imply the claimed continuation criterion.

As a matter of fact, in order to simplify the notation in our computations below, we are going to show a stronger property, namely that
\eqref{eq:cont_to-prove} holds true with $F$ replaced by
\begin{equation} \label{def:G}
G(t)\,:=\,\left\|\rho(t)-1\right\|^2_{H^s}\,+\,\left\|u(t)\right\|_{L^2}^2\,+\,\left\|\o(t)\right\|_{H^{s-1}}^2\,+\,\left\|\theta(t)\right\|_{H^{s-1}}^2\,.
\end{equation}
The advantage is that the $H^s$ norm of $\rho-1$ will come out several times in the computations: absorbing
it by $G$ allows us to avoid the use of interpolation (which would cause the presence of high powers of the energy functional) to control it.

Observe that the low regularity norm $\|u\|_{L^2}$ is simply transported by the flow of the solution, keep in mind \eqref{est:u-2}.
Thus, in order to prove a property like \eqref{eq:cont_to-prove} for $G$, we need to show that the $H^{s-1}$ norms of $\theta$ and $\o$ and the $H^s$ norm of $\rho-1$
remain bounded under the finiteness of the integrals of $M$ and $\wtilde M$.

Before going on, we introduce the operator $\Lambda\,=\,\Lambda(D)\,:=\,\big(\Id-\Delta\big)^{1/2}$. More precisely,
$\Lambda$ is the Fourier multiplier of symbol
\[
\Lambda(\xi)\,:=\,\left(1+|\xi|^2\right)^{1/2}\,.
\]
Recall that, by definition, for any $f\in\mc S(\R^2)$ and any $\s\in\R$ one has $\|f\|_{H^\s}\,=\,\left\|\Lambda^\s f\right\|_{L^2}$.

\paragraph{Estimates for $\rho-1$.}
Consider the equation for $\rho$, \tsl{i.e.} the first equation appearing in \eqref{eq:fluid_odd_visco1}.
Using the divergence-free condition for $u$ and applying the operator $\Lambda^s$, after an energy estimate we find that
\[
\frac{1}{2}\,\frac{\dd}{\dt}\left\|\rho-1\right\|_{H^s}^2\,=\,-\,\int \Lambda^s\Big(u\cdot\nabla\rho\Big)\,\Lambda^s\rho \dd x
\,=\,-\,\int\,\left[\Lambda^s,u\cdot\nabla\right]\rho\;\Lambda^s\rho\dd x\,.
\]
A use of the Kato-Ponce commutator estimates \cite{Kato-P} yields
\begin{align*}
\left|\int\,\left[\Lambda^s,u\cdot\nabla\right]\rho\,\Lambda^s\rho\dd x\right|\,\lesssim\,\Big(\left\|\nabla u\right\|_{L^\infty}\,\|\rho-1\|_{H^s}\,+\,
\left\|u\right\|_{H^s}\,\left\|\nabla\rho\right\|_{L^\infty}\Big)\,\left\|\rho-1\right\|_{H^s}\,,
\end{align*}
whence we deduce the following bound for the $H^s$ norm of $\rho-1$:
\begin{equation} \label{est:cont_rho}
\frac{\dd}{\dt}\|\rho-1\|_{{H}^{s}}^2\,\lesssim\,\Big(\left\|\nabla u\right\|_{L^\infty}\,+\,\left\|\nabla\rho\right\|_{L^\infty}\Big)\,G\,.
\end{equation}

\paragraph{Estimates for the vorticity.}
Next, we turn our attention to equation \eqref{eq:vort-2} for $\omega$: applying $\Lambda^{s-1}$ and performing an energy estimate, we find
$$
\frac{1}{2}\frac{\dd}{\dt}\|\omega\|_{{H}^{s-1}}^2\,=\,\int\Lambda^{s-1}\left(-\,\big(u-\nabla^\perp\log\rho\big)\cdot\nabla\omega\,-\,
\nabla^\perp\left(\frac{1}{\rho}\right)\cdot\nabla\big(\pi-\rho\o\big)\,-\,\mc B\big(\nabla u,\nabla^2\log\rho\big)\right)\Lambda^{s-1}\omega\,\dd x\,.
$$
We call $I_{1,2,3}$ the three terms appearing inside the brackets in the integral.

Let us focus on $I_1$ for a while: observe that, repeating \tsl{mutatis mutandis} the computations performed for the commutator appearing
in the equation for $\rho$ (computations based on Kato-Ponce commutator estimates), we would get an unpleasant  $\left\|\nabla\o\right\|_{L^\infty}$ to control.
Performing a finer decomposition, instead, we can trade lower regularity for the transported quantity for higher regularity
on the transport field. More precisely, we claim that the following control holds true.
\begin{lemma} \label{l:I_1}
The following estimate holds true:
\begin{align*}
\left\|\Big[\Lambda^{s-1},\big(u-\nabla^\perp\log\rho\big)\cdot \nabla\Big]\omega\right\|_{L^2}\,\lesssim\,
\|\omega\|_{H^{s-1}}\,\left\|\nabla\big(u-\nabla^\perp\log\rho\big)\right\|_{L^\infty}\,+\,
\left\|u-\nabla^\perp\log\rho\right\|_{H^{s}}\,\|\o\|_{L^\infty}\,.
\end{align*}
In particular, if we define
\[
I_1\,:=\,-\int\Lambda^{s-1}\Big(\big(u-\nabla^\perp\log\rho\big)\cdot\nabla\omega\Big)\,\Lambda^{s-1}\omega\,\dd x
\,=\,\int\Big[\Lambda^{s-1},\big(u-\nabla^\perp\log\rho\big)\cdot \nabla\Big]\omega\, \Lambda^{s-1}\omega \,\dd x\,,
\]
then we get
\[
\left|I_1\right|\,\lesssim\,\Big(\left\|\nabla u\right\|_{L^\infty}\,+\,\left\|\nabla\rho\right\|_{L^\infty}^2\,+\,\left\|\nabla^2\rho\right\|_{L^\infty}\,+\,
\|u\|_{L^\infty}\,\left\|\nabla u\right\|_{L^\infty}\Big)\,G\,.
\]
\end{lemma}

We postpone the proof of the previous lemma to the end of this section. Taking it for granted, let us resume with the proof of the continuation criterion.
As for the remaining terms $I_2$ and $I_3$, we have
\begin{align*}
I_2\,+\,I_3\,&=\,-\int \Lambda^{s-1}\left(\nabla^\perp\left(\frac{1}{\rho}\right)\cdot\nabla\big(\pi-\rho\o\big)\,+\,
\mc B\big(\nabla u,\nabla^2\log\rho\big)\right)\Lambda^{s-1}\omega\,\dd x \\
&\lesssim\,\left\|\o\right\|_{H^{s-1}}\,\left(\left\|\nabla^\perp\left(\frac{1}{\rho}\right)\cdot\nabla\big(\pi-\rho\o\big)\right\|_{H^{s-1}}\,+\,
\left\|\mc B\big(\nabla u,\nabla^2\log\rho\big)\right\|_{H^{s-1}}\right)\,.
\end{align*}

On the one hand, from Corollary \ref{c:tame} and the paralinearisation result of Proposition \ref{p:comp}, we infer
\begin{align*}
\left\|\mc B\big(\nabla u,\nabla^2\log\rho\big)\right\|_{H^{s-1}}\,&\lesssim\,\left\|\nabla u\right\|_{L^\infty}\,\left\|\nabla^2\log\rho\right\|_{H^{s-1}}\,+\,
\left\|\nabla u\right\|_{H^{s-1}}\,\left\|\nabla^2\log\rho\right\|_{L^\infty} \\
&\lesssim\,\Big(\left\|\nabla u\right\|_{L^\infty}\,+\,\left\|\nabla u\right\|_{L^\infty}\,\|u\|_{L^\infty}\,+\,
\left\|\nabla\rho\right\|^2_{L^\infty}\,+\,\left\|\nabla^2\rho\right\|_{L^\infty}\Big)\,\sqrt{G}\,.
\end{align*}
Observe that the term $\left\|\nabla^2\log\rho\right\|_{H^{s-1}}$ makes a $H^{s}$ norm of $\nabla\rho$ appear, but this norm can be reconducted to
$G$ by arguing as follows. First of all, we notice that
\begin{equation} \label{est:rho-cont_provvis}
\left\|\nabla\rho\right\|_{H^s}\,\lesssim\,\|\rho-1\|_{L^2}\,+\,\left\|\Delta\rho\right\|_{H^{s-1}}\,,
\end{equation}
which is derived from Proposition \ref{p:comp} and the usual low-high frequency decomposition \eqref{est:basic}.
Then, we make use of the definitions of $\eta$ and $\theta$ and of the tame estimates to get
\begin{align} \label{est:D^2rho}
\left\|\Delta\rho\right\|_{H^{s-1}}\,\leq\,\|\eta\|_{H^{s-1}}\,+\,\|\theta\|_{H^{s-1}}\,\lesssim\,
\rho^*\,\|u\|_{H^s}\,+\,\left\|\rho-1\right\|_{H^s}\,\|u\|_{L^\infty}\,+\,\|\theta\|_{H^{s-1}}\,\lesssim\,\sqrt{G}\,\left(1\,+\,\|u\|_{L^\infty}\right)\,.
\end{align}

On the other hand, using Bony's paraproduct decomposition, we can write
\begin{align} \label{eq:bony_1}
\nabla^\perp\left(\frac{1}{\rho}\right)\cdot\nabla\big(\pi-\rho\o\big)\,&=\,\mc T_{\nabla^\perp\left(\frac{1}{\rho}\right)\cdot}\nabla\big(\pi-\rho\o\big)\,+\,
\mc R\left(\nabla^\perp\left(\frac{1}{\rho}\right)\cdot,\nabla\big(\pi-\rho\o\big)\right)\,+\,\mc T_{\nabla\pi\cdot}\nabla^\perp\left(\frac{1}{\rho}\right)\,+\,
\mc T_{\nabla(\rho\o)\cdot}\nabla^\perp\left(\frac{1}{\rho}\right)\,.
\end{align}
Observe that, owing to Proposition \ref{p:op}, the first three terms can be controlled as
\begin{align*}
&\left\|\mc T_{\nabla^\perp\left(\frac{1}{\rho}\right)\cdot}\nabla\big(\pi-\rho\o\big)\,+\,
\mc R\left(\nabla^\perp\left(\frac{1}{\rho}\right)\cdot,\nabla\big(\pi-\rho\o\big)\right)\,+\,\mc T_{\nabla\pi\cdot}\nabla^\perp\left(\frac{1}{\rho}\right)
\right\|_{H^{s-1}} \\
&\qquad\qquad\qquad\qquad\qquad\qquad\qquad\qquad
\lesssim\,\left\|\nabla\rho\right\|_{L^\infty}\,\left\|\nabla\big(\pi\,-\,\rho\,\o\big)\right\|_{H^{s-1}}\,+\,\left\|\nabla\pi\right\|_{L^\infty}\,
\left\|\nabla\rho\right\|_{H^{s-1}}\,.
\end{align*}
As for the fourth term, instead, we decide to bound it as follows:
\begin{align*}
\left\|\mc T_{\nabla(\rho\o)\cdot}\nabla^\perp\left(\frac{1}{\rho}\right)\right\|_{H^{s-1}}\,&\lesssim\,\left\|\nabla\big(\rho\o\big)\right\|_{B^{-1}_{\infty,\infty}}\,
\left\|\nabla\rho\right\|_{H^s}\,\lesssim\,\left\|\rho\,\o\right\|_{B^{0}_{\infty,\infty}}\,\left\|\nabla\rho\right\|_{H^s} \\
&\lesssim\,\left\|\o\right\|_{L^\infty}\,\sqrt{G}\,\left(1\,+\,\|u\|_{L^\infty}\right)\,,
\end{align*}
where we have also used the paralinearisation result of Proposition \ref{p:comp}, the embedding $L^\infty\,\hookrightarrow\,B^{0}_{\infty,\infty}$ and the bounds
\eqref{est:rho-cont_provvis} and \eqref{est:D^2rho} for treating the $H^{s}$ norm of $\nabla\rho$.

Putting these last inequalities together with the one for the bilinear term $\mc B$, we gather the bound
\begin{align*}
I_2\,+\,I_3\,\lesssim\,G\,\Big(\left\|\nabla u\right\|_{L^\infty}\,+\,\left\|\nabla u\right\|_{L^\infty}\,\|u\|_{L^\infty}\,+\,
\left\|\nabla\rho\right\|^2_{L^\infty}\,+\,\left\|\nabla^2\rho\right\|_{L^\infty}\,+\,\left\|\nabla\pi\right\|_{L^\infty}\Big)\,+\,
\sqrt{G}\,\left\|\nabla\rho\right\|_{L^\infty}\,\left\|\nabla\big(\pi\,-\,\rho\,\o\big)\right\|_{H^{s-1}}\,.
\end{align*}
Observe that using tame estimates directly to control the left-hand side of \eqref{eq:bony_1} would have yielded to a similar inequality, with
$\left\|\nabla\pi\right\|_{L^\infty}$ replaced by $\left\|\nabla\big(\pi-\rho\o\big)\right\|_{L^\infty}$.

We now claim the following bound for the pressure term.
\begin{lemma} \label{l:press-cont}
Define the quantity $A$ as 
\[
 \forall\,t\geq0\,,\qquad\qquad A(t)\,:=\,\|u(t)\|_{L^\infty}\,+\,\left\|\nabla u(t)\right\|_{L^\infty}\,+\,\left\|\nabla \rho(t)\right\|_{L^\infty}\,.
\]
Then, the following estimates hold true:
\begin{align*}
\left\|\nabla\big(\pi-\rho\o\big)\right\|_{H^{s-1}}\,&\lesssim\,
\sqrt{G}\,\Big(\left\|\nabla\pi\right\|_{L^\infty}\,+\,1\,+\,A^2\,+\,\left\|\nabla\rho\right\|_{L^\infty}^{s-1}\,\left\|\nabla u\right\|_{L^\infty}
\,+\,\left\|\nabla\rho\right\|_{L^\infty}^{s}\,+\,\left\|\nabla^2 \rho\right\|_{L^\infty}\Big) \\
\left\|\nabla\big(\pi-\rho\o\big)\right\|_{H^{s-1}}\,&\lesssim\,
\sqrt{G}\,\Big(\left\|\nabla\big(\pi-\rho\o\big)\right\|_{L^\infty}\,+\,1\,+\,A^2\,+\,\left\|\nabla\rho\right\|_{L^\infty}^{\max\{2,s-1\}}\,\left\|\nabla u\right\|_{L^\infty}
\,+\,\left\|\nabla\rho\right\|_{L^\infty}^{s}\,+\,\left\|\nabla^2 \rho\right\|_{L^\infty}\Big) \,.
\end{align*}
\end{lemma}
Again, we postpone the proof of the previous lemma to the end of this subsection. Using the claimed bound in the estimate for $I_2+I_3$ and using Lemma \ref{l:I_1}
for bounding $I_1$, we finally obtain either the inequality
\begin{align} \label{est:cont_omega}
\frac{\dd}{\dt}\|\omega\|_{{H}^{s-1}}^2\,&\lesssim\,G\,\Big(1\,+\,N\,+\,
\left\|\nabla\rho\right\|^{s-1}_{L^\infty}\,\left\|\nabla u\right\|_{L^\infty}\,+\,
\left\|\nabla^2\rho\right\|_{L^\infty}\,+\,\left\|\nabla\pi\right\|_{L^\infty}\,+\,
\left\|\nabla\rho\right\|_{L^\infty}\,\left\|\nabla\pi\right\|_{L^\infty}\Big)
\end{align}
or, working with $\nabla\big(\pi-\rho\o\big)$ instead of $\nabla\pi$, the inequality
\[
\frac{\dd}{\dt}\|\omega\|_{{H}^{s-1}}^2\,\lesssim\,G\,\Big(1+N+
\left\|\nabla\rho\right\|^{\max\{2,s-1\}}_{L^\infty}\,\left\|\nabla u\right\|_{L^\infty}+
\left\|\nabla^2\rho\right\|_{L^\infty}+\left\|\nabla\big(\pi-\rho\o\big)\right\|_{L^\infty}+
\left\|\nabla\rho\right\|_{L^\infty}\,\left\|\nabla\big(\pi-\rho\o\big)\right\|_{L^\infty}\Big)\,,
\]
where we have defined, for all $t\geq0$, the quantity
\begin{align}
N(t)\,&:=\,\big\|u(t)\big\|_{L^\infty}^2\,+\,\left\|\nabla u(t)\right\|_{L^\infty}^2\,+\,
\left\|\nabla\rho(t)\right\|^s_{L^\infty} \label{def:N} 
\end{align}
and where we have also used the fact that $s>2$.

\paragraph{Estimates for $\theta$.}
To conclude, consider the equation \eqref{eq:theta} for $\theta$. Applying $\Lambda^{s-1}$ and performing an energy estimate, we find
$$
\frac{1}{2}\frac{\dd}{\dt}\left\|\theta\right\|_{{H}^{s-1}}^2\,=\,\int\Lambda^{s-1}\left(-u\cdot \nabla \theta\,+\,
\frac{1}{2}\,\nabla^\perp\rho \cdot \nabla\av{u}^2\,+\,\mc B\big(\nabla u,\nabla^2\rho\big)
\right)\Lambda^{s-1}\theta \,\dd x\,.
$$
We are now going to estimate each term appearing in the right-hand side of the previous expression, which we call $J_{1,2,3}$.
First of all, we use the divergence-free condition over $u$ to get 
\begin{align*}
J_1\,&:=\,-\int\Lambda^{s-1}\big(u\cdot \nabla \theta\big)\,\Lambda^{s-1}\theta \,\dd x\,=\,-\int
\left[\Lambda^{s-1}, u\cdot \nabla\right]\theta\, \Lambda^{s-1}\theta \,\dd x \\
&\lesssim\, \left\|\left[\Lambda^{s-1}, u\cdot \nabla\right]\,\theta\right\|_{L^2}\,\|\theta\|_{H^{s-1}}\,\lesssim\,
\Big(\left\|\theta\right\|_{H^{s-1}}\,\left\|\nabla u\right\|_{L^\infty}\,+\,\left\|u\right\|_{H^{s}}\,\left\|\theta\right\|_{L^\infty}\Big)\,
\|\theta\|_{H^{s-1}}\,. 
\end{align*}
where we have used the analogous estimate of Lemma \ref{l:I_1} to bound the commutator term. Using \eqref{def:eta} and \eqref{def:theta} together with
\eqref{est:rho-inf}, we thus infer
\[
J_1\,\lesssim\,G\,\Big(\left\|\nabla u\right\|_{L^\infty}\,+\,\left\|\nabla\rho\right\|_{L^\infty}\,\left\|u\right\|_{L^\infty}\,+\,
\left\|\Delta\rho\right\|_{L^\infty}\Big)\,.
\]

On the other hand, as $H^{s-1}$ is an algebra (recall that $s>2$ under our assumptions), from product estimates of Corollary \ref{c:tame} we get
\begin{align*}
J_2\,&:=\,\frac{1}{2}\,\int \Lambda^{s-1}\left(\nabla^\perp\rho \cdot \nabla\av{u}^2\right)\,\Lambda^{s-1}\theta\,\dd x \\
&\lesssim\,
\left\|\nabla^\perp\rho \cdot \nabla\av{u}^2\right\|_{H^{s-1}}\,\left\|\theta\right\|_{H^{s-1}}\,\lesssim\,
\|u\|_{L^\infty}\Big(\left\|\nabla\rho\right\|_{L^\infty}\,\|u\|_{H^s}\,+\,\left\|\nabla\rho\right\|_{H^{s-1}}\,\left\|\nabla u\right\|_{L^\infty}\Big)
\|\theta\|_{H^{s-1}} \\
&\lesssim\,G\,\Big(\left\|u\right\|_{L^\infty}\,\left\|\nabla\rho\right\|_{L^\infty}\,+\,\left\|u\right\|_{L^\infty}\,\left\|\nabla u\right\|_{L^\infty}\Big)\,.
\end{align*}

Finally, using the explicit expression \eqref{def:B} of the operator $\mc B$ together with Corollary \ref{c:tame}, 
we find that
\begin{align*}
J_3\,&:=\,\int \Lambda^{s-1}\mc B\big(\nabla u,\nabla^2\rho\big)\,\Lambda^{s-1}\theta\,\dd x \\
&\lesssim\,\left\|\mc B\big(\nabla u,\nabla^2\rho\big)\right\|_{H^{s-1}}\,\|\theta\|_{H^{s-1}}\,\lesssim\,
\Big(\left\|\nabla u\right\|_{L^\infty}\,\left\|\nabla^2\rho\right\|_{H^{s-1}}\,+\,\|u\|_{H^{s}}\,\left\|\nabla^2\rho\right\|_{L^\infty}\Big)\,\|\theta\|_{H^{s-1}}\,.
\end{align*}
We now use Calder\'on-Zygmund theory to bound $\left\|\nabla^2\rho\right\|_{H^{s-1}}\,\lesssim\,\left\|\Delta\rho\right\|_{H^{s-1}}$,
and inequality \eqref{est:D^2rho} to control the latter term.
This in turn gives us
\begin{align*}
J_3\,&\lesssim\,\Big(\left\|\nabla u\right\|_{L^\infty}\,\left\|u\right\|_{H^{s}}\,+\,\left\|\nabla u\right\|_{L^\infty}\,\left\|\rho-1\right\|_{H^{s}}\,\|u\|_{L^\infty}
\,+\,\left\|\nabla u\right\|_{L^\infty}\,\|\theta\|_{H^{s-1}}\,+\,\|u\|_{H^{s}}\,\left\|\nabla^2\rho\right\|_{L^\infty}\Big)\,\|\theta\|_{H^{s-1}} \\
&\lesssim\,G\,\Big(\left\|\nabla u\right\|_{L^\infty}\,+\,\left\|u\right\|_{L^\infty}\,\left\|\nabla u\right\|_{L^\infty}\,+\,
\left\|\nabla^2\rho\right\|_{L^\infty}\Big)\,.
\end{align*}

Collecting the estimates for $J_{1}$, $J_2$ and $J_3$, we finally find
\begin{equation} \label{est:cont_theta}
\frac{\dd}{\dt}\|\theta\|_{{H}^{s-1}}^2\,\lesssim\,\Big(\left\|\nabla u\right\|_{L^\infty}\,+\,
\left\|\nabla\rho\right\|_{L^\infty}\,\left\|u\right\|_{L^\infty}\,+\,\left\|\nabla u\right\|_{L^\infty}\,\left\|u\right\|_{L^\infty}\,+\,
\left\|\nabla^2\rho\right\|_{L^\infty}\Big)\,G\,.
\end{equation}

\paragraph{End of the argument.}
At this point, we sum up inequalities \eqref{est:cont_rho}, \eqref{est:cont_omega} and \eqref{est:cont_theta}: recalling the definitions \eqref{def:G}
and \eqref{def:N} of the functions $G$ and $N$, we get
\begin{align} \label{est:G_partial}
\frac{1}{2}\,\frac{\dd}{\dt}G\,\lesssim\,G\,\left(1\,+\,N\,+\,
\left\|\nabla\rho\right\|^{s-1}_{L^\infty}\,\left\|\nabla u\right\|_{L^\infty}\,+\,
\left\|\nabla^2\rho\right\|_{L^\infty}\,+\,\left\|\nabla\pi\right\|_{L^\infty}\,\left(1\,+\,\left\|\nabla\rho\right\|_{L^\infty}\right)\right)\,,
\end{align}
or, in case we replace $\left\|\nabla\pi\right\|_{L^\infty}$ by $\left\|\nabla\big(\pi-\rho\o\big)\right\|_{L^\infty}$, we get the bound
\begin{align} \label{est:G_partial_2}
\frac{1}{2}\,\frac{\dd}{\dt}G\,\lesssim\,G\,\left(1\,+\,N\,+\,
\left\|\nabla\rho\right\|^{\max\{2,s-1\}}_{L^\infty}\,\left\|\nabla u\right\|_{L^\infty}\,+\,
\left\|\nabla^2\rho\right\|_{L^\infty}\,+\,\left\|\nabla\big(\pi-\rho\o\big)\right\|_{L^\infty}\,\left(1\,+\,\left\|\nabla\rho\right\|_{L^\infty}\right)\right)\,.
\end{align}

Let us focus on \eqref{est:G_partial} for a while.
Our concern, now, is to remove the term $\left\|\nabla^2\rho\right\|_{L^\infty}$ in favour of $\left\|\Delta\rho\right\|_{L^\infty}$.
Of course, this cannot be done directly, because the Calder\'on-Zygmund operator $\nabla^2(-\Delta)^{-1}$ does not act continuously
on $L^\infty$. Instead, since for any $t\geq0$ the quantity $\nabla^2\rho(t)$ belongs to $H^{s-1}$, with $s-1>1$, we can use the logarithmic inequality
\[
\|f\|_{L^\infty}\,\lesssim\,1\,+\,\|f\|_{BMO}\,\Big(1\,+\,\log\left(1\,+\,\|f\|_{H^{s-1}}\right)\Big)
\]
of \cite{KT}, together with continuity of the Riesz operators over $BMO$ and the embedding property
\[
\|f\|_{BMO}\,\lesssim\,\|f\|_{L^\infty}\,,
\]
to get the bound
\[
 \left\|\nabla^2\rho\right\|_{L^\infty}\,\lesssim\,1\,+\,\left\|\Delta\rho\right\|_{L^\infty}\,
\Big(1\,+\,\log\left(1\,+\,\left\|\Delta \rho\right\|_{H^{s-1}}\right)\Big)\,.
\]
Owing to \eqref{est:D^2rho} and the fact that $\|u\|_{L^\infty}$ is controlled by $\|u\|_{L^2}+\|\o\|_{H^{s-1}}$, in turn we deduce
\begin{align*}
\left\|\nabla^2\rho\right\|_{L^\infty}\,&\lesssim\,1\,+\,\left\|\Delta\rho\right\|_{L^\infty}\,
\Bigg(1\,+\,\log\Big(1\,+\,\sqrt{G}\,\left(1+\sqrt{G}\right)\Big)\Bigg) \\
&\lesssim\,1\,+\,\left\|\Delta\rho\right\|_{L^\infty}\,
\Big(1\,+\,\log\big(1\,+\,G\big)\Big)\,.
\end{align*}

Inserting this inequality into \eqref{est:G_partial}, we then obtain
\begin{align*} 
\frac{\dd}{\dt}G\,\lesssim\,G\,\Big(1\,+\,\log\big(1\,+\,G\big)\Big)\,\left(1\,+\,N\,+\,
\left\|\nabla\rho\right\|^{s-1}_{L^\infty}\,\left\|\nabla u\right\|_{L^\infty}\,+\,
\left\|\Delta\rho\right\|_{L^\infty}\,+\,\left\|\nabla\pi\right\|_{L^\infty}\,+\,
\left\|\nabla\rho\right\|_{L^\infty}\,\left\|\nabla\pi\right\|_{L^\infty}\right)\,,
\end{align*}
with the usual modification concerning the $L^\infty$ norm of $\nabla\pi$ and $\nabla\big(\pi-\rho\o\big)$ in case we started with \eqref{est:G_partial_2} instead.
Therefore, the Osgood lemma (see \tsl{e.g.} Lemma 3.4 of \cite{B-C-D}) guarantees us that, given $T>0$, one has
\begin{align*}
&\int^T_0\left(N+\left\|\nabla\rho\right\|^{s-1}_{L^\infty}\left\|\nabla u\right\|_{L^\infty}+\left\|\Delta\rho\right\|_{L^\infty}+
\left\|\nabla\pi\right\|_{L^\infty}+\left\|\nabla\rho\right\|_{L^\infty}\left\|\nabla\pi\right\|_{L^\infty}\right)\,\dt\,<\,+\infty \\
&\qquad\qquad\qquad\qquad\qquad\qquad\qquad\qquad\qquad\qquad\qquad\qquad\qquad\qquad\qquad \Longrightarrow\qquad\qquad
\sup_{t\in[0,T[\,}G(t)\,<\,+\infty\,.
\end{align*}
Observe that an application of the H\"older and Young inequalities yields
\[
\int^T_0\left\|\nabla\rho\right\|_{L^\infty}\,\left\|\nabla\pi\right\|_{L^\infty}\,\dt\,\lesssim\,\int^T_0\left\|\nabla\rho\right\|^s_{L^\infty}\,\dt\,+\,
\int^T_0\left\|\nabla\pi\right\|_{L^\infty}^{s/(s-1)}\,\dt\,,
\]
where again $\nabla\pi$ has to be replaced by $\nabla\big(\pi-\rho\o\big)$ in case we are interested in working with the quantity $\wtilde M$.
Hence, to conclude the proof of \eqref{eq:cont_to-prove}, thus of the continuation criteria, it remains us to get rid of the term $\|u\|_{L^\infty}$
appearing in the definition of $N$. This is an easy task: starting from the usual low-high frequencies decomposition, we gather
\[
\|u\|_{L^\infty}\,\lesssim\,\|u\|_{L^2}\,+\,\|\nabla u\|_{L^\infty}\,\lesssim\,\left\|u_0\right\|_{L^2}\,+\,\left\|\nabla u\right\|_{L^\infty}\,,
\]
where we have also used the energy estimate \eqref{est:u-2}.

This completes the proof of the continuation criteria, thus of the whole Theorem \ref{t:uniform_estimates}, provided we prove Lemmas \ref{l:I_1}
and \ref{l:press-cont}.
Before doing that (see below), let us conclude with a remark about possible simpler forms of the continuation criterion,
which nonetheless would require more stringent assumptions to be verified.

\begin{rem} \label{r:cont_no-better}
We observe that, from our continuation criterion, we could derive other blow-up conditions. For instance, the following modifications would be possible.
\begin{enumerate}[(i)]
 \item We could use the inequality $\left\|\nabla\rho\right\|_{L^\infty}\,\lesssim\,\rho^*\,+\,\left\|\nabla^2\rho\right\|_{L^\infty}$ to get rid of
the lower order term $\left\|\nabla\rho\right\|_{L^\infty}$ in the definition of the function $M$; however, this would
require a higher integrability in time of the term $\left\|\nabla^2\rho\right\|_{L^\infty}$, which does not look well suited for the application of
the Osgood lemma.

\item On the other hand, we observe that the condition $\left\|\nabla\rho\right\|_{L^\infty}\in L^s\big([0,T]\big)$ is a true condition, which is needed.
As a matter of fact, from $\rho\in L^\infty\big([0,T]\times\R^2\big)$ and $\left\|\nabla^2\rho\right\|_{L^\infty}\in L^1\big([0,T]\big)$,
we can only deduce that $\left\|\nabla\rho\right\|_{L^\infty}$ belongs to $L^2\big([0,T]\big)$, while here $s>2$.

\item We could use the Young inequality to get rid of the product $\left\|\nabla\rho\right\|^{s-1}_{L^\infty}\left\|\nabla u\right\|_{L^\infty}$,
but this would require higher integrability in time of either $\left\|\nabla\rho\right\|_{L^\infty}$ (namely, $L^{2(s-1)}$ in time)
or $\left\|\nabla u\right\|_{L^\infty}$ (namely, $L^s$ in time). Of course, the same can be said when dealing
with $\left\|\nabla\rho\right\|^{\max\{2,s-1\}}_{L^\infty}\left\|\nabla u\right\|_{L^\infty}$.

\end{enumerate}
\end{rem}

We are now ready to prove the technical lemmas we have used in the previous argument.

\paragraph{Proof of the technical lemmas.}

Here we prove Lemmas \ref{l:I_1} and \ref{l:press-cont}. Let us start with the proof of the former statement.

\begin{proof}[Proof of Lemma \ref{l:I_1}]
To begin with, we notice that the bound for $I_1$ easily follows from the claimed estimate for the commutator, after using
inequality \eqref{est:rho-cont_provvis} and \eqref{est:D^2rho} for handling the $H^s$ norm of $\nabla^\perp\log\rho$.

Thus, we are left with the proof of the commutator estimate. Remark that, by Littlewood-Paley decomposition, we actually have
\[
\left\|\Big[\Lambda^{s-1},\big(u-\nabla^\perp\log\rho\big)\cdot \nabla\Big]\omega\right\|_{L^2}\,\approx\,
\left\|\left(2^{j(s-1)}\,\left\|\left[\Delta_j,\z\cdot\nabla\right]\o\right\|_{L^2}\right)_{j\geq-1}\right\|_{\ell^2}\,,
\]
where we have defined $\z\,:=\,u-\nabla^\perp\log\rho$. Recall that $\z\in H^s$ and $\nabla\cdot \z=0$. Hence, we are going to work
with the expression of the commutator on the right. We use the Bony paraproduct decomposition to write
\begin{align*}
\left[\Delta_j,\z\cdot\nabla\right]\o\,=\,\left[\Delta_j,\mc T_{\z\cdot}\right]\nabla\o\,+\,\Delta_j\mc T_{\nabla\o\cdot}\z\,-\,\mc T_{\Delta_j\nabla\o\cdot}\z\,+\,
\Delta_j\mc R\big(\z\cdot,\nabla\o\big)\,-\,\mc R\big(\z\cdot,\nabla\Delta_j\o\big)\,,
\end{align*}
where, with a little abuse of notation, we have adopted a vectorial notation for the paraproduct and remainder operators.

We start by considering the commutator term involving the paraproduct. By following the same lines of the proof to Lemma 2.99 of \cite{B-C-D}
(take $m=1$ and replace $s$ by $s-1$ in that statement), we get the estimate
\[
\left\|\left(2^{j(s-1)}\,\left\|\left[\Delta_j,\mc T_{\z\cdot}\right]\nabla\o\right\|_{L^2}\right)_{j\geq-1}\right\|_{\ell^2}\,\lesssim\,
\left\|\nabla \z\right\|_{L^\infty}\,\left\|\o\right\|_{H^{s-1}}\,.
\]

Next, we consider the term $\Delta_j\mc T_{\nabla\o\cdot}\z$. Observe that, by spectral localisation, we have
\[
\Delta_j\mc T_{\nabla\o\cdot}\z\,=\,\sum_{|k-j|\leq 3}S_{k-1}\nabla\o\cdot\,\Delta_k\z\,.
\]
For simplicity, let us focus on the term $k=j$, the treatement of the other ones being absolutely analogous (this will make a multiplicative constant $C>0$ appear,
with $C$ only depending on $s$). By Bernstein's inequalities of Lemma \ref{l:bern}, we have
\[
\left\|S_{j-1}\nabla\o\cdot\,\Delta_j\z\right\|_{L^2}\,\lesssim\,2^j\,\left\|S_{j-1}\o\right\|_{L^\infty}\,\left\|\Delta_j\z\right\|_{L^2}\,.
\]
Therefore, we deduce that
\[
\left\|\left(2^{j(s-1)}\,\left\|\Delta_j\mc T_{\nabla\o\cdot}\z\right\|_{L^2}\right)_{j\geq-1}\right\|_{\ell^2}\,\lesssim\,
\left\|\o\right\|_{L^\infty}\,\|\z\|_{H^s}\,,
\]
for a multiplicative constant depending on $s$.
The control of the terms $\mc T_{\Delta_j\nabla\o\cdot}\z$ and $\mc R\big(\z\cdot,\nabla\Delta_j\o\big)$ works similarly, after observing that, by spectral localisation,
these terms are composed by only a finite sum of terms localised at frequencies of size approximately $2^j$.

Finally, we use Proposition \ref{p:op} and embedding properties to get
\begin{align*}
\left\|\left(2^{j(s-1)}\,\left\|\Delta_j\mc R\big(\z\cdot,\nabla\o\big)\right\|_{L^2}\right)_{j\geq-1}\right\|_{\ell^2}\,\approx\,
\left\|\mc R\big(\z\cdot,\nabla\o\big)\right\|_{H^{s-1}}\,\lesssim\,\left\|\z\right\|_{H^s}\,\left\|\nabla\o\right\|_{B^{-1}_{\infty,\infty}}\,\lesssim\,
\left\|\z\right\|_{H^s}\,\left\|\o\right\|_{L^\infty}\,.
\end{align*}

Collecting all the previous inequalities, we see that the claimed commutator estimate holds true. In light of what we have said at the beginning
of the argument, this completes also the proof of the whole lemma.
\end{proof}

Next, we switch to the proof of Lemma \ref{l:press-cont}.

\begin{proof}[Proof of Lemma \ref{l:press-cont}]
In order to prove the claimed estimates, we resort to the analysis of Pragraph \ref{sss:pressure}.
Let us start by bounding the high regularity norm.
Owing to the usual low-high frequencies decomposition, we can write
\begin{align*}
\left\|\nabla\big(\pi-\rho\o\big)\right\|_{H^{s-1}}\,&\lesssim\,\left\|\nabla\big(\pi-\rho\o\big)\right\|_{L^2}\,+\,
\left\|\Delta\big(\pi-\rho\o\big)\right\|_{H^{s-2}}\,.
\end{align*}

For the low frequencies part, we can use \eqref{est:p-L^2_mild}, combined with the fact that $s-1>1$, to get
\begin{align} \label{est:pi-o_cont}
\left\|\nabla\big(\pi-\rho\o\big)\right\|_{L^2}\,&\lesssim\,\left\|\nabla\pi\right\|_{L^2}\,+\,\left\|\rho\,\o\right\|_{H^1}\,\lesssim\,
\sqrt{G}\,\Big(1\,+\,\left\|\nabla u\right\|_{L^\infty}\,+\,\left\|\nabla\rho\right\|_{L^\infty}\Big)\,.
\end{align}

For the high frequencies part, instead, we take advantage of relation \eqref{eq:p-rho_vort} and Bony's paraproduct decomposition to gather
\begin{align} \label{est:Delta-pi-o}
\left\|\Delta\big(\pi-\rho\o\big)\right\|_{H^{s-2}}\,&\lesssim\,\left\|-\nabla\log\rho\cdot\nabla\pi\,+\,
\rho\,\nabla\cdot\left((u\cdot\nabla)u\,+\,(\nabla\log\rho\cdot\nabla)u^\perp\right)\,-\,\big[\rho-1,\Delta\big]\o\right\|_{H^{s-2}} \\
\nonumber &\lesssim\,\left\|\nabla\rho\right\|_{H^{s-2}}\,\left\|\nabla\pi\right\|_{L^\infty}\,+\,\left\|\nabla\rho\right\|_{L^\infty}\,
\Big(\left\|\nabla\big(\pi-\rho\o\big)\right\|_{H^{s-2}}\,+\,\left\|\rho\o\right\|_{H^{s-1}}\Big) \\
\nonumber &\qquad +\,\left\|\rho\,\nabla\cdot\left((u\cdot\nabla)u\,+\,(\nabla\log\rho\cdot\nabla)u^\perp\right)\right\|_{H^{s-2}}
\,+\,\left\|\nabla\rho\cdot\nabla\o\right\|_{H^{s-2}}\,+\,\left\|\Delta\rho\,\o\right\|_{H^{s-2}}\,.
\end{align}
Observe that we have
$\nabla\rho\cdot\nabla\o\,=\,\mc T_{\nabla\rho\cdot}\nabla\o\,+\,\mc T_{\nabla\o\cdot}\nabla\rho\,+\,\mc R\big(\nabla\rho\cdot,\nabla\o\big)$, where,
owing to Proposition \ref{p:op} and the condition $s-2>0$, the first and last terms can be bounded by
\[
\left\|\mc T_{\nabla\rho\cdot}\nabla\o\right\|_{H^{s-2}}\,+\,\left\|\mc R\big(\nabla\rho\cdot,\nabla\o\big)\right\|_{H^{s-2}}\,\lesssim\,
\left\|\nabla\rho\right\|_{L^\infty}\,\|\o\|_{H^{s-1}}\,,
\]
whereas, arguing as done in Lemma \ref{l:I_1}, the second term can be controlled by
\[
\left\|\mc T_{\nabla\o\cdot}\nabla\rho\right\|_{H^{s-2}}\,\lesssim\,\|\rho-1\|_{H^s}\,\|\o\|_{L^\infty}\,.
\]
This yields the inequality
\begin{equation} \label{est:cont_1}
\left\|\nabla\rho\cdot\nabla\o\right\|_{H^{s-2}}\,\lesssim\,\sqrt{G}\,\left(\left\|\nabla\rho\right\|_{L^\infty}\,+\,\left\|\nabla u\right\|_{L^\infty}\right)\,.
\end{equation}

In addition, by using paraproduct decomposition once more, we easily get
\begin{equation} \label{est:cont_2}
\left\|\Delta\rho\,\o\right\|_{H^{s-2}}\,\lesssim\,
\left\|\Delta\rho\right\|_{H^{s-2}}\,\|\o\|_{L^\infty}\,+\,\left\|\Delta\rho\right\|_{L^\infty}\,\|\o\|_{H^{s-2}}\,\lesssim\,
\sqrt{G}\,\left(\left\|\Delta\rho\right\|_{L^\infty}\,+\,\left\|\nabla u\right\|_{L^\infty}\right)\,.
\end{equation}

Look now at the first term appearing in the last line of \eqref{est:Delta-pi-o}. Integrating by parts, we see that we can control it by
\begin{align*}
&\left\|\rho\,\nabla\cdot\left((u\cdot\nabla)u\,+\,(\nabla\log\rho\cdot\nabla)u^\perp\right)\right\|_{H^{s-2}}\\
&\qquad\qquad \lesssim\,\left\|(u\cdot\nabla)u\,+\,(\nabla\log\rho\cdot\nabla)u^\perp\right\|_{H^{s-1}}\,+\,\left\|(\rho-1)\Big((u\cdot\nabla)u\,+\,(\nabla\log\rho\cdot\nabla)u^\perp\Big)\right\|_{H^{s-1}}  \\
&\qquad\qquad\qquad\qquad\qquad\qquad\qquad\qquad\qquad\qquad\qquad\qquad\qquad +\,
\left\|\nabla\rho\cdot\Big((u\cdot\nabla)u\,+\,(\nabla\log\rho\cdot\nabla)u^\perp\Big)\right\|_{H^{s-2}}\,.
\end{align*}
By a systematic use of tame estimates (keep in mind Corollary \ref{c:tame}), we thus deduce that
\begin{align} \label{est:cont_3}
&\left\|\rho\,\nabla\cdot\left((u\cdot\nabla)u\,+\,(\nabla\log\rho\cdot\nabla)u^\perp\right)\right\|_{H^{s-2}} \\
\nonumber &\lesssim\,\sqrt{G}\,\Big(\left\|\nabla u\right\|_{L^\infty}\,+\,\left\|u\right\|_{L^\infty}\,+\,\left\|\nabla\rho\right\|_{L^\infty}\,+\,
\left\|\nabla u\right\|_{L^\infty}\,\|u\|_{L^\infty}\,+\,\left\|\nabla u\right\|_{L^\infty}\,\left\|\nabla\rho\right\|_{L^\infty}\,+\,
\|u\|_{L^\infty}\,\left\|\nabla \rho\right\|_{L^\infty}\,+\,\left\|\nabla\rho\right\|^2_{L^\infty} \Big)\,.
\end{align}
Observe that all the terms in the parenthesis can be bounded (up to a multiplicative constant) by the quantity $1+A^2$, where $A$ has been defined
in the statement of the lemma.

Then, inserting inequalities \eqref{est:cont_1}, \eqref{est:cont_2} and \eqref{est:cont_3} into \eqref{est:Delta-pi-o}, and using also an interpolation argument,
we find, for $\alpha\,=\,\alpha(s)\,=\,(s-2)/(s-1)\,\in\;]0,1[\,$, the bound
\begin{align*}
\left\|\Delta\big(\pi-\rho\o\big)\right\|_{H^{s-2}}\,&\lesssim\,\sqrt{G}\,\Big(\left\|\nabla\pi\right\|_{L^\infty}+1+A^2+\left\|\nabla^2 \rho\right\|_{L^\infty}\Big)\,+\,
\left\|\nabla\rho\right\|_{L^\infty}\,\left\|\nabla\big(\pi-\rho\o\big)\right\|_{L^2}^{1-\alpha}\,\left\|\Delta\big(\pi-\rho\o\big)\right\|^{\alpha}_{H^{s-2}} \\
&\leq\,C\sqrt{G}\Big(\left\|\nabla\pi\right\|_{L^\infty}+1+A^2+\left\|\nabla^2 \rho\right\|_{L^\infty}\Big)\,+\,C(\de)
\left\|\nabla\rho\right\|_{L^\infty}^{s-1}\left\|\nabla\big(\pi-\rho\o\big)\right\|_{L^2}\,+\,\de\left\|\Delta\big(\pi-\rho\o\big)\right\|_{H^{s-2}}\,,
\end{align*}
where we have also used the Young inequality once. The previous bound holds true for any $\de>0$: taking it small enough and using
estimate \eqref{est:pi-o_cont}, we get
\begin{align*}
\left\|\Delta\big(\pi-\rho\o\big)\right\|_{H^{s-2}}\,&\lesssim\,\sqrt{G}\,\Big(\left\|\nabla\pi\right\|_{L^\infty}\,+\,1\,+\,A^2\,+\,
\left\|\nabla^2 \rho\right\|_{L^\infty}\Big)\,+\,
\left\|\nabla\rho\right\|_{L^\infty}^{s-1}\,\sqrt{G}\,\Big(1\,+\,\left\|\nabla u\right\|_{L^\infty}\,+\,\left\|\nabla\rho\right\|_{L^\infty}\Big) \\
&\lesssim\,\sqrt{G}\,\Big(\left\|\nabla\pi\right\|_{L^\infty}\,+\,1\,+\,A^2\,+\,\left\|\nabla\rho\right\|_{L^\infty}^{s-1}\,\left\|\nabla u\right\|_{L^\infty}
\,+\,\left\|\nabla\rho\right\|_{L^\infty}^{s}\,+\,\left\|\nabla^2 \rho\right\|_{L^\infty}\Big)\,.
\end{align*}
Therefore, combining the previous estimate with inequality \eqref{est:pi-o_cont}, we finally deduce the first
claimed bound for $\nabla\big(\pi-\rho\o\big)$ in $H^{s-1}$.

In order to get the second estimate for $\nabla\big(\pi-\rho\o\big)$ in $H^{s-1}$, the only change pertains to estimate \eqref{est:Delta-pi-o}.
As a matter of fact, it is enough to notice that
\[
-\,\nabla\log\rho\cdot\nabla\pi\,-\,\big[\rho-1,\Delta\big]\o\,=\,-\,\nabla\log\rho\cdot\nabla\big(\pi-\rho\o\big)\,+\,\nabla\rho\cdot\nabla\o\,-\,
\frac{1}{\rho}\,\o\,\big|\nabla\rho\big|^2\,+\,\Delta\rho\,\o\,.
\]
From that relation, we deduce the estimate
\begin{align} \label{est:Delta-pi-o_2}
\left\|\Delta\big(\pi-\rho\o\big)\right\|_{H^{s-2}}\,
&\lesssim\,\left\|\nabla\rho\right\|_{H^{s-2}}\,\left\|\nabla\big(\pi-\rho\o\big)\right\|_{L^\infty}\,+\,\left\|\nabla\rho\right\|_{L^\infty}\,
\left\|\nabla\big(\pi-\rho\o\big)\right\|_{H^{s-2}} \\
\nonumber &\quad +\,\left\|\nabla\rho\right\|_{H^{s-2}}\,\left\|\o\right\|_{L^\infty}\,\left\|\nabla\rho\right\|_{L^\infty}\,+\,
\left\|\o\right\|_{H^{s-2}}\,\left\|\nabla\rho\right\|^2_{L^\infty}\,+\,
\left\|\rho-1\right\|_{H^{s-2}}\,\left\|\o\right\|_{L^\infty}\,\left\|\nabla\rho\right\|_{L^\infty}^2 \\
\nonumber &\qquad +\,\left\|\rho\,\nabla\cdot\left((u\cdot\nabla)u\,+\,(\nabla\log\rho\cdot\nabla)u^\perp\right)\right\|_{H^{s-2}}
\,+\,\left\|\nabla\rho\cdot\nabla\o\right\|_{H^{s-2}}\,+\,\left\|\Delta\rho\,\o\right\|_{H^{s-2}}\,.
\end{align}
in place of \eqref{est:Delta-pi-o}. In turn, this bound yields
\begin{align*}
\left\|\Delta\big(\pi-\rho\o\big)\right\|_{H^{s-2}}\,&\lesssim\,
\sqrt{G}\,\Big(\left\|\nabla\big(\pi-\rho\o\big)\right\|_{L^\infty}+1+A^2+\left\|\nabla^2 \rho\right\|_{L^\infty}+
\left\|\nabla\rho\right\|_{L^\infty}^2\,\left\|\nabla u\right\|_{L^\infty}\Big) \\
&\qquad\qquad\qquad\qquad\qquad\qquad\qquad\qquad \,+\,
\left\|\nabla\rho\right\|_{L^\infty}\,\left\|\nabla\big(\pi-\rho\o\big)\right\|_{L^2}^{1-\alpha}\,\left\|\Delta\big(\pi-\rho\o\big)\right\|^{\alpha}_{H^{s-2}}\,.
\end{align*}
At this point, the rest of the argument is identical to what already done. Observing that
\[
\left(\left\|\nabla\rho\right\|_{L^\infty}^2\,+\,\left\|\nabla\rho\right\|_{L^\infty}^{s-1}\right)\,\left\|\nabla u\right\|_{L^\infty}\,\lesssim\,
\left(\left\|\nabla\rho\right\|_{L^\infty}\,+\,\left\|\nabla\rho\right\|_{L^\infty}^{\max\{2,s-1\}}\right)\,\left\|\nabla u\right\|_{L^\infty}\,\lesssim\,
A^2\,+\,\left\|\nabla\rho\right\|_{L^\infty}^{\max\{2,s-1\}}\,\left\|\nabla u\right\|_{L^\infty}\,,
\]
we finally get the sought control and conclude the proof of the lemma.
\end{proof}

\section{Analysis of a higher order Stokes system} \label{s:Stokes}

In this section, we study the following Stokes-type system:
\begin{equation}\label{eq:Stokes_var}
\system{
\begin{aligned}
& \partial_tu\, +\, (v\cdot\nabla)u\,+\,b\,\nabla\Pi\,+\,\nu\,b\,\Delta^2 u\,+\,\nu_0\,b\,\nabla\cdot\left(h\,\nabla u^\perp\right)\,=\,g \\
& \nabla\cdot u\, =\,0 \\
& u_{|t=0}\,=\,u_0\,.
\end{aligned}
}
\end{equation}

When $\nu_0=0$, system \eqref{eq:Stokes_var} shares strong similarities with the classical Stokes system, except for the presence of a variable coefficient $b$
and of the higher order viscosity term $\Delta^2u$ (which we will sometimes call ``hyperviscosity'') instead of the classical one $-\Delta u$.

When $\nu_0\neq0$ as in our case, a second order perturbation appears in the Stokes system.
Our goal is to prove similar results as the ones proved in \cite{D_2006} (see Section 3 therein) for system \eqref{eq:Stokes_var},
namely global well-posedness for smooth solutions.
Such an analysis will be needed in Subsection \ref{ss:proof-e}, when constructing smooth approximate solutions to our original model \eqref{eq:fluid_odd_visco1}.

For the sake of coherence with the rest of this work, from now on we set $\nu_0=1$. As in the inviscid counterpart \eqref{eq:fluid_odd_visco1}, the
sign and the precise value of $\nu_0$ do not play any role in our analysis. On the contrary, throughout this section we assume $\nu>0$ fixed, but we will
keep track of the dependence of the various constants on it.
Moreover, we limit our analysis to the case $h\,=\,b^{-1}$, the only relevant one for our scopes.

The main result of this section is contained in the following statement.
\begin{theorem} \label{t:Stokes_var}
Let $\s>1$ and $\nu>0$ fixed. Take $T>0$ arbitrarily large. Assume that the function $b$ satisfies $b\,\geq\,b_*\,>\,0$, for some positive constant $b_*\in\R$,
and $b-1\,\in\,\wtilde L^\infty_T\big(H^{\s+2}(\R^2)\big)$. Take $h\,=\,b^{-1}$.
In addition, assume that the vector field $v$ belongs to $\wtilde L^\infty_T\big(H^\s(\R^2)\big)$, is divergence-free, \tsl{i.e.} $\nabla\cdot v=0$,
and is such that the function $V'(t)$, defined as in Theorem \ref{th:transport}, belongs to $L^1\big([0,T]\big)$.
Finally, let $g\,\in\,\wtilde L^1_T\big(H^\s\big)$.

Then, for any $u_0\in H^\s$ there exists a unique solution $\big(u,\nabla\Pi\big)$ to the Stokes-type system \eqref{eq:Stokes_var} such that
\[
u\,\in\,\wtilde{C}_T\big(H^{\s}(\R^2)\big)\,,\qquad \nu\,u\,\in\,\wtilde L^1_T\big(H^{\s+4}(\R^2)\big)\qquad\qquad \mbox{ and }\qquad\qquad
\nabla\Pi\,\in\,\wtilde L^1_T\big(H^\s(\R^2)\big)\,.
\]
Moreover, $u$ and $\nabla\Pi$ satisfy the estimates stated in Proposition \ref{p:Stokes_var-a-priori} below.
\end{theorem}

The rest of this section is devoted to the proof of Theorem \ref{t:Stokes_var}. In a first time (see Subsection \ref{ss:Stokes_const}),
we study a sort of ``homogeneous version'' of system \eqref{eq:Stokes_var}, in the sense that the odd term is assumed to be constant coefficient.
In Subsection \ref{ss:Stokes_e} we apply the previous analysis to infer the existence part of Theorem \ref{t:Stokes_var}. The proof of uniqueness
is the object of Subsection \ref{ss:Stokes_u}.

\subsection{Study of a simplified Stokes model} \label{ss:Stokes_const}

In this section, we study a simplification of system \eqref{eq:Stokes_var}, inasmuch as the odd viscosity term is assumed to be constant coefficient, \tsl{i.e.}
we set $h\equiv1$.
Taking again $\nu_0=1$, the system we are interested in writes
\begin{equation}\label{eq:Stokes}
\system{
\begin{aligned}
& \partial_tu\, +\, (v\cdot\nabla)u\,+\,b\,\nabla\Pi\,+\,\nu\,b\,\Delta^2 u\,+\,\Delta u^\perp\,=\,f \\
& \nabla\cdot u\, =\,0 \\
& u_{|t=0}\,=\,u_0\,.
\end{aligned}
}
\end{equation}

We are going to prove the following well-posedness result. 
´

\begin{theorem} \label{t:Stokes}
Let $\s>1$ and $\nu>0$ fixed. Take $T>0$ arbitrarily large. Let the functions $b$ and $v$ satisfy the same assumptions fixed in Theorem \ref{t:Stokes_var} above,
and let $f\,\in\,\wtilde L^1_T\big(H^\s\big)$.

Then, for any $u_0\in H^\s$ there exists a unique solution $\big(u,\nabla\Pi\big)$ to the Stokes-type system \eqref{eq:Stokes} such that
\[
u\,\in\,\wtilde{C}_T\big(H^{\s}(\R^2)\big)\,,\qquad \nu\,u\,\in\,\wtilde L^1_T\big(H^{\s+4}(\R^2)\big)\qquad\qquad \mbox{ and }\qquad\qquad
\nabla\Pi\,\in\,\wtilde L^1_T\big(H^\s(\R^2)\big)\,.
\]
Moreover, $u$ and $\nabla\Pi$ satisfy the estimates stated in Proposition \ref{p:Stokes_a-priori} below.
\end{theorem}

The proof follows the main lines of the analysis performed in Section 3
of \cite{D_2006} for the classical variable-viscosity Stokes system, where the diffusion operator is $-\Delta u$. Hence, we decide to
focus only on the proof of \tsl{a priori} estimates, see Proposition \ref{p:Stokes_a-priori} below.
However, the precise argument to derive, from \tsl{a priori} estimates, the existence of solutions
will be recalled in Subsection \ref{ss:Stokes_e} for the complete system \eqref{eq:Stokes_var}.
Instead, uniqueness of solutions will be an immediate consequence of the \tsl{a priori} estimates.

With respect to the analysis of \cite{D_2006}, here we have to pay attention to the higher order diffusion operator and, above all,
tp the loss of derivatives caused by the odd viscosity term.
The key point is that, in the constant coefficient case,
the odd viscosity term is a perturbation which is skew-symmetric with respect to the $H^\s$ scalar product, for any $\s\in\R$.
Thus, the $H^\s$ estimates for $u$ will not be affected by the presence of that term. On the other hand,
the analysis of the pressure term will be sensitive to it:  this is essentially the reason why, in system \eqref{eq:Stokes},
we consider the higher order viscosity instead of the classical one.

Thus, let us show \tsl{a priori} estimates for smooth solutions of system \eqref{eq:Stokes}. 
They are contained in the following statement, which is the equivalent of Proposition 3.2 of \cite{D_2006} in our context. Notice that
the proof therein does not apply to our case, so we prefer to give all the details.
\begin{prop} \label{p:Stokes_a-priori}
Assume that the assumptions of Theorem \ref{t:Stokes} are in force. Define the function 
$\mc A_T$ by
\[
\mc A_T\,:=\,1\,+\,\left\|b-1\right\|_{\wtilde L^\infty_T(H^{\s+1})}
\]
and the function $V(t)$ as in the statement of Theorem \ref{th:transport}.

Then, any smooth solution $\big(u,\nabla\Pi\big)$ to system \eqref{eq:Stokes} satisfies
the following estimates:
\begin{align*}
&\hspace{-0.5cm}
\left\|u\right\|_{\wtilde L^\infty_T(H^\s)}\,+\,\nu\,\left\|u\right\|_{\wtilde L^1_T(H^{\s+4})} \\
&\quad\leq\,C\,e^{C\,V(T)}\, \Biggl(
\left\|u_0\right\|_{H^\s}\,+\,\mc A_T^{\s+1}\,\left\|f\right\|_{\wtilde L^1_T(H^\s)} \\
&\qquad\qquad\qquad\quad +\,\left(1+\nu^4\left\|b-1\right\|_{\wtilde L^\infty_T(H^{\s+2})}^4\right)\,\mc A_T^{4(\s+1)}\,T\,\left\|u\right\|_{\wtilde L^\infty_T(H^\s)}\,
+\,\left\|v\right\|_{\wtilde L^\infty_T(H^{\s})}^{4/3}\,\mc A_T^{4(\s+1)/3}\,T\,\left\|u\right\|_{\wtilde L^\infty_T(H^\s)}\Biggr) \\
&\hspace{-0.5cm}
\left\|\nabla\Pi\right\|_{\wtilde L^1_T(H^\s)}\,\leq\,C\,\left(\mc A_T^\s\,
\Big(\left\|v\right\|_{\wtilde L^\infty_T(H^\s)}\,\left\|u\right\|_{\wtilde{L}^1_T(H^{\s+1})}\,+\,
\|f\|_{\wtilde L^1_T(H^\s)}\Big)\,+\,
\mc A_T^\s\,\Biggl(1\,+\,\nu\left\|b-1\right\|_{\wtilde L^\infty_T(H^{\s+2})}\right)\,\left\|u\right\|_{\wtilde L^1_T(H^{\s+3})}\Biggr)\,,
\end{align*}
for a suitable constant $C>0$, depending only on the regularity index $\s$ and on the viscosity $\nu$.
In the case when $v=u$, the previous estimates reduce to
\begin{align*}
&\hspace{-0.5cm}
\left\|u\right\|_{\wtilde L^\infty_T(H^\s)}\,+\,\nu\,\left\|u\right\|_{\wtilde L^1_T(H^{\s+4})} \\
&\qquad\qquad\leq\,C\,e^{C\,\mc A_T^{\s+1}\,\wtilde V(T)}\, \Biggl(
\left\|u_0\right\|_{H^\s}\,+\,\mc A_T^{\s+1}\,\left\|f\right\|_{\wtilde L^1_T(H^\s)}\,+\,
\left(1+\nu^4\left\|b-1\right\|_{\wtilde L^\infty_T(H^{\s+2})}^4\right)\,\mc A_T^{4(\s+1)}\,T\,\left\|u\right\|_{\wtilde L^\infty_T(H^\s)}\Biggr) \\
&\hspace{-0.5cm}
\left\|\nabla\Pi\right\|_{\wtilde L^1_T(H^\s)}\,\leq\,C\,\left(\mc A_T^\s\,
\left(\int^T_0\left\|\nabla u(t)\right\|_{L^\infty}\,\left\|u(t)\right\|_{H^{\s}}\,\dd t\,+\,
\|f\|_{\wtilde L^1_T(H^\s)}\right)\,+\,
\mc A_T^\s\,\Biggl(1\,+\,\nu\,\left\|b-1\right\|_{\wtilde L^\infty_T(H^{\s+2})}\right)\,\left\|u\right\|_{\wtilde L^1_T(H^{\s+3})}\Biggr)\,,
\end{align*}
where this time we define $\wtilde V'(t)\,=\,\|\nabla u(t)\|_{L^\infty}$.
\end{prop}

\begin{proof}
Essentially, we have to repeat the transport-diffusion
estimates of Subsection \ref{ss:transp-diff}, taking into account variable coefficients, the higher order diffusion operator and the divergence-free constraint
$\nabla\cdot u=0$, which entails the presence of a pressure term.
It goes without saying that the use of Chemin-Lerner spaces will appear naturally also here.

To begin with, for $j\geq-1$, let us apply the frequency-localisation operator $\Delta_j$ to equation \eqref{eq:Stokes} in order to get an evolution equation
for $\Delta_ju$: we obtain
\begin{equation} \label{eq:u_j}
\d_t\Delta_ju\,+\,(v\cdot\nabla)\Delta_ju\,+\,\Delta_j\nabla\Pi\,+\,\nu\,\Delta\big(b\,\Delta\Delta_ju\big)\,+\,\Delta \Delta_ju^\perp\,=\,
\Delta_jf\,-\,\Delta_j\big((b-1)\,\nabla\Pi\big)\,+\,\nu\,\Delta_j\B\,+\,
\mc C_j\,,
\end{equation}
where we have defined 
\[
\B\,:=\,\,-\,\Big(2\,(\nabla b\cdot\nabla)\Delta u\,+\,\Delta b\,\Delta u\Big)
\qquad\qquad\mbox{ and }\qquad\qquad
\mc C_j\,:=\,\big[v\cdot\nabla,\Delta_j\big]u\,+\,\nu\,\Delta\big([b-1,\Delta_j]\Delta u\big)\,.
\]

Performing an energy estimate for equation \eqref{eq:u_j} and using that both $v$ and $\Delta_ju$ are divergence-free, we easily find
\begin{align*}
\frac{1}{2}\,\frac{\dd}{\dd t}\left\|\Delta_ju\right\|_{L^2}^2\,+\,\nu\,b_*\,\left\|\Delta\Delta_ju\right\|^2\,\leq\,
\left\|\Delta_ju\right\|_{L^2} 
\Big(\left\|\Delta_jf\right\|_{L^2}+\left\|\Delta_j\big((b-1)\,\nabla\Pi\big)\right\|_{L^2}+\nu\left\|\Delta_j\B\right\|_{L^2}
+\left\|\mc C_j\right\|_{L^2}\Big)
\end{align*}
where we have also used the assumption that $b\geq b_*>0$. Thus, when $j=-1$ we deduce
\begin{equation} \label{est:u_j-low}
\frac{\dd}{\dd t}\left\|\Delta_{-1}u\right\|_{L^2}\,\leq\,
\left\|\Delta_{-1}f\right\|_{L^2}+\left\|\Delta_{-1}\big((b-1)\,\nabla\Pi\big)\right\|_{L^2}+\nu\left\|\Delta_{-1}\B\right\|_{L^2}
+\left\|\mc C_{-1}\right\|_{L^2}\,,
\end{equation}
whereas, for $j\geq0$, Bernstein's inequalities imply, for a suitable constant $\k>0$, that
\begin{equation} \label{est:u_j-high} 
\frac{\dd}{\dd t}\left\|\Delta_ju\right\|_{L^2}\,+\,\nu\,b_*\,\k\,2^{4j}\,\left\|\Delta_ju\right\|_{L^2}\,\leq\,
\left\|\Delta_jf\right\|_{L^2}+\left\|\Delta_j\big((b-1)\,\nabla\Pi\big)\right\|_{L^2}+\nu\left\|\Delta_j\B\right\|_{L^2}
+\left\|\mc C_j\right\|_{L^2}\,.
\end{equation} 

Now, set $\wnu\,=\,\nu\,b_*\,\k$.
%
After integrating in time \eqref{est:u_j-low} and \eqref{est:u_j-high}, we multiply the resulting expressions by $2^{\s\,j}$ and perform
a $\ell^2$ summation over $j\geq-1$:
we find
\begin{align}
\left\|u\right\|_{\wtilde L^\infty_T(H^\s)}\,+\,\wnu\,\left\|u\right\|_{\wtilde L^1_T(H^{\s+4})}\,&\leq\,
\left\|u_0\right\|_{H^\s}\,+\,\wnu\,2^{-4}\,\left\|\Delta_{-1}u\right\|_{L^1_T(L^2)}\,+\,\left\|f\right\|_{\wtilde L^1_T(H^\s)} \label{est:u-tilde} \\
&\qquad\quad +\,\left\|(b-1)\,\nabla\Pi\right\|_{\wtilde L^1_T(H^\s)}\,+\,\nu\,\left\|\B\right\|_{\wtilde L^1_T(H^\s)}\,+\,
\left(\sum_{j\geq-1}2^{2\,\s\,j}\,\left\|\mc C_j\right\|^2_{L^1_T(L^2)}\right)^{1/2}\,. \nonumber
\end{align}
At this point, we have to bound all the terms appearing in the right-hand side of the previous estimate.

Let us start with $\B$. Since $\s>1$, the space $H^\s$ is a Banach algebra (recall Corollary \ref{c:tame}). Hence, using the regularity assumption on
$b-1$, we can bound
\begin{align*}
\left\|\B\right\|_{\wtilde L^1_T(H^\s)}\,&\lesssim\,
\left\|\nabla b\right\|_{\wtilde L^\infty_T(H^\s)}\,\left\|\nabla\Delta u\right\|_{\wtilde L^1_T(H^{\s})}\,+\,
\left\|\Delta b\right\|_{\wtilde L^\infty_T(H^\s)}\,\left\|\Delta u\right\|_{\wtilde L^1_T(H^{\s})} \\
&\lesssim\,T^{1/4}\,
\left\|b-1\right\|_{\wtilde L^\infty_T(H^{\s+1})}\,\left\|u\right\|_{\wtilde L^{4/3}_T(H^{\s+3})}\,+\,
T^{1/2}\,
\left\|b-1\right\|_{\wtilde L^\infty_T(H^{\s+2})}\,\left\|u\right\|_{\wtilde L^2_T(H^{\s+2})}\,.
\end{align*}
Employing the interpolation inequality \eqref{est:interp} then yields
\begin{equation} \label{est:B}
\left\|\B\right\|_{\wtilde L^1_T(H^\s)}\,\lesssim\,
\left\|b-1\right\|_{\wtilde L^\infty_T(H^{\s+2})}
\Big(T^{1/4}\,\left\|u\right\|^{1/4}_{\wtilde L^{\infty}_T(H^{\s})}\,\left\|u\right\|^{3/4}_{\wtilde L^{1}_T(H^{\s+4})}\,+\,
T^{1/2}\,\left\|u\right\|^{1/2}_{\wtilde L^{\infty}_T(H^{\s})}\,\left\|u\right\|^{1/2}_{\wtilde L^{1}_T(H^{\s+4})}\Big)\,.
\end{equation}

Next, we turn our attention to the estimates for the pressure term. 
By taking the divergence of the first equation in \eqref{eq:Stokes}, we find an equation for $\nabla \Pi$:
\begin{equation} \label{eq:pressure-Stokes}
-\,\nabla\cdot\big(b\,\nabla\Pi\big)\,=\,\nabla\cdot\Big((v\cdot\nabla)u\,+\,\Delta u^\perp\,+\,\nu\,b\,\Delta^2u\,-\,f\Big)\,.
\end{equation}
From the estimates of Proposition \ref{p:ell-time}, the fact that $\s>1$ and the continuity of the projector $\Q$ over $H^\s$, we infer that
\[
\left\|\nabla\Pi\right\|_{\wtilde L^1_T(H^\s)}\,\lesssim\,\mc A_T^\s\,
\Big(\left\|v\right\|_{\wtilde L^\infty_T(H^\s)}\,\left\|u\right\|_{\wtilde{L}^1_T(H^{\s+1})}\,+\,\left\|u\right\|_{\wtilde L^1_T(H^{\s+2})}\,+\,
\|f\|_{\wtilde L^1_T(H^\s)}\,+\,\nu\,\left\|\Q\big(b\,\Delta^2u\big)\right\|_{\wtilde L^1_T(H^\s)}\Big)\,,
\]
where the implicit multiplicative constant may depend on $b_*$ and $b^*\,:=\,\sup_{[0,T]\times\R^2}b$.
Now, the key observation (already appearing in \cite{D_2006}) is that $\Q\Delta^2u=0$. Hence, using also the Bony paraproduct decomposition
\eqref{eq:bony}, we can write
\begin{align*}
\Q\big(b\,\Delta^2u\big)\,=\,\Q\big((b-1)\,\Delta^2u\big)\,&=\,\Q\left(\mc T_{b-1}\Delta^2u\,+\,\mc T_{\Delta^2u}(b-1)
\,+\,\mc R(b-1,\Delta^2u)\right) \\
&=\,\big[\Q,\mc T_{b-1}\big]\Delta^2u\,+\,\Q\left(\mc T_{\Delta^2u}(b-1)
\,+\,\mc R(b-1,\Delta^2u)\right)\,.
\end{align*}
On the one hand, using the continuity of $\Q$ over $H^\s$, Proposition \ref{p:op} and embedding \eqref{est:emb-time}, we get
\begin{align*}
\left\|\Q\mc T_{\Delta^2u}(b-1)\right\|_{\wtilde L^1_T(H^\s)}\,+\,\left\|\Q\mc R(b-1,\Delta^2u)\right\|_{\wtilde L^1_T(H^\s)}\,&\lesssim\,
\left\|b-1\right\|_{\wtilde L^\infty_T(H^{\s+2})}\,\left\|\Delta^2u\right\|_{L^1_T(B^{-2}_{\infty,\infty})} \\
&\lesssim\,\left\|b-1\right\|_{\wtilde L^\infty_T(H^{\s+2})}\,\left\|u\right\|_{L^1_T(B^{3}_{2,\infty})} \\
&\lesssim\,\left\|b-1\right\|_{\wtilde L^\infty_T(H^{\s+2})}\,\left\|u\right\|_{\wtilde L^1_T(H^{\s+2})}\,.
\end{align*}
On the other hand, the commutator estimate of Proposition \ref{l:ParaComm} yields
\begin{align*}
\left\|\big[\Q,\mc T_{b-1}\big]\Delta^2u\right\|_{\wtilde L^1_T(H^\s)}\,&\lesssim\,\left\|\nabla b\right\|_{L^\infty_T(L^\infty)}\,
\left\|\Delta^2u\right\|_{\wtilde L^1_T(H^{\s-1})}\,
\lesssim\,\left\|b-1\right\|_{\wtilde L^\infty_T(H^{\s+2})}\,\left\|u\right\|_{\wtilde L^1_T(H^{\s+3})}\,,
\end{align*}
where we have used again \eqref{est:emb-time} to treat the term depending on $b$.

Putting all those estimates together, we can finally bound the pressure term in the following way:
\begin{align*}
\left\|\nabla\Pi\right\|_{\wtilde L^1_T(H^\s)}\,&\lesssim\,\mc A_T^\s\,
\Big(\left\|v\right\|_{\wtilde L^\infty_T(H^\s)}\,\left\|u\right\|_{\wtilde{L}^1_T(H^{\s+1})}\,+\,
\|f\|_{\wtilde L^1_T(H^\s)}\Big)\,+\,
\mc A_T^\s\,\left(1\,+\,\nu\,\left\|b-1\right\|_{\wtilde L^\infty_T(H^{\s+2})}\right)\,\left\|u\right\|_{\wtilde L^1_T(H^{\s+3})} \\
&\lesssim\,\mc A_T^\s\,\|f\|_{\wtilde L^1_T(H^\s)}\,+\,
\mc A_T^\s\,T^{3/4}\,\left\|v\right\|_{\wtilde L^\infty_T(H^\s)}\,
\left\|u\right\|^{3/4}_{\wtilde{L}^\infty_T(H^{\s})}\,\left\|u\right\|^{1/4}_{\wtilde{L}^1_T(H^{\s+4})} \\
&\qquad\qquad\qquad\qquad\qquad +\,
\mc A_T^\s\,T^{1/4}\,\left(1\,+\,\nu\,\left\|b-1\right\|_{\wtilde L^\infty_T(H^{\s+2})}\right)\,
\left\|u\right\|^{1/4}_{\wtilde L^{\infty}_T(H^{\s})}\,\left\|u\right\|^{3/4}_{\wtilde L^{1}_T(H^{\s+4})}\,,
\end{align*}
where we have also used \eqref{est:interp} for passing from the first inequality to the second one. As a consequence, we deduce the following estimate:
\begin{align} \label{est:pi-Stokes}
\left\|(b-1)\,\nabla\Pi\right\|_{\wtilde L^1_T(H^\s)}\,&\lesssim\,\mc A_T^{\s+1}\,\|f\|_{\wtilde L^1_T(H^\s)}\,+\,
\mc A_T^{\s+1}\,T^{3/4}\,\left\|v\right\|_{\wtilde L^\infty_T(H^\s)}\,
\left\|u\right\|^{3/4}_{\wtilde{L}^\infty_T(H^{\s})}\,\left\|u\right\|^{1/4}_{\wtilde{L}^1_T(H^{\s+4})} \\
&\qquad\qquad\qquad\qquad\qquad +\,
\mc A_T^{\s+1}\,T^{1/4}\,\left(1\,+\,\nu\,\left\|b-1\right\|_{\wtilde L^\infty_T(H^{\s+2})}\right)\,
\left\|u\right\|^{1/4}_{\wtilde L^{\infty}_T(H^{\s})}\,\left\|u\right\|^{3/4}_{\wtilde L^{1}_T(H^{\s+4})}\,. \nonumber
\end{align}

Therefore, it remains us to control the commutator terms appearing in \eqref{est:u-tilde}. For this, we first use Proposition \ref{p:comm-time}
at $\s$ level of regularity to infer
\begin{equation} \label{est:comm-transp}
\left(\sum_{j\geq-1}2^{2\,\s\,j}\,\left\|\big[v\cdot\nabla,\Delta_j\big]u\right\|^2_{L^1_T(L^2)}\right)^{1/2}\,\lesssim\,
\int^T_0V'(t)\,\left\|u(t)\right\|_{H^{\s}}\,\dd t \,\lesssim\,\int^T_0V'(t)\,\left\|u\right\|_{\wtilde L^\infty_t(H^{\s})}\,\dd t \,.
\end{equation}
For the second term appearing in the definition of $\mc C_j$, 
we use a decomposition similar to the one performed in the analysis of the pressure term: we have
\[
[b-1,\Delta_j]\Delta u\,=\,\big[\mc T_{b-1},\Delta_j\big]\Delta u\,+\,\mc T_{\Delta_j\Delta u}(b-1)\,-\,\Delta_j\mc T_{\Delta u}(b-1)\,+\,
\mc R(b-1,\Delta_j\Delta u)\,-\,\Delta_j\mc R(b-1,\Delta u)\,.
\]
On the one hand, using the spectral properties \eqref{eq:loc-prop}, we can write
\[
\big[\mc T_{b-1},\Delta_j\big]\Delta u\,=\,\sum_{|j-k|\leq 4}\big[S_{k-1}(b-1),\Delta_j\big]\Delta_k\Delta u\,.
\]
Therefore, Lemma 2.97 from \cite{B-C-D} ensures that
\begin{align*}
\left\|\big[\mc T_{b-1},\Delta_j\big]\Delta u\right\|_{L^1_T(L^2)}\,\lesssim\,
2^{-j}\sum_{|j-k|\leq 4}\left\|\nabla S_k(b-1)\right\|_{L^\infty_T(L^\infty)}\,\left\|\Delta_k\Delta u\right\|_{L^1_T(L^2)}\,.
\end{align*}
From this estimate and from the spectral localisation of $\big[\mc T_{b-1},\Delta_j\big]\Delta u$, we can easily bound
\begin{align*}
\left(\sum_{j\geq-1}2^{2\,\s\,j}\,\left\|\Delta\big(\big[\mc T_{b-1},\Delta_j\big]\Delta u\big)\right\|^2_{L^1_T(L^2)}\right)^{1/2}\,&\lesssim\,
\left(\sum_{j\geq-1}2^{2\,(\s+2)\,j}\,\left\|\big[\mc T_{b-1},\Delta_j\big]\Delta u\right\|^2_{L^1_T(L^2)}\right)^{1/2} \\
&\lesssim\,
\left\|\nabla b\right\|_{L^\infty_T(L^\infty)}\,\left(\sum_{j\geq-1}2^{2\,(\s+1)\,j}\left\|\Delta_j\Delta u\right\|^2_{L^1_T(L^2)}\right)^{1/2} \\
&\lesssim\,\left\|b-1\right\|_{\wtilde L^\infty_T(H^{\s+2})}\,\left\|u\right\|_{\wtilde L^1_T(H^{\s+3})}\,.
\end{align*}
On the other hand, using Proposition \ref{p:op} and embeddings \eqref{est:emb-time}, we have
\[
\left\|\mc T_{\Delta u}(b-1)\right\|_{\wtilde L^1_T(H^{\s+2})}\,+\,\left\|\mc R(b-1,\Delta u)\right\|_{\wtilde L^1_T(H^{\s+2})}\,\lesssim\,
\left\|b-1\right\|_{\wtilde L^\infty_T(H^{\s+2})}\,\left\|\Delta u\right\|_{L^1_T(L^\infty)}\,\lesssim\,
\left\|b-1\right\|_{\wtilde L^\infty_T(H^{\s+2})}\,\left\|u\right\|_{\wtilde L^1_T(H^{\s+3})}\,,
\]
whence, after setting $\mc T'_{\Delta u}(b-1)\,:=\,\mc T_{\Delta u}(b-1)\,+\,\mc R(b-1,\Delta u)$, we deduce that
\begin{align*}
\left(\sum_{j\geq-1}2^{2\,\s\,j}\,\left\|\Delta\Delta_j\mc T'_{\Delta u}(b-1)\big)\right\|^2_{L^1_T(L^2)}\right)^{1/2}\,&\lesssim\,
\left(\sum_{j\geq-1}2^{2\,(\s+2)\,j}\,\left\|\Delta_j\mc T'_{\Delta u}(b-1)\right\|^2_{L^1_T(L^2)}\right)^{1/2} \\
&\lesssim\,
\left\|b-1\right\|_{\wtilde L^\infty_T(H^{\s+2})}\,\left\|u\right\|_{\wtilde L^1_T(H^{\s+3})}\,.
\end{align*}
Finally, we can use again the spectral localisation properties \eqref{eq:loc-prop} to write
\begin{align*}
\mc T'_{\Delta_j\Delta u}(b-1)\,&=\,\mc T_{\Delta_j\Delta u}(b-1)\,+\,\mc R(b-1,\Delta_j\Delta u) \\
&=\,\sum_{k\geq-1}S_{k+2}\Delta_j\Delta u\,\Delta_k(b-1)\,=\,
\sum_{k\geq j-2}S_{k+2}\Delta_j\Delta u\,\Delta_k(b-1)\,.
\end{align*}
Therefore, we can estimate
\begin{align*}
2^{2j\s}\,\left\|\Delta\mc T'_{\Delta_j\Delta u}(b-1)\right\|_{L^1_T(L^2)}\,&\lesssim\,2^{2j\s}
\sum_{k\geq j-2}2^{2k}\,\left\|S_{k+2}\Delta_j\Delta u\right\|_{L^1_T(L^\infty)}\,\left\|\Delta_k(b-1)\right\|_{L^\infty_T(L^2)} \\
&\lesssim\,\sum_{k\geq j-2}2^{(j-k)(\s+2)}\,\,\left\|S_{k+2}\Delta_j\Delta u\right\|_{L^1_T(L^\infty)}\,2^{k(\s+2)}\,\left\|\Delta_k(b-1)\right\|_{L^\infty_T(L^2)} \\
&\lesssim\,\left\|b-1\right\|_{\wtilde L^\infty_T(H^{\s+2})}\,2^{j}\,\left\|\Delta_j\Delta u\right\|_{L^1_T(L^2)}\sum_{k\geq j-2}2^{(j-k)(\s+2)}\,c_k\,,
\end{align*}
for a suitable sequence $\big(c_k\big)_k\,\in\,\ell^2$ of unitary norm. Notice that, in the last step, we have also made use of the first Bernstein inequality.
Hence, from the Young inequality for sequences we finally get
\begin{align*}
\left(\sum_{j\geq-1}2^{2\,\s\,j}\,\left\|\Delta\mc T'_{\Delta_j\Delta u}(b-1)\right\|^2_{L^1_T(L^2)}\right)^{1/2}\,&\lesssim\,
\left\|b-1\right\|_{\wtilde L^\infty_T(H^{\s+2})}\,\left(\sum_{j\geq-1}2^{2\,j}\,\left\|\Delta_j\Delta u\right\|_{L^1_T(L^2)}^2\right)^{1/2} \\
&\lesssim\,\left\|b-1\right\|_{\wtilde L^\infty_T(H^{\s+2})}\,\left\|u\right\|_{\wtilde L^1_T(H^{\s+2})}\,,
\end{align*}
where we have used also that $\s+2>3$.

Putting all these inequalities together and keeping in mind \eqref{est:comm-transp}, we finally discover that the commutator terms can be controlled in the following
way:
\begin{align} \label{est:comm}
 \left(\sum_{j\geq-1}2^{2\,\s\,j}\,\left\|\mc C_j\right\|^2_{L^1_T(L^2)}\right)^{1/2}\,&\lesssim\,
 \int^T_0V'(t)\,\left\|u\right\|_{\wtilde L^\infty_t(H^{\s})}\,\dd t\,+\,
\nu\,\left\|b-1\right\|_{\wtilde L^\infty_T(H^{\s+2})}\,\left\|u\right\|_{\wtilde L^1_T(H^{\s+3})} \\
&\lesssim\,
 \int^T_0V'(t)\,\left\|u\right\|_{\wtilde L^\infty_t(H^{\s})}\,\dd t\,+\,T^{1/4}\,
\nu\,\left\|b-1\right\|_{\wtilde L^\infty_T(H^{\s+2})}\,\left\|u\right\|^{1/4}_{\wtilde L^\infty_T(H^{\s})}\,\left\|u\right\|^{3/4}_{\wtilde L^1_T(H^{\s+4})}\,.
 \nonumber
\end{align}

It is now time to plug inequalities \eqref{est:B}, \eqref{est:pi-Stokes} and \eqref{est:comm} into \eqref{est:u-tilde}: observing
that $\mc A_T>1$ and using several times Young's inequality to absorbe all terms of type $\|u\|_{\wtilde L^1_T(H^{\s+4})}$ on
the left-hand side, we find
\begin{align*}
\left\|u\right\|_{\wtilde L^\infty_T(H^\s)}\,+\,\wnu\,\left\|u\right\|_{\wtilde L^1_T(H^{\s+4})}\,&\lesssim\,
\left\|u_0\right\|_{H^\s}\,+\,\wnu\,2^{-4}\,\left\|\Delta_{-1}u\right\|_{L^1_T(L^2)}\,+\,\mc A_T^{\s+1}\,\left\|f\right\|_{\wtilde L^1_T(H^\s)}\,+\,
\int^T_0V'(t)\,\left\|u\right\|_{\wtilde L^\infty_t(H^{\s})}\,\dd t\\
&\quad +\,\left(1+\wnu^4\left\|b-1\right\|_{\wtilde L^\infty_T(H^{\s+2})}^4\right)\mc A_T^{4(\s+1)}T\left\|u\right\|_{\wtilde L^\infty_T(H^\s)}
+\left\|v\right\|_{\wtilde L^\infty_T(H^{\s})}^{4/3}\mc A_T^{4(\s+1)/3}T\left\|u\right\|_{\wtilde L^\infty_T(H^\s)}\,.
\end{align*}
for a new (implicit) multiplicative constant, also depending on $b_*$ and $b^*$, but not on $\wnu$, hence not on $\nu$.
Since the low frequency term can be bounded by
\[
\left\|\Delta_{-1}u\right\|_{L^1_T(L^2)}\,\lesssim\, 
\left\|u\right\|_{L^1_T(H^\s)}
\,\lesssim\,\int^T_0\left\|u\right\|_{\wtilde L^\infty_t(H^\s)}\,\dd t\,,
\]
an application of the Gr\"onwall lemma finally implies the claimed bound for a general transport field $v$.

In the case when $v=u$, the same argument applies, with only (classical) minor modifications. The first change concerns the estimate for the pressure
term: using the embedding $L^1_T\big(H^\s\big)\,\hookrightarrow\,\wtilde L^1_T\big(H^\s\big)$ and the tame estimates
of Corollary \ref{c:tame}, we can bound
\begin{align*}
 \left\|(u\cdot\nabla)u\right\|_{\wtilde L^1_T(H^\s)}\,\lesssim\,\left\|(u\cdot\nabla)u\right\|_{L^1_T(H^\s)}\,\lesssim\,
 \int^T_0\left\|\nabla u(t)\right\|_{L^\infty}\,\left\|u(t)\right\|_{H^\s}\,\dd t\,.
\end{align*}
Analogously, for the first commutator estimate \eqref{est:comm-transp}, we can use the last part of Proposition \ref{p:comm-time}.
Putting those facts together, we easily derive the claimed estimates also in the case when $v=u$.

The proof of the proposition is thus completed.
\end{proof}

As already mentioned, the \tsl{a priori} bounds of Proposition \ref{p:Stokes_a-priori} also conclude our proof of Theorem \ref{t:Stokes}.
As a matter of fact, the arguments for existence and uniqueness of solutions are analogous to the ones used in \cite{D_2006}, and in any case will be reminded
in Subsections \ref{ss:Stokes_e} and \ref{ss:Stokes_u} for the full Stokes-type system \eqref{eq:Stokes_var}.

Let us conclude this part by formulating a remark about the estimates of Proposition \ref{p:Stokes_a-priori}.
\begin{rem} \label{r:Stokes_a-priori}
In the estimates of Proposition \ref{p:Stokes_a-priori}, it is nice that the higher order norm of $b-1$ is always multiplied by a factor $\nu$. Unfortunately,
this will not be the case in the estimates of Proposition \ref{p:Stokes_var-a-priori} below, concerning \tsl{a priori} bounds for the full Stokes
system \eqref{eq:Stokes_var}. We will explain the reason for that in the course of the proof, see especially the comments after estimate \eqref{est:Stokes_add}.
\end{rem}

\subsection{Existence of solutions to the Stokes-type system} \label{ss:Stokes_e}

We can now tackle the proof of the existence of solutions to system \eqref{eq:Stokes_var}, claimed in Theorem \ref{t:Stokes_var}.
As done above, we start by stating \tsl{a priori} estimates for smooth solutions to that system. With respect to Proposition \ref{p:Stokes_a-priori},
there are two important changes. First of all, we have to consider a more complicated forcing term, namely
$f\,=\,g\,+\,\wtilde f$, where $\wtilde f\,=\,\wtilde f(b,u)$ depends on the solution $u$ itself: this will be handled by an additional interpolation of Sobolev
norms. Secondly, we want to apply a finer interpolation, in order to make the $L^1_T\big(L^2\big)$ norm of the solution appear: of course, the price to pay
is to increase the exponent of the function $\mc A_T$ in the exponential term.

\begin{prop} \label{p:Stokes_var-a-priori}
Under the assumptions of Theorem \ref{t:Stokes_var}, there exist an abolute constant $C\,=\,C(\s,\nu)\,>\,0$ and an exponent $\lam\,=\,\lam(\s)\,>\,1$,
both depending only on the quantities inside the brackets (in fact, $\lam(\s)\,=\,(\s+4)/3$),
such that, for any smooth solution $\big(u,\nabla\Pi\big)$ to system \eqref{eq:Stokes_var}, the following estimates hold true:
\begin{align*}
&\hspace{-0.5cm}
\left\|u\right\|_{\wtilde L^\infty_T(H^\s)}\,+\,\nu\,\left\|u\right\|_{\wtilde L^1_T(H^{\s+4})} \\
&\qquad\quad\leq\,C\,e^{C\,V(T)}\, \Biggl(
\left\|u_0\right\|_{H^\s}\,+\,\mc A_T^{\s+1}\,\left\|g\right\|_{\wtilde L^1_T(H^\s)}\,+\,
\left(1+\left\|b-1\right\|_{\wtilde L^\infty_T(H^{\s+2})}^{3\lam}\,+\,\left\|v\right\|_{\wtilde L^\infty_T(H^{\s})}^{\lam}\right)\,\mc A_T^{3\lam(\s+1)}\,\left\|u\right\|_{L^1_T(L^2)}\Biggr) \\
&\hspace{-0.5cm}
\left\|\nabla\Pi\right\|_{\wtilde L^1_T(H^\s)}\,\leq\,C\,\left(\mc A_T^\s\,
\Big(\left\|v\right\|_{\wtilde L^\infty_T(H^\s)}\,\left\|u\right\|_{\wtilde{L}^1_T(H^{\s+1})}\,+\,
\|g\|_{\wtilde L^1_T(H^\s)}\Big)\,+\,
\mc A_T^\s\,\Biggl(1\,+\,\left\|b-1\right\|_{\wtilde L^\infty_T(H^{\s+2})}\right)\,\left\|u\right\|_{\wtilde L^1_T(H^{\s+3})}\Biggr)\,.
\end{align*}
As usual, the function $V(t)$ is defined as in Theorem \ref{th:transport}.
In the case when $v=u$, the previous estimates reduce to
\begin{align*}
&\hspace{-0.5cm}
\left\|u\right\|_{\wtilde L^\infty_T(H^\s)}\,+\,\nu\,\left\|u\right\|_{\wtilde L^1_T(H^{\s+4})} \\
&\qquad\qquad\leq\,C\,e^{C\,\mc A_T^{\s+1}\,\wtilde V(T)}\, \Biggl(
\left\|u_0\right\|_{H^\s}\,+\,\mc A_T^{\s+1}\,\left\|g\right\|_{\wtilde L^1_T(H^\s)}\,+\,
\left(1+\left\|b-1\right\|_{\wtilde L^\infty_T(H^{\s+2})}^{3\lam}\right)\,\mc A_T^{3\lam(\s+1)}\,\left\|u\right\|_{L^1_T(L^2)}\Biggr) \\
&\hspace{-0.5cm}
\left\|\nabla\Pi\right\|_{\wtilde L^1_T(H^\s)}\,\leq\,C\,\left(\mc A_T^\s\,
\left(\int^T_0\left\|\nabla u(t)\right\|_{L^\infty}\,\left\|u(t)\right\|_{H^{\s}}\,\dd t\,+\,
\|g\|_{\wtilde L^1_T(H^\s)}\right)\,+\,
\mc A_T^\s\,\Biggl(1\,+\,\left\|b-1\right\|_{\wtilde L^\infty_T(H^{\s+2})}\right)\,\left\|u\right\|_{\wtilde L^1_T(H^{\s+3})}\Biggr)\,,
\end{align*}
where, as above, we have defined $\wtilde V'(t)\,=\,\|\nabla u(t)\|_{L^\infty}$.
\end{prop}

\begin{proof}
Most of the work has already been done in the proof of Proposition \ref{p:Stokes_a-priori}. Hence we will be a bit sketchy here, and explain only the main
changes to apply in order to get the present result.

To begin with, we observe that we can write
\[
b\,\nabla\cdot\left(\frac{1}{b}\,\nabla u^\perp\right)\,=\,\Delta u^\perp\,-\,(\nabla\log b\cdot\nabla)u^\perp\,.
\]
Thus, system \eqref{eq:Stokes_var} can be recasted, in fact, as its ``homogeneous'' version \eqref{eq:Stokes}, up to take
\[
f\,=\,g\,+\,(\nabla\log b\cdot\nabla)u^\perp\,.
\]
Then,  by repeating the same computations as above, we arrive at the analogue of inequality \eqref{est:u-tilde}:
\begin{align*}
\left\|u\right\|_{\wtilde L^\infty_T(H^\s)}\,+\,\wnu\,\left\|u\right\|_{\wtilde L^1_T(H^{\s+4})}\,&\leq\,
\left\|u_0\right\|_{H^\s}\,+\,\wnu\,2^{-4}\,\left\|\Delta_{-1}u\right\|_{L^1_T(L^2)}\,+\,\left\|g\right\|_{\wtilde L^1_T(H^\s)}\,+\,
\left\|(\nabla\log b\cdot\nabla)u^\perp\right\|_{\wtilde L^1_T(H^\s)} \\
&\qquad\qquad\qquad
+\,\left\|(b-1)\,\nabla\Pi\right\|_{\wtilde L^1_T(H^\s)}\,+\,\nu\,\left\|\B\right\|_{\wtilde L^1_T(H^\s)}\,+\,
\left(\sum_{j\geq-1}2^{2\,\s\,j}\,\left\|\mc C_j\right\|^2_{L^1_T(L^2)}\right)^{1/2}\,. \nonumber
\end{align*}
Now, we use that the operator $\Delta_{-1}$ is continuous over $L^2$ to gather
\[
\left\|\Delta_{-1}u\right\|_{L^1_T(L^2)}\,\lesssim\,\left\|u\right\|_{L^1_T(L^2)}\,.
\]
Next, using an interpolation inequality at the level $\s+3$ of regularity together with the embedding
$L^1_T\big(L^2\big)\hookrightarrow\wtilde L^1_T\big(L^2\big)$, we find that the estimate for $\B$ becomes
\begin{equation} \label{est:B_var}
\left\|\B\right\|_{\wtilde L^1_T(H^\s)}\,\lesssim\,\left\|b-1\right\|_{\wtilde L^\infty_T(H^{\s+2})}\,\|u\|_{\wtilde L^1_T(H^{\s+3})}
\,\lesssim\,
\left\|b-1\right\|_{\wtilde L^\infty_T(H^{\s+2})}\,\|u\|^{1/(\s+4)}_{L^1_T(L^2)}\,\left\|u\right\|^{(\s+3)/(\s+4)}_{\wtilde L^1_T(H^{\s+4})}\,,
\end{equation}
which replaces the previous estimate \eqref{est:B}.
Observe that
\begin{align}
 \left\|(\nabla\log b\cdot\nabla)u^\perp\right\|_{\wtilde L^1_T(H^\s)}\,&\lesssim\,\left\|\nabla\log b\right\|_{\wtilde L^\infty_T(H^\s)}\,
\left\|\nabla u\right\|_{\wtilde L^1_T(H^\s)}\,\lesssim\,\|b-1\|_{\wtilde L^\infty_T(H^{\s+1})}\,\left\|u\right\|_{\wtilde L^1_T(H^{\s+1})}\,,
\label{est:Stokes_add}
\end{align}
where we have used the paralinearisation result of Proposition \ref{p:comp_time} to control the term in $\log b$..
Observe that we lose one derivative on $u$ in \eqref{est:Stokes_add}; to absorbe that loss, we have to use a suitable interpolation,
thus losing also uniformity of the final bound with respect the viscosity coefficient $\nu>0$.
Hence, we decide not to track anymore the explicit dependence of the estimates on $\nu$ and to simply control 
the left-hand side of \eqref{est:Stokes_add} by the quantity appearing in the right-hand side of \eqref{est:B_var}.

Also the analysis of the pressure applies straight away to give
\begin{align*}
\hspace{-0.5cm}
\left\|\nabla\Pi\right\|_{\wtilde L^1_T(H^\s)}\,&\lesssim\,\mc A_T^\s\,
\Big(\left\|v\right\|_{\wtilde L^\infty_T(H^\s)}\,\left\|u\right\|_{\wtilde{L}^1_T(H^{\s+1})}\,+\,
\|g\|_{\wtilde L^1_T(H^\s)}\,+\, \left\|(\nabla\log b\cdot\nabla)u^\perp\right\|_{\wtilde L^1_T(H^\s)}\Big) \\
&\qquad\qquad\qquad\qquad\qquad\qquad\qquad\qquad\qquad\qquad
\,+\,
\mc A_T^\s\,\left(1\,+\,\left\|b-1\right\|_{\wtilde L^\infty_T(H^{\s+2})}\right)\,\left\|u\right\|_{\wtilde L^1_T(H^{\s+3})}\,.
\end{align*}
Notice that, as in \eqref{est:Stokes_add} above, the term $(\nabla\log b\cdot\nabla )u^\perp$ can be controlled by the quantity appearing
in the last line of the previous relation. Thus, arguing again by interpolation, we infer
\begin{align*}
\hspace{-0.5cm}
\left\|\nabla\Pi\right\|_{\wtilde L^1_T(H^\s)}\,
&\lesssim\,\mc A_T^\s\,\|g\|_{\wtilde L^1_T(H^\s)}\,+\,
\mc A_T^\s\,\left\|v\right\|_{\wtilde L^\infty_T(H^\s)}\,\|u\|^{3/(\s+4)}_{L^1_T(L^2)}\,\left\|u\right\|^{(\s+1)/(\s+4)}_{\wtilde L^1_T(H^{\s+4})} \\
&\qquad\qquad\qquad\qquad\qquad\qquad\qquad\qquad +\,
\mc A_T^\s\,\left(1\,+\,\left\|b-1\right\|_{\wtilde L^\infty_T(H^{\s+2})}\right)\,
\|u\|^{1/(\s+4)}_{L^1_T(L^2)}\,\left\|u\right\|^{(\s+3)/(\s+4)}_{\wtilde L^1_T(H^{\s+4})}\,,
\end{align*}
which then implies the equivalent of \eqref{est:pi-Stokes} in this context:
\begin{align*}
\left\|(b-1)\,\nabla\Pi\right\|_{\wtilde L^1_T(H^\s)}\,&\lesssim\,\mc A_T^{\s+1}\,\|g\|_{\wtilde L^1_T(H^\s)}\,+\,
\mc A_T^{\s+1}\,\left\|v\right\|_{\wtilde L^\infty_T(H^\s)}\,\|u\|^{3/(\s+4)}_{L^1_T(L^2)}\,\left\|u\right\|^{(\s+1)/(\s+4)}_{\wtilde L^1_T(H^{\s+4})} \\
&\qquad\qquad\qquad\qquad\qquad\qquad +\,
\mc A_T^{\s+1}\,\left(1\,+\,\left\|b-1\right\|_{\wtilde L^\infty_T(H^{\s+2})}\right)\,
\|u\|^{1/(\s+4)}_{L^1_T(L^2)}\,\left\|u\right\|^{(\s+3)/(\s+4)}_{\wtilde L^1_T(H^{\s+4})}\,.
\end{align*}

The bound of the commutator terms is also analogous to what done in Subsection \ref{ss:Stokes_const} above. First of all, estimate \eqref{est:comm-transp}
still holds true. Next, for the control of the second commutator we can argue as before, just paying attention to use the right interpolation
inequalities. Therefore, we finally get
\begin{align*}
 \left(\sum_{j\geq-1}2^{2\,\s\,j}\,\left\|\mc C_j\right\|^2_{L^1_T(L^2)}\right)^{1/2}\,&\lesssim\,
 \int^T_0V'(t)\,\left\|u\right\|_{\wtilde L^\infty_t(H^{\s})}\,\dd t\,+\,\nu\,
 \left\|b-1\right\|_{\wtilde L^\infty_T(H^{\s+2})}\,\|u\|^{1/(\s+4)}_{L^1_T(L^2)}\,\left\|u\right\|^{(\s+3)/(\s+4)}_{\wtilde L^1_T(H^{\s+4})}\,,
\end{align*}
which plays the role of \eqref{est:comm} in our context.

In the end, we can put all those bounds together: using Young's inequality several times and observing that $3\lam(\s+1)>\lam(\s+1)$ to simplify the
obtained expression, we finally deduce the claimed estimate for $\big(u,\nabla\Pi\big)$.
The proof in the case $v=u$ follows the same lines as the previous argument, so we omit the details.

The proof of Proposition \ref{p:Stokes_var-a-priori} is thus completed.
\end{proof}

To end this part, let us spend a few words on how deducing existence of solutions to \eqref{eq:Stokes_var} from the \tsl{a priori} estimates
of Proposition \ref{p:Stokes_var-a-priori}.
The argument being analogous to the one used in \cite{D_2006} for proving Proposition 3.4 therein, we only briefly recall it.

\begin{proof}[Proof of Theorem \ref{t:Stokes_var} (existence part)]
First of all, we can construct a sequence of smooth, global in time, approximate solutions $\big(u_n,\nabla\Pi_n\big)_{n\in\N}$ to system \eqref{eq:Stokes_var},
for instance by employing the Friedrichs method. Since all $u_n$ and $\nabla\Pi_n$ are smooth and satisfy $\nabla\cdot u_n=0$, we can apply Proposition
\ref{p:Stokes_var-a-priori} to deduce, for any time $t\leq T$, estimates on the $\wtilde L^\infty_t\big(H^\s\big)\,\cap\,\wtilde L^1_t\big(H^{\s+4}\big)$
norm of $u_n$ and on the $\wtilde L^1_t\big(H^\s\big)$ norm of $\nabla \Pi_n$.
At this point, we notice that, employing \eqref{est:emb-time}, the lower order term can be controlled in the following way:
\[
\|u_n\|_{L^1_T(L^2)}\,\leq\,\left\|u_n\right\|_{\wtilde L^1_T(H^{\s})}\,\leq\,T\,\left\|u_n\right\|_{\wtilde L^\infty_T(H^{\s})}\,.
\]
As a consequence, by choosing a $t^*\leq T$ small enough, we can find uniform bounds for $\big(u_n,\nabla\Pi_n\big)_n$ in the space
$\Big(\wtilde L^\infty_{t^*}\big(H^\s\big)\,\cap\,\wtilde L^1_{t^*}\big(H^{\s+4}\big)\Big)\,\cap\, \wtilde L^1_{t^*}\big(H^\s\big)$.
From those uniform bounds, and using the equation to infer uniform boundedness of the time derivatives $\big(\d_tu_n\big)_n$, we can deduce
(by applying \tsl{e.g.} the Ascoli-Arzel\`a theorem)
the strong convergence (up to the extraction of a suitable subsequence) of the $u_n$'s: this allows us to pass to the limit
in the equations and find a solution $\big(u,\nabla\Pi\big)$ of the original system \eqref{eq:Stokes_var}.

The fact that $u$ belongs to $C_{t^*}\big(H^\s\big)$ can be obtained in a rather standard way, hence we omit the argument here. Proposition \ref{p:ell-time}
implies the claimed regularity for the pressure term $\nabla\Pi$, whereas the estimate for it comes from the uniform bounds for $\big(\nabla\Pi_n\big)_n$
combined with the Fatou property, which holds also for (time-dependent) Besov spaces.

Finally, coming back to the estimates of Proposition \ref{p:Stokes_var-a-priori}, we see that the time $t^*$ depends only the regularity index $\s>1$, on
$\mc A_T$, on $V(T)$ and on $\left\|b-1\right\|_{\wtilde L^\infty_T(H^{\s+2})}$.
Therefore, iterating the previous argument a finite number of times, we can cover the whole interval $[0,T]$ and deduce existence of a global in time solution.
\end{proof}

\subsection{Stability estimates and uniqueness} \label{ss:Stokes_u}

In this subsection, we discuss uniqueness, and more general stability issues, for solutions to the Stokes-type system \eqref{eq:Stokes_var}.
In particular, we complete the proof of Theorem \ref{t:Stokes_var}.

\medbreak
The first remark is that, system \eqref{eq:Stokes_var} being linear in the unknowns $u$ and $\nabla\Pi$,
the uniqueness of solutions at the claimed level of regularity is an immediate consequence of the estimates of Proposition \ref{p:Stokes_var-a-priori}.

On the contrary, and quite surprisingly, simple energy estimates in $L^2$ seem not working for uniqueness, owing to the derivative
loss appearing in the pressure equation \eqref{eq:pressure-Stokes} because of the presence of the hyperviscosity term. In order to absorbe that loss, one has to exploit
the full gain of two derivatives given by the diffusion operator: this requires working in the class of Chemin-Lerner spaces,
hence applying a frequency localisation technique and repeating the estimates which follow \eqref{eq:pressure-Stokes}.

Nonetheless, there is a special case in which simple energy estimates yield a uniqueness result, up to imposing a suitable assumption which links
the function $b^{-1}\,=\,h$ and the transport field $v$.
Before presenting the details, let us recast system \eqref{eq:Stokes_var} as
\begin{equation}\label{eq:Stokes_var_2}
\system{
\begin{aligned}
& h\,\big( \partial_tu\, +\, (v\cdot\nabla)u\big)\,+\,\nabla\Pi\,+\,\nu\,\Delta^2 u\,+ \,\nabla\cdot\left(h\,\nabla u^\perp\right)\,=\,h\,g \\
& \nabla\cdot u\, =\, 0 \\ 
& u_{|t=0}\,=\, u_0\,.
\end{aligned}
}
\end{equation}
The result is contained in the following statement. Although it is not strictly necessary
for uniqueness, we include it for the sake of completeness, due to its simplicity and also because it will apply straight away in the computations of Section \ref{s:proof}.
Notice that it requires very low regularity assumptions on the coefficient $h$, on the transport field $v$ and on the external force $g$.

\begin{prop} \label{p:unique-Stokes}
Let $\nu>0$ and $T>0$ be arbitrarily large (possibly $T=+\infty$). Assume that the function $h\in L^\infty\big([0,T]\times\R^2\big)$ is such that $0<h_*\leq h$
almost everywhere. Let $v\in L^1_T\big(W^{1,\infty}(\R^2)\big)$ be such that $\nabla\cdot v=0$ and let $g\in L^1_T\big(L^2(\R^2)\big)$.
Assume that there exists a $p\in[2,+\infty]$ such that the distribution
\[
\d_th\,+\,v\cdot\nabla h\qquad\qquad \mbox{ belongs to }\quad L^{2p/(2p-1)}_T\big(L^p(\R^2)\big)\,.
\]

Then, given any divergence-free initial datum $u_0\in L^2(\R^2)$, there exists at most one solution $\big(u,\nabla\Pi\big)$ of system \eqref{eq:Stokes_var_2},
or equivalently of system \eqref{eq:Stokes_var}, such that $u$ belongs to the energy space
$C_T\big(L^2(\R^2)\big)$, with $\Delta u\,\in\,L^2_T\big(L^2(\R^2)\big)$.
\end{prop}

\begin{proof}
Assume that we have two solutions $\big(u_1,\nabla\Pi_1\big)$ and $\big(u_2,\nabla\Pi_2\big)$ of system \eqref{eq:Stokes_var}, or equivalently of system
\eqref{eq:Stokes_var_2}. Then, if we define the differences
\[
\de u\,:=\,u_1\,-\,u_2\qquad\qquad\quad \mbox{ and }\qquad\qquad\quad \de\nabla\Pi\,:=\,\nabla\Pi_1\,-\,\nabla\Pi_2\,,
\]
the couple $\big(\de u,\de\nabla\Pi\big)$ solves
\[
h\,\big( \partial_t\de u\, +\, (v\cdot\nabla)\de u\big)\,+\,\de\nabla\Pi\,+\,\nu\,\Delta^2\de u\,+ \,\nabla\cdot\left(h\,\nabla\de u^\perp\right)\,=\,0\,,
\]
with $\nabla\cdot(\de u)\,=\,0$. We now multiply the previous equation by $\de u$ in $L^2$. Using the equalities
\begin{align*}
\int h\,\d_t\de u\cdot\de u\,\dd x\,&=\,\frac{1}{2}\,\frac{\dd}{\dt}\left\|\sqrt{h}\,\de u\right\|_{L^2}^2\,-\,\frac{1}{2}\,\int\d_th\,\left|\de u\right|^2\,\dd x \\
\int h\,(v\cdot\nabla)\de u\cdot\de u\,\dd x\,&=\,-\,\frac{1}{2}\,\int v\cdot\nabla h\,\left|\de u\right|^2\,\dd x\,,
\end{align*}
where we have also exploited the fact that $\nabla\cdot v=0$, and the skew-symmetry of the odd-viscosity operator, it is easy to get
\begin{align*}
\frac{1}{2}\,\frac{\dd}{\dt}\left\|\sqrt{h}\,\de u\right\|_{L^2}^2\,+\,\nu\,\left\|\Delta\de u\right\|^2_{L^2}\,\lesssim\,
\int\Big|\d_th\,+\,v\cdot h\Big|\,\left|\de u\right|^2\,\dd x\,.
\end{align*}
Applying the well-known two-dimensional interpolation inequality $\|f\|_{L^\infty}\,\lesssim\,\|f\|^{1/2}_{L^2}\,\|\Delta f\|_{L^2}^{1/2}$
(see \tsl{e.g.} Proposition A.8 in \cite{C-F_Non} for a proof), for any $t\geq0$ we find
\begin{align*}
\left\|\de u(t)\right\|_{L^2}^2\,+\,\nu\,\int^t_0\left\|\Delta\de u\right\|^2_{L^2}\,\dd\t\,&\lesssim\,
\int^t_0\left\|\d_th\,+\,v\cdot h\right\|_{L^2}\,\left\|\de u\right\|^{3/2}_{L^2}\,\left\|\Delta\de u\right\|^{1/2}_{L^2}\,\dd\t \\
&\leq\,\de\,\nu\,\int^t_0\left\|\Delta\de u\right\|^{2}_{L^2}\,\dd\t\,+\,C\,\int^t_0\left\|\d_th\,+\,v\cdot h\right\|_{L^2}^{4/3}\,\left\|\de u\right\|^{2}_{L^2}\,\dd\t\,,
\end{align*}
where we have also used that $h\geq h_*>0$ and we applied the Young inequality to pass from the first to the second line. Here, as usual, the real number $\de>0$
can be as small as we want and the constant $C\,=\,C(\de,\nu)\,>\,0$ only depends on the quantities inside the brackets.

Absorbing the first term on the right with the diffusion term appearing in the left-hand side of the previous inequality and applying the Gr\"onwall lemma
thus yield the desired result in the case when $p=2$.

When $p=+\infty$, instead, we can simply bound
\[
\int^t_0\int\Big|\d_th\,+\,v\cdot h\Big|\,\left|\de u\right|^2\,\dd x\,\dd\tau\,\lesssim\,
\int^t_0\left\|\d_th\,+\,v\cdot\nabla h\right\|_{L^\infty}\,\left\|\de u\right\|^2_{L^2}\,\dd\tau\,,
\]
and an immediate application of the Gr\"onwall lemma gives the result.

The general case $p\in\,]2,+\infty[\,$ follows by interpolation, as we can estimate
\begin{align*}
\int\Big|\d_th\,+\,v\cdot h\Big|\,\left|\de u\right|^2\,\dd x\,&\lesssim\,
\big\|\d_th\,+\,v\cdot h\big\|_{L^p}\,\left\|\de u\right\|_{L^{2p/(p-2)}}\,\left\|\de u\right\|_{L^2} \\
&\lesssim\,
\big\|\d_th\,+\,v\cdot h\big\|_{L^p}\,\left\|\de u\right\|^{(3+\theta)/2}_{L^2}\,\left\|\Delta\de u\right\|^{(1-\theta)/2}_{L^2}\,,
\end{align*}
with $\theta=(p-2)/p\,\in\;]0,1[\,$.

The proposition is then completely proved.
\end{proof}

\section{Proof of Theorem \ref{teo1}} \label{s:proof}

In this section, we use the analysis developed in Sections \ref{s:a-priori} and \ref{s:Stokes} to carry out the proof of Theorem \ref{teo1}.
First of all, we devote our attention to the proof of existence: this is the matter of Subsection \ref{ss:proof-e}.
After that, in Subsection \ref{ss:proof-u} we turn our attention to the proof of uniqueness of solutions in the considered functional framework.
Notice that the continuation criteria claimed in Theorem \ref{teo1} are an immediate consequence of the corresponding criteria
of Theorem \ref{t:uniform_estimates}.

\subsection{Proof of existence} \label{ss:proof-e}

This subsection is devoted to the proof of the existence of a solution to system \eqref{eq:fluid_odd_visco1} at the claimed level of regularity.
The strategy is fairly standard and consists of three main steps. First of all, we construct smooth approximate solutions to a regularised system.
Then, we prove uniform bounds (uniform in the approximation parameter) for the previous family of solutions. Finally, we pass to the limit
in the regularisation parameter and prove convergence (up to the extraction of a suitable subsequence) to a true solution
of the original system.

\subsubsection{Regularisation of the system} \label{sss:reg}

In this paragraph, we construct a sequence of smooth approximate solutions. For this, we decide to proceed by a viscous regularisation of the system.

First of all, we regularise the initial data by use of a smoothing kernel. Thus, take some smooth function $\mu\in C^\infty_0(\R^2)$ such that
$0\leq\mu\leq1$, $\mu$ is radially symmetric and non-increasing along radial directions and verifies 
$\int\mu(x)\,\dd x=1$. Thanks to $\mu$, we can define the family of mollifiers $\big(\mu_\veps\big)_{\veps\in\,]0,1]}$ \tsl{via} the formula
\[
\forall\,\veps\in\,]0,1]\,,\qquad\qquad \mu_{\veps}(x)\,:=\,\frac{1}{\veps^2}\;\mu\left(\frac{x}{\veps}\right)\,.
\]
For notational convenience, we will always use the continuous parameter $\veps\in\,]0,1]$, but (somewhat improperly) speak
about \textit{sequences} of functions $\big(f_\veps\big)_{\veps\in\,]0,1]}$: this is not wrong, if one keeps in mind the choice
\[
 \forall\,n\in\N\,,\qquad\qquad \veps_n\,:=\,\frac{1}{n}\,,
\]
which will be always tacitly made from now on.

Thanks to the mollifying kernel $\big(\mu_\veps\big)_\veps$, we can smooth out the initial data: for any $\veps\in\,]0,1]$, we define
\begin{equation} \label{eq:smooth-data}
\rho_{0,\veps}\,:=\,\mu_\veps\,*\,\rho_0\qquad\qquad\mbox{ and }\qquad\qquad u_{0,\veps}\,:=\,\mu_\veps\,*\,u_0\,.
\end{equation}
Notice that all $\rho_{0,\veps}$ and $u_{0,\veps}$ are $C^\infty$ functions such that
both $\rho_{0,\veps}-1$ and $u_{0,\veps}$ belong to the space $H^\infty\,:=\,\bigcap_{\s\in\R}H^\s$.
However, they are \emph{uniformly bounded} with respect to the parameter $\veps\in\,]0,1]$ only in the regularity spaces given by the assumptions
of Theorem \ref{teo1}: for any $\veps\in\,]0,1]$, we have
\begin{equation} \label{ub:smooth-data}
 \left\|\rho_{0,\veps}\,-\,1\right\|_{H^{s+1}}\,\lesssim\,\left\|\rho_{0}\,-\,1\right\|_{H^{s+1}}\qquad\qquad \mbox{ and }\qquad\qquad
 \left\|u_{0,\veps}\right\|_{H^{s}}\,\lesssim\,\left\|u_0\right\|_{H^{s}}\,,
\end{equation}
together with the strong convergence properties
\begin{equation} \label{conv:smooth-data}
\rho_{0,\veps}\,-\,1\,\longrightarrow\,\rho_0\,-\,1\qquad \mbox{ in }\quad H^{s+1}\qquad\qquad\qquad \mbox{ and }\qquad\qquad\qquad
u_{0,\veps}\,\longrightarrow\,u_0\qquad \mbox{ in }\quad H^s
\end{equation}
in the limit $\veps\,\ra\,0^+$.
In addition, one has
\begin{equation} \label{eq:smooth-u_div}
\forall\,\veps\in\,]0,1]\,,\qquad\qquad \nabla\cdot u_{0,\veps}\,=\,0\,.
\end{equation}
Furthermore, because $\mu$ has unit integral over $\R^2$, owing to \eqref{est:rho_0-inf} we infer that
\begin{equation} \label{ub:rho_0,eps-inf}
\forall\,\veps\in\,]0,1]\,,\qquad\qquad
\rho_*\,\leq\,\rho_{0,\veps}\,\leq\,\rho^*\,.
\end{equation}

After having regulised the initial data, we proceed to the construction of smooth solutions to an approximate system.
For $\veps\in\,]0,1]$ fixed, we introduce the regularised system
\begin{equation}\label{eq:odd_approx}
\system{
\begin{aligned}
& \d_t\rho_\veps\, +\, \nabla\cdot \big(\rho_\veps\,u_\veps\big)\, =\,0 \\
& \partial_t \pare{\rho_\veps\, u_\veps}\, +\, \nabla\cdot \pare{\rho_\veps\, u_\veps \otimes u_\veps}\, +\, \nabla \pi_\veps\,  +\, 
\, \nabla\cdot \pare{\rho_\veps \ \nabla u_\veps^\perp}\,+\,\veps\,\Delta^2u_\veps =\, 0 \\
& \nabla\cdot u_\veps\, =\,0\,,
\end{aligned}
}
\end{equation}
supplemented with the initial datum
\begin{equation} \label{eq:in-dat_approx}
\big(\rho_\veps,u_\veps\big)_{|t=0}\,=\,\big(\rho_{0,\veps},u_{0,\veps}\big)\,.
\end{equation}

System \eqref{eq:odd_approx} is a kind of non-homogeneous incompressible Navier-Stokes system with a higher viscosity term, of fourth degree, and
a second order perturbation due to the odd viscosity term.
Notice that this perturbation is linear in $u$. Therefore, applying the analysis performed in Section \ref{s:Stokes},
we claim that, for any $\veps\in\,]0,1]$ \emph{fixed}, we can uniquely solve system \eqref{eq:odd_approx}, at least locally in time.

\begin{prop} \label{p:viscous}
Let $\veps\in\,]0,1]$ fixed. 
Then, for any $\s>1$ arbitrarily large, there exists a unique smooth local in time solution to system
\eqref{eq:odd_approx} supplemented with the initial datum \eqref{eq:in-dat_approx}.

More precisely, there exists 
a triplet $\big(\rho_\veps,u_\veps,\nabla\pi_\veps\big)$ of smooth functions such that the following holds true:
for any $\s>1$ fixed, there exists a time $T\,=\,T^{(\s)}_\veps\,>\,0$ such that
$\big(\rho_\veps,u_\veps,\nabla\pi_\veps\big)$ is the unique solution to system \eqref{eq:odd_approx}-\eqref{eq:in-dat_approx} on $[0,T]\times\R^2$
having the following properties:
\begin{itemize}
 \item $\rho_\veps\,-\,1$ belongs to $\wtilde{C}\big([0,T];H^{\s+2}\big)$;
 \item the velocity field $u_\veps$ belongs to $\wtilde C\big([0,T];H^\s\big)\,\cap\,\wtilde{L}^1\big([0,T];H^{\s+4}\big)$;
 \item the pressure term $\nabla\pi_\veps$ belongs to the space $\wtilde{L}^1\big([0,T];H^{\s}\big)$.
\end{itemize}
\end{prop}

Before proving the previous result, some comments are in order.
\begin{rem} \label{r:viscous}
Because of Propositions \ref{p:Stokes_var-a-priori} and \ref{p:unique-Stokes}, we see that the triplet $\big(\rho_\veps,u_\veps,\nabla\pi_\veps\big)$
of smooth solutions does not depend on the regularity index $\s>1$.

On the contrary, so far the existence time $T_\veps^{(\s)}$ of that solution may depend on $\s$: of course, we expect this is not the case
(keep in mind the continuation criteria of Theorem \ref{t:uniform_estimates}), but, strictly speaking, we have not proved
this for system \eqref{eq:odd_approx} yet. For the time being, by embeddings we can only say that the map $\s\,\mapsto\,T^{\s}_\veps$ is a non-increasing function.
The fact that indeed $T_\veps^{(\s)}$ may be chosen indedepend of $\s>1$ is an important point of our analysis,
and will be shown in the course of the proof of Proposition
\ref{p:unif-bounds_approx} below.
\end{rem}

This having been clarified, we can now prove Proposition \ref{p:viscous}.

\begin{proof}[Proof of Proposition \ref{p:viscous}]
For convenience, since $\veps>0$ is fixed, let us drop it from the indices throughout this proof.
On the contrary, we will keep writing the index $\veps$ in the notation for the initial data, to stress the fact that we are considering the smoothed data.

\tbf{Step 1: approximate solutions.}
In order to solve system \eqref{eq:odd_approx} and prove the previous propostion, we implement an iterative procedure.
Firstly, we set
\[
\big(\rho^0, u^0, \nabla\pi^0\big)\,:=\,\big(\rho_{0,\veps}, u_{0,\veps}, 0\big)\,.
\]
Obviously, it follows from our assumptions that, for any $T>0$ and any $\s>1$, one has
\[
\big(\rho^0, u^0, \nabla\pi^0\big)\,\in\,
E^\s_T\,,
\]
where we have defined the space
\begin{align*}
E^\s_T\,&:=\,\Big\{\big(a,v,\nabla\Pi\big)\,\in\,
L^\infty\big([0,T]\times\R^2\big)\times\Big(\wtilde{C}_T\big(H^\s\big)\,\cap\,\wtilde{L}^1_T\big(H^{\s+4}\big)\Big)\times\wtilde{L}^1_T\big(H^{\s}\big)\;\Big| \\
&\qquad\qquad\qquad\qquad\qquad\qquad\qquad\qquad\qquad\qquad\qquad\qquad\qquad\qquad
a-1\in\wtilde{C}_T\big(H^{\s+2}\big)\qquad \mbox{ and }\qquad
\div v\,=\,0\;\Big\}\,.
\end{align*}
Notice that the space $E^\s_T$ is a Banach space, if endowed with the natural norm
\[
\left\|\big(a,v,\nabla\Pi\big)\right\|_{E^\s_T}\,:=\,\left\|a-1\right\|_{\wtilde L^\infty_T(H^{\s+2})}\,+\,
\left\|v\right\|_{\wtilde L^\infty_T(H^\s)}\,+\,\left\|v\right\|_{\wtilde L^1_T(H^{\s+4})}\,+\,\left\|\nabla\Pi\right\|_{\wtilde L^1_T(H^\s)}\,.
\]

Next, for $n\in\N$, assume that a triplet $\big(\rho^n,u^n,\nabla\pi^n\big)$ of smooth functions over $\R_+\times\R^2$ is given, such
that it belongs to the space $E^\s_T$ for any $T>0$ and any $\s>1$.
Then, 
we define the new triplet $\big(\rho^{n+1},u^{n+1},\nabla\pi^{n+1}\big)$ as follows.
First of all, we define $\rho^{n+1}$ as the unique global in time solution to the transport equation
\begin{equation} \label{eq:rho^n}
\d_t\rho^{n+1}\,+\,u^n\cdot\nabla\rho^{n+1}\,=\,0\,,\qquad\qquad\mbox{ with }\qquad \rho^{n+1}_{|t=0}\,=\,\rho_{0,\veps}\,.
\end{equation}
Then, Theorem \ref{t:tr-diff} implies that $\rho^{n+1}$ belongs to $\wtilde C_T(H^\s)$, for any time $T>0$
and for any regularity index $\s>1$. In addition, since $\rho^{n+1}$ is transported by the smooth divergence-free velocity field $u^n$,
in view of \eqref{ub:rho_0,eps-inf} we can guarantee that
\begin{equation} \label{est:rho^n_inf}
\forall\,(t,x)\in\R_+\times\R^2\,,\qquad\qquad \rho_*\,\leq\,\rho^{n+1}(t,x)\,\leq\,\rho^*\,.
\end{equation}
In order to define $\big(u^{n+1},\nabla\pi^{n+1}\big)$, we can now 
solve, globally in time, the equation
\begin{equation} \label{eq:u^n}
\left\{\begin{array}{l}
\partial_t u^{n+1}\, +\, \left(u^{n}\cdot\nabla\right) u^{n+1}\, +\,\dfrac{1}{\rho^{n+1}}\,\nabla \pi^{n+1}\,+\,\veps\,
\dfrac{1}{\rho^{n+1}}\,\Delta^2 u^{n+1}\,+\,\dfrac{1}{\rho^{n+1}}\,\nabla\cdot\left(\rho^{n+1}\,\nabla\big(u^{n+1}\big)^\perp\right)\,=\,0 \\[1ex]
\nabla\cdot u^{n+1}_\veps\, =\,0 \\[1ex]
u^{n+1}_{|t=0}\,=\,u_{0,\veps}\,.
       \end{array}
\right.
\end{equation}
Indeed, Theorem \ref{t:Stokes_var} yields the existence of a unique solution $\big(u^{n+1},\nabla\pi^{n+1}\big)$ to the previous system such that,
for any $T>0$ and any $\s>1$, we have
\[
u^{n+1}\,\in\,\wtilde C_T\big(H^\s\big)\,\cap\,\wtilde L^1_T\big(H^{\s+4}\big)\qquad\qquad\mbox{ and }\qquad\qquad
\nabla\pi^{n+1}\,\in\,\wtilde L^1_T\big(H^\s\big)\,.
\]
Therefore, the new triplet $\left(\rho^{n+1},u^{n+1},\nabla\pi^{n+1}\right)$ belongs to the space $E^\s_T$, for any time $T>0$ and any space regularity $\s>1$.

In the end, this construction yields a sequence of solutions $\Big(\big(\rho^n,u^n,\nabla\pi^n\big)\Big)_{n\in\N}\,\subset\,E^\s_T$ for any $T>0$ and any $\s>1$.
The next step of the proof is to find, for any $\s>1$ \emph{fixed}, uniform bounds for that sequence in the space $E^\s_T$,
at least for some small enough $T=T^{(\s)}>0$ (of course $T^{(\s)}$ will depend also on $\veps>0$).

\tbf{Step 2: uniform bounds.}
To begin with, keeping in mind the bound \eqref{est:rho^n_inf} and the definition of $\rho^{n+1}$, by induction it is an easy matter to see that
\begin{equation} \label{est:ub-rho}
\forall\,n\in\N\,,\quad \forall\,(t,x)\in\R_+\times\R^2\,,\qquad\qquad \rho_*\leq\rho^n(t,x)\leq\rho^*\qquad \mbox{ and }\qquad
\left\|\rho^n(t)-1\right\|_{L^2}\,=\,\left\|\rho_{0,\veps}-1\right\|_{L^2}\,.
\end{equation}
In addition, the first equation in \eqref{eq:u^n} is equivalent to
\begin{equation} \label{eq:u^n_reform}
\rho^{n+1}\,\d_tu^{n+1}\,+\,\rho^{n+1}\,\left(u^n\cdot\nabla\right) u^{n+1}\,+\,\nabla\pi^{n+1}\,+\,\veps\,\Delta^2u^{n+1}\,+\,
\nabla\cdot\left(\rho^{n+1}\,\nabla\big(u^{n+1}\big)^\perp\right)\,=\,0\,.
\end{equation}
Thus, an energy estimate, combined with the condition $\nabla\cdot u^{n+1}=0$ and the previous uniform $L^\infty$ bounds for $\big(\rho^n\big)_n$, implies that
\begin{equation} \label{est:u^n_en}
\forall\,n\in\N\,,\quad \forall\,t\in\R_+\,,\qquad\qquad \left\|u^n(t)\right\|_{L^2}\,\lesssim\,\left\|u_{0,\veps}\right\|_{L^2}\,.
\end{equation}

Now, we derive uniform estimates for higher order regularity norms. For notational convenience, let us define, for any $n\in\N$ and any $t\geq0$, the quantity
\[
U^n(t)\,:=\,\left\|u^n\right\|_{\wtilde L^\infty_t(H^\s)}\,+\,\veps\,\left\|u^n\right\|_{\wtilde L^1_t(H^{\s+4})}\,.
\]
We also observe that, by interpolation, we gather
\[
u^n\,\in\,L^2_t\big(H^{\s+2}\big)\,,\qquad\qquad \mbox{ with }\qquad
\left\|u^n\right\|_{L^2_t(H^{\s+2})}\,\lesssim\,\left\|u^n\right\|_{\wtilde L^\infty_t(H^\s)}^{1/2}\,\left\|u^n\right\|^{1/2}_{\wtilde L^1_t(H^{\s+4})}\,\lesssim\,U^n(t)\,.
\]
Of course, we omit to track the explicit dependence of the various multiplicative constants on the viscosity coefficient $\veps>0$.
Notice that $\s+2>3$.
Therefore, applying the estimates of Theorem \ref{t:tr-diff} (with $\nu=0$ and $q=+\infty$) to equation \eqref{eq:rho^n}, we infer,
for any $n\in\N$ and any time $T>0$ fixed, the bound
\begin{equation} \label{est:rho^n_high}
\left\|\rho^{n+1}-1\right\|_{\wtilde L^\infty_T(H^{\s+2})}\,\leq\,C\,\left\|\rho_{0,\veps}-1\right\|_{H^{\s+2}}\,
\exp\left(C\,\int^T_0\left\|u^n(t)\right\|_{H^{\s+2}}\,\dd t\right)\,\leq\,C_1\,\left\|\rho_{0,\veps}-1\right\|_{H^{\s+2}}\,
e^{C_2\,\sqrt{T}\,U^n(T)}\,,
\end{equation}
for suitable constants $C_1>0$ and $C_2>0$.

As for the estimates for $u^{n+1}$, we can apply Proposition \ref{p:Stokes_var-a-priori}, where we take $\nu=\veps$, $g=0$, $b\,=\,1/\rho^{n+1}$
and $v=u^n$, to system \eqref{eq:u^n}. Observe that, owing to Proposition \ref{p:comp_time}, in this case one has
\[
\mc A_T\,\lesssim\,1\,+\,\left\|\rho^{n+1}-1\right\|_{\wtilde L^\infty_T(H^{\s+2})}\,.
\]
Therefore, in the estimate of Proposition \ref{p:Stokes_var-a-priori}, we can simply bound
\[
\mc A_T^{3\lam(\s+1)}\,\leq\,\left(1\,+\,\left\|\rho^{n+1}-1\right\|_{\wtilde L^\infty_T(H^{\s+2})}\,+\,\left\|u^n\right\|_{\wtilde L^\infty_T(H^\s)}\right)^{3\lam(\s+1)}\,.
\]
Keeping this in mind and using also \eqref{est:u^n_en}, for any $T>0$ we find
\begin{align}
U^{n+1}(T)\,&\leq\,C\,e^{C\,\sqrt{T}\,U^n(T)}\,\left(\left\|u_{0,\veps}\right\|_{H^\s}\,+\,
\Big(1+\|\rho^{n+1}-1\|_{H^{\s+2}}\,+\,U^n(T)\Big)^{3\lam(\s+2)}\,T\,\left\|u_{0,\veps}\right\|_{L^2}\right) \label{est:u^n_high} \\
&\leq\,C_1\,\left\|u_{0,\veps}\right\|_{H^\s}\,e^{C_2\,\sqrt{T}\,U^n(T)}\,\left(1\,+\,\Big(1+\|\rho^{n+1}-1\|_{H^{\s+2}}\,+\,U^n(T)\Big)^{3\lam(\s+2)}\,T\right)\,.
\nonumber 
\end{align}
Without loss of generality, we can assume that the two multiplicative constanst $C_{1,2}$ appearing in the previous estimate
are the same as the ones appearing in \eqref{est:rho^n_high}, and that they are both bigger than $1$.

Assuming that $T$ has been chosen so small that
\begin{equation} \label{eq:cond_T_eps}
C_2\,\sqrt{T}\,U^n(T)\,\leq\,\log2 \qquad\qquad \mbox{ and }\qquad\qquad
\Big(1\,+\,2\,C_1\,\|\rho_{0,\veps}-1\|_{H^{\s+2}}\,+\,U^n(T)\Big)^{3\lam(\s+2)}\,T\,\leq\,3 \,,
\end{equation}
from \eqref{est:rho^n_high} and \eqref{est:u^n_high} we immediately deduce that
\begin{equation} \label{est:rho^n-U^n}
\left\|\rho^{n+1}-1\right\|_{\wtilde L^\infty_T(H^{\s+2})}\,\leq\,2\,C_1\,\left\|\rho_{0,\veps}-1\right\|_{H^{\s+2}}
\qquad\mbox{ and }\qquad 
U^{n+1}(T)\,\leq\,8\,C_1\,\left\|u_{0,\veps}\right\|_{H^\s}\,.
\end{equation}

In view of the previous argument, let us define $T>0$ in the following way:
\begin{align*}
T\,&:=\,\sup\Bigg\{t\geq0\;\Big|\quad 8\,C_1\,C_2\,\left\|u_{0,\veps}\right\|_{H^\s}\,\sqrt{T}\,\leq\,\log2 \\
&\qquad\qquad\qquad\qquad\qquad \mbox{ and }\qquad
\Big(1\,+\,2\,C_1\,\|\rho_{0,\veps}-1\|_{H^{\s+2}}\,+\,8\,C_1\,\left\|u_{0,\veps}\right\|_{H^\s}\Big)^{3\lam(\s+2)}\,T\,\leq\,3\Bigg\}\,.
\end{align*}
Then, by induction it is easy to show that, for any $n\in\N$, conditions \eqref{eq:cond_T_eps} are satisfied, whence bounds
\eqref{est:rho^n-U^n} are satisfied as well. With those estimates at hand, we can apply Proposition \ref{p:Stokes_var-a-priori} again
to infer bounds for the pressure terms $\nabla\pi^{n}$, for any $n\in\N$.

In the end, we have proved that, given $T\,=\,T^{(\s)}\,>\,0$ defined as above, one has
\begin{equation} \label{est:approx_n}
\sup_{n\in\N}\left\|\big(\rho^n,u^n,\nabla\pi^n\big)\right\|_{E^\s_T}\,\leq\,C\,,
\end{equation}
for an absolute constant $C>0$, depending only on the regularity index $\s>1$, on the viscosity coefficient $\veps>0$ and on the norms of the initial datum
$\big(\rho_{0,\veps},u_{0,\veps}\big)$.
The uniform bound \eqref{est:approx_n} allows us to extract weak limit-points of the sequence $\big(\rho^n,u^n,\nabla\pi^n\big)_n$ and to deduce weak convergence
properties (up to suitable extraction), but of course this is not enough to pass to the limit in the equations.
Hence, for completing the proof of the proposition, we have to prove some strong convergence property for $\big(\rho^n,u^n,\nabla\pi^n\big)_n$.

\tbf{Step 3: convergence.}
We are going to prove stability in the energy space $C_T\big(L^2\big)$. To begin with, let us introduce the following notation: for any $(n,m)\in\N^2$, we set
\[
(\de f)^{n,m}\,:=\,f^{n+m}\,-\,f^n\,,\qquad\qquad \mbox{ where }\qquad f\,\in\,\big\{\rho,u,\nabla\pi\big\}\,.
\]

It follows from \eqref{eq:rho^n} that $(\de\rho)^{n,m}$ solves
\[
\d_t(\de\rho)^{n,m}\,+\,u^{n+m-1}\cdot\nabla(\de\rho)^{n,m}\,=\,-\,(\de u)^{n-1,m}\cdot\nabla\rho^n\,.
\]
Thus, an energy estimate immediately implies the bound
\begin{equation} \label{est:de-rho}
\frac{1}{2}\,\frac{\dd}{\dt}\left\|(\de\rho)^{n,m}\right\|_{L^2}^2\,\lesssim\,\left\|(\de u)^{n-1,m}\right\|_{L^2}\,\left\|(\de\rho)^{n,m}\right\|_{L^2}\,
\left\|\nabla\rho^n\right\|_{L^\infty}\,\lesssim\,\left\|(\de u)^{n-1,m}\right\|_{L^2}\,\left\|(\de\rho)^{n,m}\right\|_{L^2}\,,
\end{equation}
where we have also used \eqref{est:approx_n}.

Using \eqref{eq:u^n_reform}, instead, we find an equation for $(\de u)^{n,m}$: we have
\begin{align*}
&\Big(\rho^{n+m}\d_t+\rho^{n+m}\left(u^{n+m-1}\cdot\nabla\right)\Big)(\de u)^{n,m} + (\de\nabla\pi)^{n,m} + \veps\Delta^2(\de u)^{n,m} +
\nabla\cdot\left(\rho^{n+m}\nabla\left((\de u)^{n,m}\right)^\perp \right) \\
&\qquad = -\Bigg(\nabla\cdot\left((\de\rho)^{n,m}\nabla\left(u^n\right)^\perp\right) + (\de\rho)^{n,m}\,\d_tu^n + 
(\de\rho)^{n,m}\left(u^{n+m-1}\cdot\nabla\right)u^n + \rho^n \left((\de u)^{n-1,m}\cdot\nabla\right)u^n \Bigg)\,.
\end{align*}
Form another energy estimate, together with suitable integrations by parts and the use of equation \eqref{eq:rho^n} at the level $n+m$, we then infer
\begin{align*}
&\frac{1}{2}\,\frac{\dd}{\dt}\left\|(\de u)^{n,m}\right\|_{L^2}^2\,+\,\veps\,\left\|\Delta(\de u)^{n,m}\right\|_{L^2}^2 \\
&\quad \,\lesssim\,
\left\|(\de\rho)^{n,m}\right\|_{L^2}\,\left\|\nabla(\de u)^{n,m}\right\|_{L^2}\,\left\|\nabla u^n\right\|_{L^\infty}\,+\,
\left\|(\de\rho)^{n,m}\right\|_{L^2}\,\left\|(\de u)^{n,m}\right\|_{L^2}\,\left\|\d_tu^n\right\|_{L^\infty} \\
&\qquad \,+\,\left\|(\de\rho)^{n,m}\right\|_{L^2}\,\left\|(\de u)^{n,m}\right\|_{L^2}\,\left\|u^{n+m-1}\right\|_{L^\infty}\,\left\|\nabla u^n\right\|_{L^\infty}\,+\,
\left\|(\de u)^{n-1,m}\right\|_{L^2}\,\left\|(\de u)^{n,m}\right\|_{L^2}\,\left\|\rho^n\right\|_{L^\infty}\,\left\|\nabla u^n\right\|_{L^\infty}\,.
\end{align*}
We now define the ``energy'' function
\[
\mc E^{n,m}(t)\,:=\,\left\|\nabla(\de \rho)^{n,m}(t)\right\|_{L^2}^2\,+\,\left\|\nabla(\de u)^{n,m}(t)\right\|_{L^2}^2\,.
\]
Using the interpolation inequality $\|\nabla f\|_{L^2}\,\lesssim\,\|f\|_{L^2}^{1/2}\,\|\Delta f\|_{L^2}^{1/2}$ together with the Young inequality for handling the
first term on the right, from the previous estimate we deduce
\begin{align*}
&\frac{1}{2}\,\frac{\dd}{\dt}\left\|(\de u)^{n,m}\right\|_{L^2}^2\,+\,\veps\,\left\|\Delta(\de u)^{n,m}\right\|_{L^2}^2\,\lesssim\,\mc E^{n,m}\,
\Big(\left\|\nabla u^n\right\|_{L^\infty}\,+\,\left\|\d_tu^n\right\|_{L^\infty}\Big)\,+\,C\,\left\|\nabla u^n\right\|_{L^\infty}\,
\sqrt{\mc E^{n-1,m}}\;\sqrt{\mc E^{n,m}}\,.
\end{align*}
where we have also made use of \eqref{est:approx_n} and \eqref{est:ub-rho}.
Notice that the multiplicative constant depends on $\veps>0$, which however is fixed at this level.
Summing up the previous inequality with \eqref{est:de-rho} and integrating in time, we gather
\[
\mc E^{n,m}(t)\,+\,\veps\int^t_0\left\|\Delta(\de u)^{n,m}\right\|_{L^2}^2\,\leq\,
C\int^t_0\Big(1\,+\,\left\|\nabla u^n\right\|^2_{L^\infty}\,+\,\left\|\d_tu^n\right\|_{L^\infty}\Big)\,\mc E^{n,m}(\t)\,\dd\t\,+\,
\frac{1}{2}\int^t_0\mc E^{n-1,m}(\t)\,\dd\t\,.
\]
At this point, we observe that, owing to \eqref{est:approx_n} and embeddings, we have that
\[
\left(\nabla u^n\right)\qquad \mbox{ is bounded in }\quad L^2_T\big(H^{\s+1}\big)\,\hookrightarrow\,L^2_T\big(L^\infty\big)\,,
\]
because $\s+1>2$. In addition, from the equation \eqref{eq:u^n} for $u^n$ and product rules, we easily deduce that
\[
\left(\d_tu^n\right)\qquad \mbox{ is bounded in }\quad \wtilde{L}^1_T\big(H^{\s}\big)\,\hookrightarrow\,L^1_T\big(H^{\s-\de}\big)\,
\hookrightarrow\,L^1_T\big(L^\infty\big)\,,
\]
where we have also made use of \eqref{est:emb-time} and we have chosen $\de>0$ so small that $\s-\de>1$. Thus, an application of the
Gr\"onwall lemma yields
\[
\mc E^{n,m}(t)\,+\,\veps\int^t_0\left\|\Delta(\de u)^{n,m}\right\|_{L^2}^2\,\leq\,\frac{C_T}{2}\int^t_0\mc E^{n-1,m}(\t)\,\dd\t\,,
\]
for a positive constant $C_T$ also depending on time, but independent of $n$ and $m$. Observe that, as the initial data in \eqref{eq:rho^n} and \eqref{eq:u^n}
do not depend on $n$, the quantity $\mc E^{n,m}(0)$ 
always  vanishes, \tsl{i.e.} one has $\mc E^{n,m}(0)\,\equiv\,0$.

Define now
\[
\mc F^n(t)\,:=\,\sup_{m\geq0}\sup_{\t\in[0,t]}\mc E^{n,m}(t)\,.
\]
Then, the previous inequality immediately implies, for any $t\in[0,T]$, the bound
\[
\mc F^n(t)\,\leq\,\frac{C_T}{2}\,\int^t_0\mc F^{n-1}(\t)\,\dd\t\,,
\]
from which we deduce, by an induction argument, the estimate
\[
\forall\,t\in[0,T]\,,\qquad\qquad \mc F^n(t)\,\leq\,
\frac{\big(C_T\,T\big)^n}{n!}\,\mc F^0(t)\,.
\]
From this estimate it immediately follows that both $\big(\rho^n-1\big)_n$ and $\big(u^n\big)_n$ are Cauchy sequences in $C_T\big(L^2\big)$, thus, by interpolation,
in any intermediate space $C_T\big(H^{\wtilde \s}\big)$, for any $\wtilde\s<\s+2$ (for the density functions) and $\wtilde\s<\s$ (for the velocity fields).

These convergence properties are strong enough to pass to the limit (for $n\ra+\infty$, at any $\veps>0$ fixed) in the weak formulation of equations
\eqref{eq:rho^n} and \eqref{eq:u^n}.
In the end, we deduce the existence of a solution $\big(\rho,u,\nabla\pi\big)$, which in fact depends on the $\veps>0$ fixed, to system
\eqref{eq:odd_approx}-\eqref{eq:in-dat_approx}. The space regularity of this solution follows from \eqref{est:approx_n} and the Fatou property of Sobolev spaces.

This concludes the proof of existence of solutions to the approximate system \eqref{eq:odd_approx}, for any fixed value $\veps>0$ of the artificial viscosity parameter.
\end{proof}

\subsubsection{Uniform bounds for the approximate solutions} \label{sss:unif}

Proposition \ref{p:viscous} provides us with a family $\big(\rho_\veps,u_\veps,\nabla\pi_\veps\big)_{\veps>0}$ of smooth solutions to the approximate
problem \eqref{eq:odd_approx}, for any value $\veps>0$ fixed.
At a first glance, there is no chance for $\big(\rho_\veps,u_\veps,\nabla\pi_\veps\big)_{\veps}$ to be uniformly bounded in any high regularity space.
Indeed, the odd viscosity term is responsible for a two derivatives loss, which can be absorbed only using viscosity.

The goal of this paragraph is to show that, resorting to the structure highlighted in Section \ref{s:a-priori} above,
the sequence of approximate solutions $\big(\rho_\veps,u_\veps,\nabla\pi_\veps\big)_{\veps}$
is actually bounded in the space $H^{s+1}\times H^s\times H^{s-2}$ on some short time interval $[0,T_0]$.
More precisely, we are going to prove the following statement.

\begin{prop} \label{p:unif-bounds_approx}
There exist a time $T_0>0$, a constant $C>0$ and a small $\veps_0\in\,]0,1]$ such that
\[
\inf_{\veps\in\,]0,\veps_0]}T_\veps\,\geq\,T_0\,>\,0\qquad\quad \mbox{ and }\qquad\quad
\sup_{\veps\in\,]0,\veps_0]}\Big(\left\|\rho_\veps-1\right\|_{L^\infty_{T_0}(H^{s+1})}\,+\,\left\|u_\veps\right\|_{L^\infty_{T_0}(H^s)}\,+\,
\left\|\nabla\pi_\veps\right\|_{\wtilde L^1_{T_0}(H^{s-2})}\Big)\,\leq\,C\,.
\]
The constant $C$ only depends on the quantity $E_0:=\left\|\rho_0-1\right\|_{H^{s+1}}+\left\|u_0\right\|_{H^s}$, introduced after \eqref{def:E},
and on $\rho_*$ and $\rho^*$ appearing in \eqref{est:rho_0-inf}.
\end{prop}

\begin{proof}
Let us fix $\s\gg s$ in Proposition \ref{p:viscous}.
Let us call $T^{(\s)}_\veps$ the lifespan of the solution $\big(\rho_\veps,u_\veps,\nabla\pi_\veps\big)$ at $H^\s$ level of regularity. 
As already noticed in Remark \ref{r:viscous}, for any $\veps>0$ and any $1<\wtilde\s\leq\s$
the inequalities $T_\veps^{(\wtilde\s)}\,\geq\,T_\veps^{(\s)}\,>\,0$ hold true.
Nonetheless, as a byproduct of our estimates, we will see that in fact one has $T_\veps^{(\wtilde\s)}=T_\veps^{(\s)}$ for any  $1<\wtilde\s\leq\s$,
so the stated property $\inf_{\veps\in\,]0,\veps_0]}T_\veps\geq T_0>0$ is not ambiguous.

In order to prove the previous statement, essentially we have to reproduce the estimates of Section \ref{s:a-priori}, paying attention that the presence
of the hyperviscosity term $+\veps\Delta^2u_\veps$ does not destroy the underlying hyperbolic structure of the odd system.

We start by considering the low regularity norms. It is clear that inequalities \eqref{est:ub-rho} and \eqref{est:u^n_en} pass to the limit: owing to
the smoothness of all the quantities involved in the computations, classical properties of the transport equation for $\rho_\veps$ imply that
\begin{align*}
\forall\,\veps>0\,,\quad \forall\,(t,x)\in[0,T_\veps]\times\R^2\,,\qquad \qquad &\rho_*\,\leq\,\rho_\veps(t,x)\,\leq\,\rho^* \\
&\qquad \mbox{ and }\qquad
\left\|\rho_\veps(t)-1\right\|_{L^2}\,=\,\left\|\rho_{0,\veps}-1\right\|_{L^2}\,\lesssim\,\left\|\rho_{0}-1\right\|_{L^2}\,,
\end{align*}
whereas, by the use of the previous positivity bounds on $\rho_\veps$, an energy estimate for the equation for $u_\veps$ yields
\[
\forall\,\veps>0\,,\quad \forall\,t\in[0,T_\veps]\,,\qquad \frac{1}{2}\,\left\|u_\veps(t)\right\|_{L^2}^2\,+\,
\veps\int^t_0\left\|\Delta u_\veps(\t)\right\|^2_{L^2}\,\dd\t\,\leq\,\frac{1}{2}\,\left\|u_{0,\veps}\right\|_{L^2}^2\,\lesssim\,\left\|u_{0}\right\|_{L^2}^2\,.
\]

Similarly to what done in Section \ref{s:a-priori}, higher order regularity estimates will be obtained by working on $\theta_\veps\,:=\,\eta_\veps-\Delta\rho_\veps$,
where $\eta_\veps\,:=\,\nabla^\perp\cdot\big(\rho_\veps\,u_\veps\big)$, and on $\o_\veps\,:=\,\nabla^\perp\cdot u_\veps$. More precisely, we are going to
bound the quantity
\[
F_\veps(t)\,:=\,\left\|\rho_\veps-1\right\|_{L^\infty_t(L^2)}\,+\,\left\|u_\veps\right\|_{L^\infty_t(L^2)}\,+\,
\left\|\theta_\veps\right\|_{\wtilde{L}^\infty_t(H^{s-1})}\,+\,
\left\|\o_\veps\right\|_{\wtilde L^\infty_t(H^{s-1})}
\]
on some time interval $[0,T_0]$, \emph{uniformly} both in $t$ and $\veps\in\,]0,1]$. As we will see, the use of the time-dependent norms will be absolutely
crucial in our argument.

To begin with, let us derive an equation for $\theta_\veps$:
performing the same computations as in Paragraph \ref{sss:eq-theta}, we easily obtain
\begin{align*}
\d_t\theta_\veps\,+\,u_\veps\cdot\nabla\theta_\veps\,+\,\veps\,\Delta^2\o_\veps\,=\,
\frac{1}{2}\,\nabla^\perp\rho_\veps\cdot\nabla\big|u_\veps\big|^2\,+\,\mc B\big(\nabla u_\veps,\nabla^2\rho_\veps\big)\,.
\end{align*}
We remark that the hyperviscosity term seems to destroy the nice structure of the equations put in evidence in Subsection \ref{ss:new-en}, as it gives rise to the
term $\veps\,\Delta^2\o$ on the left, but the latter does not yield any parabolic smoothing in the estimates for $\theta_\veps$.
As a matter of fact, employing Theorem \ref{t:tr-diff} with $\nu=0$ and $q=+\infty$ and repeating the estimates of Paragraph \ref{sss:theta-est},
we find, for any $\veps\in\,]0,1]$ and any $T\in[0,T_\veps]$, the bound
\begin{align} \label{est:theta_eps}
\left\|\theta_\veps\right\|_{\wtilde L^\infty_T(H^{s-1})}\,\lesssim\,\exp\left(C\int^T_0F_\veps(t)\,\dt\right)\,
\Bigg(\left\|\theta_{0,\veps}\right\|_{H^{s-1}}\,+\,T\,\Big(\big(F_\veps(T)\big)^2\,+\,\big(F_\veps(T)\big)^{s+2}\Big)\,+\,
\veps\,\left\|\o_\veps\right\|_{\wtilde L^1_T(H^{s+3})}\Bigg)\,,
\end{align}
which is the analogue of estimate \eqref{est:theta-transp} for the approximate system \eqref{eq:odd_approx}.
In the previous inequality, we have set $\theta_{0,\veps}\,:=\,\nabla^\perp\cdot\big(\rho_{0,\veps}\,u_{0,\veps}\big)-\Delta\rho_{0,\veps}$.

As observed above, the presence of the last term on the right-hand side of \eqref{est:theta_eps} seems to destroy all the estimates. The key point, which will save all
the argument, is that the estimates for $\o_\veps$ are \emph{closed}, in the sense that they do not require higher regularity of $\theta_\veps$ and $\rho_\veps$
(hence the use of some hyperviscous effect on those quantities), but just the norms appearing in the definition of $F_\veps$.
Another key fact is the appearing of the $\wtilde L^1_T(H^{s-1})$ norm of $\o_\veps$ in \eqref{est:theta_eps}: as a matter of fact, as we will see in a while,
there is no hope to bound $\o_\veps$ in $L^1_T(H^{s-1})$, whereas the corresponding bound in the larger space $\wtilde L^1_T(H^{s-1})$
is at reach (keep in mind the statement of Theorem \ref{t:tr-diff}).

Let us make rigorous what we have just announced. We compute the equation for $\o_\veps$, getting
\begin{align*}
\d_t\o_\veps+\left(u_\veps-\nabla^\perp\log\rho_\veps\right)\cdot\nabla\o_\veps+\veps\frac{1}{\rho_\veps}\Delta^2\o_\veps\,=\,
-\nabla^\perp\left(\frac{1}{\rho_\veps}\right)\cdot\nabla\left(\pi_\veps-\rho_\veps\o_\veps\right)-\mc B\big(\nabla u_\veps,\nabla^2\log\rho_\veps\big)
-\veps\nabla^\perp\left(\frac{1}{\rho_\veps}\right)\cdot\Delta^2u_\veps\,.
\end{align*}
Next, we reproduce the estimates of Subsection \ref{ss:Stokes_const}. Actually, here the situation is simpler, as no pressure nor odd viscosity term appear;
we have only to pay attention to the new terms arising from the hyperviscosity.
Hence, arguing as in Proposition \ref{p:Stokes_a-priori} with $\s=s-1$ and using the analogue of estimates \eqref{est:B} and \eqref{est:comm}, we find
the following bound:
\begin{align*}
&\left\|\o_\veps\right\|_{\wtilde L^\infty_T(H^{s-1})}\,+\,\veps\,\frac{\k}{\rho^*}\,\left\|\o_\veps\right\|_{\wtilde L^1_T(H^{s+3})}\,\lesssim\,
\left\|\o_{0,\veps}\right\|_{H^{s-1}}\,+\,\int^T_0\left(\left\|\nabla u_\veps\right\|_{H^{s-1}}+\left\|\nabla\rho_\veps\right\|_{H^{s}}\right)\,\left\|\o_\veps\right\|_{\wtilde L^\infty_t(H^{s-1})}\,\dd t \\
&\qquad\qquad\qquad\,+\,
\left\|\nabla^\perp\left(\frac{1}{\rho_\veps}\right)\cdot\nabla\left(\pi_\veps-\rho_\veps\o_\veps\right)\right\|_{\wtilde L^1_T(H^{s-1})} 
\,+\,\left\|\mc B\big(\nabla u_\veps,\nabla^2\log\rho_\veps\big)\right\|_{\wtilde L^1_T(H^{s-1})}\,+\,\veps\,
\left\|\nabla^\perp\left(\frac{1}{\rho_\veps}\right)\cdot\Delta^2u_\veps\right\|_{\wtilde L^1_T(H^{s-1})} \\
&\qquad\qquad\qquad\qquad +\,\veps\,\left\|\rho_\veps-1\right\|_{\wtilde L^\infty_T(H^{s+1})}
\Big(T^{1/4}\,\left\|\o_\veps\right\|^{1/4}_{\wtilde L^{\infty}_T(H^{s-1})}\,\left\|\o_\veps\right\|^{3/4}_{\wtilde L^{1}_T(H^{s+3})}\,+\,
T^{1/2}\,\left\|\o_\veps\right\|^{1/2}_{\wtilde L^{\infty}_T(H^{s-1})}\,\left\|\o_\veps\right\|^{1/2}_{\wtilde L^{1}_T(H^{s+3})}\Big)
\,,
\end{align*}
with $\o_{0,\veps}\,:=\,\nabla^\perp\cdot u_{0,\veps}$.
Observe that the previous estimate holds for any $\veps\in\,]0,1]$ and any $T\in[0,T_\veps]$.

At this point, we have to control the terms on the right-hand side. It is apparent that the terms in the first line are good, as well as
the ones on the last line (which can be treated by a use of the Young inequality).
It remains us to find suitable bounds for the terms in the second line. For this, we perform the same analysis as in Paragraph \ref{sss:omega}:
on the one hand, by product rules we get
\begin{align*}
\left\|\nabla^\perp\left(\frac{1}{\rho_\veps}\right)\cdot\nabla\left(\pi_\veps-\rho_\veps\o_\veps\right)\right\|_{\wtilde L^1_T(H^{s-1})}
&\lesssim\,\left\|\nabla^\perp\left(\frac{1}{\rho_\veps}\right)\right\|_{\wtilde L^\infty_T(H^{s-1})}\,
\left\|\nabla\left(\pi_\veps-\rho_\veps\o_\veps\right)\right\|_{\wtilde L^1_T(H^{s-1})} \\
&\lesssim\,\left\|\nabla\rho_\veps\right\|_{\wtilde L^\infty_T(H^{s-1})}\,
\left\|\nabla\left(\pi_\veps-\rho_\veps\o_\veps\right)\right\|_{\wtilde L^1_T(H^{s-1})} \\
\veps\,\left\|\nabla^\perp\left(\frac{1}{\rho_\veps}\right)\cdot\Delta^2u_\veps\right\|_{\wtilde L^1_T(H^{s-1})}\,&\lesssim\,\veps\,
\left\|\nabla\rho_\veps\right\|_{\wtilde L^\infty_T(H^{s-1})}\,\left\|\Delta^2u_\veps\right\|_{\wtilde L^1_T(H^{s-1})}\,\lesssim\,
\veps\,\left\|\nabla\rho_\veps\right\|_{\wtilde L^\infty_T(H^{s-1})}\,\left\|\o_\veps\right\|_{\wtilde L^1_T(H^{s+2})} \\
&\lesssim\,\veps\,T^{1/4}\,\left\|\nabla\rho_\veps\right\|_{\wtilde L^\infty_T(H^{s-1})}\,\left\|\o_\veps\right\|_{\wtilde L^\infty_T(H^{s-1})}^{1/4}
\,\left\|\o_\veps\right\|^{3/4}_{\wtilde L^1_T(H^{s+3})}\,.
\end{align*}
On the other hand, analogously to \eqref{est:B-omega}, we find
\begin{align*}
 \left\|\mc B\big(\nabla u_\veps,\nabla^2\log\rho_\veps\big)\right\|_{\wtilde L^1_T(H^{s-1})}\,&\lesssim\,T\,\left(F_\veps^2\,+\,F_\veps^{4s+1}\right)\,.
\end{align*}
So, it remains us to bound the pressure term, namely the quantity $\nabla\big(\pi_\veps\,-\,\rho_\veps\o_\veps\big)$. First of all, let us write an equation
for $\nabla\pi_\veps$: by taking the divergence of the momentum equation in \eqref{eq:odd_approx}, similarly to \eqref{eq:ell-p_2} we find
\begin{equation} \label{eq:p_eps}
-\nabla\cdot\left(\frac{1}{\rho_\veps}\,\nabla \pi_\veps\right)\,=\,\nabla\cdot\left((u_\veps\cdot\nabla)u_\veps\,+\,
(\nabla\log\rho_\veps\cdot\nabla)u_\veps^\perp\,+\,\veps\,\frac{1}{\rho_\veps}\,\Delta^2u_\veps\right)\,-\,\Delta\omega_\veps\,.
\end{equation}
Lemma \ref{l:laxmilgram} thus yields, exactly as for \eqref{est:p-L^2}, the bound
\begin{align*}
\left\|\nabla\pi_\veps\right\|_{L^1_T(L^2)}\,&\lesssim\,T\,\Big(\left\|\nabla\rho_\veps\right\|_{\wtilde L^\infty_T(H^{s-1})}\,+\,
\left\|u_\veps\right\|_{\wtilde L^\infty_T(H^s)}\Big)\,+\,T\,\left\|u_\veps\right\|_{\wtilde L^\infty_T(H^s)}\,+\,
\veps\,\left\|\Delta^2u_\veps\right\|_{L^1_T(L^2)} \\
&\lesssim\,T\,\Big(\left\|\nabla\rho_\veps\right\|_{\wtilde L^\infty_T(H^{s-1})}\,+\,
\left\|u_\veps\right\|_{\wtilde L^\infty_T(H^s)}\Big)\,+\,
\veps\,\left\|\o_\veps\right\|_{\wtilde L^1_T(H^{s+1})}\,, \nonumber
\end{align*}
where we have also used that $s>2$, so that $s+1>3$.
From this, we infer 
\begin{align}
\label{est:pi_eps-low}
\left\|\nabla\big(\pi_\veps\,-\,\rho_\veps\,\o_\veps\big)\right\|_{\wtilde L^1_T(L^2)}\,\lesssim\,
\left\|\nabla\big(\pi_\veps\,-\,\rho_\veps\,\o_\veps\big)\right\|_{L^1_T(L^2)}\,\lesssim\,T\,\left(F_\veps\,+\,F_\veps^{2s}\right)\,+\,
\veps\,T^{1/2}\,\left\|\o_\veps\right\|^{1/2}_{\wtilde L^\infty_T(H^{s-1})}\,\left\|\o_\veps\right\|^{1/2}_{\wtilde L^1_T(H^{s+3})}\,,
\end{align}
which of course implies the same bound also for the low frequency part $\Delta_{-1}\nabla\big(\pi_\veps-\rho_\veps\o_\veps\big)$
in $L^1_T(L^2)$. 
As for the high frequency part, similarly to \eqref{est:p-rho-omega}, we have to bound
\[
\left(\sum_{j\geq0}2^{2j(s-1)}\,\left\|\Delta_j\nabla(-\Delta)^{-1}\Phi_\veps\right\|_{L^1_T(L^2)}^2\right)^{1/2}\,,
\]
where this time we have defined
\begin{align*}
 \Phi_\veps\,&:=\,-\,\nabla\log\rho_\veps\cdot\nabla \pi_\veps\,+\,\rho_\veps\,\nabla\cdot\Big((u_\veps\cdot\nabla)u_\veps\,+\,
 (\nabla\log\rho_\veps\cdot\nabla)u_\veps^\perp\Big)\,-\, \big[\rho_\veps-1,\Delta\big]\omega_\veps\,+\,\veps\,\nabla\left(\frac{1}{\rho_\veps}\right)\cdot\Delta^2u_\veps \\
 &=\,\Phi_{\veps,1}\,+\,\Phi_{\veps,2}\,+\,\Phi_{\veps,3}\,+\,\veps\,\Phi_{\veps,4}\,.
\end{align*}
The control of the norms involving $\Phi_{\veps,1}$, $\Phi_{\veps,2}$ and $\Phi_{\veps,3}$ is exactly the same as in Paragraph \ref{sss:omega}: from \eqref{est:Phi_3}, \eqref{est:Phi_2} and \eqref{est:Phi_1} we deduce
\[
\left(\sum_{j\geq0}2^{2j(s-1)}\,\left\|\Delta_j\nabla(-\Delta)^{-1}\Big(\Phi_{\veps,1}\,+\,\Phi_{\veps,2}\,+\,\Phi_{\veps,3}\Big)\right\|_{L^1_T(L^2)}^2\right)^{1/2}
\,\lesssim\,T\,\left(F_\veps^2\,+\,F_\veps^{2s+1}\right)\,+\,\left(F_\veps\,+\,F_\veps^s\right)\,\left\|\nabla\pi_\veps\right\|_{\wtilde L^1_T(H^{s-2})}\,.
\]
For the term $\Phi_{\veps,4}$, we can use the product rules of Proposition \ref{p:op} (recall that $s+1>3$) to write
\begin{align*}
\left(\sum_{j\geq0}2^{2j(s-1)}\,\left\|\Delta_j\nabla(-\Delta)^{-1}\Phi_{\veps,4}\right\|_{L^1_T(L^2)}^2\right)^{1/2}\,&\lesssim\,
\left\|\Phi_{\veps,4}\right\|_{\wtilde L^1_T(H^{s-2})}\,\lesssim\,
\left\|\nabla\rho_\veps\right\|_{\wtilde L^\infty_T(H^{s-1})}\,\left\|\o_\veps\right\|_{\wtilde L^1_T(H^{s+1})} \\
&\lesssim\,T^{1/2}\,\left\|\nabla\rho_\veps\right\|_{\wtilde L^\infty_T(H^{s-1})}\,
\left\|\o_\veps\right\|^{1/2}_{\wtilde L^\infty_T(H^{s-1})}\,\left\|\o_\veps\right\|^{1/2}_{\wtilde L^1_T(H^{s+3})}\,.
\end{align*}
At this point, we deal with the quantity $\left\|\nabla\pi_\veps\right\|_{\wtilde L^1_T(H^{s-2})}$ as done at the end of Paragraph \ref{sss:omega}: paying
attention to the $O(\veps)$ term appearing in \eqref{est:pi_eps-low} and to the additional $\Phi_{\veps,4}$ term in the estimates, we find
\begin{align*}
&\left\|\nabla\big(\pi_\veps-\rho_\veps\o_\veps\big)\right\|_{\wtilde L^1_T(H^{s-1})} \\
&\qquad\qquad
\lesssim\,\veps\,T^{1/2}\,
\left\|\nabla\rho_\veps\right\|_{\wtilde L^\infty_T(H^{s-1})}\,
\left\|\o_\veps\right\|^{1/2}_{\wtilde L^\infty_T(H^{s-1})}\,\left\|\o_\veps\right\|^{1/2}_{\wtilde L^1_T(H^{s+3})}\,+\,
T\,\left(F_\veps^2\,+\,F_\veps^{2s+1}\right) \\
&\qquad\qquad\qquad +\,\left(F_\veps\,+\,F_\veps^s\right)\,\Big(T\,\left(F_\veps\,+\,F_\veps^{2s}\right)\,+\,
\veps T^{1/2}\left\|\o_\veps\right\|^{1/2}_{\wtilde L^\infty_T(H^{s-1})}\,\left\|\o_\veps\right\|^{1/2}_{\wtilde L^1_T(H^{s+3})}\Big)^{1/(s-1)}\,
\left\|\nabla\big(\pi_\veps-\rho_\veps\o_\veps\big)\right\|_{\wtilde L^1_T(H^{s-1})}^{(s-2)/(s-1)}\,,
\end{align*}
which in turn yields, after an application of the Young inequality, the bound
\begin{equation} \label{est:pressure_eps}
\left\|\nabla\big(\pi_\veps-\rho_\veps\o_\veps\big)\right\|_{\wtilde L^1_T(H^{s-1})}\,\lesssim\,T\,\left(F_\veps\,+\,F_\veps^{\g_0}\right)\,+\,
\veps\,T^{1/2}\,\left(F_\veps\,+\,F_\veps^{s^2-s}\right)
\left\|\o_\veps\right\|^{1/2}_{\wtilde L^\infty_T(H^{s-1})}\,\left\|\o_\veps\right\|^{1/2}_{\wtilde L^1_T(H^{s+3})}\,,
\end{equation}
where $\g_0=s^2+s+1$ as in \eqref{est:p-r-o_final} and where we have bounded $\left\|\nabla\rho_\veps\right\|_{\wtilde L^\infty_T(H^{s-1})}$
following Lemma \ref{l:F}. Observe that, as $s>2$, one has $s^2-s>s$, which justifies the power of $F_\veps$ appearing in the second term on the right-hand
side of the previous inequality.

Now, we collect all the estimates found so far, and plug them into the first bound obtained for the vorticity $\o_\veps$. By doing so, we find
\begin{align*}
&\left\|\o_\veps\right\|_{\wtilde L^\infty_T(H^{s-1})}\,+\,\veps\,\frac{\k}{\rho^*}\,\left\|\o_\veps\right\|_{\wtilde L^1_T(H^{s+3})}\,\lesssim\,
\left\|\o_{0,\veps}\right\|_{H^{s-1}}\,+\,\int^T_0\left(F_\veps\,+\,F_\veps^s\right)\,\left\|\o_\veps\right\|_{\wtilde L^\infty_t(H^{s-1})}\,\dd t \\
&\qquad\,+\,T\,\left(F_\veps^2\,+\,F_\veps^{4s+1}\right)\,+\,\veps\,\left(F_\veps\,+\,F_\veps^s\right)\,
\Big(T^{1/4}\,\left\|\o_\veps\right\|^{1/4}_{\wtilde L^{\infty}_T(H^{s-1})}\,\left\|\o_\veps\right\|^{3/4}_{\wtilde L^{1}_T(H^{s+3})}\,+\,
T^{1/2}\,\left\|\o_\veps\right\|^{1/2}_{\wtilde L^{\infty}_T(H^{s-1})}\,\left\|\o_\veps\right\|^{1/2}_{\wtilde L^{1}_T(H^{s+3})}\Big) \\
&\qquad\qquad\qquad\qquad +\,\left(F_\veps\,+\,F_\veps^s\right)\,T\,\left(F_\veps\,+\,F_\veps^{\g_0}\right)\,+\,
\veps\,\left(F_\veps\,+\,F_\veps^s\right)\,T^{1/2}\,\left(F_\veps\,+\,F_\veps^{s^2-s}\right)
\left\|\o_\veps\right\|^{1/2}_{\wtilde L^\infty_T(H^{s-1})}\,\left\|\o_\veps\right\|^{1/2}_{\wtilde L^1_T(H^{s+3})}\,.
\end{align*}
Gathering together all the similar terms and observing that $\g:=\g_0+s=(s+1)^2$, we deduce
\begin{align*}
&\left\|\o_\veps\right\|_{\wtilde L^\infty_T(H^{s-1})}\,+\,\veps\,\frac{\k}{\rho^*}\,\left\|\o_\veps\right\|_{\wtilde L^1_T(H^{s+3})}\,\lesssim\,
\left\|\o_{0,\veps}\right\|_{H^{s-1}}\,+\,\int^T_0\left(F_\veps\,+\,F_\veps^s\right)\,\left\|\o_\veps\right\|_{\wtilde L^\infty_t(H^{s-1})}\,\dd t \\
&\quad\,+\,T\,\left(F_\veps^2\,+\,F_\veps^{\g}\right)\,+\,\veps\,\left(F_\veps\,+\,F_\veps^{s^2}\right)\,
\Big(T^{1/4}\,\left\|\o_\veps\right\|^{1/4}_{\wtilde L^{\infty}_T(H^{s-1})}\,\left\|\o_\veps\right\|^{3/4}_{\wtilde L^{1}_T(H^{s+3})}\,+\,
T^{1/2}\,\left\|\o_\veps\right\|^{1/2}_{\wtilde L^{\infty}_T(H^{s-1})}\,\left\|\o_\veps\right\|^{1/2}_{\wtilde L^{1}_T(H^{s+3})}\Big)\,.
\end{align*}
Applying the Young inequality first and the Gr\"onwall lemma then, we finally infer
\begin{align}
\label{est:o_eps}
\left\|\o_\veps\right\|_{\wtilde L^\infty_T(H^{s-1})}\,+\,\veps\,\frac{\k}{\rho^*}\,\left\|\o_\veps\right\|_{\wtilde L^1_T(H^{s+3})}\,&\lesssim\,
\exp\left(C\int^T_0\Big(F_\veps(t)\,+\,\big(F_\veps(t)\big)^s\Big)\,\dt\right) \\
&\qquad
\times\,\Bigg(\left\|\o_{0,\veps}\right\|_{H^{s-1}}\,+\,T\,\left(F_\veps^2\,+\,F_\veps^{\g}\right)\,+\,\veps\,T\,
\left(F_\veps^2\,+\,F_\veps^{4s^2}\right)\,\left\|\o_\veps\right\|_{\wtilde L^{\infty}_T(H^{s-1})}\Bigg)\,. \nonumber 
\end{align}

With \eqref{est:o_eps} at hand, we can conclude our argument. To begin with, 
we plug \eqref{est:o_eps} into \eqref{est:theta_eps}, and then we sum up the resulting expression with \eqref{est:o_eps} itself. We obtain, for any $\veps\in\,]0,1]$
and any $T\in[0,T_\veps]$, the following estimate:
\begin{align} \label{est:F_eps}
F_\veps(T)\,+\,\veps\,\left\|\o_\veps\right\|_{\wtilde{L}^1_T(H^{s+3})}\,&\leq\,\wtilde{C}_1
\exp\Bigg(\wtilde C_2\,T\,\Big(F_\veps(T)\,+\,\big(F_\veps(T)\big)^s\Big)\Bigg)\,\times\,\Bigg(F_{\veps}(0)\,+\,
T\,\Big(\big(F_\veps(T))^2\,+\,\big(F_\veps(T)\big)^{\g_1}\Big)\Bigg)\,,
\end{align}
for two suitable constants $\wtilde{C}_{1,2}>0$, 
where we have defined
\[
 F_\veps(0)\,:=\,\left\|\rho_{0,\veps}-1\right\|_{L^2}+\left\|u_{0,\veps}\right\|_{L^2}+\left\|\o_{0,\veps}\right\|_{H^{s-1}}+\left\|\theta_{0,\veps}\right\|_{H^{s-1}}
\]
and we have set $\g_1\,:=\,4s^2+1$ (observe that $4s^2+1>(s+1)^2=\g$ for $s>1$).

Next, we observe that, in the previous argument, there is nothing specific to the regularity value $s$: we can reproduce the previous estimates
at any level of regularity $H^\s$, for any $\s>2$. In addition, making a thorough use of the tame estimates, it is possible to get a continuation
criterion for smooth solutions to system \eqref{eq:odd_approx}, at any value of the parameter $\veps>0$ \emph{fixed}: the solution can be continued beyond a time $T>0$
whenever
\begin{align*}
&\int_0^{T}
\Big(\left\|\nabla u_\veps\right\|^2_{L^\infty}\,+\,\left\|\nabla \rho_\veps\right\|^{\s}_{L^\infty}\,+\,
\left\|\nabla\rho_\veps\right\|^{\max\{2,\s-1\}}_{L^\infty}\,\left\|\nabla u_\veps\right\|_{L^\infty} \\
&\qquad\qquad\qquad\qquad\qquad\qquad\qquad\qquad
\,+\,\left\|\nabla^2 \rho_\veps\right\|_{L^\infty}\,+\,
\left\|\nabla\big(\pi_\veps-\rho_\veps\o_\veps\big)\right\|^{\s/(\s-1)}_{L^\infty} 
\Big)\,\dt\,+\,\int^T_0\left\|\Delta u_\veps\right\|_{L^\infty}\,\dd t\,<\,+\,\infty\,.
\end{align*}
For this, it is enough to repeat the computations of Subsection \ref{ss:a-priori_cont} and to combine them with the analysis of Subsection
\ref{ss:Stokes_const}. There, we have to apply tame estimates in all the terms containing the variable coefficient $b$
in front of the hyperviscosity term: exactly those terms are responsible for the appearing of the additional
$L^1_T\big(L^\infty\big)$ norm of $\Delta u_\veps$ in the continuation criterion.

As a consequence, \emph{for any $\veps>0$ fixed},
the lifespan of the solution $\big(\rho_\veps,u_\veps,\nabla\pi_\veps\big)$ constructed in Proposition \ref{p:viscous}
does not depend on the regularity index $\s>1$, so we can safely speak of its lifespan $T_\veps>0$.

This having been clarified, let us come back to inequality \eqref{est:F_eps}, which we consider at level $s$ of regularity.
Observe that, in view of \eqref{ub:smooth-data}, there exists a constant $K_0>0$ 
such that
\[
\sup_{\veps\in\,]0,1]}F_{\veps}(0)\,\leq\,K_0\,.
\]
Then,
a standard argument, similar to the one used in the proof of Proposition \ref{p:viscous},
allows us to close the estimate and deduce the uniform boundedness (both on $t$ and $\veps$) of $F_\veps(t)$ on some interval $[0,T_0]$.
Indeed, assume that $T>0$ has been chosen so small that
\begin{equation} \label{cond:small_eps}
\wtilde C_2\,T\,\Big(F_\veps(T)\,+\,\big(F_\veps(T)\big)^s\Big)\,\leq\,\log2\qquad\qquad \mbox{ and }\qquad\qquad
T\,\Big(\big(F_\veps(T))^2\,+\,\big(F_\veps(T)\big)^{\g_1}\Big)\,\leq\,2\,K_0\,.
\end{equation}
It follows from the previous inequalities and from \eqref{est:F_eps} that
\begin{equation} \label{est:F_eps_T}
F_\veps(T)\,\leq\,8\,\wtilde{C}_1\,K_0\,.
\end{equation}
At this point, assuming without loss of generality that $8\wtilde C_{1}\geq1$, let us define the time $T_0>0$
(notice that the definition is independent of $\veps\in\,]0,1]$) as
\begin{align*}
T_0\,:=\,\sup\Bigg\{T\geq0\;\;\Big|\quad \left(8\, \wtilde C_1\right)^s\,\wtilde C_2\,T\,\Big(K_0\,+\,\big(K_0\big)^s\Big)\,\leq\,\log2\qquad \mbox{ and }\qquad 
\left(8\,\wtilde C_1\right)^{\g_1}\,T\,\left(K_0\,+\,\big(K_0\big)^{\g_1-1}\right)\,\leq\,2\;\Bigg\}\,.
\end{align*}
We claim that estimate \eqref{est:F_eps_T} holds true with $T=T_0$, namely
\begin{equation} \label{est:F_eps_T_0}
F_\veps(T_0)\,\leq\,8\,\wtilde{C}_1\,K_0\,.
\end{equation}
For proving this claim, let us call $T_\veps^*>0$ the largest time for which both inequalities in \eqref{cond:small_eps} are satisfied.
Notice that $T_\veps^*$ may \tsl{a priori} depend on $\veps>0$, in fact.
By time continuity, we see that at least one of the two conditions \eqref{cond:small_eps} becomes
an equality at time $T=T_\veps^*$: one has
\begin{equation} \label{eq:T_eps^*}
\wtilde C_2\,T_\veps^*\,\Big(F_\veps(T_\veps^*)\,+\,\big(F_\veps(T_\veps^*)\big)^s\Big)\,=\,\log2\qquad\qquad \mbox{ or }\qquad\qquad
T_\veps^*\,\Big(\big(F_\veps(T_\veps^*))^2\,+\,\big(F_\veps(T_\veps^*)\big)^{\g_1}\Big)\,=\,2\,K_0\,,
\end{equation}
while the other one remains an inequality
This fact, together with estimate \eqref{est:F_eps}, implies that \eqref{est:F_eps_T} holds for $T=T_\veps^*$, \tsl{i.e.} that
$F_\veps(T_\veps^*)\,\leq\,8\,\wtilde{C}_1\,K_0$. Of course, on the one hand this yields that,
for any $\veps\in\,]0,1]$, we have $T_\veps\geq T_\veps^*$. On the other hand, plugging this latter bound into \eqref{eq:T_eps^*} gives that
\[
\log 2\,\leq\,\left(8\, \wtilde C_1\right)^s\,\wtilde C_2\,T_\veps^*\,\Big(K_0\,+\,\big(K_0\big)^s\Big)\qquad\qquad \mbox{ or }\qquad\qquad
2\,K_0\,\leq\,\left(8\,\wtilde C_1\right)^{\g_1}\,T_\veps^*\,\left(\big(K_0\big)^2\,+\,\big(K_0\big)^{\g_1}\right)\,,
\]
where we have used again that $8\wtilde C_1\geq1$. By maximality of $T_0$, we deduce that $T_\veps^*\geq T_0$ for any $\veps\in\,]0,1]$. In particular,
this implies that estimate \eqref{est:F_eps_T_0} holds true and, at the same time, that $T_\veps\geq T_0$ for any $\veps\in\,]0,1]$.

As a consequence of \eqref{est:F_eps_T_0}, by using Lemma \ref{l:F}, we obtain that
\[
\sup_{\veps\in\,]0,1]}\left(\left\|\rho_\veps-1\right\|_{L^\infty_{T_0}(H^{s+1})}\,+\,\left\|u_\veps\right\|_{L^\infty_{T_0}(H^s)}\right)\,\leq\,C\,,
\]
for a suitable ``universal'' constant $C>0$. Notice that, in the above inequality, we have also used the embedding
$\wtilde L^\infty_T(H^{\s})\,\hookrightarrow\,L^\infty_T(H^{\s})$.

In order to conclude the proof of the proposition, it remains us to establish the uniform boundedness of the pressure gradient
in the pertinent norm. For this, we observe that, owing to equation \eqref{eq:p_eps} and estimates similar to the ones performed above, see in particular
\eqref{est:pressure_eps}, we have
\[
\sup_{\veps\in\,]0,1]}\left\|\nabla\pi_\veps\right\|_{\wtilde L^1_T(H^{s-2})}\,\lesssim\,\left(\sup_{\veps\in\,]0,1]}F_\veps(T)\right)^{\g_2}\,+\,
\veps\,\left\|\o_\veps\right\|_{\wtilde L^1_T(H^{s+3})}\,,
\]
for a suitable high power $\g_2=\g_s(s)>1$. Therefore, by inequalities \eqref{est:F_eps} and \eqref{est:F_eps_T_0}, we also get
\[
 \sup_{\veps\in\,]0,1]}\left\|\nabla\pi_\veps\right\|_{\wtilde L^1_{T_0}(H^{s-2})}\,\leq\,C\,,
\]
for another suitable ``universal'' constant $C>0$.
The proof of Proposition \ref{p:unif-bounds_approx} is now complete.
\end{proof}

\subsubsection{Passing to the limit in the approximation parameter} \label{sss:limit}
The next step consists in showing that the sequence $\big(\rho_\veps,u_\veps,\nabla\pi_\veps\big)_\veps$ of approximate solutions constructed
above converges to some limit (up to a suitable extraction of a subsequence), and that this limit is in fact a solution of the original problem \eqref{eq:fluid_odd_visco1}. For this, we will use a compactness argument, for which Proposition \ref{p:unif-bounds_approx} turns out to be fundamental.

Indeed, owing to the uniform bounds of Proposition \ref{p:unif-bounds_approx} and Banach-Alaoglu theorem, 
we can claim the existence of a couple
$\Big(a,u\Big)\,\in\,L^\infty_{T_0}\big(H^{s+1}\big)\times L^\infty_{T_0}\big(H^{s}\big)$, 
with $\nabla\cdot u=0$,
such that, up to the extraction of a suitable subsequence (not relabelled here), one has
\begin{align*}
\rho_\veps-1\,\overset{*}{\rightharpoonup}\,a 
\qquad\qquad \mbox{ and }\qquad\qquad
u_\veps\,\overset{*}{\rightharpoonup}\,u 
\end{align*}
in the weak-$*$ topologies of the respective spaces.
In addition, from product rules in Sobolev spaces we gather that
\begin{equation} \label{ub:v_eps}
\Big(\rho_\veps\,u_\veps\Big)_{\veps}\qquad\qquad \mbox{ is uniformly bounded in }\qquad L^\infty_{T_0}\big(H^s\big)\,.
\end{equation}
From this property and the first equation in \eqref{eq:odd_approx}, we get that
$\big(\d_t(\rho_\veps-1)\big)_\veps$ is uniformly bounded in $L^\infty_{T_0}\big(H^{s-1}\big)$. In a standard way, this
yields, by Ascoli-Arzel\`a theorem, the strong convergence
\[
 \rho_\veps-1\,\longrightarrow\,a\qquad\qquad \mbox{ in }\qquad C_{T_0}\big(H^\s\big)\,,\qquad \forall\,\s<s+1\,.
\]
In light of this convergence, it is easy to see that $a$ satisfies the pure transport equation $\d_ta+u\cdot\nabla a=0$,
Thus, 
if we define
\[
 \rho\,:=\,a\,+\,1\,,
\]
we see that also $\rho$ is transported by the divergence-free Lipschitz-continuous vector field $u$,
\tsl{i.e.} $\rho$ satisfies the first equation in \eqref{eq:fluid_odd_visco1}, and then it enjoys also the bounds \eqref{est:rho-inf}.

Our next goal is to pass to the limit in the momentum equation. However, because of the presence of the convective term in that equation, we need to derive
compactness properties also for the sequence $\big(u_\veps\big)_\veps$ of the velocity fields. This requires some work, because the low regularity
in time and space of the pressure and hyperviscosity terms poses some problems.

To begin with, we remark that, as a byproduct of Proposition \ref{p:unif-bounds_approx}, keep in mind estimates \eqref{est:o_eps} and \eqref{est:F_eps_T_0}
in its proof, we deduce that
\[
\big(\omega_\veps\big)_\veps\,\subset\,L^\infty_{T_0}\big(H^{s-1}\big)\qquad\qquad \mbox{ and }\qquad\qquad
\big(\veps\,\omega_\veps\big)_\veps\,\subset\,\wtilde L^1_{T_0}\big(H^{s+3}\big)
\]
are \emph{bounded} sequences in the respective spaces. By interpolation, this implies that,
for any $\de\in\,[0,1]$, the family
$\left(\veps^\de\,\omega_\veps\right)_\veps$ is uniformly bounded in $\wtilde L^{1/\de}_{T_0}\big(H^{s-1+4\de}\big)$, hence
\[ 
\left(\veps^\de\,u_\veps\right)_\veps\qquad\qquad \mbox{ is bounded in }\qquad \wtilde L^{1/\de}_{T_0}\big(H^{s+4\de}\big)\,,
\] 
where we have also used the uniform control of the $L^2$ norm of $u_\veps$.
Owing to \eqref{est:Chem-Ler}, we get in particular the following properties:
\begin{align}
&\left(\veps^\de\,u_\veps\right)_\veps\qquad\qquad \mbox{ is bounded in }\qquad L^{1/\de}_{T_0}\big(H^{s+4\de}\big) \qquad\qquad \mbox{ if }\qquad \de\leq\dfrac{1}{2}\,,
\label{ub:eps-u_eps_2} \\
&
\left(\veps^\de\,u_\veps\right)_\veps\qquad\qquad \mbox{ is bounded in }\qquad L^{1/\de}_{T_0}\big(H^{s+3\de}\big) \qquad\qquad \mbox{ if }\qquad \de>\dfrac{1}{2}\,.
 \label{ub:eps-u_eps}
\end{align}

With the previous information at hand, we can study the regularity of the pressure functions in a more precise way than what was done in Proposition
\ref{p:unif-bounds_approx}.
Starting from equation \eqref{eq:p_eps} and writing
\[
-\,\Delta\o_\veps\,=\,-\,\nabla\cdot\left(\frac{1}{\rho_\veps}\,\nabla\big(\rho_\veps\,\o_\veps\big)\right)\,-\,
\nabla\cdot\left(\rho_\veps\,\o_\veps\,\nabla\left(\frac{1}{\rho_\veps}\right)\right)\,,
\]
we deduce, for any $\de\in\,]0,1[\,$, the relation
\begin{equation} \label{eq:pi_e-rho-o_e}
-\,\nabla\cdot\left(\frac{1}{\rho_\veps}\,\nabla\big(\pi_\veps\,-\,\rho_\veps\,\o_\veps\big)\right)\,=\,\nabla\cdot W_\veps\,+\,
\veps^{1-\de}\,\nabla\cdot Z_\veps\,,
\end{equation}
where we have set
\begin{align*}
&W_\veps\,:=\,(u_\veps\cdot\nabla)u_\veps\,+\,(\nabla\log\rho_\veps\cdot\nabla)u_\veps^\perp\,-\,\rho_\veps\,\o_\veps\,\nabla\left(\frac{1}{\rho_\veps}\right) \\
&Z_\veps\,:=\,\veps^\de\,\frac{1}{\rho_\veps}\,\Delta^2 u_\veps\,=\,\Delta\left(\frac{1}{\rho_\veps}\,\veps^\de\,\Delta u_\veps\right)
\,+\,\left[\frac{1}{\rho_\veps},\Delta\right]\,\veps^\de\,\Delta u_\veps\,.
\end{align*}
From the uniform bounds of Proposition \ref{p:unif-bounds_approx} and the fact that $H^{s-1}$ is a Banach algebra for $s>2$, it is easily seen that
$\big(W_\veps\big)_\veps$ is uniformly bounded in the space $L^\infty_{T_0}\big(H^{s-1}\big)$.
Next, we consider the term $Z_\veps$. Assume that $2<s\leq 3$ for a while, and fix $2/3<\de<1$ such that $\de>(5-s)/3$. Then we deduce that
$s+3\de-4>1$, which in particular implies that both $H^{s+3\de-3}$ and $H^{s+3\de-2}$ are Banach algebras. Therefore, using
also \eqref{ub:eps-u_eps}, we infer that $\big(Z_\veps\big)_\veps$ is uniformly bounded in $L^{1/\de}_{T_0}\big(H^{s+3\de-4}\big)$.
Using again the fact that $s+3\de-4>1$, we can apply Proposition 7 of \cite{D_2010} (see the second item of that proposition)
to gather that
\[ 
\Big(\nabla\big(\pi_\veps-\rho_\veps\o_\veps\big)\Big)_\veps\qquad\qquad \mbox{ is bounded in }\qquad L^{1/\de}_{T_0}\big(H^{s+3\de-4}\big)\,.
\] 
In the case $s>3$, instead, we take $\de=1/2$ and use \eqref{ub:eps-u_eps_2} together with Proposition 7 of \cite{D_2010} to get that
\[
\Big(\nabla\big(\pi_\veps-\rho_\veps\o_\veps\big)\Big)_\veps\qquad\qquad \mbox{ is bounded in }\qquad L^{2}_{T_0}\big(H^{s-2}\big)\,.
\]

In any case, as $\big(\nabla(\rho_\veps\o_\veps)\big)_\veps$ is bounded in $L^\infty_{T_0}\big(H^{s-2}\big)$
and $s+3\de-4>s-2$ for $\de>2/3$,
in the end we deduce that there exists some $1<p\leq 2$ such that $\big(\nabla\pi_\veps\big)_\veps$ is bounded in $L^{p}_{T_0}\big(H^{s-2}\big)$,
which implies that, up to a suitable extraction, $\big(\nabla\pi_\veps\big)_\veps$ converges weakly to some $\nabla\pi$ belonging to that space.

At this point, we are ready to derive a uniform bound for the sequence $\big(\d_tu_\veps\big)_\veps$,
from which we will infer compactness of the velocity fields. First of all, we rewrite the second equation in \eqref{eq:odd_approx} in the following form:
\begin{align*}
\d_tu_\veps\,&=\,-\,\nabla\cdot\big(u_\veps\otimes u_\veps\big)\,-\,\frac{1}{\rho_\veps}\,\nabla\cdot\big(\rho_\veps\,\nabla u_\veps^\perp\big)\,-\,
\frac{\veps}{\rho_\veps}\,\Delta^2u_\veps\,-\,\frac{1}{\rho_\veps}\,\nabla\pi_\veps \\
&=\,-\,\nabla\cdot\big(u_\veps\otimes u_\veps\big)\,-\,\Delta u_\veps^\perp\,-\,\big(\nabla\log\rho_\veps\cdot\nabla\big)u_\veps^\perp\,-\,
\frac{\veps}{\rho_\veps}\,\Delta^2u_\veps\,-\,\frac{1}{\rho_\veps}\,\nabla\big(\pi_\veps-\rho_\veps\o_\veps\big)\,-\,
\frac{1}{\rho_\veps}\,\nabla\big(\rho_\veps\o_\veps\big)\,.
\end{align*}
Then, we estimate each term appearing in the right-hand side of the previous relation.

Thanks to the bounds established in Proposition \ref{p:unif-bounds_approx}, we have that both $\Big(\nabla\cdot\big(u_\veps\otimes u_\veps\big)\Big)_\veps$
and $\Big(\big(\nabla\log\rho_\veps\cdot\nabla\big)u_\veps^\perp\Big)_\veps$ are uniformly bounded in $L^\infty_{T_0}\big(H^{s-1}\big)$.
The sequence $\left(\Delta u_\veps^\perp\right)_\veps$, instead, is uniformly bounded in $L^\infty_{T_0}\big(H^{s-2}\big)$; by writing
\[
\frac{1}{\rho_\veps}\,\nabla\big(\rho_\veps\o_\veps\big)\,=\,\nabla\o_\veps\,+\,\o_\veps\,\nabla\log\rho_\veps\,,
\]
we see that also $\Big(\frac{1}{\rho_\veps}\,\nabla\big(\rho_\veps\o_\veps\big)\Big)_\veps$ is bounded in the same space $L^\infty_{T_0}\big(H^{s-2}\big)$.
Finally, the term $\frac{\veps}{\rho_\veps}\,\Delta^2u_\veps$ can be treated as  $Z_\veps$, while the same analysis performed above applies to
the pressure term
$\frac{1}{\rho_\veps}\,\nabla\big(\pi_\veps-\rho_\veps\o_\veps\big)$, giving that both terms are uniformly bounded in \tsl{e.g.}
$L^p_{T_0}\big(H^{s-2}\big)$, where $1<p\leq2$ is the same Lebesgue index fixed before
(recall here that, under our choice of $\de$, we have $s+3\de-4>1$, so that $H^{s+3\de-4}$ is a Banach algebra, and $s+3\de-4>s-2$).

All in all, we have shown that there exists some $p\in\,]1,2]$ such that the sequence $\big(\d_tu_\veps\big)_\veps$ is bounded in the space
$L^p_{T_0}\big(H^{s-2}\big)$. This fact, combined with Ascoli-Arzel\`a theorem again, gives us the sought compactness of the velocity fields,
hence the strong convergence property
\[
u_\veps\,\longrightarrow\,u\qquad\qquad \mbox{ in }\qquad C_{T_0}\big(H^\s\big)\,,\qquad \forall\,\s<s\,.
\]

In light of all these properties, we can pass to the limit $\veps\ra0^+$ in equations \eqref{eq:odd_approx} and get that the triplet
$\big(\rho,u,\nabla\pi\big)$ is indeed a solution of system \eqref{eq:fluid_odd_visco1}.

\begin{rem} \label{r:cancellation}
Observe that, owing to the Biot-Savart law $u_\veps\,=\,-\nabla^\perp(-\Delta)^{-1}\o_\veps$, one has
\[
-\,\Delta u_\veps^\perp\,-\,\nabla\o_\veps\,=\,0\,,
\]
which is well-justified at least at high frequencies.
\end{rem}

\subsubsection{Time regularity of the solution} \label{sss:time_reg}
In order to complete the proof of the existence part in Theorem \ref{teo1}, we still have to establish suitable regularity in time of
the solution $\big(\rho,u,\nabla\pi\big)$.
Recall that, so far, we have proven that
$\big(\rho-1,u,\nabla\pi\big)$ belongs to $L^\infty_{T_0}\big(H^{s+1}\big)\times L^\infty_{T_0}\big(H^{s}\big)\times L^p_{T_0}\big(H^{s-2}\big)$,
with in addition the time continuity properties $\rho-1\in C_{T_0}\big(H^\s\big)$ for any $\s<s+1$ and $u\in C_{T_0}\big(H^\s\big)$ for any $\s<s$.

Observe that, differently from the usual approach, time regularity of the solution cannot simply come from an inspection of the equations satisfied
from $\rho$ and $u$. As a matter of fact, as $u$ has at best $H^s$ regularity in space, we can only get
 $\rho-1\in C_{T_0}\big(H^{s}\big)$, whereas for the velocity field,
due to the loss of derivatives produced by the odd term, we can only obtain
$u\in C_{T_0}\big(H^{s-1}\big)$, keep in mind Remark \ref{r:cancellation} above; but those two facts are already known from the
analysis of the previous Paragraph \ref{sss:limit}.

\medbreak
First of all, we already know that $u\in C_{T_0}\big(L^2\big)$. Notice that this could be derived also from the momentum equation, written in the form
\begin{equation} \label{eq:mom_new}
\d_tu\,+\,(u\cdot\nabla)u\,=\,-\,\frac{1}{\rho}\,\nabla\cdot\big(\rho\,\nabla u^\perp\big)\,-\,\frac{1}{\rho}\,\nabla\pi\,,
\end{equation}
by using that, since $s>2$, the right-hand side belongs to $L^p_{T_0}\big(L^2\big)$.

In addition, define the vorticity $\omega\,:=\,\nabla^\perp\cdot u$ of the fluid.
By taking the divergence of the previous equation and arguing as for getting \eqref{eq:pi_e-rho-o_e}, we can establish the properties
\[
\nabla\big(\pi-\rho\o\big)\in L^\infty_{T_0}\big(H^{s-1}\big)\qquad\qquad \mbox{ and }\qquad\qquad
\nabla\pi\in L^\infty_{T_0}\big(H^{s-2}\big)\,.
\]

Next, we resort to the variables $\theta\,:=\,\nabla^\perp\cdot\big(\rho u\big)-\Delta\rho$ and $\o$.
Owing to the previous properties for the pressure gradient, we see that the right-hand sides
of both equations \eqref{eq:theta} and \eqref{eq:vort-2} belong to $L^\infty_{T_0}\big(H^{s-1}\big)$, whereas the transport fields
belong to $L^\infty_{T_0}\big(H^s\big)$. Then, Theorem \ref{th:transport}
implies that the two quantities $\theta$ and $\o$ both belong to $C_{T_0}\big(H^{s-1}\big)$.

In light of the previous properties,
by cutting into low and high frequencies and using the Biot-Savart law, we see that
\begin{equation} \label{eq:u-decomp}
 u\,=\,\Delta_{-1}u\,+\,\big(\Id-\Delta_{-1}\big)u\,=\,\Delta_{-1}u\,-\,\big(\Id-\Delta_{-1}\big)\,\nabla^\perp(-\Delta)^{-1}\o
\end{equation}
is the sum of two functions which belong to $C_{T_0}\big(H^{s}\big)$, so in turn we deduce that also $u$ belongs to the same space: we have
$u\in C_{T_0}\big(H^s\big)$.

From this latter property $u\in C_{T_0}\big(H^{s}\big)$ and the fact that $\rho\in C_{T_0}\big(H^{s}\big)$ as well, we see that
their product $\rho u$ also belongs to $C_{T_0}\big(H^{s}\big)$, hence
\[
-\Delta\rho\,=\,\theta\,-\,\nabla^\perp\cdot\big(\rho\,u\big)\;\in\,C_{T_0}\big(H^{s-1}\big)\,,
\]
which in turn implies that $\rho-1$ belongs to $C_{T_0}\big(H^{s+1}\big)$.

Finally, resorting again to the analogue of equation \eqref{eq:pi_e-rho-o_e}, namely
\[
-\,\nabla\cdot\left(\frac{1}{\rho}\,\nabla\big(\pi\,-\,\rho\,\o\big)\right)\,=\,\nabla\cdot\left(
(u\cdot\nabla)u\,+\,(\nabla\log\rho\cdot\nabla)u^\perp\,-\,\rho\,\o\,\nabla\left(\frac{1}{\rho}\right)\right)\,,
\]
and using again Proposition 7 of \cite{D_2010}, we can improve the previous properties for the pressure function and find that
$\nabla\big(\pi-\rho\o\big)\in C_{T_0}\big(H^{s-1}\big)$ together with
$\nabla\pi\in C_{T_0}\big(H^{s-2}\big)$.

The existence part of Theorem \ref{teo1} is now completely proved.

\subsection{Proof of uniqueness} \label{ss:proof-u}

Here we prove the uniqueness statement of \Cref{teo1}. This will conclude the proof of the whole theorem.

As already commented in the Introduction (see Paragraph \ref{sss:intro_rest}), 
while following a classical approach which consists in performing stability estimates in $L^2$, we have to face several complications arising in the argument.
First of all, owing to the loss of derivatives induced by the odd viscosity term, we need instead to perform
$H^1$ stability estimates: hence, we will need a $L^2$ estimate on the difference of the vorticities,
which in turn requires also an estimate for the difference of the functions $\theta$.

In addition, if we were to use classical Kato-Ponce commutator estimates, we would need to impose a $L^\infty$ control of some quantities, namely
 $\nabla\pi$, $\nabla\o$ and $\nabla \theta$, which looks not so natural in our context: as a matter of fact, 
all those quantites are only $H^{s-2}$, but this space fails to embed in $L^\infty$ when $2<s\leq3$.
We remark that the problem cannot be solved by passing in Lagrangian coordinates, as $\o$ and $\theta$ are transported by two different velocity fields
and, in any case, the pressure term will be not eliminated anyway.

The idea to unlock this situation without requiring a gap between existence theory and uniqueness theory will be given by a careful use of
Sobolev embeddings and Gagliardo-Nirenberg inequalities in dimension $d=2$.
Also in this argument, the gain of one derivative for the density will be essential. We remark that the pressure term will require a special treatment:
we will detail in the course of the proof the difficulties and the strategy to follow.

\subsubsection{A stability result} \label{sss:stab}

After the previous presentation, we can state the main result of this section, namely a stability estimates for solutions to
system \eqref{eq:fluid_odd_visco1}.
Afterwards, in Paragraph \ref{sss:uniq-odd} we will see how to deduce uniqueness from this statement.

\begin{prop} \label{p:stability}
Let the triplets $\big(\rho_1,u_1,\nabla\pi_1\big)$ and $\big(\rho_2,u_2,\nabla\pi_2\big)$ be two solutions of system \eqref{eq:fluid_odd_visco1}.
Assume that there exists some $T>0$ such that the following conditions hold:
\begin{enumerate}[(i)]
\item there exist two constants $0<\rho_*\leq\rho^*$ such that $\rho_*\leq\rho_j\leq\rho^*$ for both $j=1$ and $j=2$;
\item the initial density $\rho_{0,1}$ and initial velocity field $u_{0,2}$ belong to $W^{1,\infty}(\R^2)$;
 \item the quantities $\rho_1-\rho_2$ and $u_1-u_2$ belong respectively to $C^1\big([0,T];H^2(\R^2)\big)$ and to $C^1\big([0,T];H^1(\R^2)\big)$;
 \item after defining, for $j\in\{1,2\}$, the quantities $\theta_j\,:=\,\nabla^\perp\cdot\big(\rho_j\,u_j\big)\,-\,\Delta\rho_j$, their difference
$\theta_1-\theta_2$ belongs to the space $C^1\big([0,T];L^2(\R^2)\big)$;
\item the gradients $\nabla\rho_{1,2}$ belong to $L^\infty\big([0,T]\times \R^2\big)$;
\item the velocity fields $u_{1,2}$ belong to $L^\infty\big([0,T];W^{1,\infty}(\R^2)\big)$;
\item the quantity $\nabla\big(\pi_2\,-\,\rho_2\,\o_2\big)$ belongs to $L^2\big([0,T];L^\infty(\R^2)\big)$,
and one has $\nabla^2\rho_2\in L^{1}\big([0,T];L^\infty(\R^2)\big)$;
 \item there exists $q\in\,]2,+\infty]$ such that the quantities $\nabla\pi_2$ and $\nabla\Delta\rho_2$
both belong to $L^1\big([0,T];L^q(\R^2)\big)$, whereas $\nabla^2 u_2$ belongs to $L^{2q/(q-2)}\big([0,T];L^q(\R^2)\big)$.
\end{enumerate}

Then, there exists a function $I\in L^1\big([0,T]\big)$ such that one has the inequality
\begin{align*}
&\sup_{t\in[0,T]}\left(\left\|\rho_1(t)-\rho_2(t)\right\|^2_{H^2}\,+\,\left\|u_1(t)-u_2(t)\right\|^2_{H^1}\right) \\
&\qquad\qquad\qquad\qquad\qquad\qquad \,\leq\,\mc K_0\,
\left(\left\|\rho_{0,1}-\rho_{0,2}\right\|^2_{H^2}\,+\,\left\|u_{0,1}-u_{0,2}\right\|^2_{H^1}\right)
\,\exp\left(C\,\int^T_0\Big(1\,+\,I(\t)\Big)\,\dd\t\right)\,,
\end{align*}
for an absolute constant $C>0$, independent on the data and solutions of the equations, and a constant
$\mc K_0$ depending only on $\rho_*$, $\rho^*$ and on the $L^\infty$ norm of $\,\nabla\rho_{0,1}$, of $u_{0,2}$ and of its gradient $\nabla u_{0,2}$.
\end{prop}

Before proceeding to the proof of the previous proposition, we point out that our assumptions imply that the differences of the initial data
belong to the right spaces, namely that $\rho_{0,1}-\rho_{0,2}\in H^2$ and $u_{0,1}-u_{0,2}\in H^1$; see also inequalities \eqref{est:D-Theta}
and \eqref{est:Theta-D} for more details. So, the claimed inequality makes sense.

In addition, let us formulate the following remark.

\begin{rem} \label{r:stability}
Following our computations, the $L^\infty$ condition in time assumed in items (v) and (vi) of the statement
could be relaxed to some $L^\alpha$ integrability, for $\alpha\gg1$  large enough. However, the stability estimates are quite intricate and there is
a choice to make, in order to optimise the integrability index $\alpha$ and the integrability for the higher order norms appearing
in items (vii) and (viii).

Here, we have made the choice of taking $\alpha=+\infty$ to have a simpler statement and a simpler proof. Besides, this choice allows us to lower
the integrability indices in items (vii) and (viii).
In any case, this choice is enough to deduce uniqueness for system \eqref{eq:fluid_odd_visco1}
at the claimed level of regularity.
\end{rem}

We can now show the proof of the claimed stability estimates.

\begin{proof}[Proof of Proposition \ref{p:stability}]
To begin with, we introduce some preliminaries. First of all, throughout this proof we will adopt the following notation:
\[
\forall\,f\in\big\{\rho,u,\pi,\o,\theta\big\}\,,\qquad \mbox{ we set }\qquad\qquad
 \de f\,:=\,f_1\,-\,f_2\,.
\]
We will also set $\de f_0\,:=\,f_{0,1}-f_{0,2}$ to be the difference of the initial data.

Next, we recall the classical Gagliardo-Nirenberg inequality in dimension $d=2$:
\begin{equation} \label{est:G-N}
\forall\, f\in H^1(\R^2)\,,\qquad \forall\,p\in[2,+\infty[\;,\qquad\qquad\qquad
\left\|f\right\|_{L^p}\,\lesssim\,C(p)\,\|f\|_{L^2}^{2/p}\;\left\|\nabla f\right\|_{L^2}^{1-2/p}\,.
\end{equation}
Throughout the following computations, $p\in[2,+\infty[$ will always be chosen in the following way:
\begin{equation} \label{def:p}
\frac{1}{p}\,+\,\frac{1}{q}\,=\,\frac{1}{2}\,,
\end{equation}
where $q$ is the integrability index given in the assumptions of the proposition.

Finally, we define the energy of the difference of solutions by
\[
\forall\,t\in[0,T]\,,\qquad\qquad 
\mc D(t)\,:=\,\left\|\de\rho(t)\right\|^2_{L^2}\,+\,\left\|\Delta\de\rho(t)\right\|^2_{L^2}\,+\,\left\|\de u(t)\right\|^2_{L^2}\,+\,
\left\|\de\o(t)\right\|^2_{L^2}\,.
\]
Thus, our final goal is to find an estimate for $\mc D(t)$ in terms of its value $\mc D_0$ at initial time. For later use, let us also define
\[
\forall\,t\in[0,T]\,,\qquad\qquad 
\Theta(t)\,:=\,\left\|\de\rho(t)\right\|^2_{L^2}\,+\,\left\|\de u(t)\right\|^2_{L^2}\,+\,
\left\|\de\o(t)\right\|^2_{L^2}\,+\,\left\|\de\theta(t)\right\|^2_{L^2}
\]
to be the same quantity, with the exception that the higher order norm of the density function $\de\rho$ is now replaced by the $L^2$ norm of $\de\theta$.
In fact, we are rather going to estimate the function $\Theta$, which is somehow easier to deal with than $\mc D$.

Observe that, from the definition of the functions $\theta_{1,2}$, we gather that
\begin{equation} \label{eq:diff-theta_Delta-rho}
\de\theta\,=\,\nabla^\perp\cdot\big(\rho_1\,u_1\,-\,\rho_2\,u_2\big)\,-\,\Delta\de\rho\,=\,\rho_1\,\de\o\,+\,\de\rho\,\o_2\,+\,
\nabla^\perp\de\rho\cdot u_2\,+\,\nabla^\perp\rho_1\cdot\de u\,-\,\Delta\de\rho\,.
\end{equation}
This relation implies that
\begin{align*}
\left\|\Delta\de\rho\right\|_{L^2}\,&\lesssim\,\left\|\de\theta\right\|_{L^2}\,+\,\left\|\de\o\right\|_{L^2}\,+\,
\left\|\de\rho\right\|_{L^2}\,\left\|\o_2\right\|_{L^\infty}\,+\,\left\|\nabla\de\rho\right\|_{L^2}\,\left\|u_2\right\|_{L^\infty}\,+\,
\left\|\nabla\rho_1\right\|_{L^\infty}\,\left\|\de u\right\|_{L^2} \\
&\lesssim\,\left\|\de\theta\right\|_{L^2}\,+\,\left\|\de\o\right\|_{L^2}\,+\,\left\|\de\rho\right\|_{L^2}\,\left\|\nabla u_2\right\|_{L^\infty}\,+\,
\left\|\de\rho\right\|_{L^2}\,\left\|u_2\right\|^2_{L^\infty}\,+\,\left\|\nabla\rho_1\right\|_{L^\infty}\,\left\|\de u\right\|_{L^2}\,,
\end{align*}
where, in passing from the first to the second inequality, we have also applied the interpolation inequality
$\left\|\nabla\de\rho\right\|^2_{L^2}\,\lesssim\,\left\|\de\rho\right\|_{L^2}\,\left\|\Delta\de\rho\right\|_{L^2}$, together with the Young inequality
to absorbe the higher order term into the left-hand side.
From the previous inequality, we infer that
\begin{equation} \label{est:D-Theta}
\forall\,t\in[0,T]\,,\qquad\qquad 
\mc D(t)\,\lesssim\,\Theta(t)\,\Big(1+\left\|\nabla u_2\right\|_{L^\infty}+\left\|u_2\right\|^2_{L^\infty}+\left\|\nabla\rho_1\right\|_{L^\infty}\Big)^2\,.
\end{equation}

Mutually, from \eqref{eq:diff-theta_Delta-rho} we also gather that
\begin{align*}
\left\|\de\theta\right\|_{L^2}\,&\lesssim\,\left\|\de\o\right\|_{L^2}\,+\,
\left\|\de\rho\right\|_{L^2}\,\left\|\o_2\right\|_{L^\infty}\,+\,\left\|\nabla\de\rho\right\|_{L^2}\,\left\|u_2\right\|_{L^\infty}\,+\,
\left\|\nabla\rho_1\right\|_{L^\infty}\,\left\|\de u\right\|_{L^2}\,+\,\left\|\Delta\de\rho\right\|_{L^2}\,, 
\end{align*}
which implies the following inequality, reciprocal to \eqref{est:D-Theta}:
\begin{equation} \label{est:Theta-D}
\forall\,t\in[0,T]\,,\qquad\qquad 
\Theta(t)\,\lesssim\,\mc D(t)\,\Big(1\,+\,\left\|u_2\right\|_{W^{1,\infty}}\,+\,\left\|\nabla\rho_1\right\|_{L^\infty}\Big)^2\,.
\end{equation}

\medbreak
After those preliminaries, we are ready to compute stability estimates for solutions of system \eqref{eq:fluid_odd_visco1}.
First of all, by taking the difference between the mass equations, we see that $\de\rho$ satisfies the relation
\[
\d_t\de\rho\,+\,u_1\cdot\nabla\de\rho\,=\,-\,\de u\cdot\nabla\rho_2\,,
\]
related to the initial datum $\de\rho_{|t=0}\,=\,\de\rho_0$. A simple energy estimates thus yields
\begin{equation} \label{est:diff-rho}
\forall\,t\in[0,T]\,,\qquad\qquad\qquad \left\|\de\rho(t)\right\|^2_{L^2}\,\leq\,\left\|\de\rho_0\right\|^2_{L^2}\,+\,
\int^t_0\left\|\de u\right\|_{L^2}\,\left\|\nabla\rho_2\right\|_{L^\infty}\,\left\|\de\rho\right\|_{L^2}\,\dd\tau\,.
\end{equation}

Next, we want to derive an equation for $\de u$. Let us recast the momentum equation in the same form as in \eqref{eq:mom_new}, namely
\[
\d_tu\,+\,u\cdot\nabla u\,+\,\frac{1}{\rho}\,\nabla\pi\,+\,\frac{1}{\rho}\,\nabla\cdot\left(\rho\,\nabla u^\perp\right)\,=\,0\,.
\]
Taking the difference between the equation for $u_1$ and the one for $u_2$ and then multiplying the resulting expression by $\rho_1$, we easily find the following
expression:
\[
\rho_1\,\d_t\de u\,+\,\rho_1\,u_1\cdot\nabla\de u\,+\,\nabla\de\pi\,+\,\nabla\cdot\left(\rho_1\,\nabla\de u^\perp\right)\,=\,-\,\rho_1\,\de u\cdot\nabla u_2\,+\,
\frac{\de\rho}{\rho_2}\,\nabla\pi_2\,-\,\nabla\cdot\left(\de\rho\,\nabla u_2^\perp\right)\,+\,\frac{\de\rho}{\rho_2}\,\nabla\cdot\left(\rho_2\,\nabla u_2^\perp\right)\,.
\]
Hence, we can perform a $L^2$ estimate for $\de u$: using that $\rho_1$ is transported by $u_1$, the assumptions on the two density functions
and the divergence-free condition on $\de u$, we find
\begin{align*}
\left\|\de u(t)\right\|^2_{L^2}\,&\lesssim\,\left\|\de u_0\right\|^2_{L^2}\,+\,
\int^t_0\left(\left\|\de u\right\|_{L^2}^2\,+\,\left\|\de\rho\right\|_{L^2}\,\left\|\nabla\de u\right\|_{L^2}\right)\,\left\|\nabla u_2\right\|_{L^\infty}\,\dd\tau \\
&\qquad\qquad\qquad\qquad
\,+\,\int^t_0\left\|\de\rho\right\|_{L^2}\,\left\|\de u\right\|_{L^p}\,\left(\left\|\nabla\pi_2\right\|_{L^q}\,+\,\left\|\Delta u_2\right\|_{L^q}\,+\,
\left\|\nabla\rho_2\right\|_{L^\infty}\,\left\|\nabla u_2\right\|_{L^\infty}\right)\,\dd\tau\,,
\end{align*}
where we have also made suitable integrations by parts and the implicit multiplicative constant depends also on $\rho_*$ and $\rho^*$.
Using that
\[
\left\|\nabla\de u\right\|_{L^2}\,\approx\,\left\|\de\o\right\|_{L^2}
\]
 by linearity of the Biot-Savart law, and that, if $p>2$, we can write
\begin{equation} \label{est:du_interpol}
\left\|\de u\right\|_{L^p}\,\lesssim\,\left\|\de u\right\|_{L^2}^{2/p}\;\left\|\nabla\de u\right\|_{L^2}^{1-2/p}\,\lesssim\,
\left\|\de u\right\|_{L^2}\,+\,\left\|\de\o\right\|_{L^2}
\end{equation}
by Gagliardo-Nirenberg inequality \eqref{est:G-N} and interpolation, from the previous estimate we deduce
\begin{align}
\label{est:diff-u}
\forall\,t\in[0,T]\,,\qquad\qquad
\left\|\de u(t)\right\|^2_{L^2}\,&\lesssim\,\left\|\de u_0\right\|^2_{L^2}+
\int^t_0\Theta\,\left(1+\left\|\nabla\pi_2\right\|_{L^q}+\left\|\Delta u_2\right\|_{L^q}+
\left\|\nabla\rho_2\right\|^2_{L^\infty}+\left\|\nabla u_2\right\|^2_{L^\infty}\right)\dd\tau\,.
\end{align}

In order to bound $\Theta$, it remains us to find an estimate for $\de\theta$ and $\de\o$ in $L^2$. We consider $\de\theta$ first:
taking the difference of equation \eqref{eq:theta} for $\theta_1$ with the one for $\theta_2$, we get
\begin{equation} \label{eq:diff-theta}
\d_t\de\theta\,+\,u_1\cdot\nabla\de\theta\,=\,-\,\de u\cdot\nabla\theta_2\,+\,\de\mc T\big(\nabla\rho,u,\nabla u\big)\,+\,
\de\mc B\big(\nabla u,\nabla^2\rho\big)\,,
\end{equation}
where $\de\mc T\big(\nabla\rho,u,\nabla u\big)$ and $\de\mc B\big(\nabla u,\nabla^2\rho\big)$ stand, respectively, for the differences of the trilinear and bilinear
terms, namely
\begin{align*}
\de\mc T\big(\nabla\rho,u,\nabla u\big)\,&:=\,\nabla^\perp\rho_1\cdot\nabla\left|u_1\right|^2\,-\,\nabla^\perp\rho_2\cdot\nabla\left|u_2\right|^2 \\
\de\mc B\big(\nabla u,\nabla^2\rho\big)\,&:=\,\mc B\big(\nabla u_1,\nabla^2\rho_1\big)\,-\,\mc B\big(\nabla u_2,\nabla^2\rho_2\big) \,.
\end{align*}
Keeping in mind formulas \eqref{eq:trilinear} and \eqref{def:B} above, it is easy to obtain the following bounds for those terms:
\begin{align*}
\left\|\de\mc T\big(\nabla\rho,u,\nabla u\big)\right\|_{L^2}\,&\lesssim\,
\left\|\nabla\rho_1\right\|_{L^\infty}\,\left\|u_1\right\|_{L^\infty}\,\left\|\nabla\de u\right\|_{L^2}\,+\,\left\|\nabla\rho_1\right\|_{L^\infty}\,
\left\|\de u\right\|_{L^2}\,\left\|\nabla u_2\right\|_{L^\infty}\,+\,\left\|\nabla\de\rho\right\|_{L^2}\,\left\|u_2\right\|_{L^\infty}\,
\left\|\nabla u_2\right\|_{L^\infty} \\
&\lesssim\,\sqrt{\Theta}\,\Big(\left\|\nabla\rho_1\right\|_{L^\infty}^2\,+\,\left\|u_1\right\|_{L^\infty}^2\,+\,
\left\|\nabla u_2\right\|^2_{L^\infty}\Big) \\
&\qquad\qquad\qquad\qquad\qquad\qquad
+\,\left\|\nabla u_2\right\|_{L^\infty}\,\sqrt{\Theta}\,
\Big(1+\left\|\nabla u_2\right\|_{L^\infty}+\left\|u_2\right\|^2_{L^\infty}+\left\|\nabla\rho_1\right\|_{L^\infty}\Big) \\
\left\|\de\mc B\big(\nabla u,\nabla^2\rho\big)\right\|_{L^2}\,&\lesssim\,\left\|\nabla u_1\right\|_{L^\infty}\,\left\|\nabla^2\de\rho\right\|_{L^2}\,+\,
\left\|\nabla\de u\right\|_{L^2}\,\left\|\nabla^2\rho_2\right\|_{L^\infty} \\
&\lesssim\,\sqrt{\Theta}\,\left\|\nabla u_1\right\|_{L^\infty}\,
\Big(1+\left\|\nabla u_2\right\|_{L^\infty}+\left\|u_2\right\|^2_{L^\infty}+\left\|\nabla\rho_1\right\|_{L^\infty}\Big)\,+\,
\sqrt{\Theta}\,\left\|\nabla^2\rho_2\right\|_{L^\infty}\,,
\end{align*}
where we have also used \eqref{est:D-Theta}. In addition, for $k=1,2$ we can compute
\[
\d_k\theta_2\,=\,\d_k\rho_2\,\o_2\,+\,\rho_2\d_k\o_2\,+\,\nabla^\perp\d_k\rho_2\cdot u_2\,+\,\nabla^\perp\rho_2\cdot\d_k u_2\,+\,\d_k\Delta\rho_2\,,
\]
from which we discover that all terms of $\nabla\theta_2$ belong to $L^1_T(L^\infty)$, except for the ones presenting $\nabla\o_2$ and $\nabla\Delta\rho_2$,
which only belong to $L^1_T(L^q)$.

Therefore, from an energy estimate for equation \eqref{eq:diff-theta} we infer
\begin{align*}
\left\|\de\theta(t)\right\|^2_{L^2}\,&\lesssim\,\left\|\de\theta_0\right\|^2_{L^2}+
\int^t_0\left\|\de u\right\|_{L^p}\left\|\de\theta\right\|_{L^2}\Big(\left\|\nabla^2u_2\right\|_{L^q}+\left\|\nabla\Delta\rho_2\right\|_{L^q}\Big)\dd\tau \\
&\qquad
\,+\,\int^t_0\Theta\,\left(1+\left\|\nabla\rho_1\right\|^3_{L^\infty}+\left\|\nabla\rho_2\right\|^3_{L^\infty}+
\left\|u_1\right\|^3_{L^\infty}+\left\|\nabla u_2\right\|^3_{L^\infty}+
\left\|u_2\right\|^3_{L^\infty}+\left\|\nabla^2\rho_2\right\|_{L^\infty}\left\|u_2\right\|_{L^\infty}\right)\dd\tau\,.
\end{align*}
Observe that we can bound $\left\|\de u\right\|_{L^p}$ as done in \eqref{est:du_interpol}, thus yielding the estimate
\begin{align}
\label{est:diff-theta}
\forall\,t\in[0,T]\,,\qquad\qquad
\left\|\de\theta(t)\right\|^2_{L^2}\,&\lesssim\,\left\|\de\theta_0\right\|^2_{L^2}+
\int^t_0\Theta\,\Big(\left\|\nabla^2u_2\right\|_{L^q}\,+\,\left\|\nabla\Delta\rho_2\right\|_{L^q}\,+\,\Psi\Big)\dd\tau\,,
\end{align}
where we have defined
\[
\Psi(t)\,:=\,1+\left\|\nabla\rho_1\right\|^3_{L^\infty}+\left\|\nabla\rho_2\right\|^3_{L^\infty}+
\left\|u_1\right\|^3_{L^\infty}+\left\|\nabla u_2\right\|^3_{L^\infty}+
\left\|u_2\right\|^3_{L^\infty}+\left\|\nabla^2\rho_2\right\|_{L^\infty}\left\|u_2\right\|_{L^\infty}\,.
\]
Notice that $\Psi\in L^1\big([0,T]\big)$, owing to our assumptions.

Let us now consider the quantity $\de\o$. Observe that the pressure term is a source of troubles also in this case. As a matter of fact,
from estimate \eqref{est:p-L^2_mild} we see that a control of $\nabla\de\pi$ in $L^2$ would require a control in the same space of $\nabla\de\o$, causing in this way
another loss of derivatives in the estimates. The idea to circumvent this problem is to resort again to the new formulation \eqref{eq:vort-2} of the vorticity
equation.
Starting from that relation, direct computations yield
\begin{align}
\label{eq:diff-o}
\d_t\de\o\,+\,\left(u_1\,-\,\nabla^\perp\log\rho_1\right)\cdot\nabla\de\o\,&=\,
-\,\left(\de u\,-\,\frac{1}{\rho_1}\,\nabla^\perp\de\rho\,+\,\frac{\de\rho}{\rho_1\,\rho_2}\,\nabla^\perp\rho_2\right)\cdot\nabla\o_2\,-\,
\de\mc B\big(\nabla u,\nabla^2\log\rho\big) \\
&\qquad\,+\,
\nabla^\perp\left(\frac{\de\rho}{\rho_1\,\rho_2}\right)\cdot\nabla\left(\pi_2\,-\,\rho_2\,\o_2\right)\,-\,
\nabla^\perp\left(\frac{1}{\rho_1}\right)\cdot\nabla\left(\de\pi\,-\,\rho_1\,\de\o\,-\,\de\rho\,\o_2\right)\,, \nonumber
\end{align}
where, similarly as above, we have defined
\begin{align*}
\de\mc B\big(\nabla u,\nabla^2\log\rho\big)\,&:=\,\mc B\big(\nabla u_1,\nabla^2\log\rho_1\big)\,-\,\mc B\big(\nabla u_2,\nabla^2\log\rho_2\big) \\
&\,=\,
\mc B\big(\nabla\de u,\nabla^2\log\rho_2\big)\,+\,\mc B\big(\nabla u_1,\nabla^2\log\rho_1-\nabla^2\log\rho_2\big)\,.
\end{align*}
We now compute
\begin{align*}
\nabla^2\log\rho_1-\nabla^2\log\rho_2\,=\,\frac{1}{\rho_1}\,\nabla^2\de\rho\,-\,\frac{\de\rho}{\rho_1\,\rho_2}\,\nabla^2\rho_2\,-\,
\frac{1}{\rho_1^2}\,\nabla\rho_1\otimes\nabla\de\rho\,-\,\frac{1}{\rho_1^2}\,\nabla\de\rho\otimes\nabla\rho_2\,+\,
\frac{\rho_1+\rho_2}{\rho_1\,\rho_2}\,\de\rho\,\nabla\rho_2\otimes\nabla\rho_2\,.
\end{align*}
From these formulas, we deduce the following bound for the bilinear term:
\begin{align*}
\left\|\de\mc B\big(\nabla u,\nabla^2\log\rho\big)\right\|_{L^2}\,&\lesssim\,\left\|\nabla\de u\right\|_{L^2}\,\Big(\left\|\nabla\rho_2\right\|_{L^{\infty}}^2+
\left\|\nabla^2\rho_2\right\|_{L^{\infty}}\Big)\,+\,\left\|\nabla u_1\right\|_{L^\infty}\,\left\|\Delta\de\rho\right\|_{L^2} \\
&\qquad \,+\,\left\|\nabla u_1\right\|_{L^\infty}\,\left\|\de\rho\right\|_{L^2}\,\left\|\nabla^2\rho_2\right\|_{L^\infty}\,+\,
\left\|\nabla u_1\right\|_{L^\infty}\,\left\|\nabla\de\rho\right\|_{L^2}\,\Big(\left\|\nabla\rho_1\right\|_{L^\infty}+\left\|\nabla\rho_2\right\|_{L^\infty}\Big) \\
&\lesssim\,\left\|\de\o\right\|_{L^2}\,\Big(\left\|\nabla\rho_2\right\|_{L^{\infty}}^2+
\left\|\nabla^2\rho_2\right\|_{L^{\infty}}\Big)\,+\,\Big(1+\left\|\nabla u_1\right\|_{L^\infty}\Big)\,\left\|\Delta\de\rho\right\|_{L^2} \\
&\qquad \,+\,\Big(\left\|\nabla u_1\right\|_{L^\infty}\,\left\|\nabla^2\rho_2\right\|_{L^\infty}\,+\,
\left\|\nabla u_1\right\|_{L^\infty}^2\left(\left\|\nabla\rho_1\right\|_{L^\infty}+\left\|\nabla\rho_2\right\|_{L^\infty}\right)^2\Big)\,\left\|\de\rho\right\|_{L^2}\,.
\end{align*}
Keeping relation \eqref{est:D-Theta} in mind, we then find
\begin{align}
\label{est:DB_o}
&\left\|\de\mc B\big(\nabla u,\nabla^2\log\rho\big)\right\|_{L^2} \\
\nonumber
&\qquad\lesssim\,\sqrt{\Theta}\,\Big(1+\left\|\nabla\rho_2\right\|_{L^{\infty}}^4+
\left\|\nabla u_1\right\|_{L^\infty}\left\|\nabla^2\rho_2\right\|_{L^{\infty}}+\left\|\nabla u_1\right\|_{L^\infty}^4+\left\|\nabla\rho_1\right\|_{L^\infty}^4+
\left\|u_2\right\|_{L^\infty}^4+\left\|\nabla u_2\right\|_{L^\infty}^4\Big)\,.
\end{align}

Next, we estimate
\begin{align*}
\left\|\nabla^\perp\left(\frac{\de\rho}{\rho_1\,\rho_2}\right)\cdot\nabla\left(\pi_2\,-\,\rho_2\,\o_2\right)\right\|_{L^2}\,&\lesssim\,
\left\|\nabla\left(\pi_2\,-\,\rho_2\,\o_2\right)\right\|_{L^\infty}\,\Big(
\left\|\de\rho\right\|_{L^2}\,\left(\left\|\nabla\rho_1\right\|_{L^\infty}+\left\|\nabla\rho_2\right\|_{L^\infty}\right)\,+\,\left\|\nabla\de\rho\right\|_{L^2}\Big) \\
&\lesssim\,\left\|\de\rho\right\|_{L^2}\,\Big(1+\left\|\nabla\left(\pi_2\,-\,\rho_2\,\o_2\right)\right\|_{L^\infty}^2+
\left\|\nabla\rho_1\right\|_{L^\infty}^2+\left\|\nabla\rho_2\right\|^2_{L^\infty}\Big)\,+\,\left\|\Delta\de\rho\right\|_{L^2}\,.
\end{align*}
Dealing with the last term on the right as je have done just when estimating $\de\mc B$, this inequality in turn implies
\begin{align}
\label{est:pi_2-o}
\left\|\nabla^\perp\left(\frac{\de\rho}{\rho_1\,\rho_2}\right)\cdot\nabla\left(\pi_2\,-\,\rho_2\,\o_2\right)\right\|_{L^2}\,&\lesssim\,\sqrt{\Theta}\,
\Big(1+\left\|\nabla\left(\pi_2\,-\,\rho_2\,\o_2\right)\right\|_{L^\infty}^2+
\left\|\nabla\rho_1\right\|_{L^\infty}^2+\left\|\nabla\rho_2\right\|^2_{L^\infty}+\left\|u_2\right\|_{W^{1,\infty}}^2\Big)\,.
\end{align}

Arguing similarly, we also get
\begin{align*}
\left\|\left(\de u\,-\,\frac{1}{\rho_1}\,\nabla^\perp\de\rho\,+\,\frac{\de\rho}{\rho_1\,\rho_2}\,\nabla^\perp\rho_2\right)\cdot\nabla\o_2\right\|_{L^2}\,&\lesssim\,
\left\|\nabla\o_2\right\|_{L^q}\,\Big(\left\|\de u\right\|_{L^p}\,+\,\left\|\nabla\de\rho\right\|_{L^p}\,+\,
\left\|\de\rho\right\|_{L^p}\,\left\|\nabla\rho_2\right\|_{L^\infty}\Big)\,.
\end{align*}
Now, we use Gagliardo-Nirenberg inequality and interpolation between Sobolev spaces to bound
\begin{align*}
 \left\|\de\rho\right\|_{L^p}\,&\lesssim\,\left\|\de\rho\right\|_{L^2}^{2/p}\;\left\|\nabla\de\rho\right\|_{L^2}^{1-2/p}\;\lesssim\;
\left\|\de\rho\right\|_{L^2}^{1/2+1/p}\;\left\|\Delta\de\rho\right\|_{L^2}^{1/2-1/p} 
\\
 \left\|\nabla\de\rho\right\|_{L^p}\,&\lesssim\,\left\|\nabla\de\rho\right\|_{L^2}^{2/p}\;\left\|\Delta\de\rho\right\|_{L^2}^{1-2/p}\;\lesssim\;
\left\|\de\rho\right\|_{L^2}^{1/p}\;\left\|\Delta\de\rho\right\|_{L^2}^{1-1/p}\,.
\end{align*}
From the previous estimate and \eqref{est:D-Theta} again, we gather
\begin{align*}
\left\|\left(\de u-\frac{1}{\rho_1}\nabla^\perp\de\rho+\frac{\de\rho}{\rho_1\,\rho_2}\nabla^\perp\rho_2\right)\cdot\nabla\o_2\right\|_{L^2}\,&\lesssim\,
\sqrt{\Theta}\Big(1+\left\|\nabla\o_2\right\|_{L^q}^p+\left(\left\|\nabla\o_2\right\|_{L^q}\left\|\nabla\rho_2\right\|_{L^\infty}\right)^{\frac{2p}{p+2}}\Big)+
\left\|\Delta\de\rho\right\|_{L^2} \\
&\lesssim\,
\sqrt{\Theta}\,\Big(1+\left\|\nabla\o_2\right\|_{L^q}^p+\left\|\nabla\rho_2\right\|_{L^\infty}^{p}+
\left\|u_2\right\|_{W^{1,\infty}}^2+\left\|\nabla\rho_1\right\|_{L^\infty}\Big)\,. \nonumber
\end{align*}
where we have 
used a Young inequality in passing from the first line to the second one, combined with the fact that $2p/(p+2)\leq p/2$ for $p\geq 2$.
Using \eqref{def:p}, we finally get
\begin{align}
\label{est:diff-transp-o}
&\left\|\left(\de u-\frac{1}{\rho_1}\nabla^\perp\de\rho+\frac{\de\rho}{\rho_1\,\rho_2}\nabla^\perp\rho_2\right)\cdot\nabla\o_2\right\|_{L^2} \\
&\qquad\qquad\qquad\qquad\qquad\qquad\qquad
\lesssim\,
\sqrt{\Theta}\,\Big(1+\left\|\nabla\o_2\right\|_{L^q}^{2q/(q-2)}+\left\|\nabla\rho_2\right\|_{L^\infty}^{2q/(q-2)}+
\left\|u_2\right\|_{W^{1,\infty}}^2+\left\|\nabla\rho_1\right\|_{L^\infty}\Big)\,. \nonumber
\end{align}

It remains us to bound the last term on the right of equation \eqref{eq:diff-o}. Of course, we have
\[
\left\|\nabla^\perp\left(\frac{1}{\rho_1}\right)\cdot\nabla\left(\de\pi\,-\,\rho_1\,\de\o\,-\,\de\rho\,\o_2\right)\right\|_{L^2}\,\lesssim\,
\left\|\nabla\rho_1\right\|_{L^\infty}\,\left\|\nabla\left(\de\pi\,-\,\rho_1\,\de\o\,-\,\de\rho\,\o_2\right)\right\|_{L^2}\,,
\]
so from now on we will focus on the bound of the last $L^2$ norm. For this, we observe that, contrarily to what it would seem natural, we cannot rely
on equation \eqref{eq:p-rho_vort} for the pressure, because the difference of that equation for $\pi_1-\rho_1\o_1$ with that for $\pi_2-\rho_2\o_2$
would make dangerous $\nabla\de\pi$ and $\de\o$ terms appear on the right.
Instead, we are going to work on relation \eqref{eq:ell-p_2}, which looks better suited for performing $L^2$ estimates.

Observe that, as done in Paragraph \ref{sss:limit}, the last term in \eqref{eq:ell-p_2} can be written as
\[
-\Delta\o\,=\,-\nabla\cdot\left(\frac{1}{\rho}\,\rho\nabla\o\right)\,=\,-\nabla\cdot\left(\frac{1}{\rho}\,\nabla\big(\rho\o\big)\right)\,+\,
\nabla\cdot\left(\frac{1}{\rho}\,\o\nabla\rho\right)\,.
\]
This relation allows us to recast \eqref{eq:ell-p_2} in the following form:
\begin{equation} \label{eq:ell_p-o_2}
 -\nabla\cdot\left(\frac{1}{\rho}\,\nabla\big(\pi\,-\,\rho\,\o\big)\right)\,=\,\nabla\cdot\Big((u\cdot\nabla)u\,+\,(\nabla\log\rho\cdot\nabla)u^\perp\Big)\,+\
\nabla\cdot\left(\frac{1}{\rho}\,\o\nabla\rho\right) \,.
\end{equation}
The avantage of this formula is that the last term on its right-hand side still involves an $\o$, but,
in view of Lemma \ref{l:laxmilgram}, this term will be of lower order in the $L^2$ estimates.

From \eqref{eq:ell_p-o_2}, we immediately deduce that
\begin{align} \label{eq:diff_pi-o}
-\nabla\cdot\left(\frac{1}{\rho_1}\,\nabla\big(\de\pi\,-\,\rho_1\,\de\o\,-\,\de\rho\,\o_2\big)\right)\,&=\,
\nabla\cdot\left(\frac{\de\rho}{\rho_1\,\rho_2}\,\nabla\big(\pi_2\,-\,\rho_2\,\o_2\big)\right)\,+\,\nabla\cdot\de \Gamma \\
\nonumber &\qquad\qquad\qquad\qquad\qquad
\,+\,
\nabla\cdot\left(\frac{1}{\rho_1}\,\o_1\,\nabla\de\rho\,+\,\frac{1}{\rho_1}\,\de\o\,\nabla\rho_2\,-\,\frac{\de\rho}{\rho_1\,\rho_2}\,\o_2\,\nabla\rho_2\right)\,,
\end{align}
where we have set
\begin{align*}
\de \Gamma\,&:=\,(u_1\cdot\nabla)u_1\,+\,(\nabla\log\rho_1\cdot\nabla)u_1^\perp\,-\,(u_2\cdot\nabla)u_2\,+\,(\nabla\log\rho_2\cdot\nabla)u_2^\perp \\
&\,=\,(u_1\cdot\nabla)\de u\,+\,(\de u\cdot\nabla) u_2\,+\,(\nabla\log\rho_1\cdot\nabla)\de u\,+\,\big((\nabla\log\rho_1-\nabla\log\rho_2)\cdot\nabla\big)u_2\,,
\end{align*}
from which we deduce the estimate
\begin{align*}
\left\|\de \Gamma\right\|_{L^2}\,&\lesssim\,\left\|\de\o\right\|_{L^2}\Big(\left\|u_1\right\|_{L^\infty}+\left\|\nabla\rho_1\right\|_{L^\infty}\Big)\,+\,
\left\|\de u\right\|_{L^2}\,\left\|\nabla u_2\right\|_{L^\infty}\,+\,\left\|\nabla u_2\right\|_{L^\infty}
\left\|\nabla\rho_2\right\|_{L^\infty}\,\left\|\de\rho\right\|_{L^2}\,+\,\left\|\nabla u_2\right\|_{L^\infty}\,\left\|\nabla\de\rho\right\|_{L^2} \\
&\lesssim\,\sqrt{\Theta}\,\Big(1+\left\|u_1\right\|_{L^\infty}+\left\|\nabla\rho_1\right\|_{L^\infty}+\left\|\nabla u_2\right\|_{L^\infty}^2+
\left\|\nabla\rho_2\right\|_{L^\infty}^2\Big)\,+\,\left\|\Delta\de\rho\right\|_{L^2}\,.
\end{align*}
Similarly, we also have
\begin{align*}
\left\|\frac{\de\rho}{\rho_1\,\rho_2}\,\nabla\big(\pi_2\,-\,\rho_2\,\o_2\big)\right\|_{L^2}\,&\lesssim\,\left\|\de\rho\right\|_{L^2}\,
\left\|\nabla\big(\pi_2\,-\,\rho_2\,\o_2\big)\right\|_{L^\infty} \\
\left\|\frac{1}{\rho_1}\,\o_1\,\nabla\de\rho\right\|_{L^2}\,&\lesssim\,\left\|\o_1\right\|_{L^\infty}\,\left\|\nabla\de\rho\right\|_{L^2}\;\lesssim\;
\left\|\nabla u_1\right\|_{L^\infty}^2\,\left\|\de\rho\right\|_{L^2}\,+\,\left\|\Delta\de\rho\right\|_{L^2}
\\
\left\|\frac{1}{\rho_1}\,\de\o\,\nabla\rho_2\,-\,\frac{\de\rho}{\rho_1\,\rho_2}\,\o_2\,\nabla\rho_2\right\|_{L^2}\,&\lesssim\,
\left\|\de\o\right\|_{L^2}\,\left\|\nabla\rho_2\right\|_{L^\infty}\,+\,\left\|\de\rho\right\|_{L^2}\,\left\|\nabla\rho_2\right\|_{L^\infty}\,
\left\|\nabla u_2\right\|_{L^\infty}\,.
\end{align*}
Therefore, applying the estimates of Lemma \ref{l:laxmilgram} to equation \eqref{eq:diff_pi-o} and using \eqref{est:D-Theta} again, we get
\begin{align}
\label{est:diff_p-o}
&\left\|\nabla\left(\de\pi\,-\,\rho_1\,\de\o\,-\,\de\rho\,\o_2\right)\right\|_{L^2}\, \\
&\qquad\qquad\qquad\qquad\qquad
\,\lesssim\,\sqrt{\Theta}\,\Big(1+\left\|u_1\right\|^2_{W^{1,\infty}}+\left\|\nabla\rho_1\right\|_{L^\infty}+\left\|u_2\right\|_{W^{1,\infty}}^2+
\left\|\nabla\rho_2\right\|_{L^\infty}^2+\left\|\nabla\big(\pi_2\,-\,\rho_2\,\o_2\big)\right\|_{L^\infty}\Big)\,. \nonumber
\end{align}

At this point, we can come back to equation \eqref{eq:diff-o} for $\de\o$ and perform a $L^2$ estimate: summing up inequalities \eqref{est:DB_o},
\eqref{est:pi_2-o}, \eqref{est:diff-transp-o} and \eqref{est:diff_p-o}, we find
\begin{align} \label{est:diff-o_final}
\forall\,t\in[0,T]\,,\qquad\qquad \left\|\de\o(t)\right\|_{L^2}^2\,&\lesssim\,\left\|\de\o_0\right\|_{L^2}^2\,+\,
\int^t_0\Theta\,\Big(\wtilde\Psi+
\left\|\nabla\o_2\right\|_{L^q}^{2q/(q-2)}+\left\|\nabla\rho_2\right\|_{L^\infty}^{2q/(q-2)}\Big)\,\dd\t\,,
\end{align}
where we have defined
\[
\wtilde\Psi(t)\,:=\,1+\left\|\nabla\rho_2\right\|_{L^{\infty}}^4+\left\|u_1\right\|_{W^{1,\infty}}^4+\left\|\nabla\rho_1\right\|_{L^\infty}^4+
\left\|u_2\right\|_{W^{1,\infty}}^4+\left\|\nabla\left(\pi_2\,-\,\rho_2\,\o_2\right)\right\|_{L^\infty}^2+\left\|\nabla u_1\right\|_{L^\infty}\left\|\nabla^2\rho_2\right\|_{L^{\infty}}\,.
\]

Finally, putting estimates \eqref{est:diff-rho}, \eqref{est:diff-u}, \eqref{est:diff-theta} and \eqref{est:diff-o_final} together,
we discover an estimate for the function $\Theta$: we have
\begin{align} \label{est:THETA}
\forall\,t\in[0,T]\,,\qquad\qquad
\Theta(t)\,\lesssim\,\Theta_0\,+\int^t_0\Big(1\,+\,I(\t)\Big)\,\Theta(\t)\,\dd\t\,,
\end{align}
where $\Theta_0\,=\,\Theta(0)$ is defined as $\Theta$, but replacing the different functions appearing therein with the corresponding initial data,
and where we have defined
\begin{align*}
I(t)\,&:=\,\left\|\nabla\rho_2\right\|_{L^{\infty}}^4+\left\|u_1\right\|_{W^{1,\infty}}^4+\left\|\nabla\rho_1\right\|_{L^\infty}^4+
\left\|u_2\right\|_{W^{1,\infty}}^4+ \left\|\nabla\left(\pi_2\,-\,\rho_2\,\o_2\right)\right\|_{L^\infty}^2 \\
&\qquad\qquad
+\left\|\nabla\pi_2\right\|_{L^q}+
\left\|\Delta u_2\right\|_{L^q}^{2q/(q-2)}+\left\|\nabla\rho_2\right\|_{L^\infty}^{2q/(q-2)} + 
\Big(\left\|\nabla u_1\right\|_{L^\infty}+\left\|u_2\right\|_{L^\infty}\Big)\,\left\|\nabla^2\rho_2\right\|_{L^{\infty}}+
\left\|\nabla\Delta\rho_2\right\|_{L^{q}}\,.
\end{align*}

Since, by assumption, we have $I\in L^1\big([0,T]\big)$, an application of the Gr\"onwall lemma gives us
\[
\sup_{t\in[0,T]}\Theta(t)\,\lesssim\,\Theta_0\,\exp\left(C\int^T_0\Big(1\,+\,I(\t)\Big)\,\dd\t\right)\,.
\]
Coming back to \eqref{est:THETA} and using both inequalities \eqref{est:D-Theta} and \eqref{est:Theta-D}, we see that
an analogous estimate holds true also for the function $\mc D$. Notice that here we have made use of assumption (ii) of the statement.
This concludes the proof of the stablity estimates.
\end{proof}

\subsubsection{Uniqueness of solutions to the odd system} \label{sss:uniq-odd}

To conclude, let us show how to deduce the uniqueness statement of Theorem \ref{teo1} from the stability estimates of Proposition \ref{p:stability}.

\begin{proof}[Proof of uniqueness in Theorem \ref{teo1}]
Let $\big(\rho_0,u_0\big)\in L^\infty(\R^2)\times H^s(\R^2)$ satisfy the assumptions of Theorem \ref{teo1}. Let $\big(\rho_1,u_1,\nabla\pi_1\big)$
and $\big(\rho_2,u_2,\nabla\pi_2\big)$ be two corresponding solutions, verifying the regularity conditions stated in the same theorem.
Let $T>0$ be such that both solutions are defined up to time $T$ (namely, $T$ is the minimum of the two lifespans).

As $s>2$, by Sobolev embeddings we immediately see that the assumptions formulated in items (i), (ii), (v) and (vi) of Proposition \ref{p:stability}
are matched, together with the second condition (the one involving $\nabla^2\rho_2$) appearing in item (vii). On the other hand,
the analysis of Paragraph \ref{sss:omega} (see also Paragraph \ref{sss:time_reg}) shows that, for $j\in\{1,2\}$, the quantity
$\nabla\big(\pi_j\,-\,\rho_j\,\o_j\big)$ belongs to $C_T\big(H^{s-1}\big)$, hence to $L^\infty\big([0,T]\times\R^2\big)$ by Sobolev embeddings.
So, also the conditions of item (vii) are completely satisfied.

It remains us to verify the assumptions of items (iii), (iv) and (viii) of Proposition \ref{p:stability}. Let us focus on (iv) for a while:
it is easy to check that the right-hand side of equation \eqref{eq:theta} for $\theta_j$ belongs to $C_T\big(H^{s-1}\big)$, whereas
$u_j\cdot\nabla\theta_j$ is in $C_T\big(H^{s-2}\big)$. Since $s>2$, this implies that $\d_t\theta_j$ belongs to $C_T\big(L^2\big)$,
whence $\theta_j\in C_T^1\big(L^2\big)$, as sought.

On the other hand, the equation for $\rho_j$ implies that $\d_t\rho_j$ belongs to $C_T\big(H^s\big)\hookrightarrow C_T\big(H^2\big)$. This immediately
implies that each $\rho_j$ belongs to $C^1_T\big(H^2\big)$, and so does the difference $\rho_1-\rho_2$.
Hence, in order to verify the assumptions of item (iii), it remains us to show that $u_j$ is a $C_T^1\big(H^1\big)$ function. We proceed similarly
as done in Paragraph \ref{sss:time_reg}. First of all, as $s>2$ and $\nabla\pi_j\in C_T\big(H^{s-2}\big)$, from the momentum equation \eqref{eq:mom_new}
we easily gather $\d_tu_j\in C_T\big(L^2\big)$, so $u_j\in C^1_T\big(L^2\big)$. Next, from equation \eqref{eq:vort-2} we deduce that
$\d_t\o_j$ belongs to $C_T\big(L^2\big)$ as well, hence $\o_j\in C^1_T\big(L^2\big)$. Finally, by using decomposition \eqref{eq:u-decomp} again,
the previous properties imply that $u_j\in C^1_T\big(H^1\big)$ for all $j=1,2$, whence the difference $u_1-u_2$ belongs to the same space,
as sought.

Finally, we consider the requirements appearing in item (viii) of Proposition \ref{p:stability}.
From our existence theory, we know that
all the quantities appearing therein, namely $\nabla\pi_j$, $\nabla\Delta\rho_j$ and $\nabla^2u_j$ (for $j=1,2$), belong to
$L^\infty_T\big(H^{s-2}\big)$. Now, three cases may happen:
\begin{itemize}
 \item in the case $s>3$, by Sobolev embedding we can verify the conditions of item (viii) by taking $q=+\infty$;
 \item if $s=3$, we can take instead any $q\in\,]2,+\infty[\,$, say $q=4$ to fix ideas;
 \item if $2<s<3$, instead, by Sobolev embeddings we get that $H^{s-2}\hookrightarrow L^q$ with $q=2/(3-s)$.
\end{itemize}

Therefore, all the assumptions of Proposition \ref{p:stability} are matched, hence the estimates claimed therein apply. We thus deduce that
$\rho_1-\rho_2$ is zero in $C_T\big(H^2\big)$ and $u_1-u_2$ is zero in $C_T(H^1)$, so in particular
\[
\rho_1\,\equiv\rho_2\qquad \mbox{ and }\qquad u_1\equiv u_2\qquad\qquad \mbox{ everywhere in }\; [0,T]\times\R^2\,,
\]
because they are continuous functions which must coincide almost everywhere.
In particular, also $\o_1\equiv\o_2$ and $\theta_1\equiv\theta_2$ everywhere in $[0,T]\times \R^2$.

This having been established, we can go back to equation \eqref{eq:ell-p_2} to deduce that also $\nabla\pi_1\equiv\nabla\pi_2$ as
$C_T\big(H^{s-2}\big)$ functions, thus almost everywhere in $[0,T]\times \R^2$.

The uniqueness is then proved.
\end{proof}


{\small
\section*{Acknowledgements} \addcontentsline{toc}{section}{Acknowledgements}

%
The work of the first author has been partially supported by the LABEX MILYON (ANR-10-LABX-0070) of Universit\'e de Lyon, within the program ``Investissement d'Avenir''
(ANR-11-IDEX-0007), and by the projects BORDS (ANR-16-CE40-0027-01), SingFlows (ANR-18-CE40-0027) and CRISIS (ANR-20-CE40-0020-01), all operated by the French National Research Agency (ANR).

The work of the second author has been partially supported by the project ``Mathematical Analysis of Fluids and Applications'' with reference PID2019-109348GA-I00/AEI/ 10.13039/501100011033 and acronym ``MAFyA'' funded by Agencia Estatal de Investigaci\'on and the Ministerio de Ciencia, Innovacion y Universidades (MICIU). Project supported by a 2021 Leonardo Grant for Researchers and Cultural Creators, BBVA Foundation. The BBVA Foundation accepts no responsability for the opinions, statements and contents included in the project and/or the results thereof, which are entirely the responsability of the authors.

The research of the third author is supported by the ERC through the Starting Grant project H2020-EU.1.1.-639227.

}

\end{document}